\documentclass[letterpaper]{article}
\usepackage{arxiv}
\usepackage[T1]{fontenc}
\usepackage{amsmath,amssymb,amsthm,mathtools}
\usepackage{graphicx,booktabs,tabularx,array,multirow}
\usepackage{enumitem,xcolor,microtype,float,placeins}
\usepackage[authoryear,round]{natbib}
\newcommand{\todo}[2][]{}

\newcommand{\E}{\mathbb E}
\newcommand{\Prob}{\mathbb P}

\newcommand{\Gap}{\operatorname{Gap}}
\newcommand{\CompGap}{\Gap_{R_0}^{\mathrm{comp}}}

\newtheorem{theorem}{Theorem}[section]
\newtheorem{lemma}[theorem]{Lemma}
\newtheorem{proposition}[theorem]{Proposition}
\newtheorem{corollary}[theorem]{Corollary}
\newtheorem{assumption}[theorem]{Assumption}
\theoremstyle{definition}
\newtheorem{definition}[theorem]{Definition}
\theoremstyle{remark}
\newtheorem{remark}[theorem]{Remark}

\floatstyle{ruled}
\newfloat{algorithm}{tbp}{loa}
\floatname{algorithm}{Algorithm}
\allowdisplaybreaks[2]
\newcommand{\InfoDepth}{D_{\mathrm{info}}}
\newcommand{\SchedDepth}{D_{\mathrm{sched}}}
\newcommand{\BinDepth}{D_{\mathrm{bin}}}
\newcommand{\CompSeq}[1]{r_{\mathrm{seq}}(#1)}
\newcommand{\CompMini}[1]{r_{\mathrm{mb}}(#1)}

\providecommand{\ind}{\mathbf 1}

\usepackage{xurl}
\definecolor{citeblue}{HTML}{1F4E79}
\definecolor{crossrefred}{HTML}{A32638}
\definecolor{revisionpink}{HTML}{C2185B}
\newif\ifshowrevisions
\showrevisionsfalse
\newcommand{\revision}[1]{%
  \ifshowrevisions\textcolor{revisionpink}{#1}\else#1\fi}
\newenvironment{revisionblock}
  {\ifshowrevisions\color{revisionpink}\fi}
  {}
\usepackage{hyperref}
\hypersetup{
  colorlinks=true,
  citecolor=citeblue,
  urlcolor=citeblue,
  filecolor=citeblue,
  linkcolor=crossrefred,
  menucolor=crossrefred,
  linktoc=all,
  breaklinks=true,
  bookmarksnumbered=true,
  pdftitle={Optimal Rates for Federated Variational Inequalities under Different Relaxations of the Bounded Variance Assumption},
  pdfauthor={Kirill Novoselov, Yuriy Dorn, Darina Dvinskikh, Alexander Gasnikov}
}
\ExplSyntaxOn
\cs_new_eq:NN \AppendixLetter \int_to_Alph:n
\ExplSyntaxOff

\title{Optimal Rates for Federated Variational Inequalities \\
under different relaxations of bounded variance assumption}
\author{%
  Kirill Novoselov \\
  Innopolis University \\
  \texttt{k.novoselov@innopolis.university}
  \And
  Yuriy Dorn \\
  Lomonosov Moscow State University 
  \AND
  Darina Dvinskikh \\
  HSE University \\
  \texttt{dmdvinskikh@hse.ru}
  \And
  Alexander Gasnikov \\
  Innopolis University, Trusted AI Research Center, RAS, \\
  Steklov Mathematical Institute, RAS \\
  \texttt{gasnikov@yandex.ru}
}
\date{}
\renewcommand{\shorttitle}{Federated Variational Inequalities}
\renewcommand{\headeright}{Preprint}
\renewcommand{\undertitle}{Preprint}
\begin{document}
\maketitle
\begin{abstract}
We study how interpolation affects local steps and parallel sampling in monotone variational inequalities (VIs), with $m$ workers, $K$ oracle calls per worker per round, and $N$ rounds.
Under the strong growth condition (SGC), we determine the minimax rate for affine VIs, attained by local extragradient with worker averaging.
For nonlinear monotone VIs and convex--concave saddle problems with Lipschitz operators satisfying SGC, we improve the stochastic terms of bounded-variance expected-gap guarantees from $(KN)^{-1/2}$ to $(KN)^{-1}$ with local steps and from $(mKN)^{-1/2}$ to $(mKN)^{-1}$ with minibatches.
Under strong monotonicity, we improve the noise-dependent condition-number scaling of existing extragradient bounds from quadratic to linear.
The resulting monotone-gap and strongly monotone accuracy bounds match lower bounds up to logarithmic factors in $m$.
Under generalized smoothness and SGC, adaptive steps give polynomial dependence on the initial radius.
For affine VIs with globally bounded oracle variance, we also attain the minimax rate: the deterministic rate over $KN$ sequential calls plus the statistical rate over all $mKN$ samples.
For composite VIs with smooth convex sample interpolation and an exactly evaluated monotone Lipschitz component, accelerated mirror-prox sliding achieves an $N^{-2}$ gradient term, a $(KN)^{-1}$ monotone term, and a $(mKN)^{-1}$ sampling term, without assuming SGC.
\end{abstract}

\section{Introduction}\label{sec:intro}

Large-scale variational inequalities (VIs) arise in convex--concave
optimization, monotone games, and adversarial and robust learning. When
such problems are solved on several machines, stochastic-oracle calls
are often much cheaper than communication. We study the standard
intermittent-communication model: $m$ workers perform at most $K$ local
oracle calls between two consecutive synchronizations, over $N$
communication rounds.

The budget can be spent in two ways. A worker can make its $K$ calls
sequentially, choosing each query from its earlier responses; over $N$
rounds this gives a chain of $KN$ sequential calls. Alternatively, all
workers can query a common point and average their responses, which gives
only $N$ sequential updates, each with up to $mK$ samples. Local
computation uses sequential depth, and minibatching uses the total number
$mKN$ of samples.

Under uniformly bounded variance, sequential progress controls the
deterministic part of the error and the number of samples controls the
statistical part; for convex optimization, this trade-off between local
steps and minibatching has been studied in
detail~\citep{woodworth2020,woodworth2021}. In the interpolation regime, however, the
noise decreases together with the operator, and it is not clear which
resource matters. We ask whether local steps can turn the depth $KN$ into
faster progress without wasting the samples of the other workers, when
common-query minibatches are preferable, and which limitations come from
the communication pattern itself rather than from a particular algorithm.

We consider unbiased stochastic oracles under three noise models. Under
the strong growth condition (SGC) with constant $\rho\ge1$, the second
moment of the oracle at every query is at most $\rho$ times the squared
norm of its mean, so every sample vanishes at a solution. The mixed model
adds an absolute noise level $\sigma\ge0$: the variance is at most
$(\rho-1)$ times the squared norm of the mean plus $\sigma^2$, and
$\sigma=0$ gives SGC. The root-noise model controls the oracle in a
solution, where its second moment is $\sigma_\star^2$; $\sigma_\star=0$ is
interpolation.

We study these questions for affine, nonlinear monotone, strongly
monotone, generalized-smooth, and composite operators. Under structured
noise, local computation can be more useful than a bounded-variance
analysis suggests, but its gain is limited by the longest chain of
informative calls that a worker can form between two communications.
Under SGC we determine the minimax rate for affine operators up to
universal constants. For nonlinear Lipschitz monotone VIs the upper and
lower bounds match up to logarithmic factors in $m$, and for composite VIs
we identify a regime with matching bounds.

\paragraph{Informative depth.}
The hard oracle in our SGC lower bounds hides each exact response
independently and reveals a rescaled response with probability $1/\rho$.
Within a round, informative calls can be chained only on one worker, so
after communication the progress is that of the longest chain over the
workers, not the sum over the workers. Combined with exact-oracle hard
instances, this gives our SGC lower bounds; the composite lower bound uses
the same construction.

\paragraph{Contributions.}

\begin{itemize}[leftmargin=*,nosep]

\item \emph{Sharp rates under SGC.}
We show that the relevant resource is the longest chain of informative
local oracle calls. This gives the affine minimax rate up to constants
and matching upper and lower bounds up to logarithmic factors for
nonlinear monotone VIs and convex--concave saddle problems, replacing
square-root stochastic terms by inverse-linear ones in the number of
sequential local calls or of common-query samples.

\item \emph{Mixed noise and generalized smoothness.}
We give monotone gap rates and strongly monotone distance rates whose
upper and lower bounds match up to logarithmic factors; in the latter, the
relative-noise penalty in the exponent is linear rather than quadratic in
the condition number. For affine VIs,
the guarantee combines the better deterministic rate of local updates
and common-query minibatches with the statistical accuracy of all $mKN$
samples. Beyond global
Lipschitz continuity, an
adaptive restarted method gives a strongly monotone rate under SGC
with polynomial dependence on the initial radius.

\item \emph{Composite VI without SGC.}
For a sampled smooth convex term plus an exactly evaluated monotone
term, we combine pooled minibatches with local solves in accelerated
mirror-prox sliding. Local computation improves the monotone term by a
factor $K$, the minibatches use all $mKN$ samples, and matching lower bounds
give optimality up to logarithmic factors when the mean smoothness is
positive and the sample smoothness is at least $K$ times as large.

\end{itemize}

\begin{table*}[!t]
\centering
\caption{Comparison With Prior Bounds in the Intermittent-Communication Model}
\label{tab:results}
\footnotesize
\setlength{\tabcolsep}{4pt}
\renewcommand{\arraystretch}{1.45}
\begin{tabular}{@{}llll@{}}
\toprule
\textbf{PROBLEM} & \textbf{NOISE} & \textbf{SOURCE} & \textbf{RATE}\\
\midrule
Monotone & Uniform & \citet{juditsky2011}$^\ddagger$
 & $O\bigl(\frac{LD^2}{T}+\frac{\sigma D}{\sqrt T}\bigr)$\\
 & Uniform & \citet{wang2026faster}
 & $O\bigl(\frac{LD^2}{\sqrt K N}+\frac{\sigma D}{\sqrt{mKN}}
   +\frac{L^{1/3}\sigma^{2/3}D^{4/3}}{K^{1/3}N^{2/3}}\bigr)$\\
 & Mixed & \textbf{Thms.~\ref{vi:main-monotone}, \ref{vi:main-mixed}}
 & $\widetilde\Theta\bigl(\min\bigl\{\frac{\rho LD^2}{KN}+\frac{\sigma D}{\sqrt{KN}},\
   \frac{LD^2}{N}+\frac{\rho LD^2}{mKN}+\frac{\sigma D}{\sqrt{mKN}}\bigr\}\bigr)$\\
\midrule
Affine & Uniform & \citet{wang2026faster}
 & $O\bigl(\frac{LR_0^2}{KN}+\frac{\sigma R_0}{\sqrt{mKN}}
   +\frac{L^{1/3}\sigma^{2/3}R_0^{4/3}}{K^{1/3}N^{2/3}}\bigr)$\\
 & SGC & \textbf{Thm.~\ref{thm:affine-integrated}}
 & $\Theta\bigl(\max\bigl\{\frac{LR_0^2}{KN},\ \frac{\rho LR_0^2}{mKN},\
   \frac{LR_0^2}{N(K/\rho+\log(2m)/\log(2+\rho\log(2m)/K))}\bigr\}\bigr)$\\
 & Mixed & \textbf{Thm.~\ref{thm:affine-additive-main}}
 & $\widetilde\Theta\bigl(\min\bigl\{\frac{\rho LR_0^2}{KN},\
   \frac{LR_0^2}{N}+\frac{\rho LR_0^2}{mKN}\bigr\}+\frac{\sigma R_0}{\sqrt{mKN}}\bigr)$\\
\midrule
Str.\ monotone & Uniform & \citet{beznosikov2024lower}
 & $O\bigl(R_0^2e^{-cN/\kappa}+\frac{\sigma^2}{\mu^2mKN}\bigr)$\\
 & Mixed & \textbf{Thm.~\ref{vi:main-mixed-accuracy}}
 & \begin{tabular}[t]{@{}l@{}}$\widetilde\Theta\bigl(R_0^2e^{-cN\max\{K/\rho,\,\min\{1,mK/\rho\}\}/\kappa}$\\
   $\quad+\frac{\sigma^2}{\mu^2mKN}\bigl(1+m\,e^{-cN\min\{1,mK/\rho\}/\kappa}\bigr)\bigr)$\end{tabular}\\
\midrule
Composite & Uniform & \citet{chen2017amp}$^\ddagger$
 & $O\bigl(\frac{L_fR_0^2}{T^2}+\frac{L_HR_0^2}{T}+\frac{\sigma R_0}{\sqrt T}\bigr)$\\
 & Interpolation & \textbf{Thm.~\ref{thm:composite-local-main}}
 & \begin{tabular}[t]{@{}l@{}}$\widetilde O\bigl(\min\bigl\{\frac{L_fR_0^2}{(KN)^2}
   +\frac{(L_H+L_{\max})R_0^2}{KN}+\frac{\sigma_\star R_0}{\sqrt{KN}},$\\
   $\quad\frac{L_fR_0^2}{N^2}+\frac{L_HR_0^2}{KN}
   +\min\bigl\{\frac{L_{\max}R_0^2}{mKN},R_0^2\sqrt{\frac{L_{\max}L_f}{mKN}}\bigr\}
   +\frac{\sigma_\star R_0}{\sqrt{mKN}}\bigr\}\bigr)$;\\
   $\widetilde\Theta$ of the second entry if $L_{\max}\ge KL_f>0$, $N\ge2$
   (Cor.~\ref{cor:composite-matching-main})\end{tabular}\\
\bottomrule
\end{tabular}
\end{table*}

\paragraph{Comparison with prior work.}
Table~\ref{tab:results} compares the bounds under the same $(m,K,N)$
budget. For nonlinear VIs, $D>0$ bounds the distance from the initial
point to every comparison point over which the Minty gap is maximized;
it does not restrict the algorithm's queries or output. For affine and
composite VIs, $R_0>0$ bounds the initial distance to a root. Monotone
and affine rows use the Minty gap; the cited results specify when the
same rate holds for the primal--dual gap. Strongly monotone rows bound
the expected squared distance $\E\|\widehat z-z_\star\|^2$. Here $L>0$ is
the operator Lipschitz constant, $\mu>0$ is the strong-monotonicity
constant, and $\kappa=L/\mu$. In composite rows, $L_f$ is the smoothness
constant of the mean gradient, $L_H$ is the Lipschitz constant of the
exact monotone component, and $L_{\max}$ is the uniform smoothness
constant of the sample gradients; $\sigma_\star$ is the root-noise level. 
\textsuperscript{$\ddagger$} indicates that the corresponding source uses
a single machine with $T$ sequential iterations.
Throughout, $\widetilde{O}$ and $\widetilde{\Theta}$ hide logarithmic factors in the relevant problem parameters; in particular, $\widetilde{\Theta}(f)$ denotes matching upper and lower bounds up to logarithmic factors.

The comparison shows where the structural assumptions change the rate:
for composite VIs, local work
reduces the exact monotone term to $L_H/(KN)$. When
$L_{\max}\ge KL_f>0$, Corollary~\ref{cor:composite-matching-main} shows
that choosing between the initial point and the parallel-minibatch method
is minimax-optimal up to logarithmic factors.

\paragraph{Related work.}
\revision{Mirror-prox and extragradient are standard for stochastic
VIs~\citep{korpelevich1976,nemirovski2004,juditsky2011}. For distributed VIs, prior work analyzes heterogeneous local extragradient,
similarity and compression, and near-optimal saddle methods
\citep{beznosikov2022local,beznosikov2023pillars,beznosikov2024lower}.
\citet{luo2026distributed} prove communication optimality within
gradient-span algorithms in a complementary deterministic distributed
saddle-point model, whereas our data-parallel lower bounds cover all randomized
algorithms in our model. \citet{wang2026faster} improve upper bounds under
bounded variance; we study the trade-off between local computation and
minibatching under structured noise and prove communication lower
bounds. }

In minimization, the strong
growth condition models interpolation by over-parameterized models,
under which stochastic gradient methods attain the rates of their
deterministic counterparts~\citep{schmidt2013sgc,vaswani2019}. For VIs,
relative noise yields $O(1/T)$ rates for cocoercive
operators~\citep{antonakopoulos2021sifting} and is also studied for
quantized optimistic methods~\citep{nguyen2025qoda}; adaptive extragradient attains $O(1/N)$ for Lipschitz root finding
under SGC and $\widetilde O(T^{-1/2})$ under affine
noise~\citep{luo2025}. We add the
dependence on $(m,K,N,\rho)$ in the intermittent-communication model,
the choice between local steps and common-query minibatches, and
matching communication lower bounds, exact for affine VIs and up to
logarithmic factors for nonlinear VIs.

\revision{Generalized-smooth VIs are studied by
\citet{vankov2024vi,vankov2024stochasticvi}, and
\citet{choudhury2025generalized} treats $(L_0,L_1)$-Lipschitz root
finding. Our bound under generalized smoothness makes the dependence on
$(m,K,N,\rho)$ explicit and chooses between local steps and
common-query minibatches.}

In the intermittent-communication model, oracle lower bounds separate
the total sample count from the available sequential
depth~\citep{woodworth2018,woodworth2021}, while local averaging for
quadratic objectives has been studied under additive and multiplicative
noise~\citep{woodworth2020,sadchikov2024local}. Under smooth-sample
interpolation, accelerated minibatching improves the sampling
term~\citep{woodworthsrebro2021,ilandarideva2025}. Our composite result
extends this phenomenon to an exactly evaluated monotone component
through accelerated mirror-prox~\citep{chen2017amp}.

\paragraph{Organization.}
\revision{Section~\ref{sec:model} gives the model and assumptions.
Sections~\ref{sec:affine-integrated}--\ref{sec:acceleration-main} treat
affine, nonlinear, and composite VIs, respectively. The supplement
contains numerical illustrations, algorithms, proofs, and auxiliary
results.}

\section{Model and assumptions}
\label{sec:model}
\paragraph{Notation.}
Throughout, $\langle\cdot,\cdot\rangle$ and $\|\cdot\|$ are Euclidean,
$\|\cdot\|_{\rm op}$ is the induced matrix norm, and logarithms are
natural. We write
for $d\in\mathbb N$, $z\in\mathbb R^d$, and $r\ge0$,
$B(z,r):=\{u\in\mathbb R^d:\|u-z\|\le r\}$ for the closed Euclidean
ball centered at $z$ with radius $r$.
The notation $O$, $\Omega$, and $\Theta$ hides universal constants; a
tilde also hides the logarithmic factors specified in the corresponding
statement.

\begin{definition}[Variational inequality]
\label{def:vi-main}
For $F:\mathbb R^d\to\mathbb R^d$, a solution $z_\star$ satisfies for all $z\in\mathbb R^d$
\[
 \langle F(z_\star),z-z_\star\rangle\ge0\;
 \;\Longleftrightarrow\;F(z_\star)=0.
\]
The algorithm starts from $z_0$ and does not know the root $z_\star$.
For the affine and composite models of
Sections~\ref{sec:affine-integrated} and~\ref{sec:acceleration-main}
we write $(G,x_0,x_\star)$ instead of $(F,z_0,z_\star)$.
\end{definition}

\begin{assumption}[Monotone Lipschitz operator]
\label{ass:monotone-lipschitz}
Let $L>0$. The map $F:\mathbb R^d\to\mathbb R^d$ has a root
$z_\star$ and satisfies, for all $z,u\in\mathbb R^d$,
\[
 \langle F(z)-F(u),z-u\rangle\ge0,
 \quad \|F(z)-F(u)\|\le L\|z-u\|.
\]
\end{assumption}

\begin{assumption}[Strong monotonicity]
\label{ass:strong-monotonicity}
For $\mu>0$, the map $F:\mathbb R^d\to\mathbb R^d$ satisfies for all $z,u\in\mathbb R^d$
\[
 \langle F(z)-F(u),z-u\rangle\ge\mu\|z-u\|^2.
\]
\end{assumption}
For a differentiable convex--concave function $\Phi(x,y)$, the operator
$F=(\nabla_x\Phi,-\nabla_y\Phi)$ is monotone and its roots are the
saddle points of $\Phi$.

\begin{definition}[Intermittent communication]
\label{def:ic-main}
Fix $m,K,N\in\mathbb N$. An algorithm runs on $m$ workers for $N$
communication rounds. In each
round, every worker makes at most $K$ sequential oracle calls; a query
may depend on the information shared at the start of the round, on the
worker's own randomness, and on its earlier responses in that round.
Workers communicate only at round boundaries. Messages, memory, and
local computation are unrestricted. Query and output maps are Borel
measurable and may take arbitrary finite values, and the randomness of
the algorithm is independent of the instance and of the oracle.

\revision{On a sample space $\Xi$,} all workers query the same oracle
$g:\mathbb R^d\times\Xi\to\mathbb R^d$. \revision{At each call, $\xi$
is drawn from a fixed law independently of all}
earlier draws and of the algorithm's randomness; the query is
measurable with respect to the information $\mathcal F$ available
before the call. Every call counts towards the budget, including both
calls of an extragradient step. Our algorithms may use the model
parameters. We call any randomized algorithm satisfying these rules
\emph{admissible}. Its output is denoted by $\widehat z$; in the affine
and composite settings, where the problem variable is $x$, the output
is denoted by $\widehat x$.
\end{definition}

\revision{Our lower bounds hold for every admissible randomized
algorithm. For a fixed algorithm, the hard instance or distribution does
not depend on the realized random seed or on the oracle draws, and it
lives in the dimension stated in the result; when one distribution works
for all algorithms, we say so.}

The three noise models below are used in different parts of the paper:
SGC and mixed noise in Sections~\ref{sec:affine-integrated}
and~\ref{vi:main}, root noise in Section~\ref{sec:acceleration-main}.

\begin{assumption}[Strong growth condition, SGC]
\label{ass:sgc}
For $\rho\ge1$, every pre-call history $\mathcal F$ and every
$\mathcal F$-measurable query $z\in\mathbb R^d$ satisfy
\begin{equation*}
 \begin{aligned}
 \mathbb E[g(z,\xi)\mid\mathcal F]&=F(z),\\
 \mathbb E[\|g(z,\xi)\|^2\mid\mathcal F]
 &\le\rho\|F(z)\|^2.
 \end{aligned}
\end{equation*}
\end{assumption}

\begin{assumption}[Mixed noise]
\label{ass:mixed}
For $\rho\ge1$ and $\sigma\ge0$, every pre-call history $\mathcal F$
and every $\mathcal F$-measurable query $z\in\mathbb R^d$ satisfy
\begin{equation*}
 \begin{aligned}
 \mathbb E[g(z,\xi)\mid\mathcal F]&=F(z),\\
 \mathbb E[\|g(z,\xi)-F(z)\|^2\mid\mathcal F]
 &\le(\rho-1)\|F(z)\|^2+\sigma^2.
 \end{aligned}
\end{equation*}
\end{assumption}

\begin{assumption}[Root noise and interpolation]
\label{ass:interpolation}
Let $F(z_\star)=0$. The oracle $g:\mathbb R^d\times\Xi\to\mathbb R^d$
satisfies
\[
 \begin{gathered}
 \mathbb E_\xi g(z,\xi)=F(z)\quad(\forall z\in\mathbb R^d),\\
 \sigma_\star^2:=\mathbb E_\xi\|g(z_\star,\xi)\|^2<\infty.
 \end{gathered}
\]
Interpolation is the case $\sigma_\star=0$, that is,
\begin{equation*}
 g(z_\star,\xi)=0\qquad\text{almost surely}.
\end{equation*}
\end{assumption}

\begin{remark}[Relations between the noise models]
\label{rem:noise-contracts-main}
Unbiasedness gives
\[
 \mathbb E[\|g(z,\xi)\|^2\mid\mathcal F]
 =\|F(z)\|^2+
   \mathbb E[\|g(z,\xi)-F(z)\|^2\mid\mathcal F].
\]
Hence SGC is the mixed condition with $\sigma=0$, and it implies
interpolation at every root. Interpolation does not imply SGC
\revision{even for smooth convex sample losses}. In the mixed model, $\sigma^2$ is
an additive budget: the variance at $z$ is at most
$(\rho-1)\|F(z)\|^2+\sigma^2$, which exceeds $\sigma^2$ away from the
roots when $\rho>1$, while the variance at a root is
$\sigma_\star^2\le\sigma^2$. For $\rho=1$, $\sigma^2$ bounds the variance at
every query and the mixed model is the bounded-variance model.
\end{remark}

\begin{definition}[Gap criteria]
\label{def:gaps-main}
For a nonempty compact set $C\subseteq\mathbb R^d$ containing a root,
the Minty gap is
\begin{equation*}
 \Gap_C^{\rm M}(z)
 :=\sup_{u\in C}\langle F(u),z-u\rangle.
\end{equation*}
For nonlinear VIs the test set satisfies
$z_\star\in C\subseteq B(z_0,D)$ with $D>0$. For a differentiable
convex--concave $\Phi(x,y)$, $F=(\nabla_x\Phi,-\nabla_y\Phi)$, and
compact sets $X,Y$ whose product contains a saddle point, the
primal--dual gap is
\[
 \Gap^{\rm PD}_{X\times Y}(x,y)
 :=\sup_{v\in Y}\Phi(x,v)-\inf_{u\in X}\Phi(u,y).
\]
We always keep the superscripts ${\rm M}$ and ${\rm PD}$. All expected-gap
bounds place the supremum inside the expectation. The test set restricts
neither queries nor outputs.
\end{definition}

\paragraph{Rate abbreviations.}
The informative depth is
\[
\InfoDepth:=\min\!\left\{
 K,\ \tfrac{mK}{\rho},\ \tfrac K\rho+
 \tfrac{\log(2m)}{\log(2+\rho\log(2m)/K)}
 \right\}.
\]
\revision{For independent $Z_1,\ldots,Z_m\sim
\operatorname{Binomial}(K,1/\rho)$,
$\E\max_{i\le m}Z_i=\Theta(\InfoDepth)$ uniformly. The three terms
reflect the per-worker round cap, the total informative-call cap, and
the binomial-maximum estimate, respectively.}
The implementable schedule depth is
\[
 \SchedDepth
 :=\max\left\{\tfrac K\rho,
          \min\left\{\tfrac{mK}{\rho},1\right\}\right\}.
\]
\revision{The two terms are the progress per round of sequential updates on
one worker and of common-query minibatches.}

\section{Affine VI}\label{sec:affine-integrated}
For affine operators we determine the minimax rate under SGC and then
show that additive noise adds a statistical term that uses all $mKN$
samples.

\begin{assumption}[Affine operator]\label{ass:affine-main}
\revision{Let $L,R_0>0$, $x_0\in\mathbb R^d$, and
$J\in\mathbb R^{d\times d}$ with $J+J^\top\succeq0$. The operator
$G:\mathbb R^d\to\mathbb R^d$ has an unknown root $x_\star$ and satisfies}
\[
 \begin{gathered}
 G(x)=J(x-x_\star),\quad \|x_0-x_\star\|\le R_0, \quad
 \|J\|_{\rm op}\le L.
\end{gathered}
\]
The initial point $x_0$ and the radius $R_0$ are known; $J$ and the root
$x_\star$ are not. The sampled maps need not be affine.
\end{assumption}

\begin{definition}[Affine minimax risks]\label{def:affine-risks-main}
Let $\mathcal E_d^{\rm aff,rel}$ and $\mathcal E_d^{\rm aff,mix}$ be the
minimax values, over the algorithms of Definition~\ref{def:ic-main}, of
the expected root-centered Minty gap $\E\Gap_\star^{\rm M}(\widehat x)$
over Assumption~\ref{ass:affine-main} combined with
Assumption~\ref{ass:sgc} and with Assumption~\ref{ass:mixed},
respectively, for $F=G$, where
\begin{equation}\label{eq:aff:self-contained-gap}
 \Gap_\star^{\rm M}(x)=\sup_{\|u-x_\star\|\le R_0}
                       \langle G(u),x-u\rangle .
\end{equation}
\end{definition}

\begin{theorem}[Affine minimax rate under SGC]
\label{thm:affine-integrated}\label{thm:pure-affine-main}\label{pv:minimax}
Under Assumptions~\ref{ass:affine-main} and~\ref{ass:sgc}, for every
$m,K,N\in\mathbb N$, $L,R_0>0$, and $\rho\ge1$, there is a finite
dimension $d_{\rm hard}=O\bigl(m(KN)^2\bigr)$ such that
\begin{equation*}
 \mathcal E_{d_{\rm hard}}^{\rm aff,rel}
 =\Theta\!\left(\frac{LR_0^2}{1+N\InfoDepth}\right).
\end{equation*}
The upper bound holds in every finite dimension. The lower bound uses
one fixed distribution of affine operators with unique roots.
\end{theorem}
\revision{The algorithm attaining the upper bound chooses the best of four
methods: returning the initial point, sequential extragradient on one
worker, common-query minibatch extragradient, and local extragradient
with averaging, in which the workers start from a common point, make $T$
local steps with batch size $b$ and $2bT\le K$, and average their
endpoints.}
The hard dimension grows with the budgets, as it must: in a fixed
dimension, enough noiseless responses identify an affine operator.

Under mixed noise, the uniform part of the noise adds a statistical
term, and averaging over workers keeps it at the level of all $mKN$
samples.

\begin{theorem}[Affine minimax rate under mixed noise]
\label{thm:affine-additive-main}\label{thm:affine-mixed-main}
Under Assumptions~\ref{ass:affine-main} and~\ref{ass:mixed}, for every
$m,K,N\in\mathbb N$, $L,R_0>0$, $\rho\ge1$, and $\sigma\ge0$, there is a
finite dimension $d_{\rm mix}=O\bigl(m(KN)^2\bigr)$ such that
\begin{equation*}
 \mathcal E_{d_{\rm mix}}^{\rm aff,mix}
 =\widetilde\Theta\!\Biggl(R_0^2\min\Biggl\{L,\ 
 \frac{L}{N}\min\!\left\{\frac\rho K,1+\frac\rho{mK}\right\}
 +\frac{\sigma}{R_0\sqrt{mKN}}\Biggr\}\Biggr).
\end{equation*}
where the upper bound holds up to a universal constant, in every finite
dimension, and the lower bound up to a universal multiple of $\log(2m)$. The hard
instances may have root noise exactly $\sigma^2$.
\end{theorem}
The deterministic term decreases at the rate of the local sequential
progress, whereas the statistical term uses every sample, as for a
single minibatch of size $mKN$. For globally bounded variance,
$\rho=1$, the logarithmic factor disappears and the minimax rate is the
deterministic rate over $KN$ sequential calls plus the statistical rate
over all $mKN$ samples:
\begin{equation}\label{eq:affine-uniform-specialization}
 \mathcal E_{d_{\rm mix}}^{\rm aff,mix}
 =\Theta\Bigl(\min\Bigl\{LR_0^2,\,\tfrac{LR_0^2}{1+KN}
   +\tfrac{\sigma R_0}{\sqrt{mKN}}\Bigr\}\Bigr).
\end{equation}
For $\sigma=0$, Theorem~\ref{thm:affine-integrated} is sharper.

\section{Nonlinear monotone VI}\label{vi:main}
For nonlinear Lipschitz monotone operators, the gap upper and lower
bounds match up to logarithmic factors in $m$. The upper bounds use
batched extragradient, either on a single worker or with queries common
to all workers. The lower bounds use monotone chains whose links are
revealed only by informative calls.

\subsection{Monotone operators}
\revision{We first compare sequential single-worker computation and
common-query minibatching under SGC, and then add the uniform component
of mixed noise.}

\begin{theorem}[Monotone VI under SGC]\label{vi:main-monotone}
Let Assumptions~\ref{ass:monotone-lipschitz} and~\ref{ass:sgc} hold,
let $D>0$, and let $C$ be a nonempty compact set with
$z_\star\in C\subseteq B(z_0,D)$.
\revision{There exists an admissible algorithm attaining}
\begin{equation*}
 \E\Gap_C^{\rm M}(\widehat z)
 =O\!\left(\frac{LD^2}{1+N\SchedDepth}\right)
\end{equation*}
for the Minty gap. If
$F=(\nabla_x\Phi,-\nabla_y\Phi)$ for a differentiable convex--concave
$\Phi$ and $C=X\times Y$ is a product set, its expected primal--dual
gap $\E\Gap_{X\times Y}^{\rm PD}(\widehat z)$ satisfies the same bound.

Conversely, \revision{in a sufficiently large finite dimension,} every
admissible algorithm has an instance in this class,
with a root in $B(z_0,D)$, such that
\begin{equation*}
 \E\Gap^{\rm M}_{B(z_0,D)}(\widehat z)
 =\Omega\!\left(\frac{LD^2}{1+N\InfoDepth}\right).
\end{equation*}
For saddle problems, one $C^2$ convex--concave instance satisfies the
same lower bound for both gaps on the product of the two balls of
radius $D/\sqrt2$ centered at the blocks of $z_0$, which contains its
root. The lower bounds hold
even if every response reveals whether it is informative and, for
saddle problems, the sample values. The sample operators share the root
and are $\rho L$-Lipschitz.
\end{theorem}
Up to universal constants, the upper bound in
Theorem~\ref{vi:main-monotone} equals
$D^2\min\{L,\rho L/(KN),(L/N)(1+\rho/(mK))\}$. Compared with
stochastic mirror-prox~\citep{juditsky2011} run with the variance bound
$9(\rho-1)L^2D^2$ that SGC gives on $B(z_0,2D)$
(Appendix~\ref{app:comparisons}), the stochastic terms
$\sqrt{\rho-1}\,L/\sqrt{KN}$ and $\sqrt{\rho-1}\,L/\sqrt{mKN}$ become
$\rho L/(KN)$ and $\rho L/(mKN)$.
\revision{The upper and lower denominators in
Theorem~\ref{vi:main-monotone} differ by at most a factor
$O(1+\log(2m)/\log(2+\log(2m)))$.}

\begin{theorem}[Monotone VIs under mixed noise]\label{vi:main-mixed}
Let Assumption~\ref{ass:mixed} hold with $\rho\ge1$ and $\sigma\ge0$,
together with the remaining conditions of Theorem~\ref{vi:main-monotone}.
Set
\begin{equation*}
 r_{\rm mix}:=\min\Biggl\{L,\ 
 \frac{\rho L}{KN}+\frac{\sigma}{D\sqrt{KN}},\ 
 \frac{L}{N}\left(1+\frac{\rho}{mK}\right)
                  +\frac{\sigma}{D\sqrt{mKN}}\Biggr\}.
\end{equation*}
\revision{There exists an admissible algorithm attaining}
\begin{equation*}
 \E\Gap_C^{\rm M}(\widehat z)=O(D^2r_{\rm mix}).
\end{equation*}
For a saddle problem with $C=X\times Y$, it also satisfies
$\E\Gap_{X\times Y}^{\rm PD}(\widehat z)=O(D^2r_{\rm mix})$.
Conversely, \revision{in a sufficiently large finite dimension,} every
admissible algorithm has an instance in this class,
with a root in $B(z_0,D)$, such that
\begin{equation*}
 \E\Gap^{\rm M}_{B(z_0,D)}(\widehat z)
 =\Omega\!\left(\frac{D^2r_{\rm mix}}{\log(2m)}\right).
\end{equation*}
\end{theorem}
The bounds agree within a universal multiple of $\log(2m)$; replacing
$1+\rho/(mK)$ by $1/\min\{1,mK/\rho\}$ changes the upper bound by at most
a factor two.
Here $\sigma^2$ is the additive term of Assumption~\ref{ass:mixed}, so the
variance at a root is at most $\sigma^2$; the bounds do not require it to
equal $\sigma^2$. \revision{The lower-bound instances have root variance
exactly $\sigma^2$.}

\paragraph{Affine and nonlinear operators.}
Theorems~\ref{vi:main-monotone} and~\ref{vi:main-mixed} show where the
affine structure matters. For affine operators, averaging local
trajectories commutes with the operator, so a single method uses the
sequential depth $KN$ and the $mKN$ calls of all workers at the same
time. Under SGC, its progress per round is the informative depth
$\InfoDepth$, which includes the longest informative chain over the
workers (Theorem~\ref{thm:affine-integrated}). Under mixed noise, the
deterministic term decreases at the rate of local sequential progress,
while the statistical term $\sigma R_0/\sqrt{mKN}$ uses every sample
(Theorem~\ref{thm:affine-additive-main}).

This split resembles the one of~\citet{woodworth2020} for convex
optimization with intermittent communication: for quadratic objectives,
local SGD is better than minibatch SGD and its accelerated variant is
minimax optimal, whereas for general convex objectives local SGD can be
worse than minibatch SGD. Here the same structure separates affine from
nonlinear VIs. For nonlinear operators, our methods choose between the
two resources: $KN$ sequential calls on one worker, with the statistical
error $\sigma D/\sqrt{KN}$ of their own samples, or common queries that
use all $mKN$ samples but make only $N$ sequential updates. Under mixed
noise, the separation holds for every admissible algorithm, not only for
local methods: by the lower bound of Theorem~\ref{vi:main-mixed}, no
algorithm does better on worst-case nonlinear instances than the better
of the two schedules, up to a factor $\log(2m)$. Under SGC, the upper
and lower bounds differ only by the ratio of $\InfoDepth$ to
$\SchedDepth$, which is at most logarithmic in $m$.

Despite this theoretical difference, nonlinear problems can behave much
like affine ones in practice. In the experiments of
Appendix~\ref{app:exp-affine}, local extragradient with averaging gains
over serial and common-query extragradient on two strongly nonlinear
problems by the same order as on a bilinear game, and the nonlinear
problems converge at least as fast as the bilinear one
(Figure~\ref{fig:affine-nonlinear-experiments}).

\subsection{Strongly monotone operators}
For strongly monotone operators we measure the squared distance to the
root. Let Assumptions~\ref{ass:monotone-lipschitz}
and~\ref{ass:strong-monotonicity} hold with $0<\mu\le L$, let
$\|z_0-z_\star\|\le R_0$ for some $R_0>0$, and write $\kappa=L/\mu$.
Under SGC, an admissible algorithm converges linearly with the schedule
depth of Theorem~\ref{vi:main-monotone} in the exponent
(Proposition~\ref{nn:strong-distance}): for a universal constant $c>0$,
\begin{equation}\label{vi:main-distance}
 \E\|\widehat z-z_\star\|^2
 =O\!\left(R_0^2\exp\!\left[-\frac{cN}{\kappa}\SchedDepth\right]\right).
\end{equation}
Mixed noise adds a statistical term.

\begin{theorem}[Strongly monotone VIs under mixed noise]
\label{vi:main-mixed-accuracy}
Under Assumptions~\ref{ass:monotone-lipschitz},
\ref{ass:strong-monotonicity}, and~\ref{ass:mixed}, let $\kappa\ge2$ and
$\|z_0-z_\star\|\le R_0$. For a universal constant $c>0$ let
$K_{\rm cq}=\min\{1,mK/\rho\}$ and
$\bar m_N=m/(1+m\,e^{-cNK_{\rm cq}/\kappa})$. There exists an admissible
algorithm attaining
\begin{equation*}
 \E\|\widehat z-z_\star\|^2\lesssim\min\Bigl\{R_0^2,\
 R_0^2e^{-cN\SchedDepth/\kappa}
 +\frac{\sigma^2}{\mu^2\bar m_NKN}\Bigr\}.
\end{equation*}
Conversely, for universal constants $C,c'>0$ let $m_N=1$ if
$N\log(2m)\le c'\kappa$ and $m_N=m$ otherwise. In a sufficiently large
finite dimension, every admissible algorithm has an instance in this
class with
\begin{equation*}
 \E\|\widehat z-z_\star\|^2\gtrsim\max\Bigl\{
 R_0^2e^{-CN\InfoDepth/\kappa},\
 \min\Bigl\{R_0^2,\frac{\sigma^2}{\mu^2m_NKN}\Bigr\}\Bigr\}.
\end{equation*}
\end{theorem}
The algorithm makes sequential steps on one worker for the first half of
the rounds and then switches to common-query minibatches
(Theorem~\ref{thm:strong-mixed-hybrid}); Theorem~\ref{sf:rate-lower}
gives the lower bound. Both bounds consist of a linear rate and the
statistical error of $\bar m_NKN$ or $m_NKN$ samples: the number of
workers whose samples can be pooled grows from one to $m$ after about
$\kappa$ rounds, up to factors in $m$. In terms of the rounds needed for
a given accuracy, the two bounds agree up to logarithmic factors in $m$
(Corollary~\ref{sf:accuracy-match}). For noise bounded by
$\delta\|z-z_\star\|^2+\sigma^2$, the stochastic extragradient bounds
of~\citet{gorbunov2022} depend on $\delta/\mu^2$. Since
$\|F(z)\|^2\le L^2\|z-z_\star\|^2$, SGC gives such a bound with
$\delta=(\rho-1)L^2$, and this dependence becomes $(\rho-1)\kappa^2$.
The exponents $NK/(\rho\kappa)$ and $N\min\{1,mK/\rho\}/\kappa$ in
Theorem~\ref{vi:main-mixed-accuracy} are linear in $\kappa$, because the
analysis uses that the noise scales with $\|F\|^2$ rather than with the
squared distance.

\subsection{Generalized smoothness}\label{gs:main}
Global Lipschitz continuity can be replaced by a Jacobian bound that
grows with the operator.

\begin{assumption}[Monotone generalized-smooth operator]
\label{ass:generalized-smoothness}
Let $L_0,R_0>0$ and $L_1\ge0$. The operator
$F:\mathbb R^d\to\mathbb R^d$ is continuously differentiable. Writing
$\mathrm DF(z)$ for its Jacobian at $z$, we require, for every
$z,u\in\mathbb R^d$,
\begin{equation*}
 \begin{aligned}
 \langle F(z)-F(u),z-u\rangle&\ge0,\\
 \|\mathrm DF(z)\|_{\rm op}&\le L_0+L_1\|F(z)\|.
 \end{aligned}
\end{equation*}
There is a root $z_\star$ with
$F(z_\star)=0$ and $\|z_0-z_\star\|\le R_0$.
\end{assumption}
The assumption replaces
Assumption~\ref{ass:monotone-lipschitz}; the oracle satisfies the mixed
condition, which contains SGC as the case $\sigma=0$.

\begin{corollary}[Generalized smoothness under SGC]
\label{gs:main-distance-complexity}
Under Assumptions~\ref{ass:generalized-smoothness},
\ref{ass:strong-monotonicity}, and~\ref{ass:sgc}, let $\kappa_0=L_0/\mu\ge1$,
$\vartheta=\min\{\rho/K,1+\rho/(mK)\}$, and
$\ell_0=\log(64\max\{(1+L_1R_0)^2,\kappa_0\})$. For every $N\ge1$ and a
universal constant $c>0$, an admissible adaptive restarted extragradient
method satisfies, in every finite dimension,
$\E\|\widehat z-z_\star\|^2\le8R_0^2e^{-cx_N}$ with
\begin{equation*}
 x_N=\min\Bigl\{\frac{N}{\ell_0\bigl(1+\vartheta(L_1R_0)^2\bigr)},\
 \sqrt{\frac{N}{\ell_0\vartheta\kappa_0}}\Bigr\}.
\end{equation*}
\end{corollary}
The initial radius enters only through the burn-in factor
$1+\vartheta(L_1R_0)^2$. The method adapts the extragradient step size,
restarts on shrinking balls, and switches once from sequential local
updates to common-query minibatches; each restart is repeated to make its
failure unlikely, which produces the square root in the exponent.
Corollary~\ref{gs:rate-sgc} proves the bound, and Theorem~\ref{ghs:upper}
also covers mixed noise.

\section{Composite VI with smooth convex samples}
\label{sec:acceleration-main}
\revision{This section studies composite VIs under the root-noise model.
For $G=\nabla f+H$, the gradient is sampled and the monotone map $H$ is
exact; smooth convex samples control the noise without a global SGC
assumption.}

\begin{assumption}[Smooth convex samples and an exact monotone component]
\label{ass:interp-main}
Let $d\in\mathbb N$, $R_0>0$, $0\le L_f\le L_{\max}<\infty$, and
$0\le L_H<\infty$. On $\mathbb R^d$, let
\[
 G=\nabla f+H,
\]
where $f$ is convex with an $L_f$-Lipschitz gradient, and $H$ is
deterministic, monotone, and $L_H$-Lipschitz. A designated root satisfies
\[
 G(x_\star)=0,\qquad \|x_0-x_\star\|\le R_0.
\]
Almost every sample function $f_\xi$ is convex on $\mathbb R^d$ and has
an $L_{\max}$-Lipschitz gradient. At every query $x$ measurable with
respect to the pre-call history $\mathcal F$, \revision{the sample drawn
for that call is independent of $\mathcal F$} and
\[
 \mathbb E[\nabla f_\xi(x)\mid\mathcal F]=\nabla f(x).
\]
All workers use the same sample distribution and the same exact map $H$;
samples are independent across calls and workers, and all maps and
outputs are measurable. The root noise is finite:
\begin{equation}\label{eq:composite-root-variance-main}
 \sigma_\star^2
 :=\mathbb E\|\nabla f_\xi(x_\star)+H(x_\star)\|^2
 =\mathbb E\|\nabla f_\xi(x_\star)-\nabla f(x_\star)\|^2<\infty.
\end{equation}
One stochastic gradient or one evaluation of $H$ costs one call.
Algorithms know the displayed parameters; an upper bound on
$\sigma_\star$ may replace its exact value.
\end{assumption}

For the sample operator $\nabla f_\xi+H$, interpolation reads
\[
 \sigma_\star=0
 \quad\Longleftrightarrow\quad
 \nabla f_\xi(x_\star)+H(x_\star)=0\quad\text{almost surely}.
\]
\revision{Interpolation does not imply SGC with any finite constant, even
for one-dimensional nonnegative convex $1$-smooth sample losses with a
common minimizer. The mean and sample smoothness constants $L_f$ and
$L_{\max}$ can differ by a large factor.}

\begin{definition}[Composite gap]
For $x\in\mathbb R^d$, define
\begin{equation*}
 \CompGap(x)
 =\sup_{\|u-x_\star\|\le R_0}
   \{f(x)-f(u)+\langle H(u),x-u\rangle\}.
\end{equation*}
The error is $\mathbb E\CompGap(\widehat x)$, with the
expectation outside the supremum.
\end{definition}
With $F=G$, convexity gives
$\CompGap(x)\ge\Gap^{\rm M}_{B(x_\star,R_0)}(x)$, and
$\CompGap(x_0)\le(L_f/2+L_H)R_0^2$.

\revision{For a root-noise level $s\ge0$, define the rates of sequential
updates on one worker and of pooled minibatches}
\begin{equation*}
 \begin{aligned}
 \CompSeq{s}&:=\frac{L_f}{(KN)^2}
   +\frac{L_H+L_{\max}}{KN}
   +\frac{s}{R_0\sqrt{KN}},\\
 \CompMini{s}&:=\frac{L_f}{N^2}+\frac{L_H}{KN}
   +\min\!\left\{\frac{L_{\max}}{mKN},
          \sqrt{\frac{L_{\max}L_f}{mKN}}\right\}
   +\frac{s}{R_0\sqrt{mKN}}.
 \end{aligned}
\end{equation*}
\revision{The second method also solves for $H$ locally between
communications.}

\subsection{Upper bound}

\begin{theorem}[\revision{Sequential local updates and parallel minibatching}]
\label{thm:composite-local-main}
Under Assumption~\ref{ass:interp-main}, for every $m,K\ge1$ and $N\ge2$,
\revision{there exists an admissible algorithm attaining}
\begin{equation*}
 \mathbb E\CompGap(\widehat x)=
 \widetilde O\!\Bigl(R_0^2\min\bigl\{L_f+L_H,
 \CompSeq{\sigma_\star},\CompMini{\sigma_\star}\bigr\}\Bigr).
\end{equation*}
\revision{The sequential single-worker guarantee has no logarithmic
factors and holds for $N=1$; the parallel-minibatch guarantee hides
powers of $\log(2mKN)$.}
\end{theorem}
\revision{The sequential guarantee uses accelerated mirror-prox on one
worker. The parallel method combines mirror-prox sliding
\citep{lan2021sliding} with independent minibatches and local $H$-solves,
charging every call to the budget. Relative to stochastic accelerated
mirror-prox~\citep{chen2017amp} run with the variance bound of
Lemma~\ref{lem:samples:variance} on $B(x_0,2R_0)$
(Appendix~\ref{app:comparisons}), it improves $L_H/N$ to $L_H/(KN)$ and replaces
$\sqrt{L_{\max}L_f/(mKN)}$ by
$\min\{L_{\max}/(mKN),\sqrt{L_{\max}L_f/(mKN)}\}$.}

Appendix~\ref{app:exp-composite} gives a demonstration
of the regimes in Theorem~\ref{thm:composite-local-main}.

\subsection{General lower bound}
We next define the minimax risk and state the lower bound for the full
parameter range.
\begin{definition}[Composite minimax risk]
Fix $L_f,L_H,L_{\max},R_0$ and $m,K,N\in\mathbb N$.
For $\nu\ge0$, let $\mathfrak I_d(\nu)$ be the class of instances of
Assumption~\ref{ass:interp-main} in $\mathbb R^d$ with
\[
 \mathbb E\|\nabla f_\xi(x_\star)+H(x_\star)\|^2\le \nu^2.
\]
\revision{Let $\mathsf A$ range over admissible randomized algorithms,
let $\mathcal I$ denote an instance in $\mathfrak I_d(\nu)$, and let
$\mathbb E_{\mathsf A,\mathcal I}$ average over all randomness;
$\mathsf A$ may use $d$ and the model parameters.}
Define
\[
 \begin{gathered}
 \mathfrak R_{\rm comp}(\nu)
 :=\sup_{d\in\mathbb N}\inf_{\mathsf A}
       \sup_{\mathcal I\in\mathfrak I_d(\nu)}
       \mathbb E_{\mathsf A,\mathcal I}\CompGap(\widehat x),
\end{gathered}
\]
\revision{Thus $\mathfrak R_{\rm comp}(0)$ is the interpolation risk.}
\end{definition}

\begin{theorem}[General composite lower bound]
\label{thm:composite-lower-main}
\label{interp:explicit-main}
\label{prop:composite-root-budget-main}
Let $0\le L_f\le L_{\max}$, $L_H\ge0$, $R_0>0$, $m,K,N\in\mathbb N$, and
$\nu\ge0$. In a common dimension of at most $2048m(KN)^2$, and in every
larger dimension,
\begin{equation*}
 \begin{multlined}
 \frac{\mathfrak R_{\rm comp}(\nu)}{R_0^2}
 =\Omega\!\Biggl(
 \frac{L_f}{(KN)^2}+\frac{L_H}{KN}
 \!+\!\min\!\left\{L_f,\frac{L_{\max}}{mKN}\right\}\\[-2pt]
 \quad ~ ~ +\min\!\left\{\frac{L_f}{(1+N\log m)^2},
              \frac{L_{\max}}{KN(1+N\log m)}\right\}\\[-2pt]
 \quad+\min\!\left\{L_f+L_H,
              \frac{\nu}{R_0\sqrt{mKN}}\right\}\Biggr).
 \end{multlined}
\end{equation*}
The first four terms are attained by interpolating instances.
\end{theorem}
\revision{The five terms come from sequential depth, the exact component,
sampling, communication, and the noise at the solution.}

\subsection{A matching regime}
The general lower bound matches the parallel-minibatch upper bound in
the following regime, up to logarithmic factors.

\begin{corollary}[Matching bounds for large sample smoothness]
\label{cor:composite-matching-main}\label{comp-proj:matching}
Let $m,K\ge1$, $N\ge2$, $R_0>0$, $\nu\ge0$, $L_H\ge0$, and
$0<KL_f\le L_{\max}$. Then
\begin{equation*}
 \mathfrak R_{\rm comp}(\nu)=
 \widetilde\Theta\!\left(
 R_0^2\min\{L_f+L_H,\CompMini{\nu}\}\right),
\end{equation*}
where the tilde hides powers of $\log(2mKN)$. The lower bound holds in
every dimension of at least
$\max\{2048m(KN)^2,\lceil8192L_{\max}/L_f\rceil\}$.
\end{corollary}
\revision{Outside this regime we do not know whether full local
acceleration and full use of the samples can be achieved together.}

\section{Discussion \texorpdfstring{\revision{and Conclusion}}
  {and Conclusion}}\label{sec:discussion}
\revision{The communication budget can be spent on sequential local
updates, which increase the optimization depth, or on common-query
minibatches, which use all $mKN$ samples to reduce the stochastic error.
Which is better depends on the operator class and the noise model. Our
bounds are sharp for affine VIs under SGC and, up to logarithmic factors,
for nonlinear monotone and strongly monotone VIs and for composite VIs
with $L_{\max}\ge KL_f>0$. Several questions remain open: closing the
logarithmic gap for nonlinear operators; the optimal
dependence on the initial radius under generalized smoothness; whether
local acceleration and pooled sampling can be combined when
$L_{\max}<KL_f$; and heterogeneous workers.}

\clearpage
\section*{AI Use Statement}
\revision{Generative AI tools were used to improve the clarity
and presentation of the manuscript and to help with the
implementation and checking of the numerical experiments. 
The authors reviewed all AI-assisted changes,
independently verified the reported results, and take
full responsibility for the final content.}

\bibliography{references}
\clearpage
\appendix
\renewcommand{\thesection}{\AppendixLetter{\value{section}}}
\renewcommand{\theHsection}{\thesection}
\raggedbottom
\section*{Supplementary Material}
\input{appendices/contents}
\begin{revisionblock}
\section{Numerical experiments}\label{app:numerical-experiments}

We illustrate two comparisons from the main text on synthetic instances. The budget is counted exactly: each of the $m$ workers makes at most $K$
oracle calls per round, and both evaluations in an extragradient step and
every evaluation of the exactly evaluated component count as calls. Step
sizes, batch sizes, and local horizons are chosen by grid search on separate
tuning seeds; we minimize the worst final error over these seeds, which rules
out unstable choices. The parameters are tuned for the last displayed round,
and the resulting method is evaluated after every round. Curves show the
median and the interquartile band over $16$ evaluation seeds. The instance is
fixed within each panel, while the initial point and the oracle randomness
vary with the seed.

\subsection{Affine and nonlinear operators}\label{app:exp-affine}

For affine operators, local extragradient with worker averaging attains the
informative depth $\InfoDepth$ (Theorem~\ref{thm:affine-integrated}). For
nonlinear operators, our upper bound only has the depth $\SchedDepth$ of the
serial and common-query schedules (Theorem~\ref{vi:main-monotone}), because
averaging commutes only with affine maps. In theory, the affine case is
therefore richer. This experiment checks whether the difference shows up in
practice. We use saddle problems on $\mathbb R^{20}\times\mathbb R^{20}$ with
root $z_\star=0$ and $\lVert z_0\rVert=1$,
\[
 F(z)=Jz+c\,h'(z),
 \qquad
 J=\begin{pmatrix}0&A\\-A^\top&0\end{pmatrix},
\]
where $A$ has singular values between $0.1$ and $1$ and $h'$ acts
coordinatewise. The affine instance has $c=0$. The kinked instance uses
$h(t)=\tfrac12\max\{t,0\}^2$; its kink is at the root, so the operator is
nonlinear at every scale. The smooth instance uses
\[
 \widetilde h_\delta(t):=\frac{\delta^2}{2}
 \operatorname{softplus}(t/\delta)^2,
 \qquad \operatorname{softplus}(s):=\log(1+e^s),
\]
and $h(t):=\widetilde h_{\delta}(t)-\widetilde h_{\delta}(0)
-\widetilde h'_{\delta}(0)t$ with $\delta=0.1$. For both nonlinear
instances $c=4$, so the nonlinear part is four to six times larger than the
bilinear part.

The oracle returns $g=(\omega/p)F$ with $\omega\sim\operatorname{Bernoulli}(p)$
and $p=1/\rho$, so SGC holds with equality; this is the oracle of our lower
bounds. We take $m=64$, $K=16$, and $\rho\in\{4,16,64,256\}$, and compare
serial extragradient on one worker, extragradient with common queries, and
local extragradient with worker averaging after every round. The error is the
saddle gap on the unit ball divided by its value at $z_0$.

\begin{figure}[h]
\centering
\includegraphics[width=.85\textwidth]{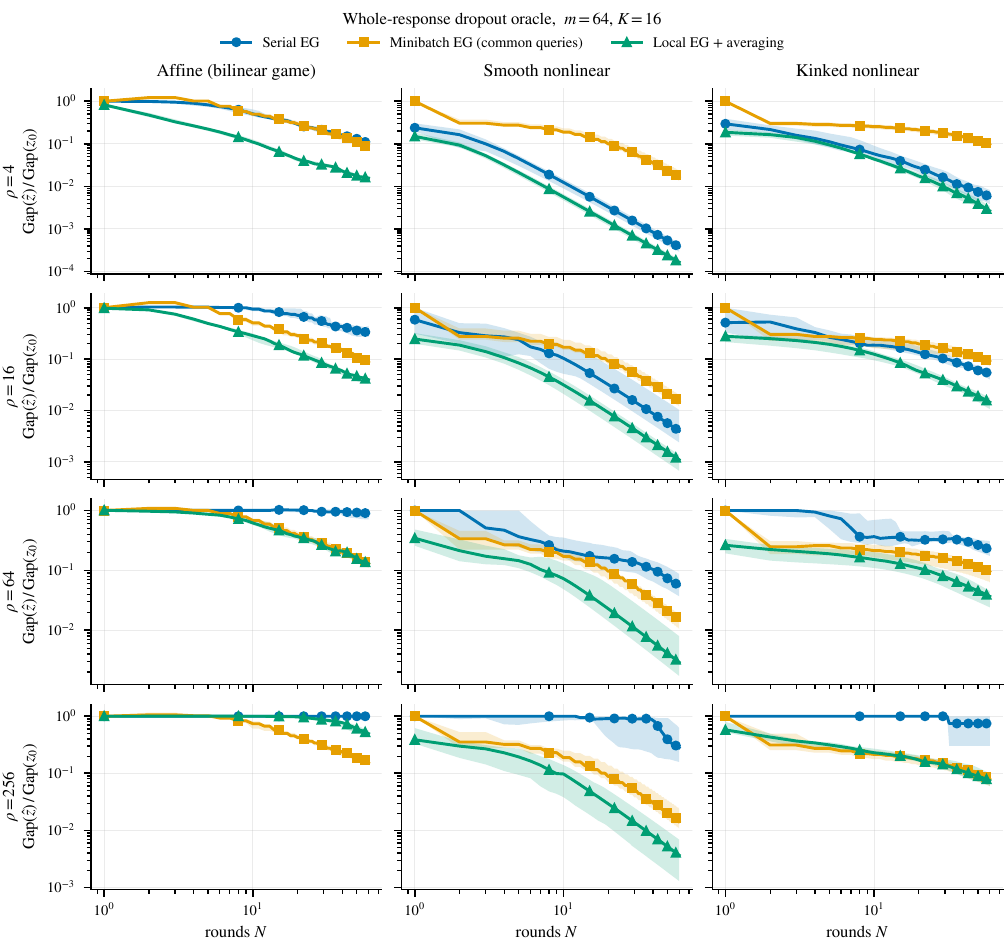}
\caption{Affine and nonlinear operators under whole-response dropout
($m=64$, $K=16$): saddle gap after $N$ rounds divided by its value at $z_0$.
Rows: SGC constant $\rho$; columns: problems.}
\label{fig:affine-nonlinear-experiments}
\end{figure}

Figure~\ref{fig:affine-nonlinear-experiments} ranks the methods in the same
way in all three columns. For $\rho\le16$, local averaging is the fastest
method on the affine problem and on both nonlinear problems. For these
parameters $\InfoDepth$ exceeds $\SchedDepth$ by a factor between $2$ and
$3.5$, and the gain of local averaging over the better of the other two
methods is of the same order in the three columns. For $\rho=256$, where
$\InfoDepth\approx1.2\,\SchedDepth$, common queries are as fast as local
averaging or faster, except on the smooth problem. The nonlinear problems
are not harder than the affine one: their errors are of the same order or
smaller, because the nonlinear part adds monotonicity. Thus, although the
theory distinguishes the affine case, convergence on these problems is about
the same with and without nonlinearity, and the gain from worker averaging
that we prove only for affine operators also appears for the nonlinear ones.
This does not close the logarithmic gap in Theorem~\ref{vi:main-monotone};
it only shows that these instances are not hard for local averaging.

\subsection{Composite VIs}\label{app:exp-composite}

We take $G=\nabla f+H$ on $\mathbb R^{40}$ with
\[
 f(x)=\frac{1}{2n}\sum_{i=1}^n(a_i^\top x-b_i)^2,
 \qquad n=4000,
\]
one row sampled per gradient call, and $H$ an exactly evaluated skew affine
map; $x_0=0$ and $\lVert x_\star\rVert=1$. The covariance of the rows has
spectrum from $1$ to $10^{-4}$, so $f$ is far from strongly convex on these
horizons. The noise level $\sigma_\star$ at the solution is set by perturbing
the targets, and $\sigma_\star=0$ is interpolation. In the first row of
Figure~\ref{fig:composite-sigma-experiments}, $L_f=0.05$ and $L_H=1$, so the
exact component dominates. In the second row, $L_f=1$, $L_H=0.02$, and a few
heavy rows give $L_{\max}\approx400$, so sampling dominates. With $m=K=16$
we compare serial accelerated mirror-prox (AMP) on one worker, AMP with
common queries and batches of $m(K-2)$ gradients, local AMP with worker
averaging, and the method of Theorem~\ref{thm:composite-local-main}, which
pools gradients and solves for $H$ locally. The error is the composite gap
divided by its value at $x_0$.

\begin{figure}[h]
\centering
\includegraphics[width=.9\textwidth]{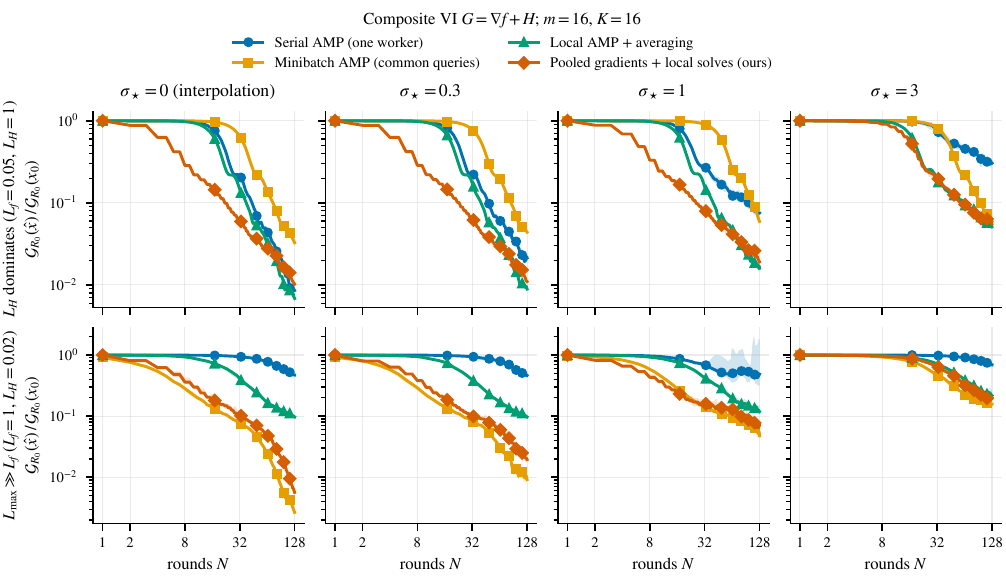}
\caption{Composite VIs ($m=K=16$): composite gap after $N$ rounds divided by
its value at $x_0$. Rows: dominant term; columns: noise level $\sigma_\star$ at
the solution.}
\label{fig:composite-sigma-experiments}
\end{figure}

The curves follow the terms of Theorem~\ref{thm:composite-local-main}. When
$L_H$ dominates, our method is the fastest for small and moderate $N$,
because the local solves reduce the exact term from $L_H/N$ for common queries
to $L_H/(KN)$; serial and local AMP also make $\Theta(KN)$ sequential steps
and catch up near the end of the horizon. As $\sigma_\star$ grows, serial AMP
degrades the most, since its statistical term is $\sigma_\star R_0/\sqrt{KN}$
rather than $\sigma_\star R_0/\sqrt{mKN}$. When sampling dominates, the two
methods that pool gradients are far ahead of serial AMP, in line with
$L_{\max}R_0^2/(mKN)$ against $L_{\max}R_0^2/(KN)$, and local AMP lies in
between. For the largest $\sigma_\star$ the differences shrink, since all
methods are limited by the noise at the solution.

\FloatBarrier
\end{revisionblock}

\section{The expected maximum of binomial variables}
\label{app:clock}
This appendix gives one probability calculation used throughout the paper.
For integers $m,K\ge1$ and $0<p\le1$, write $\operatorname{Bin}(K,p)$
for the binomial law with $K$ trials and success probability $p$; iid
means independent and identically distributed. Define
\begin{equation*}
 Z_i\stackrel{\rm iid}{\sim}\operatorname{Bin}(K,p),
 \qquad \beta_p=\E\max_{1\le i\le m}Z_i.
\end{equation*}
The maximum is taken over independent local counts. The SGC lower
bounds use $p=1/\rho$, and the composite lower bounds use
$p=\lambda/L_{\max}$, where $\lambda\in(0,L_{\max}]$ is the curvature of the
mean and $L_{\max}$ is the smoothness of the samples. These are the same
probability calculation applied to different oracle constructions.

\begin{lemma}[Elementary bounds]\label{lem:clock:elementary}
For all the above parameters,
\begin{align*}
 \max\{Kp,(1-e^{-1})\min\{1,mKp\}\}
 &\le\beta_p\\
 &\le\min\{K,mKp,(e-1)Kp+\log m\}.\nonumber
\end{align*}
\end{lemma}
\begin{proof}
We begin by establishing a lower bound for $\beta_p$. First, observe that
\begin{equation*}
\beta_p = \E\max_{1\le i\le m} Z_i \ge \E Z_1 = Kp.
\end{equation*}
On the other hand, let $Z_{\max}=\max_{1\le i\le m}Z_i$. Since
$Z_{\max}\ge1$ on the event $\{Z_{\max}\ge1\}$, we have
$Z_{\max}\ge\ind_{\{Z_{\max}\ge1\}}$, and therefore
\begin{equation*}
\beta_p = \E Z_{\max} \ge \Prob(Z_{\max} \ge 1).
\end{equation*}
Now,
\begin{align*}
\Prob(Z_{\max} \ge 1)
&= 1 - \Prob(Z_{\max} = 0) = 1 - \prod_{i=1}^m \Prob(Z_i = 0) = 1 - (1-p)^{Km} \ge 1 - e^{-mKp},
\end{align*}
where the last inequality follows from $1 - p \le e^{-p}$.

Using the elementary bound $1 - e^{-x} \ge (1 - e^{-1}) \min\{1, x\}$, we obtain
\begin{equation*}
\beta_p \ge \max\left\{Kp,\ (1 - e^{-1}) \min\{1, mKp\} \right\}.
\end{equation*}

We now turn to the upper bound. It is immediate that
\begin{equation*}
\beta_p \le K
\quad\text{and}\quad
\beta_p \le mKp.
\end{equation*}
Next, note that for any $\eta_{\rm mgf}>0$
\begin{equation*}
e^{\eta_{\rm mgf}Z_{\max}}
=\max_{1\le i\le m}e^{\eta_{\rm mgf}Z_i}
\le\sum_{i=1}^m e^{\eta_{\rm mgf}Z_i}.
\end{equation*}
By Jensen's inequality,
\begin{equation*}
e^{\eta_{\rm mgf}\E Z_{\max}}
\le \E e^{\eta_{\rm mgf}Z_{\max}}
\le\sum_{i=1}^m\E e^{\eta_{\rm mgf}Z_i}
=m\bigl(1+p(e^{\eta_{\rm mgf}}-1)\bigr)^K .
\end{equation*}
Taking logarithms yields
\begin{equation*}
\beta_p
\le \frac{\log m+K\log(1+p(e^{\eta_{\rm mgf}}-1))}{\eta_{\rm mgf}}
\le \frac{\log m+Kp(e^{\eta_{\rm mgf}}-1)}{\eta_{\rm mgf}}.
\end{equation*}
For $\eta_{\rm mgf}=1$ this completes the proof.
\end{proof}

It will be useful for us to define the following function:

\begin{equation}\label{eq:clock:explicit-app}
 \BinDepth(m,K,p)=\min\left\{K,mKp,
 Kp+\frac{\log(2m)}{\log(2+\log(2m)/(Kp))}\right\}.
\end{equation}

\begin{theorem}[An explicit uniform asymptotic]\label{thm:clock:asymptotic}

Uniformly in all $m,K\ge1$ and $0<p\le1$,
\begin{equation*}
 \frac{\BinDepth(m,K,p)}{64}\le\beta_p
 \le4\BinDepth(m,K,p).
\end{equation*}
Consequently, for every $N\ge1$,
\begin{equation}\label{eq:clock:reciprocal}
 \frac{1+N\BinDepth(m,K,p)}{64}\le1+N\beta_p
 \le4\bigl(1+N\BinDepth(m,K,p)\bigr).
\end{equation}
\end{theorem}
\begin{proof}
Write $a_{\rm clk}=\log(2m)/\log\bigl(2+\log(2m)/(Kp)\bigr)$ and
$T_{\rm clk}=\min\{K,mKp,Kp+a_{\rm clk}\}$.
Using Lemma~\ref{lem:clock:elementary},
\begin{equation*}
\beta_p\le
\frac{\log m+Kp(e^{\eta_{\rm mgf}}-1)}{\eta_{\rm mgf}}.
\end{equation*}
If $\log(2m)/(Kp)\ge1$, take
$\eta_{\rm mgf}=\log\bigl(1+\log(2m)/(Kp)\bigr)$. Since
$\log(2+x)\le2\log(1+x)$ for $x\ge1$, this gives
$\beta_p\le4a_{\rm clk}$. If $\log(2m)/(Kp)<1$,
Lemma~\ref{lem:clock:elementary}
gives $\beta_p\le(e-1)Kp+\log(2m)<3Kp$. In both cases
$\beta_p\le4(Kp+a_{\rm clk})$. Combining with $\beta_p\le K,mKp$
proves the upper bound.

For the lower bound, first suppose $p\ge1/2$. Then
$\beta_p\ge Kp\ge K/2\ge T_{\rm clk}/2$. Suppose henceforth $p<1/2$.
If $Kp\ge\log(2m)/8$, then
$Kp+a_{\rm clk}\le(1+8/\log2)Kp<13Kp$, so
$\beta_p\ge T_{\rm clk}/13$.
It remains to consider $Kp<\log(2m)/8$.
For $x>8$, $\log(2+x)<x/2$, and hence $a_{\rm clk}>2Kp$ and
$T_{\rm clk}\le Kp+a_{\rm clk}<3a_{\rm clk}/2$.
If $T_{\rm clk}\le16$, the elementary lower bound yields
\begin{equation*}
\beta_p\ge(1-e^{-1})\min\{1,mKp\}
\ge\frac{1-e^{-1}}{16}T_{\rm clk}\ge T_{\rm clk}/64.
\end{equation*}
If $T_{\rm clk}>16$, put $t=\lfloor T_{\rm clk}/16\rfloor$.
Then $T_{\rm clk}/32\le t\le K/16$ and $t\le3a_{\rm clk}/32$.
The elementary binomial estimate $\binom Kt\ge(K/t)^t$ and
$\log(1-p)\ge-2p$ imply
\begin{equation*}
p_{\rm tail}:=\Prob(Z_1\ge t)\ge\Prob(Z_1=t)
\ge(Kp/t)^t e^{-2Kp}.
\end{equation*}
Since $t/(Kp)\le(3/32)\bigl(\log(2m)/(Kp)\bigr)/\log\bigl(2+\log(2m)/(Kp)\bigr)\le2+\log(2m)/(Kp)$,
\begin{equation*}
t\log\bigl(t/(Kp)\bigr)+2Kp
\le t\max\{0,\log(t/(Kp))\}+2Kp
\le\frac{3\log(2m)}{32}+\frac{\log(2m)}{4}=\frac{11\log(2m)}{32}.
\end{equation*}
Here $m\ge2$: otherwise $T_{\rm clk}\le Kp<\log(2m)/8<1$.
Thus $\log m\ge\log(2m)/2$ and $mp_{\rm tail}\ge1$.
Independence now gives
\begin{equation*}
\beta_p\ge t[1-(1-p_{\rm tail})^m]
\ge(1-e^{-1})t
\ge\frac{1-e^{-1}}{32}T_{\rm clk}\ge T_{\rm clk}/64.
\end{equation*}
Inequality~\eqref{eq:clock:reciprocal} follows by multiplying the two
bounds by $N$ and adding one.
\end{proof}
\begin{lemma}[Comparison with the nonlinear upper bound]\label{lem:clock:constructive}
For $p=1/\rho$, $\rho\ge1$, recall that
\[
 \SchedDepth=\max\{K/\rho,\min\{mK/\rho,1\}\}.
\]
Then
\begin{equation*}
 \beta_{1/\rho}\le4\left(1+\frac{\log(2m)}{\log(2+\log(2m))}\right)\SchedDepth.
\end{equation*}
The dependence on $m$ is sharp for this comparison: as $m\to\infty$,
choosing $K=\rho=\lceil\log(2m)\rceil$ gives
\[
 \frac{\beta_{1/\rho}}{\SchedDepth}
 =\Theta\!\left(1+\frac{\log(2m)}{\log(2+\log(2m))}\right).
\]
\end{lemma}
\begin{proof}
Write $K_{\rm eff}=K/\rho$. If $mK_{\rm eff}\le1$, then
$\SchedDepth=mK_{\rm eff}$ and
Lemma~\ref{lem:clock:elementary} gives $\beta_{1/\rho}\le \SchedDepth$.
Suppose $mK_{\rm eff}>1$. If $K_{\rm eff}\le1$, then
$\SchedDepth=1$ and
Theorem~\ref{thm:clock:asymptotic} gives
\[
 \beta_{1/\rho}\le4\left(K_{\rm eff}
 +\frac{\log(2m)}{\log(2+\log(2m)/K_{\rm eff})}\right)
 \le4\left(1+\frac{\log(2m)}{\log(2+\log(2m))}\right).
\]
If $K_{\rm eff}\ge1$, then $\SchedDepth=K_{\rm eff}$ and
\[
 \frac{\beta_{1/\rho}}{\SchedDepth}
 \le4\left(1+\frac{\log(2m)/K_{\rm eff}}
 {\log(2+\log(2m)/K_{\rm eff})}\right)
 \le4\left(1+\frac{\log(2m)}{\log(2+\log(2m))}\right).
\]
The last inequality follows because $x/\log(2+x)$ is increasing for
$x\ge0$: its derivative has the sign of
$\log(2+x)-x/(2+x)>0$.
For sharpness, take $K=\rho=\lceil\log(2m)\rceil$, so
$K_{\rm eff}=\SchedDepth=1$.
For all sufficiently large $m$, both caps $K$ and $m$ exceed
$1+\log(2m)/\log(2+\log(2m))$. The lower estimate in
Theorem~\ref{thm:clock:asymptotic} completes the proof.
\end{proof}
Since the nonlinear upper bound of Theorem~\ref{vi:main-monotone} has the
denominator $1+N\SchedDepth$ and its lower bound has
$1+N\BinDepth(m,K,1/\rho)$,
the lemma and Theorem~\ref{thm:clock:asymptotic} bound the ratio of the two
by $O\bigl(1+\log(2m)/\log(2+\log(2m))\bigr)$, and the second part shows that
this factor can occur.

\section{A statistical testing lemma}
\label{app:shared-probability}
This appendix isolates the two-point Gaussian testing bound used in the
additive-noise lower bounds. The result remains valid after arbitrary
processing of the transcript.

\begin{lemma}[Gaussian observations and adaptive transcripts]
\label{lem:gaussian-testing}
Let $Q_{\rm tot}\in\mathbb N$, $\mu_{\rm obs}\ge0$, and $\sigma>0$.
Under two hypotheses $s_{\rm hyp}\in\{-1,1\}$, observe independent
random variables
$Y_j\sim\mathcal N(s_{\rm hyp}\mu_{\rm obs},\sigma^2)$,
$j=1,\ldots,Q_{\rm tot}$. Let $\mathbb P_+$ and $\mathbb P_-$ denote
probability under $s_{\rm hyp}=1$ and $s_{\rm hyp}=-1$, respectively.
Let $\mathcal T$ be any measurable function
of these observations and an independent random seed whose law does not
depend on $s_{\rm hyp}$. Every test
$\widehat s_{\rm hyp}=\widehat s_{\rm hyp}(\mathcal T)\in\{-1,1\}$
satisfies
\[
 \frac{\mathbb P_{+}\{\widehat s_{\rm hyp}=-1\}
       +\mathbb P_{-}\{\widehat s_{\rm hyp}=1\}}2
 \ge\frac12\left(1-
       \frac{\mu_{\rm obs}\sqrt{Q_{\rm tot}}}{\sigma}\right).
\]
In particular, the right side is at least $3/8$ if
$Q_{\rm tot}\mu_{\rm obs}^2/\sigma^2\le1/16$, and at least $7/16$ if
$Q_{\rm tot}\mu_{\rm obs}^2/\sigma^2\le1/64$.
\end{lemma}
\begin{proof}
Write $P_{s_{\rm hyp}}$ for the product law of the $Q_{\rm tot}$
observations. Direct
integration of the Gaussian log-likelihood ratio gives
\[
 D_{\rm KL}(P_+\Vert P_-)
 =Q_{\rm tot}\frac{(2\mu_{\rm obs})^2}{2\sigma^2}
 =\frac{2Q_{\rm tot}\mu_{\rm obs}^2}{\sigma^2}.
\]
Independent randomization does not change this divergence. Data
processing and Pinsker's inequality therefore imply
\[
 \operatorname{TV}(\mathcal L_+(\mathcal T),
                         \mathcal L_-(\mathcal T))
 \le\operatorname{TV}(P_+,P_-)
 \le\sqrt{\tfrac12D_{\rm KL}(P_+\Vert P_-)}
 =\frac{\mu_{\rm obs}\sqrt{Q_{\rm tot}}}{\sigma}.
\]
For $A=\{\widehat s_{\rm hyp}=1\}$, the average error is
$[1-P_+(A)+P_-(A)]/2\ge[1-\operatorname{TV}]/2$.
\end{proof}

\begin{remark}[Adaptive queries]\label{rem:gaussian-adaptive}
Suppose every response has the form
$g(x)=r_{\rm known}(x)+s_{\rm hyp}\mu_{\rm obs}+\sigma\zeta$, where
$r_{\rm known}$ is known and $\zeta\sim\mathcal N(0,1)$ is independent,
possibly as one coordinate of a vector whose other coordinates are known
deterministic functions of the query. (It is not enough that they do not
depend on $s_{\rm hyp}$: an extra coordinate equal to $\zeta$ would reveal
$s_{\rm hyp}$ from one call.) Then
$Y=g(x)-r_{\rm known}(x)
\sim\mathcal N(s_{\rm hyp}\mu_{\rm obs},\sigma^2)$ independently of the
query, and the transcript of an algorithm with at most $Q_{\rm tot}$ calls
is a measurable function of $Y_1,\ldots,Y_{Q_{\rm tot}}$ and of the
algorithm's seed. Hence
Lemma~\ref{lem:gaussian-testing} applies to every adaptive algorithm.
\end{remark}

\section{One step of stochastic extragradient}
\label{app:extragradient-step}

All upper bounds in this paper are built from extragradient steps with
batched responses. This appendix collects the two facts about a single
step that are used repeatedly: an exact identity and its expected form.

Let $F$ be $L$-Lipschitz and fix a step size $\gamma>0$. One step of
extragradient with errors $e$ and $\delta$ is
\begin{equation}\label{eq:eg-step}
 y=x-\gamma\bigl(F(x)+e\bigr),\qquad
 x^+=x-\gamma\bigl(F(y)+\delta\bigr).
\end{equation}
Write $F_x=F(x)$ and $q_{\rm eg}=\gamma^2L^2$.

\begin{lemma}[Extragradient identity]\label{nn:identity-lemma}
For every $u\in\mathbb R^d$,
\begin{equation}\label{nn:identity}
 2\gamma\langle F(y),y-u\rangle
 =\|x-u\|^2-\|x^+-u\|^2+\gamma^2\mathcal R_{\rm eg}
 -2\gamma\bigl\langle\delta,\,x-\gamma F(y)-u\bigr\rangle,
\end{equation}
where
$\mathcal R_{\rm eg}=-\|F_x+e\|^2+\|F(y)-F_x-e\|^2+\|\delta\|^2$.
Moreover,
\begin{equation}\label{nn:pathwise-residual}
 \gamma^2\mathcal R_{\rm eg}\le-(1-2q_{\rm eg})\gamma^2\|F_x+e\|^2
       +2\gamma^2\|e\|^2+\gamma^2\|\delta\|^2.
\end{equation}
\end{lemma}
\begin{proof}
Expanding the square,
\[
 \|x-u\|^2-\|x^+-u\|^2
 =2\gamma\langle F(y)+\delta,x-u\rangle-\gamma^2\|F(y)+\delta\|^2.
\]
Since $x-y=\gamma(F_x+e)$,
\[
 2\gamma\langle F(y),y-u\rangle
 =2\gamma\langle F(y),x-u\rangle-2\gamma^2\langle F(y),F_x+e\rangle .
\]
Subtracting the two displays and expanding
$\|F(y)+\delta\|^2=\|F(y)\|^2+2\langle F(y),\delta\rangle+\|\delta\|^2$,
\[
 2\gamma\langle F(y),y-u\rangle
 -\bigl(\|x-u\|^2-\|x^+-u\|^2\bigr)
 =\gamma^2\bigl(\|F(y)\|^2-2\langle F(y),F_x+e\rangle+\|\delta\|^2\bigr)
 -2\gamma\langle\delta,x-\gamma F(y)-u\rangle .
\]
The bracket equals
$\|F(y)-F_x-e\|^2-\|F_x+e\|^2+\|\delta\|^2=\mathcal R_{\rm eg}$, which
proves~\eqref{nn:identity}. For~\eqref{nn:pathwise-residual}, Lipschitz
continuity gives $\|F(y)-F_x\|\le L\|y-x\|=\gamma L\|F_x+e\|$, hence
\[
 \|F(y)-F_x-e\|^2\le2\|F(y)-F_x\|^2+2\|e\|^2
 \le2q_{\rm eg}\|F_x+e\|^2+2\|e\|^2 .
\]
\end{proof}

Now let the responses be random. Let $\mathcal F$ be the history before
the step, and let $\mathcal F^+$ also contain the first batch. Each
error is the average of $b$ independent centered errors of single
responses, which satisfy Assumption~\ref{ass:mixed}. Put
\[
 \alpha_{\rm rel}=\frac{\rho-1}b,\qquad
 \sigma_b^2=\frac{\sigma^2}b .
\]
Then
\begin{equation}\label{eq:eg-moments}
 \begin{aligned}
 \E[e\mid\mathcal F]&=0,&
 \E[\|e\|^2\mid\mathcal F]&\le \alpha_{\rm rel}\|F(x)\|^2+\sigma_b^2,\\
 \E[\delta\mid\mathcal F^+]&=0,&
 \E[\|\delta\|^2\mid\mathcal F^+]&\le \alpha_{\rm rel}\|F(y)\|^2+\sigma_b^2.
 \end{aligned}
\end{equation}

\begin{lemma}[Expected extragradient step]\label{nn:one-step}
\label{vi:seg-one-step}
Assume~\eqref{eq:eg-moments}, $\alpha_{\rm rel}\le1/4$ and
$0<\gamma\le1/(4L)$. Then
for every $u$ that is $\mathcal F$-measurable,
\begin{align}
 \E[\|x^+-u\|^2\mid\mathcal F]
 &\le\|x-u\|^2-2\gamma\,\E[\langle F(y),y-u\rangle\mid\mathcal F]
       -\frac7{128}\gamma^2\|F(x)\|^2+\frac{69}{32}\gamma^2\sigma_b^2,
       \label{nn:energy}\\
 \E[\|\delta\|^2\mid\mathcal F]
 &\le\frac{69}{32}\alpha_{\rm rel}\|F(x)\|^2
      +\frac{33}{32}\sigma_b^2.
       \label{nn:corrector}
\end{align}
\end{lemma}
\begin{proof}
All expectations are conditional on $\mathcal F$. Let
$v_e=\E\|e\|^2\le \alpha_{\rm rel}\|F_x\|^2+\sigma_b^2$.
Since $e$ is centered and
$\|F(y)-F_x\|\le\gamma L\|F_x+e\|$,
\[
 \E\|F_x+e\|^2=\|F_x\|^2+v_e,\qquad
 \E\|F(y)-F_x\|^2\le q_{\rm eg}(\|F_x\|^2+v_e),\qquad
 \E\|F(y)\|^2\le2\|F_x\|^2+2q_{\rm eg}(\|F_x\|^2+v_e).
\]

\emph{Corrector.} By~\eqref{eq:eg-moments} and the last bound,
\[
 \begin{aligned}
 \E\|\delta\|^2&\le \alpha_{\rm rel}\E\|F(y)\|^2+\sigma_b^2\\
 &\le2\alpha_{\rm rel}(1+q_{\rm eg})\|F_x\|^2
    +2\alpha_{\rm rel}q_{\rm eg}v_e+\sigma_b^2\\
 &\le2\alpha_{\rm rel}\bigl(1+q_{\rm eg}(1+\alpha_{\rm rel})\bigr)
       \|F_x\|^2+(1+2\alpha_{\rm rel}q_{\rm eg})\sigma_b^2.
 \end{aligned}
\]
For $\alpha_{\rm rel}\le1/4$ and $q_{\rm eg}\le1/16$ we have
$2(1+q_{\rm eg}(1+\alpha_{\rm rel}))\le69/32$ and
$1+2\alpha_{\rm rel}q_{\rm eg}\le33/32$, which is~\eqref{nn:corrector}.

\emph{Energy.} Using $\|r-e\|^2\le2\|r\|^2+2\|e\|^2$ with
$r=F(y)-F_x$ and the corrector bound,
\begin{align*}
 \E\mathcal R_{\rm eg}
 &\le-\bigl(\|F_x\|^2+v_e\bigr)
       +2q_{\rm eg}(\|F_x\|^2+v_e)+2v_e
       +\alpha_{\rm rel}\bigl(2\|F_x\|^2
          +2q_{\rm eg}(\|F_x\|^2+v_e)\bigr)+\sigma_b^2\\
 &=\bigl(-1+2\alpha_{\rm rel}+2q_{\rm eg}(1+\alpha_{\rm rel})\bigr)
       \|F_x\|^2
   +\bigl(1+2q_{\rm eg}(1+\alpha_{\rm rel})\bigr)v_e+\sigma_b^2\\
 &\le\bigl(-1+3\alpha_{\rm rel}
       +2q_{\rm eg}(1+\alpha_{\rm rel})^2\bigr)\|F_x\|^2
       +2\bigl(1+q_{\rm eg}(1+\alpha_{\rm rel})\bigr)\sigma_b^2,
\end{align*}
where the last step uses
$v_e\le \alpha_{\rm rel}\|F_x\|^2+\sigma_b^2$.
For $\alpha_{\rm rel}\le1/4$ and $q_{\rm eg}\le1/16$,
\[
 1-3\alpha_{\rm rel}-2q_{\rm eg}(1+\alpha_{\rm rel})^2
 \ge1-\frac34-\frac{2}{16}\cdot\frac{25}{16}=\frac7{128},
 \qquad 2\bigl(1+q_{\rm eg}(1+\alpha_{\rm rel})\bigr)\le\frac{69}{32}.
\]
The vector $x-\gamma F(y)-u$ is $\mathcal F^+$-measurable and $\delta$ is
centered given $\mathcal F^+$, so the last term of~\eqref{nn:identity}
has zero conditional mean. Taking the conditional expectation
of~\eqref{nn:identity} proves~\eqref{nn:energy}.
\end{proof}

\section{A fixed affine hard distribution for arbitrary queries}
\label{app:affine}

In this section, we develop a construction that will be used to prove
lower bounds in subsequent sections. The construction gives a single
distribution of affine monotone operators that is hard for every
algorithm with a prescribed number of parallel batches.

\subsection{Matrix queries in parallel batches}
\label{app:aff:haar}

Each query to the affine field constructed below will require two matrix
products. We first describe the information available from such products.
Fix an integer $d_{\rm sym}\ge1$ for the matrix dimension; the VI
constructed in Section~\ref{app:aff:hard-dimension} has dimension
$2d_{\rm sym}$.

\begin{definition}[Batched matrix-query algorithm]
\label{def:aff:batch-model}
Let $A\in\mathbb R^{d_{\rm sym}\times d_{\rm sym}}$ be unknown,
and let $b_{\rm rhs}\in\mathbb R^{d_{\rm sym}}$ be given to the algorithm.
A query at $v$ returns the exact vector $Av$.
Fix integers $w_{\rm batch}\ge1$, $T_{\rm depth}\ge0$, and $c_{\rm out}\ge0$.
In each batch $r=1,\ldots,T_{\rm depth}$, the algorithm chooses
$w_{\rm batch}$ vectors $v_{r,1},\ldots,v_{r,w_{\rm batch}}$ and then receives
$Av_{r,1},\ldots,Av_{r,w_{\rm batch}}$. All queries in a batch are chosen before
any answer from that batch is received. They may depend on $b_{\rm rhs}$ and
all earlier queries and answers. After the last batch, the algorithm
returns $c_{\rm out}$ vectors $z_1,\ldots,z_{c_{\rm out}}$.
Thus $w_{\rm batch}$ is the batch width and $T_{\rm depth}$ is its sequential depth.

All query and output rules are Borel measurable functions of the
available information and internal randomness independent of the
instance. Their norms are unrestricted but finite. Batches containing
fewer than $w_{\rm batch}$ queries are padded with zeros.
\end{definition}

Fix $\Lambda\in\mathbb R^{d_{\rm sym}\times d_{\rm sym}}$ and
$b_{\rm spec}\in\mathbb R^{d_{\rm sym}}$. Consider
\begin{equation}
 U\sim\operatorname{Haar}(O(d_{\rm sym})),\qquad
 A=U\Lambda U^\top,\qquad b_{\rm rhs}=Ub_{\rm spec}.
 \label{eq:aff:marked-matrix}
\end{equation}
Here $U$ is uniformly distributed over the orthogonal matrices:
its law is unchanged by multiplication by any fixed orthogonal matrix.
The algorithm may know $(\Lambda,b_{\rm spec})$, but not $U$.
We call $U^\top v$ the reference coordinates of a vector $v$;
in particular, $b_{\rm spec}$ gives the reference coordinates of $b_{\rm rhs}$.
Write $\Pi_E$ for orthogonal projection onto a subspace $E$.

The reduction below uses the conditional-rotation argument of
\citet[Lemmas~38--39]{chewi2023} and
\citet[Lemma~6.3 and Appendix~A]{bakshi2023}.
We give the adaptation for parallel batches, the supplied vector $b_{\rm rhs}$,
and arbitrary outputs; the standard Haar-conditioning argument is
used without reproving it.

\begin{lemma}[Polynomial representation of batched queries]
\label{lem:aff:batched}
Consider any algorithm in Definition~\ref{def:aff:batch-model} applied to
\eqref{eq:aff:marked-matrix}. Its original experiment can be realized
jointly with $w_{\rm batch}T_{\rm depth}+c_{\rm out}$ independent standard Gaussian vectors
in $\mathbb R^{d_{\rm sym}}$,
\[
 v^{\rm G}_{r,i},\quad 1\le r\le T_{\rm depth},\ 1\le i\le w_{\rm batch},
 \qquad
 v^{\rm G}_{{\rm out},i},\quad 1\le i\le c_{\rm out},
\]
so that, almost surely, each output satisfies
\begin{align}
 U^\top z_\ell\in\mathcal K_{T_{\rm depth}}:={}&
 \operatorname{span}\{\Lambda^j b_{\rm spec}:0\le j\le T_{\rm depth}\}
+\sum_{r=1}^{T_{\rm depth}}\sum_{i=1}^{w_{\rm batch}}
   \operatorname{span}\{\Lambda^j v^{\rm G}_{r,i}:0\le j\le T_{\rm depth}-r+1\}
 \nonumber\\
 &+\operatorname{span}\{v^{\rm G}_{{\rm out},i}:1\le i\le c_{\rm out}\}.
 \label{eq:aff:krylov}
\end{align}
The joint law of the instance, queries, answers, and outputs is unchanged.
The matrix $U$ has Haar law and is independent of the algorithm's
internal randomness. The Gaussian collection has its independent
product law and is also independent of the algorithm's internal
randomness. The collection and $U$ may depend on one another.
\end{lemma}

\begin{proof}
Fix the algorithm's internal random choices. We construct $U$ by
specifying its action only as needed. If that action is specified on
$E$, the remaining map from $E^\perp$ to $(UE)^\perp$ is conditionally
uniform. We use the following standard Haar property from the cited
conditioning argument: the inverse images of a fixed orthonormal
frame in $(UE)^\perp$ form a uniform orthonormal frame in $E^\perp$;
after these images are specified, the remaining map is still uniform.
Such an inverse frame is obtained by projecting independent standard
Gaussian vectors onto $E^\perp$ and applying Gram--Schmidt
(successive orthogonal projection and normalization).
The same property applies to images of a reference-coordinate frame.

\emph{Initial information.}
Set $E_0=\operatorname{span}\{b_{\rm spec}\}$.
For $b_{\rm spec}\ne0$, draw $b_{\rm rhs}$ uniformly on the sphere of radius
$\|b_{\rm spec}\|$ and specify $Ub_{\rm spec}=b_{\rm rhs}$.
For $b_{\rm spec}=0$, set $b_{\rm rhs}=0$ and $E_0=\{0\}$.
In either case this is the original law of the supplied vector,
and the remaining rotation is conditionally uniform.

\emph{One batch.}
Suppose $U$ is specified on $E=E_{r-1}$, and put $S=UE$.
The queries $v_{r,1},\ldots,v_{r,w_{\rm batch}}$ are already determined.
Choose an orthonormal basis $u_1,\ldots,u_k$ of
\[
 \operatorname{span}\{(I-\Pi_S)v_{r,i}:1\le i\le w_{\rm batch}\},
 \qquad k\le\min\{w_{\rm batch},d_{\rm sym}-\dim E\}.
\]
Draw $v^{\rm G}_{r,1},\ldots,v^{\rm G}_{r,w_{\rm batch}}$
independently of everything generated
so far. Project the first $k$ draws onto $E^\perp$ and orthonormalize
them to obtain $\widetilde u_1,\ldots,\widetilde u_k$.
Specify $U\widetilde u_j=u_j$. Since
\[
 \widetilde u_j\in E+\operatorname{span}\{v^{\rm G}_{r,1},\ldots,v^{\rm G}_{r,j}\},
\]
every query now has specified reference coordinates satisfying
\begin{equation}
 \widetilde v_{r,i}:=U^\top v_{r,i}
 \in E+\operatorname{span}\{v^{\rm G}_{r,1},\ldots,
       v^{\rm G}_{r,w_{\rm batch}}\}.
 \label{eq:aff:preimage}
\end{equation}
Conditional on all generated variables, including the full Gaussian
draws, leave the unexposed action uniformly distributed subject to the
specified action; generate its values only when needed. Thus unused
Gaussian components impose no further constraint on that action.

To answer the batch, reveal the images under $U$ of
$\Lambda\widetilde v_{r,1},\ldots,\Lambda\widetilde v_{r,w_{\rm batch}}$
according to this conditional law. Since the $u_j$ span the perpendicular
query components,
\[
 E+\operatorname{span}\{\widetilde u_j:1\le j\le k\}
 =E+\operatorname{span}\{\widetilde v_{r,i}:1\le i\le w_{\rm batch}\}.
\]
The specified subspace therefore becomes
\begin{equation}
 E_r=E+\operatorname{span}
       \{\widetilde v_{r,i},\Lambda\widetilde v_{r,i}:1\le i\le w_{\rm batch}\},
 \label{eq:aff:exposed}
\end{equation}
and the returned vectors are exactly
$U\Lambda\widetilde v_{r,i}=Av_{r,i}$.
The remaining action stays conditionally uniform.
Equations~\eqref{eq:aff:preimage}--\eqref{eq:aff:exposed} give
\[
 E_r\subseteq E_{r-1}+\Lambda E_{r-1}
       +\operatorname{span}\{v^{\rm G}_{r,i},\Lambda v^{\rm G}_{r,i}:
          1\le i\le w_{\rm batch}\}.
\]
Induction from $E_0=\operatorname{span}\{b_{\rm spec}\}$ yields
\[
 E_r\subseteq
 \operatorname{span}\{\Lambda^j b_{\rm spec}:0\le j\le r\}
 +\sum_{a=1}^r\sum_{i=1}^{w_{\rm batch}}
   \operatorname{span}\{\Lambda^j v^{\rm G}_{a,i}:0\le j\le r-a+1\}.
\]
The degree increases at most once per batch because all its queries
are fixed before any matrix product is returned.

\emph{Outputs and distribution.}
Apply the query construction to the output vectors with independently sampled
$v^{\rm G}_{{\rm out},1},\ldots,v^{\rm G}_{{\rm out},c_{\rm out}}$,
without applying
$\Lambda$ or returning further answers. Then
\begin{equation*}
 U^\top z_\ell\in E_{T_{\rm depth}}+
   \operatorname{span}\{v^{\rm G}_{{\rm out},i}:1\le i\le c_{\rm out}\},
\end{equation*}
which proves~\eqref{eq:aff:krylov}. Complete $U$ uniformly on the
remaining complements. If no unexposed direction remains, or if the
projected queries are linearly dependent, fewer Gaussian directions
are used; all stated Gaussian vectors are still drawn independently.
Bases are chosen by Gram--Schmidt in the displayed order, discarding
zero residuals, so all constructions are measurable.

Each step uses the conditional law of the original Haar rotation.
Consequently, for every fixed realization of the algorithm's internal
randomness, the completed $U$ has Haar law and the instance, queries,
answers, and outputs have their original joint law. The Gaussian draws
have their independent product law. Neither marginal law depends on
the fixed internal random choices; averaging over those choices proves
the independence assertions in the statement. No independence between
$U$ and the Gaussian collection is asserted or needed.
\end{proof}

For the residual bound below, we enlarge the Gaussian part of
\eqref{eq:aff:krylov} by allowing arbitrary powers of $\Lambda$ on each
Gaussian vector. The bound $T_{\rm depth}$ on the degree of the polynomial
multiplying $b_{\rm spec}$ is retained.

\subsection{A residual that finite-depth information cannot remove}
\label{app:aff:polynomial}

The following scalar inequality will give a lower bound on the residual
of a matrix equation. Its weights specify the right-hand side of that equation.

Fix $L>0$ and an odd integer $n\ge1$. Define
\begin{equation*}
 \lambda_j=L\cos(j\pi/n),\qquad j=0,\ldots,n.
\end{equation*}
These are $n+1$ distinct nonzero numbers, ordered as
$L=\lambda_0>\cdots>\lambda_n=-L$. For each $j$, define the polynomial
\[
 \ell_j(z)=\prod_{\substack{0\le i\le n\\i\ne j}}
       \frac{z-\lambda_i}{\lambda_j-\lambda_i}.
\]
It has degree $n$ and satisfies
$\ell_j(\lambda_i)=\mathbf 1_{\{i=j\}}$.

\begin{lemma}[A weighted bound on Chebyshev nodes]
\label{lem:aff:chebyshev}
Let
\[
 a_j=\ell_j'(0),\qquad
 a_{\rm sum}=\sum_{j=0}^n|a_j|,\qquad
 \omega_j=\frac{|a_j|}{a_{\rm sum}}.
\]
Then $\omega_j>0$, $\sum_{j=0}^n\omega_j=1$, and
$a_{\rm sum}=n/L$. Every real polynomial $\psi$ with
$\deg\psi\le n$ and $\psi'(0)=1$ satisfies
\begin{equation}
 \sum_{j=0}^n\omega_j\psi(\lambda_j)^2\ge\frac{L^2}{n^2}.
 \label{eq:aff:derivative-certificate}
\end{equation}
\end{lemma}

\begin{proof}
Set $\phi(z)=\prod_{j=0}^n(z-\lambda_j)$. Since the nodes occur in
nonzero pairs $\lambda_{n-j}=-\lambda_j$, this polynomial is even and
$\phi(0)\ne0$, $\phi'(0)=0$. Differentiating
$\ell_j(z)=\phi(z)/[(z-\lambda_j)\phi'(\lambda_j)]$ at zero gives
\begin{equation*}
a_j=\frac{-\lambda_j\phi'(0)-\phi(0)}{\lambda_j^2\phi'(\lambda_j)}
    =-\frac{\phi(0)}{\lambda_j^2\phi'(\lambda_j)}\ne0.
\end{equation*}
The product
$\phi'(\lambda_j)=\prod_{i\ne j}(\lambda_j-\lambda_i)$ has exactly
$j$ negative factors. Hence
\[
\operatorname{sign}\phi'(\lambda_j)=(-1)^j,
\qquad
\operatorname{sign}\bigl(a_j(-1)^j\bigr)
  =\operatorname{sign}(-\phi(0)),
\]
so all the terms $a_j(-1)^j$ have the same sign.

Define the Chebyshev polynomials by
$T_0(z)=1$, $T_1(z)=z$, and
$T_{k+1}(z)=2zT_k(z)-T_{k-1}(z)$ for $k\ge1$.
The cosine addition formula gives
$T_n(\cos\theta)=\cos(n\theta)$, so
\[
T_n(\lambda_j/L)=(-1)^j,
\qquad T_n'(0)=n\sin(n\pi/2),
\qquad |T_n'(0)|=n.
\]
For every polynomial $\psi$ of degree at most $n$, the interpolation identity
\[
\psi(z)=\sum_{j=0}^n\psi(\lambda_j)\ell_j(z)
\]
holds: the two sides have degree at most $n$ and agree at $n+1$ distinct
points. We can calculate $a_{\text{sum}}$ by applying this identity to $\psi(z)=T_n(z/L)$ and differentiating at zero:
\[
\frac nL
=\left|\frac{T_n'(0)}L\right|
=\left|\sum_{j=0}^na_j(-1)^j\right|
=\sum_{j=0}^n|a_j|=a_{\rm sum}.
\]
It follows that the $\omega_j$ are positive and sum to one. For an
admissible $\psi$, the same interpolation identity and Cauchy--Schwarz give
\begin{align*}
1
&=\psi'(0)^2
  =\left(\sum_{j=0}^na_j\psi(\lambda_j)\right)^2\le\left(\sum_{j=0}^n\frac{a_j^2}{\omega_j}\right)
      \left(\sum_{j=0}^n\omega_j\psi(\lambda_j)^2\right)=a_{\rm sum}^2\sum_{j=0}^n\omega_j\psi(\lambda_j)^2.
\end{align*}
This proves~\eqref{eq:aff:derivative-certificate}.
\end{proof}
\subsection{The projector estimate}
\label{sec:dimension-refinement}

The next lemma bounds the part of a fixed vector that can be removed
by projection onto a random Gaussian span.

\begin{lemma}[A uniform Gaussian projection bound]
\label{lem:dimref:projector}
Let $t,h\ge1$ and $s_{\rm G}\ge0$ be integers. For each
$j=1,\ldots,t$, let
\[
S_j=\operatorname{span}\{v^{\rm G}_{j,1},\ldots,v^{\rm G}_{j,s_{\rm G}}\}
\subseteq\mathbb R^h,
\]
where $v^{\rm G}_{j,1},\ldots,v^{\rm G}_{j,s_{\rm G}}$ are independent
$\mathcal N(0,I_h)$ vectors for each fixed $j$.
The collections corresponding to different indices $j$ may be
dependent. Let $e_1,\ldots,e_t\in\mathbb R^h$ be deterministic unit
vectors, and define
\[
p_j^{\rm proj}=\|\Pi_{S_j}e_j\|^2\in[0,1],
\]
where $\Pi_{S_j}$ denotes the orthogonal projector onto $S_j$.

If $h\ge16s_{\rm G}$, then
\begin{equation}
\mathbb{P}\!\left\{\max_{1\le j\le t}p_j^{\rm proj}>\frac12\right\}
\le2te^{-h/32}.
\label{eq:dimref:tail}
\end{equation}
Consequently, if
\begin{equation}
h\ge\max\{16s_{\rm G},32\log(8t)\},
\label{eq:dimref:dimension-condition}
\end{equation}
then
\begin{equation}
\mathbb E\sqrt{1-\max_{1\le j\le t}p_j^{\rm proj}}
\ge\frac3{4\sqrt2}>\frac12.
\label{eq:dimref:factor}
\end{equation}
For $s_{\rm G}=0$, we use the convention $S_j=\{0\}$.
Then $p_j^{\rm proj}=0$, and the expectation
in~\eqref{eq:dimref:factor} equals one for every $h\ge1$.
\end{lemma}

\begin{proof}
The case $s_{\rm G}=0$ is immediate. Assume henceforth that
$s_{\rm G}\ge1$ and $h\ge16s_{\rm G}$.

\emph{Step 1: distribution of a single projection.}
Fix $j\in\{1,\ldots,t\}$. Since $s_{\rm G}<h$, the vectors
$v^{\rm G}_{j,1},\ldots,v^{\rm G}_{j,s_{\rm G}}$ are linearly independent
almost surely.
Indeed, conditional on the preceding vectors, each new Gaussian vector
has probability zero of belonging to their proper linear span.
Thus $\dim S_j=s_{\rm G}$ almost surely.

For every deterministic orthogonal matrix $V$, rotational invariance
of the Gaussian distribution gives
\[
(Vv^{\rm G}_{j,1},\ldots,Vv^{\rm G}_{j,s_{\rm G}})
\stackrel{\rm d}=
(v^{\rm G}_{j,1},\ldots,v^{\rm G}_{j,s_{\rm G}}).
\]
Hence $VS_j\stackrel{\rm d}=S_j$, so $S_j$ is uniformly distributed
over the $s_{\rm G}$-dimensional subspaces of $\mathbb R^h$.
If $U_{\rm Haar}$ is a Haar-distributed orthogonal matrix of order $h$,
then
\[
\Pi_{S_j}\stackrel{\rm d}=
U_{\rm Haar}\begin{pmatrix}
I_{s_{\rm G}}&0\\
0&0
\end{pmatrix}U_{\rm Haar}^\top.
\]
Since an orthogonal projector satisfies
$\Pi_{S_j}^\top\Pi_{S_j}=\Pi_{S_j}$, it follows that
\[
p_j^{\rm proj}
=e_j^\top\Pi_{S_j}e_j
\stackrel{\rm d}=
\sum_{i=1}^{s_{\rm G}}(U_{\rm Haar}^\top e_j)_i^2.
\]

Because $e_j$ is a deterministic unit vector,
$U_{\rm Haar}^\top e_j$ is uniform
on the unit sphere. A standard Gaussian vector normalized by its norm
has the same distribution. Therefore, for independent
$\zeta_1,\ldots,\zeta_h\sim\mathcal N(0,1)$,
\[
U_{\rm Haar}^\top e_j\stackrel{\rm d}=
\frac{(\zeta_1,\ldots,\zeta_h)^\top}
     {\sqrt{\sum_{i=1}^h\zeta_i^2}},
\qquad
p_j^{\rm proj}\stackrel{\rm d}=\frac{S_{\rm num}}{S_{\rm den}},
\]
where
\[
S_{\rm num}:=\sum_{i=1}^{s_{\rm G}}\zeta_i^2,
\qquad
S_{\rm den}:=\sum_{i=1}^h\zeta_i^2>0
\quad\text{almost surely}.
\]
These identities describe the marginal distribution for the fixed
index $j$; they impose no independence across the different subspaces.

\emph{Step 2: reduction to two scalar tail bounds.}
On the event $\{S_{\rm num}\le h/4\}\cap\{S_{\rm den}\ge h/2\}$, we have
\[
\frac{S_{\rm num}}{S_{\rm den}}\le\frac{h/4}{h/2}=\frac12.
\]
Taking complements gives
\[
\{S_{\rm num}/S_{\rm den}>1/2\}
\subseteq
\{S_{\rm num}>h/4\}\cup\{S_{\rm den}<h/2\}.
\]
Thus
\[
\mathbb{P}\{p_j^{\rm proj}>1/2\}
\le
\mathbb{P}\{S_{\rm num}>h/4\}
+
\mathbb{P}\{S_{\rm den}<h/2\}.
\]
No independence between $S_{\rm num}$ and $S_{\rm den}$ is needed for this inequality.

\emph{Step 3: exponential bounds for the scalar tails.}
For $S_{\rm chi}=\sum_{i=1}^k\eta_i^2$, where
$\eta_1,\ldots,\eta_k$ are independent $\mathcal N(0,1)$ variables,
independence gives, for every $a<1/2$,
\begin{align*}
\mathbb E e^{aS_{\rm chi}}
&=\prod_{i=1}^k\mathbb E e^{a\eta_i^2}=\left(
  \frac1{\sqrt{2\pi}}
  \int_{\mathbb R}e^{-(1-2a)x^2/2}\,dx
  \right)^k
=(1-2a)^{-k/2}.
\end{align*}
The last equality follows by substituting
$u=\sqrt{1-2a}\,x$ in the integral.

Applying Markov's inequality to the nonnegative random variable
$e^{S_{\rm num}/4}$ yields
\begin{align*}
\mathbb{P}\{S_{\rm num}>h/4\}
&=\mathbb{P}\{e^{S_{\rm num}/4}>e^{h/16}\}\le e^{-h/16}\mathbb E e^{S_{\rm num}/4}=e^{-h/16}2^{s_{\rm G}/2}\le\exp\!\left(-\frac h{16}+\frac{h\log2}{32}\right)
\le e^{-h/32},
\end{align*}
where the last line uses $s_{\rm G}\le h/16$ and $\log2\le1$.

For the lower tail of $S_{\rm den}$, the map $z\mapsto e^{-z/2}$ is decreasing.
Consequently, Markov's inequality gives
\begin{align*}
\mathbb{P}\{S_{\rm den}<h/2\}
&=\mathbb{P}\{e^{-S_{\rm den}/2}>e^{-h/4}\}\le e^{h/4}\mathbb E e^{-S_{\rm den}/2}=e^{h/4}2^{-h/2}\\
&=\exp\!\left[-h\left(\frac{\log2}{2}-\frac14\right)\right]
\le e^{-h/12}
\le e^{-h/32}.
\end{align*}

Combining the two tail bounds, we obtain
\[
\mathbb{P}\{p_j^{\rm proj}>1/2\}\le2e^{-h/32}.
\]

\emph{Step 4: simultaneous control and the expectation bound.}
The union bound gives
\begin{align*}
\mathbb{P}\!\left\{\max_{1\le j\le t}p_j^{\rm proj}>\frac12\right\}
&=\mathbb{P}\!\left(
  \bigcup_{j=1}^t\{p_j^{\rm proj}>1/2\}
  \right)\\
&\le\sum_{j=1}^t\mathbb{P}\{p_j^{\rm proj}>1/2\}\\
&\le2te^{-h/32}.
\end{align*}
This proves~\eqref{eq:dimref:tail}, regardless of dependence among
the subspaces.

Under~\eqref{eq:dimref:dimension-condition},
\[
2te^{-h/32}
\le2te^{-\log(8t)}
=\frac14.
\]
Therefore the event
\[
\Omega_{\rm good}:=
\left\{\max_{1\le j\le t}p_j^{\rm proj}\le\frac12\right\}
\]
satisfies $\mathbb{P}(\Omega_{\rm good})\ge3/4$.
Since $0\le p_j^{\rm proj}\le1$ for every $j$, we have the pointwise bound
\[
\sqrt{1-\max_{1\le j\le t}p_j^{\rm proj}}
\ge\frac1{\sqrt2}\mathbf 1_{\Omega_{\rm good}}.
\]
Taking expectations yields
\[
\mathbb E\sqrt{1-\max_{1\le j\le t}p_j^{\rm proj}}
\ge\frac1{\sqrt2}\mathbb{P}(\Omega_{\rm good})
\ge\frac3{4\sqrt2}
>\frac12,
\]
which proves~\eqref{eq:dimref:factor}.
\end{proof}
\subsection{A residual bound from the projection estimate}

Fix $R_0>0$, an odd $n\ge1$, and an integer $h\ge1$. Set $t=n+1$ and
$d_{\rm sym}=th$. Choose an orthogonal decomposition
\[
 \mathbb R^{d_{\rm sym}}=V_0\oplus\cdots\oplus V_n,
 \qquad \dim V_j=h,
\]
and a fixed unit vector $e_j\in V_j$ for each $j$.
Using the nodes and weights from Lemma~\ref{lem:aff:chebyshev}, define
\[
 \Lambda|_{V_j}=\lambda_j I_{V_j},\qquad
 z_{\rm spec}=R_0\sum_{j=0}^n\sqrt{\omega_j}e_j,
 \qquad b_{\rm spec}=\Lambda z_{\rm spec}.
\]
Here $I_{V_j}$ is the identity operator on $V_j$. Each eigenvalue
$\lambda_j$ has multiplicity $h$;
$\|\Lambda\|_{\rm op}=L$, $\Lambda$ is invertible, and
$\|z_{\rm spec}\|=R_0$. The scalar inequality we shall use is
\begin{equation}
 \sum_{j=0}^n\omega_j\psi(\lambda_j)^2\ge\frac{L^2}{n^2}
 \quad\text{whenever}\quad\deg\psi\le n,\quad\psi'(0)=1.
 \label{eq:dimref:chebyshev}
\end{equation}

For a polynomial $p$, the matrix $p(\Lambda)$ is obtained by replacing
the variable by $\Lambda$ and the constant term by that constant times
the identity matrix.

\begin{proposition}[Residual under arbitrary Gaussian polynomial information]
\label{prop:dimref:residual}
For the preceding construction, let
$v^{\rm G}_1,\ldots,v^{\rm G}_{s_{\rm G}}$ be
independent standard Gaussian vectors in $\mathbb R^{d_{\rm sym}}$,
where $s_{\rm G}\ge0$ is an integer. Let $d_0\ge0$ be an integer with
$d_0+2\le n$. Consider any possibly random vector of the form
\[
 v=p(\Lambda)b_{\rm spec}
   +\sum_{i=1}^{s_{\rm G}}p_i(\Lambda)v_i^{\rm G},
 \qquad\deg p\le d_0.
\]
The degrees of $p_i$ are unrestricted, and all polynomial coefficients
may depend on the Gaussian vectors and additional randomness.
If~\eqref{eq:dimref:dimension-condition} holds, then
\begin{equation}
 \mathbb E\|\Lambda v+b_{\rm spec}\|
 \ge\frac{3LR_0}{4\sqrt2\,n}>\frac{LR_0}{2n}.
 \label{eq:dimref:residual}
\end{equation}
The weaker lower bound $LR_0/(2n)$ also holds under the alternative
condition $h\ge2ts_{\rm G}$.
\end{proposition}

\begin{proof}
Define
\[
 S_j=\operatorname{span}\{\Pi_{V_j}v_i^{\rm G}:1\le i\le s_{\rm G}\}
 \subseteq V_j,\qquad
 \Pi_j=\Pi_{S_j},\qquad
 p_j^{\rm proj}=\|\Pi_je_j\|^2.
\]
For each fixed $j$, the vectors $\Pi_{V_j}v_i^{\rm G}$ are independent
standard
Gaussians in $V_j$. Because $\Lambda$ acts as $\lambda_jI$ on this
subspace, every real polynomial $q$ satisfies
\[
 \Pi_{V_j}q(\Lambda)y=q(\lambda_j)\Pi_{V_j}y
 \quad\text{for every }y\in\mathbb R^{d_{\rm sym}}.
\]
Therefore
\begin{align*}
 \Pi_{V_j}(\Lambda v+b_{\rm spec})
 &=R_0\sqrt{\omega_j}
        [\lambda_j^2p(\lambda_j)+\lambda_j]e_j
   +\sum_{i=1}^{s_{\rm G}}\lambda_jp_i(\lambda_j)\Pi_{V_j}v_i^{\rm G}.
\end{align*}
The sum belongs to $S_j$, irrespective of its coefficients and degrees.
Applying $I-\Pi_j$ gives
\[
 (I-\Pi_j)\Pi_{V_j}(\Lambda v+b_{\rm spec})
 =R_0\sqrt{\omega_j}\psi(\lambda_j)(I-\Pi_j)e_j,
 \qquad \psi(\lambda)=\lambda+\lambda^2p(\lambda).
\]
For every realization, $\deg\psi\le d_0+2\le n$ and $\psi'(0)=1$.
Since $\|(I-\Pi_j)e_j\|^2=1-p_j^{\rm proj}$, orthogonality of the $V_j$
and~\eqref{eq:dimref:chebyshev} imply
\begin{equation}
 \begin{aligned}
 \|\Lambda v+b_{\rm spec}\|^2
 &=\sum_{j=0}^n\|\Pi_{V_j}(\Lambda v+b_{\rm spec})\|^2\\
 &\ge\sum_{j=0}^n\|(I-\Pi_j)\Pi_{V_j}(\Lambda v+b_{\rm spec})\|^2\\
 &=R_0^2\sum_{j=0}^n\omega_j\psi(\lambda_j)^2(1-p_j^{\rm proj})\\
 &\ge R_0^2(1-\max_j p_j^{\rm proj})
       \sum_{j=0}^n\omega_j\psi(\lambda_j)^2\\
 &\ge\frac{L^2R_0^2}{n^2}(1-\max_j p_j^{\rm proj}).
 \end{aligned}
 \label{eq:dimref:pointwise}
\end{equation}
This bound is pointwise and hence permits coefficients chosen after
observing all the Gaussian vectors. Taking square roots and then
expectations gives
\[
 \mathbb E\|\Lambda v+b_{\rm spec}\|
 \ge\frac{LR_0}{n}\,\mathbb E\sqrt{1-\max_jp_j^{\rm proj}}.
\]
Lemma~\ref{lem:dimref:projector}, applied in the $t=n+1$ subspaces $V_j$,
proves~\eqref{eq:dimref:residual}.

Under the alternative condition $h\ge2ts_{\rm G}$ and $s_{\rm G}\ge1$,
each $S_j$ has dimension $s_{\rm G}$ almost surely. Rotational invariance
makes $\mathbb E\Pi_j$ a scalar multiple of $I_{V_j}$; taking its trace gives
\[
 \mathbb E\Pi_j=\frac{s_{\rm G}}hI_{V_j},
 \qquad\mathbb E p_j^{\rm proj}=\frac{s_{\rm G}}h.
\]
Using $\sqrt{1-a}\ge1-a$ for $0\le a\le1$ yields
\[
 \mathbb E\sqrt{1-\max_jp_j^{\rm proj}}
 \ge1-\mathbb E\max_jp_j^{\rm proj}
 \ge1-\sum_{j=0}^n\mathbb E p_j^{\rm proj}
 =1-\frac{ts_{\rm G}}h\ge\frac12.
\]
If $s_{\rm G}=0$, the expectation is one, which proves the same conclusion.
\end{proof}

\subsection{The hard dimension and the depth lower bound}
\label{app:aff:hard-dimension}

For positive integers $m,T_{\rm depth}$, let $n$ be the smallest odd integer
satisfying $n\ge T_{\rm depth}+2$. Set
\begin{align}
 t&=n+1,\qquad s_{\rm G}=2mT_{\rm depth}+2,\nonumber\\
 h_{\rm ref}
 &=\min\!\left\{2ts_{\rm G},
       \left\lceil\max\{16s_{\rm G},32\log(8t)\}\right\rceil\right\},
 \qquad d_{\rm ref}(m,T_{\rm depth})=2th_{\rm ref}.
 \label{eq:dimref:explicit-dimension}
\end{align}
The two candidate multiplicities are the sufficient choices in
Proposition~\ref{prop:dimref:residual}. The factor two in $d_{\rm ref}$
comes from the two components of the VI construction below.
\begin{corollary}[A depth lower bound in dimension quadratic in depth]
\label{cor:dimref:depth}\label{lem:relmm:depth}
Let $L,R_0>0$ and let $m,T_{\rm depth}$ be positive integers. In dimension
$d_{\rm ref}(m,T_{\rm depth})$, there is a single fixed distribution over noiseless
affine monotone $L$-Lipschitz fields, each with a unique unknown root
$x_\star$ of norm $R_0$. For the root-centered restricted Minty gap
in~\eqref{eq:aff:self-contained-gap},
every randomized algorithm using at most $T_{\rm depth}$ global batches of $m$
simultaneous exact field queries satisfies
\begin{equation}
 \mathbb E\Gap_\star^{\rm M}(\widehat x)
 \ge\frac{LR_0^2}{2n}\ge\frac{LR_0^2}{6T_{\rm depth}}.
 \label{eq:dimref:depth}
\end{equation}
Queries and outputs are arbitrary finite vectors, and the root is not
given to the algorithm. The instance distribution is independent of the algorithm and its internal randomness; the expectation
is over that distribution and the algorithm's randomness.
In particular, the bound applies to $m$ workers making at most $K$
local queries per worker in each of $N$ communication rounds by taking
$T_{\rm depth}=KN$, where $K,N$ are positive integers. Moreover,
\begin{equation*}
 d_{\rm ref}(m,T_{\rm depth})
 \le\min\{4t^2s_{\rm G},64ts_{\rm G}\}
 \le1024mT_{\rm depth}^2.
\end{equation*}
\end{corollary}

\begin{proof}
Use multiplicity $h=h_{\rm ref}$ in the preceding spectral construction,
so $d_{\rm sym}=th_{\rm ref}$. Draw
\[
 U\sim\operatorname{Haar}(O(d_{\rm sym})),\qquad
 A=U\Lambda U^\top,\qquad b_{\rm rhs}=Ub_{\rm spec},
\]
and, for $x=(u,v)\in\mathbb R^{d_{\rm sym}}\times\mathbb R^{d_{\rm sym}}$,
define
\[
 G_A(u,v)=(Av+b_{\rm rhs},-Au),\qquad
 J_A=\begin{pmatrix}0&A\\-A&0\end{pmatrix}.
\]
The law of $U$, and hence that of $G_A$, is fixed by $L,R_0,m,T_{\rm depth}$.

\emph{Field properties.}
Since $A=A^\top$, the linear part satisfies $J_A^\top=-J_A$. Consequently,
\[
 \langle G_A(x)-G_A(y),x-y\rangle
 =\langle J_A(x-y),x-y\rangle=0.
\]
Moreover,
\[
 \|J_A(u,v)\|^2=\|Av\|^2+\|Au\|^2,
 \qquad
 \|J_A\|_{\rm op}=\|A\|_{\rm op}=\|\Lambda\|_{\rm op}=L,
\]
so the field is monotone and $L$-Lipschitz.
The nonzero eigenvalues of $\Lambda$ make $A$ invertible. The field $G_A$ therefore has the unique root
\[
 x_\star=(0,-A^{-1}b_{\rm rhs})=(0,-Uz_{\rm spec}),
 \qquad \|x_\star\|=R_0.
\]

\emph{Gap identity.}
For an arbitrary output $x=(u,v)$, write a comparator as
$y=x_\star+h_{\rm cmp}$. Since $G_A(y)=J_Ah_{\rm cmp}$ and
$\langle J_Ah_{\rm cmp},h_{\rm cmp}\rangle=0$,
\begin{align*}
 \Gap_\star^{\rm M}(x)
 &=\sup_{\|h_{\rm cmp}\|\le R_0}
       \langle J_Ah_{\rm cmp},x-x_\star-h_{\rm cmp}\rangle\\
 &=\sup_{\|h_{\rm cmp}\|\le R_0}
       \langle h_{\rm cmp},J_A^\top(x-x_\star)\rangle\\
 &=\sup_{\|h_{\rm cmp}\|\le R_0}\langle h_{\rm cmp},-G_A(x)\rangle
   =R_0\|G_A(x)\|\\
 &=R_0\sqrt{\|Av+b_{\rm rhs}\|^2+\|Au\|^2}
   \ge R_0\|Av+b_{\rm rhs}\|.
\end{align*}
The supremum equality uses, for any vector $r_{\rm dir}$ of the same
dimension as $h_{\rm cmp}$,
$\sup_{\|h_{\rm cmp}\|\le R_0}\langle h_{\rm cmp},r_{\rm dir}\rangle
=R_0\|r_{\rm dir}\|$,
which follows from Cauchy--Schwarz and is attained at
$h_{\rm cmp}=R_0r_{\rm dir}/\|r_{\rm dir}\|$ when $r_{\rm dir}\ne0$.

\emph{Information available to the algorithm.}
Give $b_{\rm rhs}$ to the algorithm without charge. An exact field query $(u,v)$
can then be answered by two simultaneous matrix queries:
\[
 (Au,Av)\longmapsto(Av+b_{\rm rhs},-Au)=G_A(u,v).
\]
Thus $T_{\rm depth}$ batches of $m$ field queries can be simulated by
$T_{\rm depth}$ batches of $2m$ matrix queries, followed by two arbitrary output
components $(\widehat u,\widehat v)$. Apply
Lemma~\ref{lem:aff:batched} with width $2m$ and two outputs.
It gives $s_{\rm G}=2mT_{\rm depth}+2$ independent standard Gaussian
vectors $v_1^{\rm G},\ldots,v_{s_{\rm G}}^{\rm G}$ and
\[
 U^\top\widehat v
 =p(\Lambda)b_{\rm spec}
  +\sum_{i=1}^{s_{\rm G}}p_i(\Lambda)v_i^{\rm G},
 \qquad\deg p\le T_{\rm depth}.
\]
Allowing unrestricted degrees for the $p_i$ only enlarges the space
containing the output. Since $n\ge T_{\rm depth}+2$, Proposition~\ref{prop:dimref:residual}
applies with $d_0=T_{\rm depth}$. Either multiplicity in
\eqref{eq:dimref:explicit-dimension} yields
\[
 \mathbb E\|A\widehat v+b_{\rm rhs}\|
 =\mathbb E\|\Lambda U^\top\widehat v+b_{\rm spec}\|
 \ge\frac{LR_0}{2n}.
\]
This use of the proposition is unconditional: the Gaussian vectors have
their independent product law, whereas their dependence on $U$ is
irrelevant to the pointwise inequality~\eqref{eq:dimref:pointwise}.
The joint construction preserves the uniform distribution of $U$, independently of the
algorithm's internal randomness. The gap identity therefore proves the lower bound under
the same fixed instance distribution for every randomized algorithm.
Minimality of $n$ gives $n\le T_{\rm depth}+3$. For $T_{\rm depth}=1$ one has $n=3$;
for $T_{\rm depth}\ge2$ one has $T_{\rm depth}+3\le3T_{\rm depth}$. Hence $n\le3T_{\rm depth}$, proving
\eqref{eq:dimref:depth}.

\emph{Intermittent communication.}
Index the queries of each worker within a round by $k=1,\ldots,K$.
At each global batch the simulator makes the $k$th query of every
worker, padding missing queries by zero. All these queries are determined
by earlier batches. Sharing all responses between batches can only give
more information than the intermittent-communication protocol, which restricts each worker to
its local responses until the end of the round. Thus $N$ rounds are
simulated by $KN$ global batches of width $m$.

\emph{Dimension.}
The definition of $n$ gives
$t=n+1\le2T_{\rm depth}+2\le s_{\rm G}$, and $s_{\rm G}\ge4$.
For $s\ge4$ one has $\log(8s)\le s$: the inequality holds at $s=4$,
and $\frac{d}{ds}[s-\log(8s)]=1-1/s>0$. Therefore
\[
 \log(8t)\le\log(8s_{\rm G})\le s_{\rm G},\qquad
 \left\lceil\max\{16s_{\rm G},32\log(8t)\}\right\rceil
 \le32s_{\rm G}.
\]
It follows that
\[
 d_{\rm ref}=2th_{\rm ref}
 \le2t\min\{2ts_{\rm G},32s_{\rm G}\}
 =\min\{4t^2s_{\rm G},64ts_{\rm G}\}.
\]
Finally, $t\le4T_{\rm depth}$ and
$s_{\rm G}=2mT_{\rm depth}+2\le4mT_{\rm depth}$ imply
$64ts_{\rm G}\le1024mT_{\rm depth}^2$.
\end{proof}

\section{Upper bounds for nonlinear VI}
\label{vi:appendix}\label{nn:section}

Throughout this appendix $F$ is monotone and $L$-Lipschitz with
$F(z_\star)=0$, and the oracle satisfies Assumption~\ref{ass:mixed};
SGC is the case $\sigma=0$. For gap bounds, $C$ is a nonempty compact
set with $z_\star\in C\subseteq B(x_0,D)$, where $x_0=z_0$.
We set $\diamond={\rm M}$ for the Minty gap. For a saddle operator
$F=(\nabla_x\Phi,-\nabla_y\Phi)$ with $C=X\times Y$, we instead set
$\diamond={\rm PD}$. Thus $\Gap_C^\diamond$ always displays which of the
two distinct criteria is being bounded.
All methods use the step~\eqref{eq:eg-step} of
Appendix~\ref{app:extragradient-step}, where each error is the average
of $b$ independent responses, with
\[
 b\ge\max\{1,\lceil4(\rho-1)\rceil\},\qquad\text{so that}\qquad
 \alpha_{\rm rel}=\frac{\rho-1}b\le\frac14 .
\]
Section~\ref{nn:trajectory-subsection} analyzes one trajectory,
Section~\ref{vi:schedules-subsection} distributes it over the workers and
rounds, Section~\ref{nn:strong-subsection} treats strong monotonicity,
and Section~\ref{nn:sgc-subsection} gives the explicit form under SGC
used in Theorem~\ref{vi:main-monotone}.

\subsection{One trajectory}\label{nn:trajectory-subsection}

\begin{theorem}[Gap of one trajectory]\label{nn:trajectory}
Run $T\ge1$ extragradient steps with a constant step size
$0<\gamma\le1/(4L)$ and independently sampled batches of size $b$, starting from $x_0$,
and let $\bar y=\frac1T\sum_{t<T}y_t$. Then
\begin{equation}\label{nn:trajectory-bound}
 \E\Gap_C^\diamond(\bar y)
 \le\frac{38}7\cdot\frac{D^2}{\gamma T}
       +\frac{69}{64}\cdot\frac{\gamma\sigma^2}{b}
       +\sqrt{\frac{33}{32}}\cdot\frac{D\sigma}{\sqrt{bT}}.
\end{equation}
For $\gamma=\min\{1/(4L),D\sqrt b/(\sigma\sqrt T)\}$, where the second
entry is $+\infty$ if $\sigma=0$,
\begin{equation}\label{nn:trajectory-optimized}
 \E\Gap_C^\diamond(\bar y)
 \le\frac{22LD^2}{T}
     +\frac{8\sigma D}{\sqrt{bT}}.
\end{equation}
\end{theorem}
\begin{proof}
\emph{From the gap to the regret.} For the Minty gap, monotonicity
gives $\langle F(u),y_t-u\rangle\le\langle F(y_t),y_t-u\rangle$, hence
\[
 \Gap^{\rm M}_C(\bar y)
 =\sup_{u\in C}\frac1T\sum_{t<T}\langle F(u),y_t-u\rangle
 \le\sup_{u\in C}\frac1T\sum_{t<T}\langle F(y_t),y_t-u\rangle .
\]
For a saddle operator and $u=(u_x,u_y)\in X\times Y$, convexity of
$\Phi(\cdot,y)$ and concavity of $\Phi(x,\cdot)$ give
$\langle F(y_t),y_t-u\rangle\ge\Phi(y_{t,x},u_y)-\Phi(u_x,y_{t,y})$.
Averaging over $t$ and applying Jensen's inequality to the convex
function $\Phi(\cdot,u_y)$ and the concave function $\Phi(u_x,\cdot)$,
\[
 \Phi(\bar y_x,u_y)-\Phi(u_x,\bar y_y)
 \le\frac1T\sum_{t<T}\langle F(y_t),y_t-u\rangle,
\]
and the supremum over $u$ gives the same bound for
$\Gap^{\rm PD}_{X\times Y}(\bar y)$.

\emph{The regret.} Sum~\eqref{nn:identity} over $t<T$ for a fixed
$u\in C$, drop $-\|x_T-u\|^2$, and write
$x_t-\gamma F(y_t)-u=(x_t-\gamma F(y_t)-x_0)+(x_0-u)$. With
$S_\delta=\sum_{t<T}\delta_t$,
\[
 2\gamma\sum_{t<T}\langle F(y_t),y_t-u\rangle
 \le\|x_0-u\|^2+\gamma^2\sum_{t<T}\mathcal R_{{\rm eg},t}
 -2\gamma\sum_{t<T}\langle\delta_t,x_t-\gamma F(y_t)-x_0\rangle
 +2\gamma\langle S_\delta,u-x_0\rangle .
\]
Since $\|u-x_0\|\le D$, the first and last terms are at most $D^2$ and
$2\gamma D\|S_\delta\|$, and only these depend on $u$. Take the supremum over
$u\in C$ and then the expectation. The third term has zero mean,
because each $\delta_t$ is centered given the vector it multiplies. By
the proof of Lemma~\ref{nn:one-step},
$\E\mathcal R_{{\rm eg},t}
\le-\frac7{128}\E\|F(x_t)\|^2+\frac{69}{32}\sigma_b^2$ with
$\sigma_b^2=\sigma^2/b$. Writing
$V_F=\sum_{t<T}\E\|F(x_t)\|^2$,
\begin{equation}\label{nn:regret-bound}
 2\gamma\,\E\sup_{u\in C}\sum_{t<T}\langle F(y_t),y_t-u\rangle
 \le D^2-\frac7{128}\gamma^2V_F+\frac{69}{32}\gamma^2T\sigma_b^2
 +2\gamma D\,\E\|S_\delta\| .
\end{equation}

\emph{The corrector sum.} The errors $\delta_t$ are martingale
differences, so~\eqref{nn:corrector} gives
\[
 \E\|S_\delta\|^2=\sum_{t<T}\E\|\delta_t\|^2
 \le\frac{69}{32}\alpha_{\rm rel}V_F+\frac{33}{32}T\sigma_b^2,
 \qquad
 \E\|S_\delta\|\le\sqrt{\frac{69}{32}\alpha_{\rm rel}V_F}
      +\sqrt{\frac{33}{32}T}\,\sigma_b .
\]
For $\alpha,\beta>0$,
$\max_{V_F\ge0}(-\alpha V_F+\beta\sqrt {V_F})=\beta^2/(4\alpha)$.
With $\alpha=\frac7{128}\gamma^2$,
$\beta=2\gamma D\sqrt{69\alpha_{\rm rel}/32}$ and
$\alpha_{\rm rel}\le1/4$,
\[
 -\frac7{128}\gamma^2V_F
 +2\gamma D\sqrt{\frac{69}{32}\alpha_{\rm rel}V_F}
 \le\frac{4\gamma^2D^2\cdot69\alpha_{\rm rel}/32}
          {4\cdot7\gamma^2/128}
 =\frac{276\alpha_{\rm rel}}{7}D^2\le\frac{69}7D^2 .
\]
Substituting into~\eqref{nn:regret-bound} and dividing by $2\gamma T$,
\[
 \E\Gap_C^\diamond(\bar y)
 \le\Bigl(1+\frac{69}7\Bigr)\frac{D^2}{2\gamma T}
 +\frac{69}{64}\gamma\sigma_b^2
 +\sqrt{\frac{33}{32}}\cdot\frac{D\sigma_b}{\sqrt T},
\]
which is~\eqref{nn:trajectory-bound}, since $(1+69/7)/2=38/7$.

\emph{The step size.} For the given $\gamma$,
$1/\gamma\le4L+\sigma\sqrt T/(D\sqrt b)$ and
$\gamma\sigma^2/b\le D\sigma/\sqrt{bT}$. Substituting,
\[
 \E\Gap_C^\diamond(\bar y)
 \le\frac{152}7\cdot\frac{LD^2}T
 +\Bigl(\frac{38}7+\frac{69}{64}+\sqrt{\frac{33}{32}}\Bigr)\frac{D\sigma}{\sqrt{bT}},
\]
and $152/7<22$, $38/7+69/64+\sqrt{33/32}<8$.

\emph{Integrability.} All second moments are finite by induction on
$t$, since $F$ is Lipschitz and the oracle has finite conditional second
moments. The supremum over the compact set $C$ of a continuous function
of $u$ equals the supremum over a countable dense subset, so it is
measurable.
\end{proof}

\subsection{Distributing the trajectory over workers and rounds}
\label{vi:schedules-subsection}

The trajectory of Theorem~\ref{nn:trajectory} can be run in two ways.
\emph{Serially}, one worker runs it alone: it keeps its state across
round boundaries, so that a batch interrupted by the end of a round is
completed in the next round at the same query point. With common
queries, all workers query the same point and combine responses from
consecutive rounds into one batch; the server forms the
predictor after one batch and broadcasts it, and the corrector batch is
collected in the following block of rounds. Let
\begin{equation}\label{vi:all-budget-schedules}
 \begin{aligned}
 &\text{serial:} &&b=\max\{1,\lceil4(\rho-1)\rceil\},
  &&T_{\rm sd}=\Bigl\lfloor\frac{KN}{2b}\Bigr\rfloor,\\
 &\text{common queries:} &&s_{\rm cq}=\max\Bigl\{1,\Bigl\lceil\frac{4(\rho-1)}{mK}\Bigr\rceil\Bigr\},
  &&T_{\rm cq}=\Bigl\lfloor\frac N{2s_{\rm cq}}\Bigr\rfloor .
 \end{aligned}
\end{equation}
The serial schedule makes $2bT_{\rm sd}\le KN$ calls on one worker and
at most $K$ calls per round. The common-query schedule uses batches of
size $mKs_{\rm cq}\ge4(\rho-1)$ and
$2s_{\rm cq}T_{\rm cq}\le N$ rounds. We use
\begin{equation}\label{nn:schedule-facts}
 w_{\rm cq}:=\min\Bigl\{1,\frac{mK}\rho\Bigr\},\qquad
 b\le4\rho,\qquad
 s_{\rm cq}\le1+\frac{4\rho}{mK}\le\frac5{w_{\rm cq}},
\end{equation}
where the last inequality is checked separately for $w_{\rm cq}=1$ and
$w_{\rm cq}=mK/\rho$. We also use $\lfloor x\rfloor\ge x/2$ for $x\ge1$.

\begin{theorem}[All budgets]\label{nn:all-budget}
Set $C_{\rm gap}:=440$. For all $m,K,N\in\mathbb N$, the better of the initial point and the two
schedules~\eqref{vi:all-budget-schedules}, chosen by their bounds, satisfies
\begin{equation}\label{nn:ic-upper}
 \E\Gap_C^\diamond(\widehat z)\le C_{\rm gap}\min\left\{
 LD^2,\quad
 \frac{\rho LD^2}{KN}+\frac{\sigma D}{\sqrt{KN}},\quad
 \frac{LD^2}{N\min\{1,mK/\rho\}}
       +\frac{\sigma D}{\sqrt{mKN}}
 \right\}
\end{equation}
\end{theorem}
\begin{proof}
\emph{Initial point.} For $u\in C$, $\|u-z_\star\|\le2D$ and
$\|x_0-u\|\le D$, so
$\Gap^{\rm M}_C(x_0)\le\sup_{u\in C}\|F(u)\|\,\|x_0-u\|\le2LD^2$. For a
saddle operator, convexity--concavity gives, for
$u=(u_x,u_y)\in X\times Y$,
\[
 \Phi(x_{0,x},u_y)-\Phi(u_x,x_{0,y})
 \le\langle F(x_0),x_0-u\rangle\le\|F(x_0)\|D\le LD^2,
\]
since $\|F(x_0)\|\le L\|x_0-z_\star\|\le LD$.

\emph{Serial schedule.} If $KN\ge4b$, then $KN/(2b)\ge2$, so
$T_{\rm sd}\ge KN/(4b)$ and $bT_{\rm sd}\ge KN/4$.
By~\eqref{nn:trajectory-optimized} and $b\le4\rho$,
\[
 \E\Gap_C^\diamond(\bar y)
 \le\frac{88bLD^2}{KN}+\frac{16\sigma D}{\sqrt{KN}}
 \le\frac{352\rho LD^2}{KN}+\frac{16\sigma D}{\sqrt{KN}} .
\]
If $KN<4b\le16\rho$, the initial point gives
$2LD^2\le32\rho LD^2/(KN)$.

\emph{Common queries.} If $N\ge4s_{\rm cq}$, then
$T_{\rm cq}\ge N/(4s_{\rm cq})$ and the total batch size satisfies
$mKs_{\rm cq}T_{\rm cq}\ge mKN/4$.
By~\eqref{nn:trajectory-optimized} and~\eqref{nn:schedule-facts},
\[
 \E\Gap_C^\diamond(\bar y)\le\frac{88s_{\rm cq}LD^2}{N}
      +\frac{16\sigma D}{\sqrt{mKN}}
 \le\frac{440LD^2}{Nw_{\rm cq}}+\frac{16\sigma D}{\sqrt{mKN}} .
\]
If $N<4s_{\rm cq}\le20/w_{\rm cq}$, the initial point gives
$2LD^2\le40LD^2/(Nw_{\rm cq})$.

The three bounds depend only on known parameters, which proves the
claim.
\end{proof}

\subsection{Strong monotonicity}\label{nn:strong-subsection}

\begin{lemma}[A scalar recursion with two phases]\label{lem:two-phase-recursion}
Let $d_t\ge0$ satisfy
\[
 d_{t+1}\le(1-\alpha\gamma_t)d_t+v_{\rm rec}\gamma_t^2,
 \qquad v_{\rm rec}\ge0,
\]
for step sizes $0<\gamma_t\le\gamma_0$ with $0<\alpha\gamma_0\le1$. Let
$r_{\rm sw}=\lceil2/(\alpha\gamma_0)\rceil$. If $T\le2r_{\rm sw}$,
take $\gamma_t=\gamma_0$.
Otherwise take $\gamma_t=\gamma_0$ for the first
$T_{\rm const}=\lfloor T/2\rfloor$ steps and
$\gamma_{T_{\rm const}+j-1}=2/[\alpha(r_{\rm sw}+j)]$ for the remaining
steps $j=1,\ldots,T_{\rm dec}$, where $T_{\rm dec}=T-T_{\rm const}$. Then
\begin{equation}\label{eq:two-phase-recursion}
 d_T\le d_0e^{-\alpha\gamma_0T/3}+\frac{14v_{\rm rec}}{\alpha^2T}.
\end{equation}
\end{lemma}
\begin{proof}
The decreasing steps satisfy
$2/[\alpha(r_{\rm sw}+j)]\le2/(\alpha r_{\rm sw})\le\gamma_0$.
Since $\alpha\gamma_0\le1$, we have
$r_{\rm sw}\le2/(\alpha\gamma_0)+1\le3/(\alpha\gamma_0)$.

\emph{Constant steps.} Unrolling the recursion and using
$1-x\le e^{-x}$,
\[
 d_t\le(1-\alpha\gamma_0)^td_0
      +v_{\rm rec}\gamma_0^2\sum_{j<t}(1-\alpha\gamma_0)^j
 \le e^{-\alpha\gamma_0t}d_0+\frac{v_{\rm rec}\gamma_0}\alpha .
\]
If $T\le2r_{\rm sw}$, then $T\le6/(\alpha\gamma_0)$, so
$v_{\rm rec}\gamma_0/\alpha\le6v_{\rm rec}/(\alpha^2T)$, which
proves~\eqref{eq:two-phase-recursion}.

\emph{Decreasing steps.} Let $T>2r_{\rm sw}$, so $T\ge3$,
$T_{\rm const}\ge T/3$ and $T_{\rm dec}\ge T/2>r_{\rm sw}$.
In the second phase the contraction factors are
$1-2/(r_{\rm sw}+j)=(r_{\rm sw}+j-2)/(r_{\rm sw}+j)$, and their products telescope:
\[
 \prod_{i=j+1}^{T_{\rm dec}}\Bigl(1-\frac2{r_{\rm sw}+i}\Bigr)
 =\frac{(r_{\rm sw}+j-1)(r_{\rm sw}+j)}
        {(r_{\rm sw}+T_{\rm dec}-1)(r_{\rm sw}+T_{\rm dec})} .
\]
The noise added at step $j$ is
$4v_{\rm rec}/[\alpha^2(r_{\rm sw}+j)^2]$, so its
contribution to $d_T$ is
\[
 \frac{(r_{\rm sw}+j-1)(r_{\rm sw}+j)}
      {(r_{\rm sw}+T_{\rm dec}-1)(r_{\rm sw}+T_{\rm dec})}
 \cdot\frac{4v_{\rm rec}}{\alpha^2(r_{\rm sw}+j)^2}
 \le\frac{4v_{\rm rec}}
          {\alpha^2(r_{\rm sw}+T_{\rm dec}-1)(r_{\rm sw}+T_{\rm dec})}
 \le\frac{4v_{\rm rec}}{\alpha^2T_{\rm dec}^2}.
\]
Summing over $T_{\rm dec}$ steps and using the product with $j=0$, which
is at most $r_{\rm sw}^2/T_{\rm dec}^2$,
\[
 d_T\le\frac{r_{\rm sw}^2}{T_{\rm dec}^2}d_{T_{\rm const}}
      +\frac{4v_{\rm rec}}{\alpha^2T_{\rm dec}}
 \le e^{-\alpha\gamma_0T_{\rm const}}d_0
      +\frac{r_{\rm sw}^2}{T_{\rm dec}^2}
       \cdot\frac{v_{\rm rec}\gamma_0}{\alpha}
      +\frac{4v_{\rm rec}}{\alpha^2T_{\rm dec}},
\]
where we used $r_{\rm sw}^2/T_{\rm dec}^2\le1$ in front of
$e^{-\alpha\gamma_0T_{\rm const}}d_0$. Since
$r_{\rm sw}\gamma_0\le3/\alpha$ and $r_{\rm sw}<T_{\rm dec}$,
\[
 \frac{r_{\rm sw}^2}{T_{\rm dec}^2}
 \cdot\frac{v_{\rm rec}\gamma_0}\alpha
 \le\frac{3v_{\rm rec}r_{\rm sw}}{\alpha^2T_{\rm dec}^2}
 \le\frac{3v_{\rm rec}}{\alpha^2T_{\rm dec}}.
\]
Hence
$d_T\le e^{-\alpha\gamma_0T/3}d_0
      +7v_{\rm rec}/(\alpha^2T_{\rm dec})$, and
$T_{\rm dec}\ge T/2$
gives~\eqref{eq:two-phase-recursion}.
\end{proof}

\begin{proposition}[Distance to the root]\label{nn:strong-distance}
Let $F$ be in addition $\mu$-strongly monotone with $0<\mu\le L$, and
let $\|x_0-z_\star\|\le R_0$. One trajectory of $T\ge1$ steps with the
step sizes of Lemma~\ref{lem:two-phase-recursion}, where
$\gamma_0=1/(4L)$ and $\alpha=9\mu/8$. Set
$c_{\rm traj}:=3/32$ and $C_{\rm traj}:=1932/81$. Then
\begin{equation}\label{nn:strong-trajectory}
 \E\|x_T-z_\star\|^2
 \le R_0^2e^{-c_{\rm traj}\mu T/L}
       +\frac{C_{\rm traj}\sigma^2}{\mu^2bT}.
\end{equation}
Set $c_{\rm sched}:=3/640$ and $C_{\rm sched}:=96$. Consequently, the
better of the initial point and the two schedules~\eqref{vi:all-budget-schedules},
run with these step sizes, satisfies
\begin{align}\label{nn:strong-ic}
 \E\|\widehat z-z_\star\|^2\le C_{\rm sched}\min\bigg\{&R_0^2,\quad
 R_0^2e^{-c_{\rm sched}\mu KN/(\rho L)}
      +\frac{\sigma^2}{\mu^2KN},\nonumber\\[-2pt]
 &R_0^2e^{-c_{\rm sched}\mu N\min\{1,mK/\rho\}/L}
      +\frac{\sigma^2}{\mu^2mKN}\bigg\}
\end{align}
with universal constants shown above.
For $\sigma=0$ this is linear convergence.
\end{proposition}
\begin{proof}
\emph{One step.} Let $d_t=\E\|x_t-z_\star\|^2$. Conditionally on the
history, $\E y_t=x_t-\gamma_tF(x_t)$, and since
$\|F(x_t)-F(z_\star)\|\le L\|x_t-z_\star\|$ and $\gamma_tL\le1/4$,
Jensen's inequality gives
\[
 \E\|y_t-z_\star\|^2\ge\|x_t-z_\star-\gamma_t(F(x_t)-F(z_\star))\|^2
 \ge(1-\gamma_tL)^2\|x_t-z_\star\|^2\ge\frac9{16}\|x_t-z_\star\|^2 .
\]
Apply~\eqref{nn:energy} with $u=z_\star$, use
$\langle F(y_t),y_t-z_\star\rangle\ge\mu\|y_t-z_\star\|^2$, and drop the
term $-\frac7{128}\gamma_t^2\|F(x_t)\|^2$:
\begin{equation}\label{nn:distance-recursion}
 d_{t+1}\le\Bigl(1-\frac98\mu\gamma_t\Bigr)d_t+\frac{69}{32}\cdot
 \frac{\sigma^2}b\gamma_t^2 .
\end{equation}
Lemma~\ref{lem:two-phase-recursion} with $\alpha=9\mu/8$,
$\gamma_0=1/(4L)$ (so $\alpha\gamma_0\le9/32$) and
$v_{\rm rec}=\frac{69}{32}\sigma^2/b$ gives
\[
 d_T\le R_0^2e^{-3\mu T/(32L)}
 +\frac{14\cdot69\cdot64}{32\cdot81}\cdot\frac{\sigma^2}{\mu^2bT},
\]
which is~\eqref{nn:strong-trajectory}. The step sizes depend only on $T$
and known parameters, so the bound also holds conditionally on a random
starting point.

\emph{Serial schedule.} If $KN\ge4b$, then $T_{\rm sd}\ge KN/(4b)\ge
KN/(16\rho)$ and $bT_{\rm sd}\ge KN/4$, so
\[
 \E\|x_{T_{\rm sd}}-z_\star\|^2
 \le R_0^2e^{-3\mu KN/(512\rho L)}+\frac{4\cdot1932}{81}\cdot
 \frac{\sigma^2}{\mu^2KN}.
\]
If $KN<4b\le16\rho$, then $3\mu KN/(512\rho L)<3/32$, and the initial
point gives $R_0^2\le e^{3/32}R_0^2e^{-3\mu KN/(512\rho L)}$.

\emph{Common queries.} If $N\ge4s_{\rm cq}$, then
$T_{\rm cq}\ge N/(4s_{\rm cq})\ge Nw_{\rm cq}/20$ and the batch size
satisfies $mKs_{\rm cq}T_{\rm cq}\ge mKN/4$, so
\[
 \E\|x_{T_{\rm cq}}-z_\star\|^2
 \le R_0^2e^{-3\mu Nw_{\rm cq}/(640L)}+\frac{4\cdot1932}{81}\cdot
 \frac{\sigma^2}{\mu^2mKN}.
\]
If $N<4s_{\rm cq}\le20/w_{\rm cq}$, then $Nw_{\rm cq}<20$, and the initial point gives
$R_0^2\le e^{3/32}R_0^2e^{-3\mu Nw_{\rm cq}/(640L)}$.

Since $3/640\le3/512$, $e^{3/32}\le96$ and $4\cdot1932/81<96$, the three
bounds prove~\eqref{nn:strong-ic} with the stated constants.
\end{proof}

\subsection{Explicit bounds under SGC}\label{nn:sgc-subsection}

Under SGC ($\sigma=0$) both schedules use the constant step
$\gamma=1/(4L)$. Write $D_0=\|x_0-z_\star\|$ and, under strong
monotonicity, $\kappa=L/\mu$.

\begin{corollary}[SGC]\label{vi:all-budget}
Under SGC, the better of the initial point and the two
schedules~\eqref{vi:all-budget-schedules} satisfies
\begin{equation}\label{vi:all-gap}
 \E\Gap_C^\diamond(\widehat z)\le\frac{500LD^2}{1+N\SchedDepth}.
\end{equation}
Under strong monotonicity, the endpoint of the schedule with more steps
satisfies
\begin{equation*}
 \E\|\widehat z-z_\star\|^2
 \le D_0^2\min\left\{1,\ 2\exp\left[-\frac{9N}{320\kappa}\SchedDepth\right]\right\}.
\end{equation*}
\end{corollary}
\begin{proof}
Let $z_{\rm sd}=KN/\rho$, $z_{\rm cq}=Nw_{\rm cq}$ with $w_{\rm cq}$
from~\eqref{nn:schedule-facts}, so that
$\max\{z_{\rm sd},z_{\rm cq}\}=N\SchedDepth$.

\emph{Gap.} For $\sigma=0$ and $\gamma=1/(4L)$,
\eqref{nn:trajectory-bound} gives
$\E\Gap_C^\diamond(\bar y)\le C_0LD^2/T$
with $C_0=152/7$, and the initial point gives $2LD^2$. Since $b\le4\rho$,
$KN/(2b)\ge z_{\rm sd}/8$. If $z_{\rm sd}\ge16$, then
$T_{\rm sd}\ge z_{\rm sd}/16$ and
\[
 \frac{C_0LD^2}{T_{\rm sd}}\le\frac{16C_0LD^2}{z_{\rm sd}}
 \le\frac{17C_0LD^2}{1+z_{\rm sd}}\le\frac{370LD^2}{1+z_{\rm sd}} ;
\]
if $z_{\rm sd}<16$, then $2LD^2\le34LD^2/(1+z_{\rm sd})$. Similarly,
$s_{\rm cq}\le5/w_{\rm cq}$ gives
$N/(2s_{\rm cq})\ge z_{\rm cq}/10$; if $z_{\rm cq}\ge20$, then
$T_{\rm cq}\ge z_{\rm cq}/20$ and
$C_0LD^2/T_{\rm cq}\le21C_0LD^2/(1+z_{\rm cq})=456LD^2/(1+z_{\rm cq})$,
and if $z_{\rm cq}<20$, then $2LD^2\le42LD^2/(1+z_{\rm cq})$. Taking the
schedule with the larger $z$ proves~\eqref{vi:all-gap}.

\emph{Distance.} For $\sigma=0$ and $\gamma_t=1/(4L)$,
\eqref{nn:distance-recursion} gives
$d_{t+1}\le(1-9/(32\kappa))d_t$. Since $\lfloor x\rfloor\ge x-1$,
$T_{\rm sd}\ge z_{\rm sd}/8-1$ and $T_{\rm cq}\ge z_{\rm cq}/10-1$, so the
larger of the two horizons is
$T\ge N\SchedDepth/10-1$. Hence, using $1-x\le e^{-x}$ and
$e^{9/(32\kappa)}\le e^{9/32}\le2$,
\[
 \E\|\widehat z-z_\star\|^2\le D_0^2e^{-9T/(32\kappa)}
 \le2D_0^2\exp\Bigl[-\frac{9N}{320\kappa}\SchedDepth\Bigr].
\]
The contraction also gives the bound $D_0^2$, including the case
$T=0$.
\end{proof}

\section{Proofs for affine VI}\label{pv:section}

Throughout this appendix Assumption~\ref{ass:affine-main} holds, so
$G(x)=J(x-x_\star)$ with $J+J^\top\succeq0$ and $\|J\|_{\rm op}\le L$.
Every oracle call returns an unbiased estimate of $G$ at the queried
point, computed from a sample that is independent of all earlier
samples, of the instance, and of the algorithm's randomness. We write
\[
 \beta=\beta_{1/\rho}=\E\max_{1\le i\le m}Z_i,
 \qquad Z_1,\ldots,Z_m\stackrel{\rm iid}{\sim}\operatorname{Bin}(K,1/\rho),
\]
for the quantity of Appendix~\ref{app:clock}, and $\BinDepth(m,K,p)$ for its
explicit function~\eqref{eq:clock:explicit-app}. Its specialization is
$\InfoDepth=\BinDepth(m,K,1/\rho)$, the quantity used in
Theorem~\ref{thm:affine-integrated}.

The appendix is organized as follows. Section~\ref{sec:relative-minimax}
proves the lower bound of Theorem~\ref{thm:affine-integrated} by
reducing an SGC oracle to the exact-oracle construction of
Appendix~\ref{app:affine}. Sections~\ref{pv:serial-subsection}
and~\ref{pv:algorithm-subsection} analyze three algorithms: serial
extragradient, extragradient with common queries, and local
extragradient with worker averaging. Section~\ref{pv:comparison-subsection}
chooses the local horizon and proves Theorem~\ref{thm:affine-integrated}.
Section~\ref{an:section} treats mixed noise and proves
Theorem~\ref{thm:affine-additive-main}.

\subsection{A lower bound from random masks}
\label{sec:relative-minimax}

We call an oracle \emph{masked with probability $p$} if at every call an
independent indicator $B_{\rm mask}\sim\operatorname{Bernoulli}(p)$
decides the response: if $B_{\rm mask}=1$, the response is the exact
value of the operator, transformed by a known invertible map; if
$B_{\rm mask}=0$, it is a known constant that does not depend on the
instance. The indicators are independent of the instance and of the
algorithm's randomness. A call with $B_{\rm mask}=1$ is called
\emph{informative}.

\begin{lemma}[Simulation by exact batches]\label{lem:clock:simulation}
\label{lem:masked:simulation}
Consider an algorithm of Definition~\ref{def:ic-main} with a masked
oracle, and assume without loss of generality that every worker makes
exactly $K$ calls in every round. Let $Z_{i,r}$ be the number of
informative calls of worker $i$ in round $r$, so that
$Z_{i,r}\stackrel{\rm iid}{\sim}\operatorname{Bin}(K,p)$, and set
\[
 S_{\rm info}=\sum_{r=1}^N\max_{1\le i\le m}Z_{i,r},\qquad
 T_{\rm mask}=\min\bigl\{KN,\ \max\{1,\lceil2N\beta_p\rceil\}\bigr\}.
\]
\begin{enumerate}[label=(\roman*),leftmargin=*]
\item Conditionally on the indicators, the algorithm can be simulated
by an algorithm that queries the exact operator in $S_{\rm info}$ consecutive
batches of at most $m$ simultaneous queries each.
\item $\Prob\{S_{\rm info}\le T_{\rm mask}\}\ge1/2$.
\item Suppose that a fixed distribution of instances gives expected loss
at least $\ell_{\rm dep}\ge0$ to every randomized algorithm that uses at most
$T_{\rm mask}$
batches of at most $m$ exact queries, and that the loss is nonnegative.
Then the same distribution, combined with independent indicators, gives
expected loss at least $\ell_{\rm dep}/2$ to every algorithm of
Definition~\ref{def:ic-main}.
\end{enumerate}
\end{lemma}
\begin{proof}
\emph{(i)} Reveal all indicators to the algorithm before it starts.
They are independent of the instance, so this can only help the
algorithm, and from now on every non-informative response is known in
advance. Fix the indicators and a round $r$. For $j\ge1$, the query of
the $j$th informative call of worker $i$ is a function of
\begin{itemize}[leftmargin=*,nosep]
\item the information shared at the start of round $r$,
\item the algorithm's randomness and the indicators,
\item the first $j-1$ informative responses of worker $i$ in round $r$,
\end{itemize}
because all non-informative responses in between are known constants.
Put the $j$th informative calls of all workers into one batch. All
queries of batch $j$ are then determined before any response of batch
$j$ is revealed, as required. Round $r$ needs
$\max_iZ_{i,r}$ batches, after which the whole transcript of the round
is known and the communication at its end can be reproduced. Over $N$
rounds the simulation uses $S_{\rm info}$ batches.

\emph{(ii)} By definition of $\beta_p$,
$\E S_{\rm info}=\sum_{r=1}^N\E\max_iZ_{i,r}=N\beta_p$, and
$S_{\rm info}\le KN$. If $T_{\rm mask}=KN$, then
$\Prob\{S_{\rm info}\le T_{\rm mask}\}=1$. Otherwise
$T_{\rm mask}\ge2N\beta_p$, and Markov's
inequality gives
\[
 \Prob\{S_{\rm info}>T_{\rm mask}\}
 \le\frac{\E S_{\rm info}}{T_{\rm mask}}
 \le\frac{N\beta_p}{2N\beta_p}=\frac12.
\]

\emph{(iii)} Fix the indicators. On the event
$\{S_{\rm info}\le T_{\rm mask}\}$, pad the simulator of~(i) with empty
batches to exactly $T_{\rm mask}$ batches; this is a randomized algorithm
with at most $T_{\rm mask}$ batches of at most $m$ exact
queries. The instance distribution does not depend on the indicators,
so the assumption applies to it:
\[
 \E[\text{loss}\mid\text{indicators}]\ge \ell_{\rm dep}
 \qquad\text{on }\{S_{\rm info}\le T_{\rm mask}\}.
\]
Since the loss is nonnegative, (ii) gives
$\E[\text{loss}]\ge
\ell_{\rm dep}\Prob\{S_{\rm info}\le T_{\rm mask}\}
\ge \ell_{\rm dep}/2$.
\end{proof}

\begin{lemma}[Lower bound under SGC]\label{lem:relmm:masked}
Let
\begin{equation*}
 T_{\rm mask}=\min\bigl\{KN,\ \max\{1,\lceil2N\beta\rceil\}\bigr\},\qquad
 d_{\rm hard}=d_{\rm ref}(m,T_{\rm mask})\le1024\,m T_{\rm mask}^2,
\end{equation*}
with $d_{\rm ref}$ from Corollary~\ref{cor:dimref:depth}. There is a
fixed distribution of monotone affine $L$-Lipschitz operators on
$\mathbb R^{d_{\rm hard}}$, each with a unique root at distance $R_0$
from the initial point, and an oracle satisfying SGC with constant
$\rho$, such that every algorithm of Definition~\ref{def:ic-main}
satisfies
\begin{equation}\label{pv:minimax-comparison}
 \E\Gap_\star^{\rm M}(\widehat x)
 \ge\frac{LR_0^2}{12T_{\rm mask}}
 \ge\frac{LR_0^2}{96(1+N\InfoDepth)}.
\end{equation}
\end{lemma}
\begin{proof}
Translate the initial point to $0$ and take the distribution of
Corollary~\ref{cor:dimref:depth} with $T_{\rm depth}=T_{\rm mask}$.
Against every randomized algorithm with at most $T_{\rm mask}$ batches
of $m$ exact queries it gives
$\E\Gap_\star^{\rm M}(\widehat x)\ge LR_0^2/(6T_{\rm mask})$. At every call, the oracle
returns
\[
 \widehat G(x)=\rho B_{\rm mask}G(x),\qquad
 B_{\rm mask}\sim\operatorname{Bernoulli}(1/\rho)
 \text{ independent of everything else}.
\]
It is unbiased and satisfies SGC with equality:
\[
 \E[\widehat G(x)\mid\mathcal F]=\rho\cdot\frac1\rho\,G(x)=G(x),\qquad
 \E[\|\widehat G(x)\|^2\mid\mathcal F]=\rho^2\cdot\frac1\rho\,\|G(x)\|^2
 =\rho\|G(x)\|^2.
\]
This oracle is masked with probability $p=1/\rho$, with the known
invertible map $v\mapsto\rho v$ and the constant $0$. The gap is
nonnegative, because the root is an admissible comparator and
$G(x_\star)=0$. Lemma~\ref{lem:clock:simulation}(iii) with
$\ell_{\rm dep}=LR_0^2/(6T_{\rm mask})$ and
$\beta_{1/\rho}=\beta$ gives the first inequality
in~\eqref{pv:minimax-comparison}. For the second, $\lceil x\rceil\le x+1$
gives $T_{\rm mask}\le\max\{1,2N\beta+1\}\le2(1+N\beta)$, and
Theorem~\ref{thm:clock:asymptotic} gives $\beta\le4\InfoDepth$. Hence
$12T_{\rm mask}\le24(1+N\beta)\le96(1+N\InfoDepth)$.
\end{proof}

\subsection{Serial and common-query extragradient}
\label{pv:serial-subsection}

\begin{lemma}[Serial and common-query bounds]\label{lem:relmm:upper}
There is a universal constant $C_0\ge6$ such that, under SGC and in
every dimension, an algorithm of Definition~\ref{def:ic-main} satisfies
\begin{equation}\label{eq:relmm:sharp-upper}
 \E\Gap_\star^{\rm M}(\widehat x)
 \le LR_0^2\min\left\{1,\
       \frac{C_0\rho}{NK},\
       \frac{C_0}{N}\Bigl(1+\frac{\rho}{mK}\Bigr)\right\}.
\end{equation}
\end{lemma}
\begin{proof}
The test ball satisfies $B(x_\star,R_0)\subseteq B(x_0,2R_0)$,
because $\|x_0-x_\star\|\le R_0$. Theorem~\ref{nn:all-budget} with
$\sigma=0$, $C=B(x_\star,R_0)$ and $D=2R_0$ therefore gives
an algorithm, combining a serial and a common-query schedule, with
\[
 \E\Gap_\star^{\rm M}(\widehat x)\le4C'LR_0^2\min\left\{1,\
 \frac{\rho}{KN},\ \frac1{N\min\{1,mK/\rho\}}\right\},
\]
where $C'$ is the constant of that theorem. Since
\[
 \frac1{\min\{1,mK/\rho\}}=\max\Bigl\{1,\frac\rho{mK}\Bigr\}
 \le1+\frac\rho{mK},
\]
the last two entries are covered by~\eqref{eq:relmm:sharp-upper} with
$C_0=\max\{6,4C'\}$. The entry $LR_0^2$ is attained by the initial
point: for $\|u-x_\star\|\le R_0$, monotonicity and $G(x_\star)=0$ give
$\langle G(u),x_\star-u\rangle\le0$, hence
\[
 \langle G(u),x_0-u\rangle
 \le\langle G(u),x_0-x_\star\rangle
 \le\|G(u)\|\,\|x_0-x_\star\|
 \le L\|u-x_\star\|R_0\le LR_0^2.
\]
The algorithm outputs the candidate with the smallest bound; all
bounds depend only on known parameters.
\end{proof}

\subsection{Local extragradient with worker averaging}
\label{pv:algorithm-subsection}

Fix integers $b,T\ge1$ with $2bT\le K$ and a step size $\gamma>0$. Every
stage of the method below averages $b$ responses; we call $b$ the
batch size and $T$ the local horizon.

\begin{definition}[Local extragradient with averaging]\label{pv:definition}
Set $z_0=x_0$. In round $r=0,\ldots,N-1$, every worker $i$ starts from
$x_{r,0}^i=z_r$ and performs, for $t=0,\ldots,T-1$,
\begin{align}\label{pv:algorithm}
 y_{r,t}^i&=x_{r,t}^i-\gamma\bigl(G(x_{r,t}^i)+e_{r,t}^i\bigr),\nonumber\\
 x_{r,t+1}^i&=x_{r,t}^i-\gamma\bigl(G(y_{r,t}^i)+\delta_{r,t}^i\bigr),
\end{align}
where $G(x_{r,t}^i)+e_{r,t}^i$ and $G(y_{r,t}^i)+\delta_{r,t}^i$ are
averages of $b$ independent responses at $x_{r,t}^i$ and at $y_{r,t}^i$; the
second batch is drawn after $y_{r,t}^i$ is computed. At the end of the
round the workers communicate and set
\begin{equation*}
 z_{r+1}=\frac1m\sum_{i=1}^m x_{r,T}^i.
\end{equation*}
The output is
$\widehat y=\frac1{mNT}\sum_{r<N}\sum_{i\le m}\sum_{t<T}y_{r,t}^i$;
each worker keeps the running sum of its points $y_{r,t}^i$ and sends it
at the last communication.
\end{definition}
Each worker makes $2bT\le K$ calls per round, and the method uses the $N$
communications at the round boundaries.

Under SGC, a response at $x$ has variance at most $(\rho-1)\|G(x)\|^2$,
so a batch of $b$ independent responses has variance at most
$\alpha_{\rm rel}\|G(x)\|^2$ with
\[
 \alpha_{\rm rel}=\frac{\rho-1}b.
\]
Exact extragradient on an affine operator multiplies both $x-x_\star$ and
$G(x)$ by the matrix
\[
 M_{\rm eg}=I-\gamma J+\gamma^2J^2,
\]
and its first step multiplies them by $M_{{\rm eg},1}=I-\gamma J$. We use
$q_{\rm eg}=\gamma^2L^2$ and the two constants
\begin{equation*}
 \vartheta_{\rm grow}=q_{\rm eg}\alpha_{\rm rel}
       (1+2q_{\rm eg}+q_{\rm eg}\alpha_{\rm rel}),\qquad
 c_{\rm avg}=\alpha_{\rm rel}
       (2+3q_{\rm eg}+q_{\rm eg}\alpha_{\rm rel}).
\end{equation*}
The first controls the growth of a worker's operator norm caused by
noise; the second controls the noise that remains after averaging over
workers.

\begin{lemma}[Operator energy of one round]\label{pv:field-lemma}
Assume SGC and $0<\gamma\le1/(4L)$. Fix a round, condition on its
starting point $z$, and drop the round index. For a worker $i$, let
\[
 d_t=\|M_{\rm eg}^tG(z)\|^2,\qquad
 s_t=\E\bigl[\|G(x_t^i)\|^2\bigm|z\bigr],
\]
and let $\bar x_t=\frac1m\sum_ix_t^i$. Then $s_t$ does not depend on $i$,
\begin{equation}\label{pv:field-energy}
 \sum_{t<T}s_t\le(1+\vartheta_{\rm grow})^{T-1}\sum_{t<T}d_t,
 \qquad
 \E\bigl[\|G(\bar x_t)\|^2\bigm|z\bigr]\ge d_t,
\end{equation}
and
$\E[\|\delta_t^i\|^2\mid z]
\le \alpha_{\rm rel}(1+q_{\rm eg}+q_{\rm eg}\alpha_{\rm rel})\,s_t$.
\end{lemma}
\begin{proof}
\emph{Two matrix inequalities.} For every $v$, since
$\langle Jv,v\rangle\ge0$ and $\|Jv\|\le L\|v\|$,
\begin{equation}\label{pv:H}
 \|M_{{\rm eg},1}v\|^2
 =\|v\|^2-2\gamma\langle Jv,v\rangle+\gamma^2\|Jv\|^2
 \le(1+q_{\rm eg})\|v\|^2.
\end{equation}
Moreover $M_{\rm eg}=I-\gamma JM_{{\rm eg},1}$. Put
$w=M_{{\rm eg},1}v$, so that $v=w+\gamma Jv$ and
$M_{\rm eg}v=v-\gamma Jw$. Then
\[
 \|M_{\rm eg}v\|^2
 =\|v\|^2-2\gamma\langle Jw,w\rangle-2\gamma^2\langle Jw,Jv\rangle
  +\gamma^2\|Jw\|^2
 =\|v\|^2-2\gamma\langle Jw,w\rangle-\gamma^2\|Jv\|^2
  +\gamma^2\|Jw-Jv\|^2.
\]
Since $Jw-Jv=-\gamma J(Jv)$, we have $\|Jw-Jv\|\le\gamma L\|Jv\|$, and
$\langle Jw,w\rangle\ge0$ gives
\begin{equation}\label{pv:P}
 \|M_{\rm eg}v\|^2
 \le\|v\|^2-\gamma^2(1-q_{\rm eg})\|Jv\|^2\le\|v\|^2.
\end{equation}

\emph{One step.} Fix a worker, drop the index $i$, write $F_t=G(x_t)$,
and let $\E_t$ denote conditional expectation given the history before
the first batch of step $t$. The first batch gives $\E_te_t=0$ and
$\E_t\|e_t\|^2\le \alpha_{\rm rel}\|F_t\|^2$. Since $G$ is affine,
\[
 G(y_t)=F_t-\gamma J(F_t+e_t)=M_{{\rm eg},1}F_t-\gamma Je_t,
\]
and~\eqref{pv:H} yields
\begin{equation}\label{pv:corrector}
 \E_t\|G(y_t)\|^2
 =\|M_{{\rm eg},1}F_t\|^2+\gamma^2\E_t\|Je_t\|^2
 \le(1+q_{\rm eg}+q_{\rm eg}\alpha_{\rm rel})\|F_t\|^2.
\end{equation}
The second batch is centered given $y_t$, and its variance is at most
$\alpha_{\rm rel}\|G(y_t)\|^2$; hence
$\E_t\|\delta_t\|^2
\le \alpha_{\rm rel}(1+q_{\rm eg}+q_{\rm eg}\alpha_{\rm rel})
\|F_t\|^2$. Taking
the expectation given $z$ proves the last claim of the lemma.

\emph{Recursion for the operator.} By~\eqref{pv:algorithm},
\begin{equation}\label{pv:field-recursion}
 F_{t+1}=F_t-\gamma J\bigl(G(y_t)+\delta_t\bigr)
 =M_{\rm eg}F_t+\zeta_t,\qquad
 \zeta_t=\gamma^2J^2e_t-\gamma J\delta_t.
\end{equation}
We have $\E_t\zeta_t=0$. Because $\delta_t$ is centered given $e_t$, the
two terms of $\zeta_t$ are orthogonal in $L^2$, and
\begin{equation}\label{pv:field-variance}
 \E_t\|\zeta_t\|^2
 \le q_{\rm eg}^2\alpha_{\rm rel}\|F_t\|^2
 +q_{\rm eg}\alpha_{\rm rel}
   (1+q_{\rm eg}+q_{\rm eg}\alpha_{\rm rel})\|F_t\|^2
 =\vartheta_{\rm grow}\|F_t\|^2.
\end{equation}

\emph{Summation.} Unrolling~\eqref{pv:field-recursion} gives
$F_t=M_{\rm eg}^tG(z)+\sum_{j<t}M_{\rm eg}^{t-1-j}\zeta_j$.
The terms $\zeta_j$ are martingale differences, and
$M_{\rm eg}^tG(z)$ is fixed given $z$, so all cross terms vanish in
expectation. Using $\|M_{\rm eg}\|_{\rm op}\le1$
from~\eqref{pv:P} and~\eqref{pv:field-variance},
\[
 s_t=d_t+\sum_{j<t}\E\bigl[\|M_{\rm eg}^{t-1-j}\zeta_j\|^2\bigm|z\bigr]
 \le d_t+\vartheta_{\rm grow}\sum_{j<t}s_j.
\]
(The second moments are finite by induction on $t$, since every step is
affine in the previous iterate and the noise.) We claim that
\[
 s_t\le u_t:=d_t+\vartheta_{\rm grow}
       \sum_{j<t}(1+\vartheta_{\rm grow})^{t-1-j}d_j .
\]
For $t=0$ this is $s_0=d_0$. If it holds for all $j<t$, then
$s_t\le d_t+\vartheta_{\rm grow}\sum_{j<t}u_j$, and
\[
 \sum_{j<t}u_j
 =\sum_{j<t}d_j\Bigl(1+\vartheta_{\rm grow}
       \sum_{l=j+1}^{t-1}(1+\vartheta_{\rm grow})^{l-1-j}\Bigr)
 =\sum_{j<t}(1+\vartheta_{\rm grow})^{t-1-j}d_j,
\]
which gives $s_t\le u_t$. The same identity with $t=T$ yields
\[
 \sum_{t<T}s_t\le\sum_{t<T}u_t
 =\sum_{j<T}(1+\vartheta_{\rm grow})^{T-1-j}d_j
 \le(1+\vartheta_{\rm grow})^{T-1}\sum_{j<T}d_j.
\]

\emph{The averaged iterate.} The noise is centered and the updates are
affine, so $\E[x_t^i-x_\star\mid z]=M_{\rm eg}^t(z-x_\star)$ for every $i$.
Since $J$ commutes with $M_{\rm eg}$,
$\E[G(\bar x_t)\mid z]
=JM_{\rm eg}^t(z-x_\star)=M_{\rm eg}^tG(z)$, and Jensen's
inequality gives $\E[\|G(\bar x_t)\|^2\mid z]\ge d_t$.
\end{proof}

\begin{lemma}[Averaging absorbs the relative noise]\label{pv:block}
Assume SGC, $0<\gamma\le1/(4L)$, and
\begin{equation}\label{pv:condition}
 \frac{c_{\rm avg}(1+\vartheta_{\rm grow})^{T-1}}{m}
 \le\frac{1-2q_{\rm eg}}2.
\end{equation}
Then the output of Definition~\ref{pv:definition} satisfies
\begin{equation}\label{pv:block-bound}
 \E\Gap_\star^{\rm M}(\widehat y)\le\frac{3R_0^2}{\gamma NT}.
\end{equation}
For $\gamma=1/(4L)$ the bound is $12LR_0^2/(NT)$, and
condition~\eqref{pv:condition} reads
\begin{equation}\label{pv:explicit-condition}
 \alpha_{\rm rel}(35+\alpha_{\rm rel})
 \left(1+\frac{\alpha_{\rm rel}(18+\alpha_{\rm rel})}{256}\right)^{T-1}
 \le7m.
\end{equation}
\end{lemma}
\begin{proof}
Within a round we may reveal, at each local step $t$, first the
predictor batches of all workers and then their corrector batches.
The responses are mutually independent across workers, so this order
does not change the joint law and respects every dependency.

\emph{The averaged trajectory.} Let bars denote averages over workers.
Since $G$ is affine, $G(\bar x)$ is the average of the $G(x^i)$, and
\[
 \bar y_t=\bar x_t-\gamma\bigl(G(\bar x_t)+\bar e_t\bigr),\qquad
 \bar x_{t+1}=\bar x_t-\gamma\bigl(G(\bar y_t)+\bar\delta_t\bigr).
\]
This is one extragradient step with errors $\bar e_t,\bar\delta_t$, so
identity~\eqref{nn:identity} and bound~\eqref{nn:pathwise-residual} of
Lemma~\ref{nn:identity-lemma}, with $F=G$, give for every
$u\in\mathbb R^d$
\begin{align}\label{pv:gap-identity}
 2\gamma\langle G(\bar y_t),\bar y_t-u\rangle
 \le{}&\|\bar x_t-u\|^2-\|\bar x_{t+1}-u\|^2
 +\overline{\mathcal R}_{{\rm eg},t}
 -2\gamma\bigl\langle\bar\delta_t,\bar x_t-\gamma G(\bar y_t)-u\bigr\rangle,
\end{align}
where
\begin{equation}\label{pv:remainder}
 \overline{\mathcal R}_{{\rm eg},t}
 =-(1-2q_{\rm eg})\gamma^2\|G(\bar x_t)+\bar e_t\|^2
          +2\gamma^2\|\bar e_t\|^2+\gamma^2\|\bar\delta_t\|^2 .
\end{equation}

\emph{The remainder.} Condition on the round-start point $z$. The
errors of different workers are independent and centered, so
\[
 \E[\|\bar e_t\|^2\mid z]\le\frac{\alpha_{\rm rel}s_t}m,\qquad
 \E[\|\bar\delta_t\|^2\mid z]
 \le\frac{\alpha_{\rm rel}(1+q_{\rm eg}
 +q_{\rm eg}\alpha_{\rm rel})s_t}m,
\]
by Lemma~\ref{pv:field-lemma}. Since $\bar e_t$ is centered,
$\E[\|G(\bar x_t)+\bar e_t\|^2\mid z]
=\E[\|G(\bar x_t)\|^2\mid z]+\E[\|\bar e_t\|^2\mid z]$, and
Lemma~\ref{pv:field-lemma} gives
\[
 \frac{\E[\overline{\mathcal R}_{{\rm eg},t}\mid z]}{\gamma^2}
 \le-(1-2q_{\rm eg})d_t
 +(1+2q_{\rm eg})\frac{\alpha_{\rm rel}s_t}m
 +\frac{\alpha_{\rm rel}(1+q_{\rm eg}
 +q_{\rm eg}\alpha_{\rm rel})s_t}m
 =-(1-2q_{\rm eg})d_t+\frac{c_{\rm avg}}{m}s_t.
\]
Summing over $t$ and using~\eqref{pv:field-energy}
and~\eqref{pv:condition},
\begin{equation}\label{pv:block-absorption}
 \E\Bigl[\sum_{t<T}\overline{\mathcal R}_{{\rm eg},t}\Bigm|z\Bigr]
 \le\gamma^2\left(-(1-2q_{\rm eg})
 +\frac{c_{\rm avg}(1+\vartheta_{\rm grow})^{T-1}}m\right)
 \sum_{t<T}d_t
 \le-\frac{(1-2q_{\rm eg})\gamma^2}2\sum_{t<T}d_t.
\end{equation}

\emph{Energy over all rounds.} Restore the round index $r$ and let
$\mathcal D_{\rm eg}=\sum_{r<N}\sum_{t<T}\E\|M_{\rm eg}^tG(z_r)\|^2$.
Take $u=x_\star$
in~\eqref{pv:gap-identity}. The left side is nonnegative by
monotonicity and $G(x_\star)=0$. The vector paired with $\bar\delta_t$ is
determined before the corrector batch, so the last term has zero mean.
Hence
$\E\|\bar x_{t+1}-x_\star\|^2\le\E\|\bar x_t-x_\star\|^2
+\E\overline{\mathcal R}_{{\rm eg},t}$.
Since $\bar x_{r,0}=z_r$ and $\bar x_{r,T}=z_{r+1}$,
summing over $t$ and $r$ and using~\eqref{pv:block-absorption} gives
\begin{equation}\label{pv:global-energy}
 \E\|z_N-x_\star\|^2
 +\frac{(1-2q_{\rm eg})\gamma^2}2\,\mathcal D_{\rm eg}\le R_0^2.
\end{equation}

\emph{The corrector sum.} Ordered by $(r,t)$, the averaged correctors
$\bar\delta_{r,t}$ are martingale differences. By
Lemma~\ref{pv:field-lemma},
$\alpha_{\rm rel}(1+q_{\rm eg}+q_{\rm eg}\alpha_{\rm rel})
\le c_{\rm avg}$, \eqref{pv:condition}
and~\eqref{pv:global-energy},
\begin{equation}\label{pv:martingale}
 \E\Bigl\|\sum_{r,t}\bar\delta_{r,t}\Bigr\|^2
 =\sum_{r,t}\E\|\bar\delta_{r,t}\|^2
 \le\frac{c_{\rm avg}(1+\vartheta_{\rm grow})^{T-1}}m\,
       \mathcal D_{\rm eg}
 \le\frac{1-2q_{\rm eg}}2\,\mathcal D_{\rm eg}
 \le\frac{R_0^2}{\gamma^2}.
\end{equation}

\emph{The gap.} Let $\|u-x_\star\|\le R_0$. Monotonicity gives
$\langle G(u),\bar y_t-u\rangle\le\langle G(\bar y_t),\bar y_t-u\rangle$.
Sum~\eqref{pv:gap-identity} over all $(r,t)$, drop
$-\|z_N-u\|^2$, use $\|x_0-u\|\le2R_0$, and write
$\bar x_{r,t}-\gamma G(\bar y_{r,t})-u
=(\bar x_{r,t}-\gamma G(\bar y_{r,t})-x_\star)+(x_\star-u)$. Since
$\widehat y$ is the average of the $\bar y_{r,t}$,
\begin{align}\label{pv:supremum}
 2\gamma NT\langle G(u),\widehat y-u\rangle
 \le{}&4R_0^2+\sum_{r,t}\overline{\mathcal R}_{{\rm eg},r,t}
       +2\gamma R_0\Bigl\|\sum_{r,t}\bar\delta_{r,t}\Bigr\|\nonumber\\
 &-2\gamma\sum_{r,t}\bigl\langle\bar\delta_{r,t},
              \bar x_{r,t}-\gamma G(\bar y_{r,t})-x_\star\bigr\rangle.
\end{align}
The right side does not depend on $u$, so it also bounds
$2\gamma NT\,\Gap_\star^{\rm M}(\widehat y)$. In expectation, the second term is
nonpositive by~\eqref{pv:block-absorption}, the third is at most
$2\gamma R_0\cdot R_0/\gamma=2R_0^2$ by~\eqref{pv:martingale} and
Jensen's inequality, and the last has zero mean. Therefore
$2\gamma NT\,\E\Gap_\star^{\rm M}(\widehat y)\le6R_0^2$,
which is~\eqref{pv:block-bound}.

For $\gamma=1/(4L)$ we have $q_{\rm eg}=1/16$, $1-2q_{\rm eg}=7/8$,
$c_{\rm avg}=\alpha_{\rm rel}(35+\alpha_{\rm rel})/16$ and
$\vartheta_{\rm grow}=\alpha_{\rm rel}(18+\alpha_{\rm rel})/256$;
substituting these values
into~\eqref{pv:condition} gives~\eqref{pv:explicit-condition}.
\end{proof}

\subsection{The local horizon and the proof of Theorem~\ref{thm:affine-integrated}}
\label{pv:comparison-subsection}

For $K\ge2$, let $T_\star$ be the largest integer
$T\in\{1,\ldots,\lfloor K/2\rfloor\}$ that
satisfies~\eqref{pv:explicit-condition} with $b=\lfloor K/(2T)\rfloor$
and $\alpha_{\rm rel}=(\rho-1)/b$; set $T_\star=0$ if $K=1$ or if no
such $T$ exists.
Computing $T_\star$ is a finite search over known parameters. The next
lemma shows that this horizon captures the quantity $\beta$.

\begin{lemma}[The admissible horizon]\label{pv:comparison}
For all $m,K\ge1$ and $\rho\ge1$,
\begin{equation*}
 \beta\le1024\max\Bigl\{1,\frac K\rho,T_\star\Bigr\}.
\end{equation*}
\end{lemma}
\begin{proof}
Write $K_{\rm eff}=K/\rho=\E Z_1$ and
$T_{\rm scale}=\max\{1,K_{\rm eff},T_\star\}$.
If $K=1$, then $\beta\le K=1$. If $m<1024$, then
$\beta\le\E\sum_iZ_i=mK_{\rm eff}\le1024K_{\rm eff}$. Assume from now on that $K\ge2$
and $m\ge1024$.

\emph{A horizon that satisfies the condition.} The function
$t\mapsto t\log(1+t/K_{\rm eff})$ increases from $0$ to $\infty$ on
$(0,\infty)$, so there is a unique $T_{\rm cont}>0$ with
\begin{equation*}
 T_{\rm cont}\log\Bigl(1+\frac {T_{\rm cont}}{K_{\rm eff}}\Bigr)
 =\frac14\log m.
\end{equation*}
Let $T=\lfloor\min\{K/2,T_{\rm cont}\}\rfloor$ and suppose $T\ge1$. Then
$K/(2T)\ge1$, so $b=\lfloor K/(2T)\rfloor\ge K/(4T)$ and
$\alpha_{\rm rel}\le\rho/b\le4T/K_{\rm eff}=4r_{\rm hor}$, where
$r_{\rm hor}=T/K_{\rm eff}$. Therefore
\[
 \frac{\alpha_{\rm rel}(35+\alpha_{\rm rel})}{16}
 \le\frac{35}4r_{\rm hor}+r_{\rm hor}^2\le10(1+r_{\rm hor})^2,
 \qquad
 1+\frac{\alpha_{\rm rel}(18+\alpha_{\rm rel})}{256}
 \le1+\frac9{32}r_{\rm hor}+\frac{r_{\rm hor}^2}{16}
 \le(1+r_{\rm hor})^2.
\]
Since $T\le T_{\rm cont}$ and
$t\mapsto t\log(1+t/K_{\rm eff})$ is increasing,
$(1+r_{\rm hor})^{2T}=\exp\{2T\log(1+T/K_{\rm eff})\}\le\sqrt m$.
Hence
\begin{equation*}
 \frac{\alpha_{\rm rel}(35+\alpha_{\rm rel})}{16}
 \left(1+\frac{\alpha_{\rm rel}(18+\alpha_{\rm rel})}{256}\right)^{T-1}
 \le10(1+r_{\rm hor})^{2T}\le10\sqrt m\le\frac7{16}m,
\end{equation*}
where the last step uses $m\ge1024$. This
is~\eqref{pv:explicit-condition}, so $T\le T_\star$. The inequality
$T\le T_\star$ also holds when $T=0$.

\emph{A tail bound for $\beta$.} The Chernoff bound for the binomial
law gives $\Prob\{Z_1\ge t\}\le(eK_{\rm eff}/t)^t$ for
$t\ge K_{\rm eff}$, and a union bound gives
$\Prob\{\max_iZ_i\ge t\}\le m(eK_{\rm eff}/t)^t$. Let
$t_0=4(K_{\rm eff}+T_{\rm cont}+1)$ and $c_0=\log4-1>0$.
For $t\ge t_0$ we have
$t\ge4(K_{\rm eff}+T_{\rm cont})$ and $t\ge4T_{\rm cont}$, so
\[
 \log\frac t{eK_{\rm eff}}
 \ge\log\frac{4(K_{\rm eff}+T_{\rm cont})}{eK_{\rm eff}}
 =\log\Bigl(1+\frac {T_{\rm cont}}{K_{\rm eff}}\Bigr)+c_0,
 \qquad
 t\log\Bigl(1+\frac {T_{\rm cont}}{K_{\rm eff}}\Bigr)
 \ge4T_{\rm cont}\log\Bigl(1+\frac {T_{\rm cont}}{K_{\rm eff}}\Bigr)
 =\log m .
\]
Hence $\Prob\{\max_iZ_i\ge t\}\le m\,e^{-\log m-c_0t}=e^{-c_0t}$, and
\begin{equation}\label{pv:tail-integration}
 \beta\le t_0+\int_{t_0}^\infty e^{-c_0t}\,dt
 \le t_0+\frac{e^{-4c_0}}{c_0}
 \le8(K_{\rm eff}+T_{\rm cont}+1).
\end{equation}

\emph{Conclusion.} If $T_{\rm cont}\le K/2$, then
$T=\lfloor T_{\rm cont}\rfloor$ and
$T_{\rm cont}<T+1\le T_\star+1\le2T_{\rm scale}$, so
\eqref{pv:tail-integration}
gives $\beta\le8(T_{\rm scale}+2T_{\rm scale}+T_{\rm scale})
=32T_{\rm scale}$. If $T_{\rm cont}>K/2$, then
$T=\lfloor K/2\rfloor\ge1$,
so $\lfloor K/2\rfloor\le T_\star$ and
$\beta\le K\le3\lfloor K/2\rfloor\le3T_{\rm scale}$. In all cases
$\beta\le1024T_{\rm scale}$.
\end{proof}

\begin{proof}[Proof of Theorem~\ref{thm:affine-integrated}]
\emph{Upper bound.} The algorithm outputs the best of four
candidates, chosen by their parameter-based bounds: the initial point,
serial and common-query extragradient (Lemma~\ref{lem:relmm:upper}),
and, if $T_\star\ge1$, local extragradient with $T=T_\star$,
$b=\lfloor K/(2T_\star)\rfloor$ and $\gamma=1/(4L)$
(Lemma~\ref{pv:block}). Let $C_0\ge6$ be the constant of
Lemma~\ref{lem:relmm:upper} and
$T_{\rm scale}=\max\{1,K/\rho,T_\star\}$.

Suppose $mK/\rho>1$. The common-query bound is at most
$2C_0LR_0^2/N$, the serial bound is $C_0LR_0^2/(N\cdot K/\rho)$, and
the local bound is $12LR_0^2/(NT_\star)\le2C_0LR_0^2/(NT_\star)$.
Together with the initial point,
\[
 \E\Gap_\star^{\rm M}(\widehat x)
 \le LR_0^2\min\Bigl\{1,\frac{2C_0}{NT_{\rm scale}}\Bigr\}
 \le LR_0^2\min\Bigl\{1,\frac{2048C_0}{N\beta}\Bigr\},
\]
where the last step is Lemma~\ref{pv:comparison}. This also covers
$K=1$ and $T_\star=0$, in which case
$T_{\rm scale}=\max\{1,K/\rho\}$.

Suppose $mK/\rho\le1$. Lemma~\ref{lem:clock:elementary} gives
$\beta\le mK/\rho$, and the common-query bound is at most
\[
 \frac{C_0LR_0^2}N\Bigl(1+\frac\rho{mK}\Bigr)
 \le\frac{2C_0LR_0^2}{N\cdot mK/\rho}
 \le\frac{2C_0LR_0^2}{N\beta}.
\]

In both cases
$\E\Gap_\star^{\rm M}(\widehat x)\le LR_0^2\min\{1,C_{\rm aff}/(N\beta)\}$,
where $C_{\rm aff}=2048C_0$. For $x>0$,
$\min\{1,C_{\rm aff}/x\}\le(C_{\rm aff}+1)/(1+x)$: if
$x\le C_{\rm aff}$ the left side is at most
$1\le(C_{\rm aff}+1)/(1+x)$, and if $x>C_{\rm aff}$ then
$C_{\rm aff}(1+x)\le(C_{\rm aff}+1)x$. Hence
\[
 \E\Gap_\star^{\rm M}(\widehat x)\le\frac{(C_{\rm aff}+1)LR_0^2}{1+N\beta}
 \le\frac{64(C_{\rm aff}+1)LR_0^2}{1+N\InfoDepth},
\]
by Theorem~\ref{thm:clock:asymptotic}. This holds in every dimension.

\emph{Lower bound.} Lemma~\ref{lem:relmm:masked} gives the matching
lower bound in dimension
$d_{\rm hard}\le1024mT_{\rm mask}^2\le1024m(KN)^2$,
with a single distribution of
instances that does not depend on the algorithm.
\end{proof}

\subsection{Mixed noise}\label{an:section}

Now Assumption~\ref{ass:mixed} holds with $\rho\ge1$ and $\sigma\ge0$:
\begin{equation}\label{an:noise}
 \E[\widehat G(x)\mid\mathcal F]=G(x),\qquad
 \E[\|\widehat G(x)-G(x)\|^2\mid\mathcal F]
 \le(\rho-1)\|G(x)\|^2+\sigma^2.
\end{equation}
A batch of $b$ responses then has variance at most
$\alpha_{\rm rel}\|G(x)\|^2+\sigma_b^2$, with
$\alpha_{\rm rel}=(\rho-1)/b$ and $\sigma_b^2=\sigma^2/b$.

\begin{definition}[Independent trajectories]\label{an:parallel-protocol}
For integers $b,T\ge1$ and a step size $\gamma>0$, every worker starts at
$x_0$ and performs $T$ steps of~\eqref{pv:algorithm} with batch size $b$,
without communication. At the end the workers send their running sums
and the method outputs
$\widehat y=\frac1{mT}\sum_{i=1}^m\sum_{t<T}y_t^i$.
\end{definition}

\begin{lemma}[Independent trajectories under mixed noise]\label{an:parallel}
If $\alpha_{\rm rel}=(\rho-1)/b\le1/8$ and
$0<\gamma\le1/(4L)$, then the output of
Definition~\ref{an:parallel-protocol} satisfies
\begin{equation}\label{an:parallel-certificate}
 \E\Gap_\star^{\rm M}(\widehat y)
 \le\frac{4R_0^2}{\gamma T}
       +\frac{5\gamma\sigma^2}{2mb}
       +\frac{\sqrt3\,\sigma R_0}{\sqrt{mbT}}.
\end{equation}
In particular, for $\gamma=\min\{1/(4L),R_0\sqrt{mb}/(\sigma\sqrt T)\}$,
where the second entry is $+\infty$ if $\sigma=0$,
\begin{equation}\label{an:parallel-optimized}
 \E\Gap_\star^{\rm M}(\widehat y)
 \le\frac{16LR_0^2}{T}+\frac{9\sigma R_0}{\sqrt{mbT}}.
\end{equation}
\end{lemma}
\begin{proof}
We have $q_{\rm eg}=\gamma^2L^2\le1/16$ and
$\alpha_{\rm rel}\le1/8$, hence
\begin{equation}\label{an:constants}
 \begin{aligned}
 1-2q_{\rm eg}&\ge\frac78,
 &c_{\rm avg}=\alpha_{\rm rel}(2+3q_{\rm eg}
   +q_{\rm eg}\alpha_{\rm rel})&\le\frac{281}{1024},
 &\alpha_{\rm rel}(1+q_{\rm eg}+q_{\rm eg}\alpha_{\rm rel})
   &\le\frac{137}{1024},\\
 1+q_{\rm eg}\alpha_{\rm rel}&\le\frac{129}{128},
 &2+2q_{\rm eg}+q_{\rm eg}\alpha_{\rm rel}&\le\frac{273}{128}.
 \end{aligned}
\end{equation}
In particular $1-2q_{\rm eg}-c_{\rm avg}\ge1/2$ and
$2+2q_{\rm eg}+q_{\rm eg}\alpha_{\rm rel}\le3$. All second moments are
finite by induction, since every step is affine in the previous iterate
and the noise; this justifies the expectations below.

\emph{One worker.} Let $F_t=G(x_t)$ and let $\E_t$ condition on the
history before step $t$. As in~\eqref{pv:corrector}, now with the
additional variance $\sigma_b^2$,
\begin{equation*}
 \E_t\|G(y_t)\|^2
 \le(1+q_{\rm eg}+q_{\rm eg}\alpha_{\rm rel})\|F_t\|^2
      +q_{\rm eg}\sigma_b^2,\qquad
 \E_t\|\delta_t\|^2
 \le\alpha_{\rm rel}(1+q_{\rm eg}+q_{\rm eg}\alpha_{\rm rel})
      \|F_t\|^2+(1+q_{\rm eg}\alpha_{\rm rel})\sigma_b^2.
\end{equation*}
Let $\mathcal R_{{\rm eg},t}$ be the remainder~\eqref{pv:remainder}
written for this worker.
Since $e_t$ is centered,
\begin{align*}
 \frac{\E_t\mathcal R_{{\rm eg},t}}{\gamma^2}
 & =-(1-2q_{\rm eg})\|F_t\|^2
   +(1+2q_{\rm eg})\E_t\|e_t\|^2+\E_t\|\delta_t\|^2\\
 & \le-(1-2q_{\rm eg}-c_{\rm avg})\|F_t\|^2
   +(2+2q_{\rm eg}+q_{\rm eg}\alpha_{\rm rel})\sigma_b^2\\
 & \le-\frac12\|F_t\|^2+3\sigma_b^2.
\end{align*}
Apply~\eqref{pv:gap-identity} to this worker with $u=x_\star$; the left
side is nonnegative and the last term has zero mean. Summing over $t$,
\begin{equation}\label{an:parallel-local-energy}
 \sum_{t<T}\E\|G(x_t^i)\|^2
 \le\frac{2R_0^2}{\gamma^2}+6T\sigma_b^2
 \qquad(i=1,\ldots,m).
\end{equation}

\emph{The averaged trajectory.} Let bars denote averages over workers
and $V_{F,t}=\frac1m\sum_i\E\|G(x_t^i)\|^2$. Since $G$ is affine, the
averaged iterates follow one extragradient step with errors
$\bar e_t,\bar\delta_t$, so~\eqref{pv:gap-identity}
and~\eqref{pv:remainder} apply to them. The workers' errors are
independent and centered, hence
\[
 \E\|\bar e_t\|^2\le\frac{\alpha_{\rm rel}V_{F,t}+\sigma_b^2}m,
 \qquad
 \E\|\bar\delta_t\|^2
 \le\frac{\alpha_{\rm rel}(1+q_{\rm eg}
       +q_{\rm eg}\alpha_{\rm rel})V_{F,t}
       +(1+q_{\rm eg}\alpha_{\rm rel})\sigma_b^2}m,
\]
and, as in the proof of Lemma~\ref{pv:block},
\[
 \E\overline{\mathcal R}_{{\rm eg},t}
 \le-(1-2q_{\rm eg})\gamma^2\E\|G(\bar x_t)\|^2
 +\frac{\gamma^2}m\left(c_{\rm avg}V_{F,t}
 +(2+2q_{\rm eg}+q_{\rm eg}\alpha_{\rm rel})\sigma_b^2\right).
\]
Drop the negative term and sum over $t$, using~\eqref{an:parallel-local-energy}
and~\eqref{an:constants} in the form $2c_{\rm avg}\le1$,
$6c_{\rm avg}+2+2q_{\rm eg}+q_{\rm eg}\alpha_{\rm rel}\le5$,
$2\alpha_{\rm rel}(1+q_{\rm eg}+q_{\rm eg}\alpha_{\rm rel})\le1/2$
and
$6\alpha_{\rm rel}(1+q_{\rm eg}+q_{\rm eg}\alpha_{\rm rel})
+1+q_{\rm eg}\alpha_{\rm rel}\le3$. The averaged correctors
are martingale differences, so
\begin{align*}
 \sum_{t<T}\E\overline{\mathcal R}_{{\rm eg},t}
 &\le\frac{R_0^2}m+\frac{5\gamma^2T\sigma_b^2}m,
       \\
 \E\Bigl\|\sum_{t<T}\bar\delta_t\Bigr\|^2
 =\sum_{t<T}\E\|\bar\delta_t\|^2
 &\le\frac{R_0^2}{2m\gamma^2}+\frac{3T\sigma_b^2}m.
\end{align*}

\emph{The gap.} The argument leading to~\eqref{pv:supremum}, applied to
the single uninterrupted averaged trajectory, gives
\begin{align*}
 2\gamma T\,\E\Gap_\star^{\rm M}(\widehat y)
 &\le4R_0^2+\sum_{t<T}\E\overline{\mathcal R}_{{\rm eg},t}
       +2\gamma R_0\Bigl(\E\Bigl\|\sum_{t<T}\bar\delta_t\Bigr\|^2\Bigr)^{1/2}\\
 &\le4R_0^2+\frac{R_0^2}m+\frac{5\gamma^2T\sigma_b^2}m
       +2\gamma R_0\Bigl(\frac{R_0}{\gamma\sqrt{2m}}
       +\sigma_b\sqrt{\frac{3T}m}\Bigr)\\
 &\le(5+\sqrt2)R_0^2+\frac{5\gamma^2T\sigma_b^2}m
       +2\sqrt3\,\gamma R_0\sigma_b\sqrt{\frac Tm},
\end{align*}
where we used $\sqrt{x+y}\le\sqrt x+\sqrt y$ and $m\ge1$. Dividing by
$2\gamma T$, using $(5+\sqrt2)/2<4$ and $\sigma_b^2=\sigma^2/b$
proves~\eqref{an:parallel-certificate}.

For the given step size,
\[
 \frac1\gamma\le4L+\frac{\sigma_b\sqrt T}{R_0\sqrt m},\qquad
 \frac{\gamma\sigma_b^2}m\le\frac{\sigma_b R_0}{\sqrt{mT}}.
\]
Substituting into~\eqref{an:parallel-certificate} gives
$16LR_0^2/T$ plus
$(4+5/2+\sqrt3)\,\sigma_b R_0/\sqrt{mT}$, and
$4+5/2+\sqrt3<9$ proves~\eqref{an:parallel-optimized}.
\end{proof}

\begin{corollary}[Upper bound under mixed noise]\label{an:mixed-explicit}
For every $m,K,N\in\mathbb N$ there is an algorithm of
Definition~\ref{def:ic-main} with
\begin{equation}\label{an:mixed-explicit-bound}
 \E\Gap_\star^{\rm M}(\widehat x)
 \le\min\left\{LR_0^2,\
       \frac{512LR_0^2}{N}\min\Bigl\{\frac\rho K,1+\frac\rho{mK}\Bigr\}
       +\frac{18\sigma R_0}{\sqrt{mKN}}\right\}.
\end{equation}
\end{corollary}
\begin{proof}
We combine two instances of Lemma~\ref{an:parallel} with the initial
point.

\emph{Independent trajectories.} Take
$b=\max\{1,\lceil8(\rho-1)\rceil\}$, so that
$\alpha_{\rm rel}=(\rho-1)/b\le1/8$ and
$b\le8\rho$, and $T=\lfloor KN/(2b)\rfloor$. Each worker runs its
trajectory across rounds, keeping an unfinished batch and its query
from one round to the next; it makes at most $K$ calls per round and
$2bT\le KN$ calls in total. If $T\ge1$, then $T\ge KN/(4b)$ and
$bT\ge KN/4$, and~\eqref{an:parallel-optimized} gives
\[
 \E\Gap_\star^{\rm M}(\widehat y)
 \le\frac{64bLR_0^2}{KN}+\frac{18\sigma R_0}{\sqrt{mKN}}
 \le\frac{512\rho LR_0^2}{KN}+\frac{18\sigma R_0}{\sqrt{mKN}}.
\]
If $T=0$, then $KN<2b\le16\rho$, and the first term exceeds $LR_0^2$.

\emph{Common queries.} Let all workers query the same point and pool
their responses over
$r_{\rm cq}=\max\{1,\lceil8(\rho-1)/(mK)\rceil\}$ rounds per stage,
so that each stage averages $b_{\rm cq}=mKr_{\rm cq}$ responses and
$(\rho-1)/b_{\rm cq}\le1/8$. Perform
$T_{\rm cq}=\lfloor N/(2r_{\rm cq})\rfloor$ steps, and apply
Lemma~\ref{an:parallel} with one worker and batch size $b_{\rm cq}$.
If $T_{\rm cq}\ge1$, then $r_{\rm cq}T_{\rm cq}\ge N/4$ and
$r_{\rm cq}\le1+8\rho/(mK)\le8(1+\rho/(mK))$,
so
\[
 \E\Gap_\star^{\rm M}(\widehat y)
 \le\frac{64r_{\rm cq}LR_0^2}{N}+\frac{18\sigma R_0}{\sqrt{mKN}}
 \le\frac{512LR_0^2}{N}\Bigl(1+\frac\rho{mK}\Bigr)
       +\frac{18\sigma R_0}{\sqrt{mKN}}.
\]
If $T_{\rm cq}=0$, then
$N<2r_{\rm cq}\le16(1+\rho/(mK))$, and the first term exceeds
$LR_0^2$.

The initial point satisfies $\Gap_\star^{\rm M}(x_0)\le LR_0^2$
(Lemma~\ref{lem:relmm:upper}). Outputting the candidate with the
smallest bound proves~\eqref{an:mixed-explicit-bound}.
\end{proof}

\begin{lemma}[Statistical lower bound]\label{an:statistical}
Under~\eqref{an:noise}, in every dimension $d\ge2$,
\begin{equation*}
 \mathcal E_d^{\rm aff,mix}
 \ge\frac3{128}\min\left\{LR_0^2,\
                         \frac{\sigma R_0}{\sqrt{mKN}}\right\}.
\end{equation*}
The hard instances are $L$-Lipschitz, have unique roots, and use a
Gaussian oracle whose variance equals $\sigma^2$ at every point.
\end{lemma}
\begin{proof}
If $\sigma=0$ there is nothing to prove. Let $\sigma>0$,
$Q_{\rm tot}=mKN$, and
$\lambda=\min\{L,\sigma/(4R_0\sqrt{Q_{\rm tot}})\}$. For
$\theta_{\rm test}\in\{-R_0,R_0\}$ and
$x=(x_1,x_2)\in\mathbb R^2$ define
\[
 G_{\theta_{\rm test}}(x)
 =\bigl(\lambda(x_1-\theta_{\rm test}),\,Lx_2\bigr),\qquad
 \widehat G_{\theta_{\rm test}}(x)
 =G_{\theta_{\rm test}}(x)+(\sigma\zeta,0),
 \quad \zeta\sim\mathcal N(0,1).
\]
With $x_0=0$, the operator $G_{\theta_{\rm test}}$ is monotone and
$L$-Lipschitz, its unique root is
$x_{\star,\theta_{\rm test}}=(\theta_{\rm test},0)$, and
$\|x_0-x_{\star,\theta_{\rm test}}\|=R_0$. The oracle
satisfies~\eqref{an:noise}. For $d>2$, add coordinates on which the
operator is $Lx_j$ without noise.

\emph{Reduction to testing.} Give $\lambda$ to the algorithm, allow it
to communicate after every call, and let it make exactly $Q_{\rm tot}$
calls; this can only help it. Subtracting the known value $\lambda x_1$
from the first coordinate of each response leaves $Q_{\rm tot}$
independent $\mathcal N(-\lambda\theta_{\rm test},\sigma^2)$ observations,
and the other coordinates carry no information about $\theta_{\rm test}$.
Since $Q_{\rm tot}\lambda^2R_0^2/\sigma^2\le1/16$,
Lemma~\ref{lem:gaussian-testing} shows that every test of the sign of
$\theta_{\rm test}$ based on the transcript errs with average probability
at least $3/8$. We use the test that declares the sign of $\widehat x_1$
(with either convention at $\widehat x_1=0$). When it errs,
$\widehat x_1$ and $\theta_{\rm test}$ lie on different sides of $0$, so
$|\widehat x_1-\theta_{\rm test}|\ge R_0$.

\emph{Gap on an error.} On this event put
$\Delta_{\rm test}=(R_0/2)
\operatorname{sign}(\widehat x_1-\theta_{\rm test})$ and
$u=x_{\star,\theta_{\rm test}}+(\Delta_{\rm test},0)$. Then
$\|u-x_{\star,\theta_{\rm test}}\|\le R_0$ and
$G_{\theta_{\rm test}}(u)=(\lambda\Delta_{\rm test},0)$. Hence
\[
 \Gap_{\star}^{\rm M}(\widehat x)
 \ge\langle G_{\theta_{\rm test}}(u),\widehat x-u\rangle
 =\lambda\Delta_{\rm test}
   (\widehat x_1-\theta_{\rm test}-\Delta_{\rm test})
 =\lambda\left(\frac{R_0}2|\widehat x_1-\theta_{\rm test}|
   -\frac{R_0^2}4\right)
 \ge\frac{\lambda R_0^2}4.
\]
The gap is nonnegative on the complementary event. Averaging over
$\theta_{\rm test}$,
\[
 \frac12\sum_{\theta_{\rm test}=\pm R_0}\E\Gap_\star^{\rm M}(\widehat x)
 \ge\frac38\cdot\frac{\lambda R_0^2}4
 =\frac{3\lambda R_0^2}{32}
 \ge\frac3{128}\min\left\{LR_0^2,
       \frac{\sigma R_0}{\sqrt{Q_{\rm tot}}}\right\},
\]
and one of the two signs attains at least this average.
\end{proof}

\begin{proof}[Proof of Theorem~\ref{thm:affine-additive-main}]
\emph{Upper bound.} This is Corollary~\ref{an:mixed-explicit}, which
holds in every dimension.

\emph{Lower bound.} Let $G_{\rm hard}$ and $\widehat G_{\rm hard}$ be
the operator and the SGC oracle of Lemma~\ref{lem:relmm:masked}. Add one
coordinate with known operator $Lz_{\rm aux}$ and Gaussian noise:
\[
 G(z_{\rm hard},z_{\rm aux})
 =(G_{\rm hard}(z_{\rm hard}),Lz_{\rm aux}),\qquad
 \widehat G(z_{\rm hard},z_{\rm aux})
 =(\widehat G_{\rm hard}(z_{\rm hard}),Lz_{\rm aux}+\sigma\zeta),
 \qquad \zeta\sim\mathcal N(0,1),
\]
with an independent $\zeta$ at every call. This operator is monotone and
$L$-Lipschitz and has a unique root at distance $R_0$ from the initial
point. Its conditional variance satisfies
\[
 \operatorname{Var}(\widehat G(z_{\rm hard},z_{\rm aux})\mid\mathcal F)
 \le(\rho-1)\|G_{\rm hard}(z_{\rm hard})\|^2+\sigma^2
 \le(\rho-1)\|G(z_{\rm hard},z_{\rm aux})\|^2+\sigma^2,
\]
so~\eqref{an:noise} holds, and its root variance equals $\sigma^2$. An
algorithm can generate the last coordinate of every response itself, so
the extension gives it no information, and comparators of the form
$(u,0)$ show that its gap is at least the gap of the original instance.
Hence~\eqref{pv:minimax-comparison} still holds.

Both this distribution and the one of Lemma~\ref{an:statistical} live in
dimension
\begin{equation*}
 d_{\rm mix}=\max\{d_{\rm hard}+1,2\}=O\bigl(m(KN)^2\bigr),
\end{equation*}
after padding with noiseless coordinates $Lx_j$. Let
$A_0=LR_0^2$, $A_{\rm info}=A_0/(1+N\InfoDepth)$ and
$R_{\rm stat}=\sigma R_0/\sqrt{mKN}$. The
equal mixture of the two distributions is a fixed distribution that
does not depend on the algorithm, and its expected loss is at least
\[
 \frac12\left(\frac{A_{\rm info}}{96}
       +\frac3{128}\min\{A_0,R_{\rm stat}\}\right)
 \ge\frac1{192}\bigl(A_{\rm info}+\min\{A_0,R_{\rm stat}\}\bigr)
 \ge\frac1{192}\min\{A_0,A_{\rm info}+R_{\rm stat}\}.
\]
Thus
\begin{equation}\label{an:combined-bound}
 \mathcal E_{d_{\rm mix}}^{\rm aff,mix}
 \ge\frac1{192}\min\left\{LR_0^2,\
       \frac{LR_0^2}{1+N\InfoDepth}
       +\frac{\sigma R_0}{\sqrt{mKN}}\right\}.
\end{equation}

\emph{Comparison.} Let $K_{\rm eff}=K/\rho$ and
$K_{\rm scale}=\max\{K_{\rm eff},\min\{1,mK_{\rm eff}\}\}$. Then
\[
 \frac1{K_{\rm scale}}
 =\min\Bigl\{\frac\rho K,\max\Bigl\{1,\frac\rho{mK}\Bigr\}\Bigr\}
 \le\min\Bigl\{\frac\rho K,1+\frac\rho{mK}\Bigr\}
 \le\frac2{K_{\rm scale}} .
\]
Moreover $\InfoDepth\le(1+\log(2m)/\log2)K_{\rm scale}$: if
$mK_{\rm eff}\le1$, then
$K_{\rm scale}=mK_{\rm eff}$ and
$\InfoDepth\le mK_{\rm eff}$ by definition; otherwise
$K_{\rm scale}\ge\max\{K_{\rm eff},1\}$ and
$\InfoDepth\le K_{\rm eff}+\log(2m)/\log2
\le(1+\log(2m)/\log2)K_{\rm scale}$. The
upper bound of Corollary~\ref{an:mixed-explicit}, without its constants,
is, for a universal constant $C_{\rm cmp}>0$, bounded as follows. Define
\[
 R_{\rm alg}:=\frac{A_0}N
       \min\left\{\frac\rho K,1+\frac\rho{mK}\right\}+R_{\rm stat}.
\]
Then
\[
 \min\{A_0,R_{\rm alg}\}
 \le4\min\left\{A_0,\frac{A_0}{1+NK_{\rm scale}}+R_{\rm stat}\right\}
 \le C_{\rm cmp}\log(2m)
       \min\{A_0,A_{\rm info}+R_{\rm stat}\}.
\]
For the first inequality, if $NK_{\rm scale}\ge1$ then
$(A_0/N)\cdot2/K_{\rm scale}\le4A_0/(1+NK_{\rm scale})$, and if
$NK_{\rm scale}<1$ the
right side is at least $4\min\{A_0,A_0/2\}\ge A_0$. For the second,
$1+N\InfoDepth
\le(1+\log(2m)/\log2)(1+NK_{\rm scale})$ and
$1+\log(2m)/\log2\le2\log(2m)/\log2$. Together
with~\eqref{an:combined-bound}, this proves the theorem.

\emph{The case $\rho=1$.} Then $\InfoDepth=K$ and
$\min\{\rho/K,1+\rho/(mK)\}=1/K$. Since $KN\ge1$, we have
$KN\le1+KN\le2KN$, so the upper bound~\eqref{an:mixed-explicit-bound}
and the lower bound~\eqref{an:combined-bound} agree up to universal
constants, which proves~\eqref{eq:affine-uniform-specialization},
including $K=N=1$.
\end{proof}

\section{Lower bounds for nonlinear VI under SGC}
\label{vi:lower-section}

This appendix proves the lower bounds of Theorem~\ref{vi:main-monotone}:
Theorem~\ref{vi:uvi-lower} for the Minty gap and Theorem~\ref{vi:sp-lower}
for saddle problems. It also proves the distance lower bound for strongly
monotone operators (Corollary~\ref{vi:strong-exponent}) used in
Appendix~\ref{app:strong-mixed-refinement}. All three constructions
follow the same scheme.
\begin{enumerate}[label=(\arabic*),leftmargin=*]
\item A smooth monotone \emph{chain} on a finite-dimensional Euclidean space whose operator, at a
point with a small coordinate tail, depends only on the revealed prefix
and one additional coordinate (Section~\ref{vi:robust-core}).
\item A \emph{localization} that keeps the operator monotone and
Lipschitz and makes it a known function far from the origin
(Section~\ref{vi:localization-subsection}).
\item A random orthonormal frame in a high-dimensional space that
\emph{hides} the chain, combined with a random mask so that one new
coordinate is revealed only by an informative call
(Section~\ref{vi:haar-subsection}).
\end{enumerate}
Sections~\ref{vi:monotone-lower-subsection},
\ref{vi:product-cutoff-subsection} and~\ref{vi:strong-chain-subsection}
assemble the three lower bounds. Throughout, $\psi(s)=10s^3-15s^4+6s^5$ on
$[0,1]$, extended by $0$ for $s<0$ and by $1$ for $s>1$. This function is
$C^2$ on $\mathbb R$, and
\begin{equation}\label{vi:psi-bounds}
 0\le\psi'(s)=30s^2(1-s)^2\le\frac{15}8,\qquad
 |\psi''(s)|=60s(1-s)|1-2s|\le6\qquad(0\le s\le1).
\end{equation}

\subsection{An affine chain}
\label{vi:affine-subsection}

For an integer $n\ge1$, let $D_n:\mathbb R^n\to\mathbb R^n$,
$D_nu=(u_1-u_2,\ldots,u_{n-1}-u_n,u_n)$. Write $e_j$ for the $j$th
standard basis vector and $\mathbf1$ for the vector of ones. Then
$\|D_n\|\le2$, and
\[
 (D_n^\top v)_1=v_1,\qquad (D_n^\top v)_j=v_j-v_{j-1}\ (j\ge2),\qquad
 D_n\mathbf1=e_n .
\]

\begin{lemma}[Gap of the affine chain]\label{vi:affine-comparison}
For $r,\gamma>0$, the operator
\begin{equation}\label{vi:affine-chain}
 F_n(u,v)=\gamma\Bigl(D_n^\top v+\frac r{\sqrt n}e_1,\ -D_nu\Bigr)
\end{equation}
on $\mathbb R^n\times\mathbb R^n$ is monotone and $2\gamma$-Lipschitz, and its
root is $(0,-r\mathbf1/\sqrt n)$. On the product of two balls of radius
$r$ centered at $0$, the Minty gap and the primal--dual gap of
$\Phi_n(u,v)=\gamma(v^\top D_nu+ru_1/\sqrt n)$ coincide: for every
$(\widehat u,\widehat v)$ both equal
\begin{equation}\label{vi:affine-gap}
 \gamma\Bigl\{\frac r{\sqrt n}\widehat u_1+r\|D_n\widehat u\|
 +r\Bigl\|D_n^\top\widehat v+\frac r{\sqrt n}e_1\Bigr\|\Bigr\}
 \ \ge\ \frac{\gamma r}{\sqrt n}\Bigl|\widehat v_n+\frac r{\sqrt n}\Bigr|.
\end{equation}
\end{lemma}
\begin{proof}
Write $F_n(z)=Az+h$ with $A(u,v)=\gamma(D_n^\top v,-D_nu)$ and
$h=\gamma(re_1/\sqrt n,0)$. The matrix $A$ is skew-symmetric, so $F_n$ is
monotone and $\|A\|=\gamma\|D_n\|\le2\gamma$. The root equation is
$D_n^\top v=-re_1/\sqrt n$ and $D_nu=0$, whose solution is $u=0$,
$v=-r\mathbf1/\sqrt n$.

\emph{Minty gap.} Since $\langle Aw,w\rangle=0$ and
$\langle Aw,\widehat z\rangle=-\langle w,A\widehat z\rangle$,
\[
 \langle F_n(w),\widehat z-w\rangle
 =\langle h,\widehat z\rangle+\langle w,-A\widehat z-h\rangle .
\]
Here $-A\widehat z-h=\gamma(-D_n^\top\widehat v-re_1/\sqrt n,\ D_n\widehat u)$,
and the supremum of $\langle w,\cdot\rangle$ over the product of balls is
$r$ times the sum of the norms of the two blocks. This gives the
expression in~\eqref{vi:affine-gap}.

\emph{Primal--dual gap.} Over $\|u\|,\|v\|\le r$,
\[
 \sup_v\Phi_n(\widehat u,v)-\inf_u\Phi_n(u,\widehat v)
 =\gamma\Bigl(r\|D_n\widehat u\|+\frac r{\sqrt n}\widehat u_1\Bigr)
 +\gamma r\Bigl\|D_n^\top\widehat v+\frac r{\sqrt n}e_1\Bigr\|,
\]
which is the same expression.

\emph{Lower bound.} By Cauchy--Schwarz,
$\|D_n\widehat u\|\ge|\mathbf1^\top D_n\widehat u|/\sqrt n
=|\widehat u_1|/\sqrt n$, so the first two terms have a nonnegative sum.
The coordinates of $D_n^\top\widehat v+re_1/\sqrt n$ sum to
$(D_n\mathbf1)^\top\widehat v+r/\sqrt n=\widehat v_n+r/\sqrt n$, so its
norm is at least $|\widehat v_n+r/\sqrt n|/\sqrt n$.
\end{proof}

\subsection{A flat convex--concave edge and a robust chain}
\label{vi:robust-core}

For $\epsilon>0$ define
\[
 \chi_{\rm edge}(t)=1-\psi(|t|-1),\qquad
 b_\epsilon(t)=t\,\chi_{\rm edge}(t/\epsilon),
\]
so that $\chi_{\rm edge}=1$ on $[-1,1]$,
$\chi_{\rm edge}=0$ outside $[-2,2]$, and $b_\epsilon(t)=t$
for $|t|\le\epsilon$. Let $\zeta_\epsilon$ be the even function with
\[
 \zeta_\epsilon(0)=\zeta_\epsilon'(0)=0,\qquad
 \zeta_\epsilon''(t)=2\epsilon|b_\epsilon''(t)|,
\]
and, for $\varsigma\in\{-1,1\}$,
\begin{equation}\label{vi:flat-edge}
 E_\epsilon^\varsigma(u,v)=\varsigma\bigl[uv-b_\epsilon(u)b_\epsilon(v)\bigr]
                 +\zeta_\epsilon(u)-\zeta_\epsilon(v).
\end{equation}
The edge $E_\epsilon^\varsigma$ agrees with the bilinear function
$\varsigma uv$ up to a small gradient error, and it vanishes, together with
its gradient, near the origin.

\begin{lemma}[Flat edge]\label{vi:edge-lemma}
The function $E_\epsilon^\varsigma$ is $C^2$ and convex--concave. Its value
and gradient vanish on $\{|u|,|v|\le\epsilon\}$. Its second derivatives
satisfy $|\partial_{uu}E_\epsilon^\varsigma|,|\partial_{vv}E_\epsilon^\varsigma|\le512$
and $|\partial_{uv}E_\epsilon^\varsigma|\le82$, and every coordinate of
$\nabla E_\epsilon^\varsigma-\nabla(\varsigma uv)$ has absolute value at most
$274\epsilon$.
\end{lemma}
\begin{proof}
Write $b,\zeta_{\rm edge},E$ for
$b_\epsilon,\zeta_\epsilon,E_\epsilon^\varsigma$.
By~\eqref{vi:psi-bounds}, $|\chi_{\rm edge}'|\le4$ and
$|\chi_{\rm edge}''|\le60$. From
\[
 b'(t)=\chi_{\rm edge}(t/\epsilon)
 +\frac t\epsilon\chi_{\rm edge}'(t/\epsilon),\qquad
 b''(t)=\frac1\epsilon\Bigl[2\chi_{\rm edge}'(t/\epsilon)
 +\frac t\epsilon\chi_{\rm edge}''(t/\epsilon)\Bigr]
\]
and $|t|\le2\epsilon$ on the support of $\chi_{\rm edge}(t/\epsilon)$ we get
\[
 |b|\le2\epsilon,\qquad |b'|\le1+2\cdot4=9,\qquad
 |b''|\le\frac{2\cdot4+2\cdot60}\epsilon=\frac{128}\epsilon .
\]
Since $b(t)=t$ for $|t|\le\epsilon$, the function $\zeta_{\rm edge}''$ is continuous,
even, nonnegative, bounded by $256$, and supported in
$\{\epsilon\le|t|\le2\epsilon\}$. Integrating from $0$ gives
$|\zeta_{\rm edge}'|\le256\epsilon$, and
$\zeta_{\rm edge}=\zeta_{\rm edge}'=0$ on $[-\epsilon,\epsilon]$. Direct
differentiation gives
\[
 \partial_{uu}E=\zeta_{\rm edge}''(u)-\varsigma b''(u)b(v),\qquad
 \partial_{vv}E=-\zeta_{\rm edge}''(v)-\varsigma b(u)b''(v),\qquad
 \partial_{uv}E=\varsigma\bigl[1-b'(u)b'(v)\bigr].
\]
Since $|b''(u)b(v)|\le2\epsilon|b''(u)|=\zeta_{\rm edge}''(u)$, we have
$\partial_{uu}E\ge0$, and similarly $\partial_{vv}E\le0$; hence $E$ is
convex--concave. The bounds give $|\partial_{uu}E|\le256+256=512$,
likewise for $\partial_{vv}E$, and $|\partial_{uv}E|\le1+81=82$. The
gradient error is
$\partial_uE-\varsigma v=-\varsigma b'(u)b(v)+\zeta_{\rm edge}'(u)$, of absolute value at
most $9\cdot2\epsilon+256\epsilon=274\epsilon$; the error
$\partial_vE-\varsigma u=-\varsigma b(u)b'(v)-\zeta_{\rm edge}'(v)$ obeys
the same bound. For $|u|,|v|\le\epsilon$, $b(u)=u$, $b(v)=v$ and
$\zeta_{\rm edge}=\zeta_{\rm edge}'=0$, so $E$ and $\nabla E$ vanish.
\end{proof}

For an integer $n\ge1$ and scales $r,\lambda>0$, define the chain
dimension $d_{\rm ch}=2n$ and chain amplitude $a_{\rm ch}=r/\sqrt n$,
define on $\mathbb R^n\times\mathbb R^n$
\begin{equation}\label{vi:robust-chain}
 \Phi_\epsilon(u,v)=\lambda\Bigl[a_{\rm ch}u_1+
 \sum_{i=1}^nE_\epsilon^{+1}(u_i,v_i)+
 \sum_{i=1}^{n-1}E_\epsilon^{-1}(u_{i+1},v_i)\Bigr],
 \qquad G_\epsilon=(\nabla_u\Phi_\epsilon,-\nabla_v\Phi_\epsilon).
\end{equation}
Replacing every edge by its bilinear part gives the affine reference
$\Phi_0(u,v)=\lambda(v^\top D_nu+a_{\rm ch}u_1)$, whose saddle operator
$G_0=\lambda(D_n^\top v+a_{\rm ch}e_1,-D_nu)$ is~\eqref{vi:affine-chain} with
$\gamma=\lambda$ and root $z_{\rm ref}=(0,-a_{\rm ch}\mathbf1)$. We order the
coordinates as $(u_1,v_1,u_2,v_2,\ldots,u_n,v_n)$; then every edge
couples two consecutive coordinates.

\begin{lemma}[Robust chain]\label{vi:core-root}
The operator $G_\epsilon$ is monotone and $\ell$-Lipschitz with
$\ell=1200\lambda$, and
\begin{equation}\label{vi:uniform-core-error}
 \|G_\epsilon(z)-G_0(z)\|\le
 \Delta_{\rm op}:=548\lambda\epsilon\sqrt{d_{\rm ch}}
 \qquad(z\in\mathbb R^{d_{\rm ch}}).
\end{equation}
It has a root within distance $548n\epsilon\sqrt{d_{\rm ch}}$ of
$z_{\rm ref}$. If $0\le s<d_{\rm ch}$ and all coordinates of $z$ after position $s$ have absolute
value at most $\epsilon$, then $\Phi_\epsilon(z)$ and $G_\epsilon(z)$ depend
only on the first $s+1$ coordinates of $z$.
\end{lemma}
\begin{proof}
\emph{Monotonicity and Lipschitz constant.} The symmetric part of the
Jacobian of $G_\epsilon$ is
$\operatorname{diag}(\nabla^2_{uu}\Phi_\epsilon,-\nabla^2_{vv}\Phi_\epsilon)$,
which is positive semidefinite because every edge is convex--concave
(Lemma~\ref{vi:edge-lemma}). Every coordinate belongs to at most two
edges, so every row and every column of the Jacobian has absolute sum
at most $(2\cdot512+2\cdot82)\lambda=1188\lambda$. Hence the spectral norm
is at most $1188\lambda<\ell$.

\emph{Distance to the reference.} Each coordinate of $G_\epsilon-G_0$ is
$\lambda$ times the sum of at most two edge errors, each at most
$274\epsilon$ in absolute value. Hence every coordinate is at most
$548\lambda\epsilon$, which gives~\eqref{vi:uniform-core-error}.

\emph{A root.} The linear part of $G_0$ is
$A_{\rm ref}=\lambda\begin{psmallmatrix}0&D_n^\top\\-D_n&0\end{psmallmatrix}$,
and $D_n^{-1}$ is the upper triangular matrix of ones, so
$\|D_n^{-1}\|\le\|D_n^{-1}\|_F=\sqrt{n(n+1)/2}\le n$ and
$\|A_{\rm ref}^{-1}\|\le n/\lambda$. The continuous map
\[
 T(z)=z_{\rm ref}-A_{\rm ref}^{-1}\bigl(G_\epsilon(z)-G_0(z)\bigr)
\]
satisfies
$\|T(z)-z_{\rm ref}\|\le n\Delta_{\rm op}/\lambda
=548n\epsilon\sqrt{d_{\rm ch}}$, so by
Brouwer's theorem it has a fixed point $z$ in the ball of this radius
around $z_{\rm ref}$. Since $G_0(z)=A_{\rm ref}(z-z_{\rm ref})$, the fixed
point equation reads $G_0(z)=G_0(z)-G_\epsilon(z)$, that is,
$G_\epsilon(z)=0$.

\emph{Flatness.} Let all coordinates after position $s$ be at most
$\epsilon$. Every edge whose two coordinates are both after $s$
vanishes together with its gradient by Lemma~\ref{vi:edge-lemma}. The
remaining edges involve only the first $s+1$ coordinates.
\end{proof}

\subsection{Localization}
\label{vi:localization-subsection}

\begin{lemma}[Localization of an operator]\label{vi:operator-cutoff}
Let $G$ be continuously differentiable, monotone and $\ell$-Lipschitz,
with a root of norm at most $R_{\rm loc}>0$. Let
\[
 \theta(x)=\psi\Bigl(\frac{\|x\|-2R_{\rm loc}}{408R_{\rm loc}}\Bigr),\qquad
 V_{\rm loc}(x)=84\ell(\|x\|-R_{\rm loc})_+\frac x{\|x\|},\qquad
 \overline G=(1-\theta)G+V_{\rm loc},
\]
with $V_{\rm loc}(0)=0$. Then $\overline G$ is monotone and
$90\ell$-Lipschitz, it equals $G$ on $\{\|x\|\le R_{\rm loc}\}$ and the
known map $V_{\rm loc}$ on
$\{\|x\|\ge410R_{\rm loc}\}$, and it has the same root.
\end{lemma}
\begin{proof}
Write $t=\|x\|$. The map $V_{\rm loc}$ is the gradient of the convex function
$42\ell(t-R_{\rm loc})_+^2$, so it is monotone. For $t>R_{\rm loc}$ its Jacobian has the
eigenvalue $84\ell$ in the radial direction and $84\ell(1-R_{\rm loc}/t)$ in the
tangential directions.

\emph{The region $t\le2R_{\rm loc}$.} Here $\theta=0$ and
$\overline G=G+V_{\rm loc}$, which
is monotone and $(\ell+84\ell)$-Lipschitz.

\emph{The annulus $2R_{\rm loc}\le t\le410R_{\rm loc}$.}
By~\eqref{vi:psi-bounds},
$\|\nabla\theta\|\le4/(408R_{\rm loc})$, and since the root has norm at
most $R_{\rm loc}$,
$\|G(x)\|\le\ell(t+R_{\rm loc})\le411\ell R_{\rm loc}$. Hence
$\|G(x)\nabla\theta(x)^\top\|\le4\cdot411\ell/408<5\ell$. The Jacobian
\[
 \mathrm D\overline G=(1-\theta)\mathrm DG-G\nabla\theta^\top
 +\mathrm D V_{\rm loc}
\]
has symmetric part at least $(0-5\ell+42\ell)I\succeq0$, because
$\mathrm DG$ has positive semidefinite symmetric part and the tangential
eigenvalues of $\mathrm D V_{\rm loc}$ are at least
$84\ell(1-R_{\rm loc}/(2R_{\rm loc}))=42\ell$. Its
norm is at most $\ell+5\ell+84\ell=90\ell$.

\emph{The region $t\ge410R_{\rm loc}$.} Here $\theta=1$ and
$\overline G=V_{\rm loc}$.

The map $\overline G$ is continuous and piecewise continuously
differentiable, so integrating its Jacobian along segments gives global
monotonicity and the Lipschitz constant $90\ell$. For $t\le R_{\rm loc}$,
$\theta=0$ and $V_{\rm loc}=0$, so $\overline G=G$; in particular the root is
preserved.
\end{proof}

For saddle problems the localization must act on the function. Fix
$R_{\rm loc},\ell>0$ and let
\[
 \chi_{R_{\rm loc}}(w)=1-\psi\Bigl(
 \frac{\|w\|-2R_{\rm loc}}{408R_{\rm loc}}\Bigr),\qquad
 q_\delta(t)=\begin{cases}0,&t\le0,\\t^3/(6\delta),&0<t<\delta,\\
 t^2/2-\delta t/2+\delta^2/6,&t\ge\delta,\end{cases}
\]
\begin{equation}\label{vi:smooth-hinge}
 P_{R_{\rm loc}}(w)=256\ell\,q_{R_{\rm loc}/100}(\|w\|-R_{\rm loc}).
\end{equation}
The mask $\chi_{R_{\rm loc}}$ is $C^2$, equals $1$ for
$\|w\|\le2R_{\rm loc}$ and $0$ for $\|w\|\ge410R_{\rm loc}$, and
\[
 \|\nabla\chi_{R_{\rm loc}}\|\le\frac4{408R_{\rm loc}},\qquad
 \|\nabla^2\chi_{R_{\rm loc}}\|\le\frac{60}{(408R_{\rm loc})^2}.
\]
For the Hessian, the radial eigenvalue is
$|\psi''|/(408R_{\rm loc})^2\le6/(408R_{\rm loc})^2$. With
$s=(\|w\|-2R_{\rm loc})/(408R_{\rm loc})\in(0,1)$, the tangential
eigenvalues are $\psi'(s)/(408R_{\rm loc}\|w\|)$, and
\[
 \frac{\psi'(s)}{408R_{\rm loc}\|w\|}
 =\frac1{(408R_{\rm loc})^2}\cdot\frac{408\psi'(s)}{2+408s}
 \le\frac1{(408R_{\rm loc})^2}\cdot\frac{\psi'(s)}s
 =\frac{30s(1-s)^2}{(408R_{\rm loc})^2}
 \le\frac{40}{9(408R_{\rm loc})^2}.
\]

\begin{lemma}[Localization of a saddle function]\label{vi:product-cutoff}
Let $d_{\rm amb}\ge1$, and let $\Psi$ be $C^2$ and convex--concave on
$\mathbb R^{d_{\rm amb}}\times\mathbb R^{d_{\rm amb}}$,
with $\ell$-Lipschitz gradient, $\Psi(0)=0$ and
$\|\nabla\Psi(0)\|\le\ell R_{\rm loc}$. The function
\begin{equation}\label{vi:localized-game}
 \overline\Phi(u,v)=\chi_{R_{\rm loc}}(u)\chi_{R_{\rm loc}}(v)\Psi(u,v)
 +P_{R_{\rm loc}}(u)-P_{R_{\rm loc}}(v)
\end{equation}
is $C^2$ and convex--concave, and its gradient is $356\ell$-Lipschitz. If
$\Psi$ has a saddle point $z_\star$ with $\|z_\star\|<R_{\rm loc}$, then $z_\star$
is a saddle point of $\overline\Phi$ and
$|\overline\Phi(z_\star)|\le3\ell R_{\rm loc}^2/2$. Moreover, with
$r=R_{\rm loc}/2$,
\begin{equation}\label{vi:one-sided-values}
 \overline\Phi(u,v)\ge\Psi(u,v)\quad(\|v\|\le r),\qquad
 \overline\Phi(u,v)\le\Psi(u,v)\quad(\|u\|\le r).
\end{equation}
\end{lemma}
\begin{proof}
Let $\chi(u,v)=\chi_{R_{\rm loc}}(u)\chi_{R_{\rm loc}}(v)$; it vanishes
unless $\|u\|,\|v\|\le410R_{\rm loc}$, where
$\|(u,v)\|\le580R_{\rm loc}$.

\emph{Size of $\Psi$ on the support.} Integrating the Lipschitz
gradient from $0$,
\[
 \|\nabla\Psi(z)\|\le\ell R_{\rm loc}+\ell\|z\|
 \le581\ell R_{\rm loc},\qquad
 |\Psi(z)|\le\ell R_{\rm loc}\|z\|+\tfrac\ell2\|z\|^2
 \le169000\ell R_{\rm loc}^2 .
\]
The mask bounds give $\|\nabla\chi\|\le6/(408R_{\rm loc})$ and
$\|\nabla^2\chi\|\le76/(408R_{\rm loc})^2$, where the diagonal blocks contribute
at most $60$ and the mixed blocks at most $16$ in units of
$(408R_{\rm loc})^{-2}$. By the product rule,
\begin{equation}\label{vi:cutoff-hessian}
 \|\nabla^2(\chi\Psi)\|
 \le\ell\Bigl(1+\frac{2\cdot6\cdot581}{408}+\frac{76\cdot169000}{408^2}\Bigr)<100\ell .
\end{equation}

\emph{Convexity in $u$.} If $\|u\|\le2R_{\rm loc}$, then
$\chi_{R_{\rm loc}}(u)=1$ and
$\nabla^2_{uu}(\chi\Psi)=\chi_{R_{\rm loc}}(v)\nabla^2_{uu}\Psi\succeq0$,
while $\nabla^2P_{R_{\rm loc}}(u)\succeq0$ because
$P_{R_{\rm loc}}$ is a nondecreasing convex function of the convex function
$\|u\|-R_{\rm loc}$. If $\|u\|\ge2R_{\rm loc}$, then
$\|u\|-R_{\rm loc}\ge R_{\rm loc}>R_{\rm loc}/100$, and the smallest
eigenvalue of $\nabla^2P_{R_{\rm loc}}(u)$ is its tangential eigenvalue
\[
 256\ell\,\frac{q_{R_{\rm loc}/100}'(\|u\|-R_{\rm loc})}{\|u\|}
 =256\ell\Bigl(1-\frac{R_{\rm loc}+R_{\rm loc}/200}{\|u\|}\Bigr)
 \ge256\ell\Bigl(1-\frac{201}{400}\Bigr)=127.36\,\ell,
\]
which exceeds the bound $100\ell$ of~\eqref{vi:cutoff-hessian}.
Concavity in $v$ follows in the same way from $-P_{R_{\rm loc}}(v)$. The
Hessian of $P_{R_{\rm loc}}$ has norm at most $256\ell$, so the gradient
of $\overline\Phi$ is $(100+256)\ell$-Lipschitz. All pieces are $C^2$:
$P_{R_{\rm loc}}$ vanishes near the center of each ball, and the masks are
constant near $\|w\|=2R_{\rm loc}$ and $\|w\|=410R_{\rm loc}$.

\emph{The saddle point.} If $\|z_\star\|<R_{\rm loc}$, the masks equal one
and $P_{R_{\rm loc}}$, $\nabla P_{R_{\rm loc}}$ vanish in a neighbourhood
of $z_\star$, so
$\nabla\overline\Phi(z_\star)=\nabla\Psi(z_\star)=0$, and
$|\overline\Phi(z_\star)|=|\Psi(z_\star)|
\le\ell R_{\rm loc}^2+\ell R_{\rm loc}^2/2$.

\emph{One-sided comparison.} Let $\|v\|\le r$, so
$\chi_{R_{\rm loc}}(v)=1$ and $P_{R_{\rm loc}}(v)=0$, and
$\overline\Phi-\Psi=P_{R_{\rm loc}}(u)
-(1-\chi_{R_{\rm loc}}(u))\Psi(u,v)$. If
$\|u\|\le2R_{\rm loc}$, this is $P_{R_{\rm loc}}(u)\ge0$. If
$t=\|u\|\ge2R_{\rm loc}$, then
\[
 |\Psi(u,v)|\le\ell\Bigl(\frac{t^2}2
 +\frac{3R_{\rm loc}t}2+\frac{5R_{\rm loc}^2}8\Bigr)
 \le6\ell(t-R_{\rm loc})^2,
 \qquad
 P_{R_{\rm loc}}(u)\ge126.72\,\ell(t-R_{\rm loc})^2,
\]
so $P_{R_{\rm loc}}(u)\ge(1-\chi_{R_{\rm loc}}(u))\Psi(u,v)$. This proves the first inequality;
the second follows by exchanging the roles of $u$ and $v$.
\end{proof}

\subsection{Hiding the chain}
\label{vi:haar-subsection}

Fix $p\in(0,1]$. For every round $r$, worker $i$ and call $j$ let
$B_{{\rm mask},r,i,j}\sim\operatorname{Bernoulli}(p)$ be independent, and put
\[
 S_N=\sum_{r=1}^N\max_{1\le i\le m}\sum_{j=1}^KB_{{\rm mask},r,i,j},
 \qquad
 k=\min\{KN,\lfloor4N\beta_p\rfloor\},
\]
with $\beta_p$ from Appendix~\ref{app:clock}. Then $\E S_N=N\beta_p$ and
$S_N\le KN$. If $k<KN$, then $S_N>k$ implies $S_N\ge k+1>4N\beta_p$, so
Markov's inequality gives
\begin{equation}\label{vi:clock-event}
 \Prob\{S_N\le k\}\ge\frac34 .
\end{equation}

Let $d_{\rm amb}\gg d_{\rm ch}$ be an ambient dimension. The hidden
instance lives in $\mathbb R^{d_{\rm amb}}$. Let
$U=(U_1,\ldots,U_{d_{\rm ch}})$ be a uniformly distributed orthonormal
$d_{\rm ch}$-frame in $\mathbb R^{d_{\rm amb}}$, and let a chain operator
on $\mathbb R^{d_{\rm ch}}$ with the flatness
property of Lemma~\ref{vi:core-root} be lifted as $x\mapsto UG(U^\top x)$.
We call $U_s^\top x$ the $s$th \emph{hidden coordinate} of $x$. The oracle
returns $B_{\rm mask}UG(U^\top x)/p$ together with the bit $B_{\rm mask}$.

\begin{lemma}[Hiding]\label{vi:haar}
Let $R_{\rm cut},\epsilon>0$, and consider an algorithm whose queries and
output have norm at most $R_{\rm cut}$ on every transcript. If
\[
 d_{\rm amb}\ge d_{\rm ch}
 +\frac{2R_{\rm cut}^2}{\epsilon^2}\log(32mKNd_{\rm ch}),
\]
then, jointly over the frame, the bits and the algorithm's randomness,
the following event has probability at least $5/8$: $S_N\le k$, and all
queries and the output have hidden coordinates after position $k$ of
absolute value at most $\epsilon$. For two independent $n$-frames, used
for the two blocks of a saddle problem with $d_{\rm ch}=2n$ and the interleaved
order, the same holds if every query and output block has norm at most
$R_{\rm cut}$ and
$d_{\rm amb}\ge n+2R_{\rm cut}^2\epsilon^{-2}\log(32mKNd_{\rm ch})$.
\end{lemma}
\begin{proof}
Fix the bits and the algorithm's randomness; only the frame is random.

\emph{A fictitious oracle.} For a call of worker $i$ in round $r$ at slot
$j$, let
\[
 s=\min\Bigl\{d_{\rm ch},\
 \sum_{r'<r}\max_{i'}\sum_{j'=1}^KB_{{\rm mask},r',i',j'}
 +\sum_{j'<j}B_{{\rm mask},r,i,j'}\Bigr\}.
\]
Let $c'=\min\{d_{\rm ch},s+1\}$, let $w^{(c')}$ keep the first $c'$ hidden
coordinates of the query and set the others to $0$, and let $\pi_{c'}$ keep
the first $c'$ components of a vector of $\mathbb R^{d_{\rm ch}}$. The
fictitious oracle returns $0$ if $B_{{\rm mask},r,i,j}=0$, and otherwise
\[
 \frac1p\,U_{\le c'}\,\pi_{c'}G\bigl(w^{(c')}\bigr),
 \qquad U_{\le c'}=(U_1,\ldots,U_{c'}).
\]
The projection $\pi_{c'}$ is needed because an edge with one large
coordinate has a nonzero output in its other coordinate, so
$G(w^{(c')})$ itself may have a nonzero component $c'+1$. With it, an
informative response depends on the frame only through
$U_1,\ldots,U_{c'}$. By induction over calls, every query of the
fictitious transcript depends on the frame only through
$U_1,\ldots,U_s$: an informative response adds at most one column, and at
a communication the workers share their prefixes, which are nested, so
the longest one is kept. Hence the fictitious output depends only on the
first $\min\{d_{\rm ch},S_N\}$ columns.

\emph{One hidden coordinate.} Let $x$ be a fictitious query or the
output, with $\|x\|\le R_{\rm cut}$, and let $U_i$ be a column not yet used,
$i>s$. Conditionally on $U_1,\ldots,U_s$, the column $U_i$ is uniform on
the unit sphere of the orthogonal complement of these columns, of
dimension $d_{\rm amb}-s\ge d_{\rm amb}-d_{\rm ch}$. For a uniform unit
vector $\vartheta$ in dimension $n'$ and a fixed $y$ with
$\|y\|\le R_{\rm cut}$,
$\E\langle\vartheta,y\rangle^{2j}\le(2j-1)!!\,R_{\rm cut}^{2j}/n'^j$ and the odd
moments vanish, so
\[
 \E e^{t\langle\vartheta,y\rangle}
 \le\sum_{j\ge0}\frac{t^{2j}(2j-1)!!\,R_{\rm cut}^{2j}}{(2j)!\,n'^j}
 =\sum_{j\ge0}\frac1{j!}\Bigl(\frac{t^2R_{\rm cut}^2}{2n'}\Bigr)^j
 =e^{t^2R_{\rm cut}^2/(2n')} .
\]
The Chernoff bound and the assumption on $d_{\rm amb}$ give
\[
 \Prob\bigl\{|\langle U_i,x\rangle|>\epsilon\bigm|U_1,\ldots,U_s\bigr\}
 \le2e^{-(d_{\rm amb}-d_{\rm ch})\epsilon^2/(2R_{\rm cut}^2)}
 \le\frac1{16mKNd_{\rm ch}}.
\]
For two frames, condition on both prefixes and apply the same bound in
the block of each query, whose remaining dimension is at least
$d_{\rm amb}-n$.

\emph{Union bound.} For every fictitious query, with its prefix length
$s$ at the time of the call, consider all hidden coordinates $i>s$; for the
fictitious output consider all $i>\min\{d_{\rm ch},S_N\}$. Since the bits are fixed,
these index sets are deterministic, and the vector is determined by the
columns of its prefix, so the previous bound applies to each pair. There
are at most $(mKN+1)d_{\rm ch}\le2mKNd_{\rm ch}$ pairs, hence define
$\Omega_{\rm good}$ as the event that all these coordinates are at most
$\epsilon$ in absolute value; it has probability
at least $7/8$. Together with~\eqref{vi:clock-event}, the event
$\Omega_{\rm good}\cap\{S_N\le k\}$ has probability at least
$7/8-1/4=5/8$. On it,
the final prefix has length at most $k$, so the output has all hidden
coordinates after position $k$ bounded by $\epsilon$.

\emph{The true transcript.} On $\Omega_{\rm good}$, every query has hidden
coordinates after its prefix $s$ bounded by $\epsilon$. Then all edges with
both coordinates after $s$ vanish, including the edge between $s+1$ and
$s+2$, so $G(U^\top x)$ is supported on the first $c'$ coordinates and
coincides with $\pi_{c'}G(w^{(c')})$. Hence the true response equals the
fictitious one at every informative call, and both are $0$ otherwise. By
induction the true and fictitious transcripts coincide. Unused calls
can be padded by ignored queries, so the bound covers every algorithm
with at most $K$ calls per worker and round.
\end{proof}

We use the lemma as follows. Suppose that a nonnegative loss is at least
$\Delta_{\rm loss}$ on the event of the lemma. Its probability, averaged over the
frame, is at least $5/8$, so some fixed frame (or pair of frames) gives
this event probability at least $1/2$ over the bits and the algorithm's
randomness. For this fixed instance the expected loss is at least
$\Delta_{\rm loss}/2$.
Several losses that are bounded below on the same event are bounded
below simultaneously for the same instance.

\subsection{The Minty gap}
\label{vi:monotone-lower-subsection}

\begin{theorem}[Minty gap under SGC]\label{vi:uvi-lower}
Let $m,K,N\ge1$, $\rho\ge1$, $L,D>0$, $p=1/\rho$, $k$ as in
Section~\ref{vi:haar-subsection}, $n=k+1$, and
\[
 \begin{gathered}
 d_{\rm ch}=2n,\quad r=\frac D4,\quad
 R_{\rm loc}=2r,\quad R_{\rm cut}=410R_{\rm loc},\\
 \lambda=\frac L{120000},\quad\ell=1200\lambda=\frac L{100},\quad
 \epsilon=\frac r{2000000\,n\sqrt{d_{\rm ch}}}.
 \end{gathered}
\]
If the ambient dimension satisfies
$d_{\rm amb}\ge d_{\rm ch}
+2R_{\rm cut}^2\epsilon^{-2}\log(32mKNd_{\rm ch})$, then every algorithm of
Definition~\ref{def:ic-main} on $\mathbb R^{d_{\rm amb}}$ has an instance with a
monotone $L$-Lipschitz operator, a root in $B(0,D)$ and an
oracle satisfying SGC with equality, such that
\begin{equation*}
 \E\Gap^{\rm M}_{B(0,D)}(\widehat z)
 \ge\frac{LD^2}{122880000\,(1+N\InfoDepth)}.
\end{equation*}
Every realization of the oracle is a monotone, $\rho L$-Lipschitz
operator with the same root, and the bound holds even if the bit
$B_{\rm mask}$ is revealed with every response.
\end{theorem}
\begin{proof}
\emph{The instance.} Take the chain~\eqref{vi:robust-chain} with these
parameters. By Lemma~\ref{vi:core-root} it has a root $z_\epsilon$ with
$\|z_\epsilon-z_{\rm ref}\|\le548n\epsilon\sqrt{d_{\rm ch}}
=548r/2000000<r/4$, and
$\|z_{\rm ref}\|=a_{\rm ch}\sqrt n=r$, so
$\|z_\epsilon\|\le5r/4<R_{\rm loc}$. For a frame
$U$, let $G_U(x)=UG_\epsilon(U^\top x)$, which is monotone and
$\ell$-Lipschitz, and let $F_U=(1-\theta)G_U+V_{\rm loc}$ be its localization from
Lemma~\ref{vi:operator-cutoff}. Then $F_U$ is monotone and
$90\ell\le L$-Lipschitz, and $Uz_\epsilon$ is a root with
$\|Uz_\epsilon\|<R_{\rm loc}\le D$.

\emph{Interior outputs.} Let $\|x\|<R_{\rm cut}$ and let the hidden coordinates
$U^\top x=(\widehat u,\widehat v)$ after position $k=n-1$ be at most
$\epsilon$; in particular $|\widehat v_n|\le\epsilon$. Test on
$Z_r=\{U(u,v):\|u\|,\|v\|\le r\}$, which lies in
$B(0,\sqrt2r)\subseteq B(0,R_{\rm loc})$, where
$F_U=G_U$. For
$w=U(u,v)\in Z_r$,
$\langle F_U(w),x-w\rangle=\langle G_\epsilon(u,v),U^\top x-(u,v)\rangle$.
By Lemma~\ref{vi:affine-comparison} with $\gamma=\lambda$ and
$\epsilon\le r/(4\sqrt n)$, the reference operator $G_0$ gives
\[
 \sup_{(u,v)}\langle G_0(u,v),U^\top x-(u,v)\rangle
 \ge\frac{\lambda r}{\sqrt n}\Bigl(\frac r{\sqrt n}-\epsilon\Bigr)
 \ge\frac{3\lambda r^2}{4n}.
\]
By~\eqref{vi:uniform-core-error}, replacing $G_0$ by $G_\epsilon$ changes
each tested inner product by at most
\[
 \Delta_{\rm op}\bigl(\|U^\top x\|+\sqrt2r\bigr)
 \le548\lambda\epsilon\sqrt{d_{\rm ch}}\,(820r+\sqrt2r)
 \le\frac{548\cdot822}{2000000}\cdot\frac{\lambda r^2}n<\frac{\lambda r^2}{4n}.
\]
Hence $\Gap^{\rm M}_{B(0,D)}(x)\ge\lambda r^2/(2n)$.

\emph{Exterior outputs.} Let $\|x\|\ge R_{\rm cut}$,
$u_{\rm rad}=x/\|x\|$, and test at
$w=2R_{\rm loc}u_{\rm rad}$, with $\|w\|=D$. There $\theta(w)=0$,
$V_{\rm loc}(w)=84\ell R_{\rm loc}u_{\rm rad}$, and
$\|G_U(w)\|\le\ell\|U^\top w-z_\epsilon\|
\le\ell(2R_{\rm loc}+R_{\rm loc})=3\ell R_{\rm loc}$. Since
$x-w=(\|x\|-2R_{\rm loc})u_{\rm rad}$,
\[
 \langle F_U(w),x-w\rangle
 \ge(84-3)\ell R_{\rm loc}(\|x\|-2R_{\rm loc})
 \ge81\cdot408\,\ell R_{\rm loc}^2\ge\frac{\lambda r^2}{2n}.
\]

\emph{Oracle and simulation.} At every call the oracle draws an independent bit
$B_{\rm mask}\sim\operatorname{Bernoulli}(p)$ and returns
$B_{\rm mask}F_U/p$ and $B_{\rm mask}$. It is unbiased,
$\E\|B_{\rm mask}F_U/p\|^2=\rho\|F_U\|^2$,
and every realization is a nonnegative multiple of $F_U$. We simulate
the algorithm by one that sees only the hidden chain: its oracle returns
$B_{\rm mask}G_U/p$ and $B_{\rm mask}$. At a query with
$\|x\|<R_{\rm cut}$ the
simulator makes the same query and returns
$(1-\theta(x))\cdot B_{\rm mask}G_U(x)/p
+B_{\rm mask}V_{\rm loc}(x)/p$; at a query
with $\|x\|\ge R_{\rm cut}$ it queries $0$, uses only the bit and returns
$B_{\rm mask}V_{\rm loc}(x)/p$. Its output is the original output if this has norm
below $R_{\rm cut}$ and $0$ otherwise. All its queries and its output have
norm at most $R_{\rm cut}$, so Lemma~\ref{vi:haar} applies to it; on the event of the
lemma, the output of the original algorithm is either interior with
small hidden coordinates or exterior.

\emph{Conclusion.} The gap is nonnegative, because the test set contains
the root. By the remark after Lemma~\ref{vi:haar}, some fixed frame gives
\[
 \E\Gap^{\rm M}_{B(0,D)}(\widehat z)\ge\frac{\lambda r^2}{4n}
 =\frac{LD^2}{7680000\,n}\ge\frac{LD^2}{30720000\,(1+N\beta_p)},
\]
since $n=k+1\le4N\beta_p+1\le4(1+N\beta_p)$. Finally
$1+N\beta_p\le4(1+N\InfoDepth)$ by Theorem~\ref{thm:clock:asymptotic}.
\end{proof}

\subsection{Saddle problems}
\label{vi:product-cutoff-subsection}

\begin{theorem}[Saddle and Minty gaps under SGC]\label{vi:sp-lower}
Let $k,n,d_{\rm ch}$ be as in Theorem~\ref{vi:uvi-lower}, and
\[
 \begin{aligned}
 R_{\rm loc}&=\frac D{2\sqrt2},&
 r&=\frac{R_{\rm loc}}2,&
 \ell&=\frac L{400},&
 \lambda&=\frac L{480000},\\
 \epsilon&=\frac r{2000000\,n\sqrt{d_{\rm ch}}},&
 d_{\rm amb}&\ge n+\frac{2(410R_{\rm loc})^2}{\epsilon^2}
 \log(32mKNd_{\rm ch}).
 \end{aligned}
\]
Every algorithm of Definition~\ref{def:ic-main} on
$\mathbb R^{d_{\rm amb}}\times\mathbb R^{d_{\rm amb}}$ has a $C^2$
convex--concave instance with
$L$-Lipschitz saddle operator and an oracle satisfying SGC with
equality, such that, for
$\mathcal Z_{\rm test}=B(0,2R_{\rm loc})
\times B(0,2R_{\rm loc})\subseteq B(0,D)$,
\begin{equation*}
 \E\Gap_{\mathcal Z_{\rm test}}^{\rm M}(\widehat z)\ \text{ and }\
 \E\Gap_{\mathcal Z_{\rm test}}^{\rm PD}(\widehat z)
 \ \ge\ \frac{LD^2}{983040000\,(1+N\InfoDepth)}.
\end{equation*}
The same instance works for both gaps. Every sample function is
convex--concave with a $\rho L$-Lipschitz saddle operator and the same
saddle point, and the bound holds even if sample values and the bit
$B_{\rm mask}$ are revealed.
\end{theorem}
\begin{proof}
\emph{The instance.} For independent $n$-frames $U,V$ in
$\mathbb R^{d_{\rm amb}}$, let
$\Psi(u,v)=\Phi_\epsilon(U^\top u,V^\top v)$ and
$\Psi_0(u,v)=\Phi_0(U^\top u,V^\top v)$, and let $\overline\Phi$
be~\eqref{vi:localized-game}. Since $\ell=1200\lambda$, Lemma~\ref{vi:core-root} shows that $\Psi$
has $\ell$-Lipschitz gradient; moreover $\Psi(0)=0$ and
$\|\nabla\Psi(0)\|=\lambda a_{\rm ch}\le\ell R_{\rm loc}$, because every edge has zero
gradient at the origin. Its saddle point has norm below
$R_{\rm loc}$ by Lemma~\ref{vi:core-root}. Lemma~\ref{vi:product-cutoff} shows that
$\overline\Phi$ is $C^2$ and convex--concave, that its saddle operator
$\overline F$ is $356\ell\le L$-Lipschitz, and that the saddle point
$z_\star=(u_\star,v_\star)$ is preserved and lies in $\mathcal Z_{\rm test}$.

\emph{Interior outputs.} Let both output blocks have norm below
$410R_{\rm loc}$
and all hidden coordinates
$(\widetilde u,\widetilde v)=(U^\top\widehat u,V^\top\widehat v)$ after
position $k$ be at most $\epsilon$; in particular
$|\widetilde v_n|\le\epsilon$. On the product of hidden balls of radius $r$
(a subset of $\mathcal Z_{\rm test}$ on which $\overline\Phi=\Psi$),
Lemma~\ref{vi:affine-comparison} bounds both gaps of $\Psi_0$ below by
$\lambda r^2/n-\lambda r\epsilon/\sqrt n$. For the Minty gap, replacing
the operator of $\Psi_0$ by that of $\Psi$ changes each tested inner
product by at most
\[
 \Delta_{\rm op}\sqrt2\,(410R_{\rm loc}+r)<1200\,r\Delta_{\rm op}
 =\frac{657600}{2000000}\cdot\frac{\lambda r^2}n .
\]
For the primal--dual gap, apply~\eqref{vi:one-sided-values}: it suffices
to compare $\Psi$ with $\Psi_0$, and since $\Psi-\Psi_0$ vanishes at
the origin and has gradient norm at most $\Delta_{\rm op}$, the two value errors sum
to at most
\[
 2(410R_{\rm loc}+r)\Delta_{\rm op}
 =\frac{899816}{2000000}\cdot\frac{\lambda r^2}n .
\]
The term $\lambda r\epsilon/\sqrt n$ is at most
$\lambda r^2/(2000000\,n)$. In both cases the total loss is below
$\lambda r^2/(2n)$, so both gaps are at least $\lambda r^2/(2n)$.

\emph{Exterior outputs.} Suppose
$\|\widehat u\|\ge410R_{\rm loc}$. Testing at
$u_\star$ and $v_\star$ gives
\[
 \Gap^{\rm PD}_{\mathcal Z_{\rm test}}(\widehat z)
 \ge\overline\Phi(\widehat u,v_\star)-\overline\Phi(u_\star,\widehat v)
 \ge P_{R_{\rm loc}}(\widehat u)-|\overline\Phi(z_\star)|
 \ge256\ell\Bigl(\frac{(409R_{\rm loc})^2}2
 -\frac{409R_{\rm loc}^2}{200}\Bigr)-\frac{3\ell R_{\rm loc}^2}2,
\]
which exceeds $\lambda r^2/(2n)$ by a large margin; the case
$\|\widehat v\|\ge410R_{\rm loc}$ is symmetric. For the Minty gap, test at
$u_{\rm test}=2R_{\rm loc}\widehat u/\|\widehat u\|$ and
$v_{\rm test}=2R_{\rm loc}\widehat v/\|\widehat v\|$ (a zero
block is tested at $0$). There the masks equal one with zero gradients,
every nonzero radial block of $\overline F$ has norm
$256\ell(R_{\rm loc}-R_{\rm loc}/200)=254.72\,\ell R_{\rm loc}$, and the
localized core operator has norm at most
$\ell(\|(u_{\rm test},v_{\rm test})\|+R_{\rm loc})\le4\ell R_{\rm loc}$.
With $R_{\rm sum}=\|\widehat u\|+\|\widehat v\|
\ge410R_{\rm loc}$,
\[
 \begin{aligned}
 &\langle\overline F(u_{\rm test},v_{\rm test}),
 \widehat z-(u_{\rm test},v_{\rm test})\rangle\\
 &\quad\ge\bigl(254.72\,\ell R_{\rm loc}R_{\rm sum}
 -1018.88\,\ell R_{\rm loc}^2\bigr)
 -4\ell R_{\rm loc}(R_{\rm sum}+3R_{\rm loc})\\
 &\quad\ge101469\,\ell R_{\rm loc}^2>\frac{\lambda r^2}{2n}.
 \end{aligned}
\]

\emph{Oracle and simulation.} At every call the oracle draws an independent
$B_{\rm mask}\sim\operatorname{Bernoulli}(p)$ and returns the value and the
saddle operator of $B_{\rm mask}\overline\Phi/p$, together with
$B_{\rm mask}$. A simulator with access to $B_{\rm mask}\Psi/p$ reproduces a
response at an interior query from one response at the same query, by the
product rule. At an exterior query at least one mask vanishes with its
derivatives, so the response is a known function of the query and the
bit, and the simulator queries $0$ to obtain the bit. All queries of the
simulator have blocks of norm at most $410R_{\rm loc}$, and its output is the
original output if interior and $0$ otherwise. The two-frame version of
Lemma~\ref{vi:haar} applies, since $\Psi$ has the flatness property for
values and gradients.

\emph{Conclusion.} Both gaps are nonnegative, because $\mathcal Z_{\rm test}$
contains the saddle point. Some fixed pair of frames gives,
simultaneously for both gaps,
\[
 \min\left\{
  \E\Gap_{\mathcal Z_{\rm test}}^{\rm M}(\widehat z),
  \E\Gap_{\mathcal Z_{\rm test}}^{\rm PD}(\widehat z)
 \right\}
 \ge\frac{\lambda r^2}{4n}
 \ge\frac{LD^2}{245760000\,(1+N\beta_p)},
\]
using $r^2=D^2/32$ and $n\le4(1+N\beta_p)$, and
Theorem~\ref{thm:clock:asymptotic} gives $1+N\beta_p\le4(1+N\InfoDepth)$.
Each sample function $B_{\rm mask}\overline\Phi/p$ is convex--concave, has
the same saddle point, and has a $\rho L$-Lipschitz saddle operator.
\end{proof}

\subsection{Strong monotonicity}
\label{vi:strong-chain-subsection}

Fix $0<\mu<L$, $R_0>0$ and $\rho\ge1$, and let $p=1/\rho$ and
$k,n,d_{\rm ch}$ be as in Theorem~\ref{vi:uvi-lower}. Set
\begin{align}
 \lambda&=\frac{L-\mu}{480400},\qquad\ell=1201\lambda,\qquad
 q_{\rm geo}=\frac{2\lambda}{\sqrt{\mu^2+4\lambda^2}+\mu}
 \in(0,1),\nonumber\\
 c&=\frac{R_0}2\left(
 \frac{1-q_{\rm geo}^2}{1-q_{\rm geo}^{2d_{\rm ch}}}
 \right)^{1/2},\qquad
 \epsilon=\frac{R_0q_{\rm geo}^k}
 {10000\sqrt{d_{\rm ch}}\,(1+\lambda/\mu)},\nonumber\\
 R_{\rm loc}&=R_0+\frac{\lambda c}{q_{\rm geo}\ell}.
 \label{vi:strong-params}
\end{align}
The number $q_{\rm geo}$ is the root in $(0,1)$ of
$\lambda q_{\rm geo}^2+\mu q_{\rm geo}-\lambda=0$.
Let $K_{d_{\rm ch}}\in
\mathbb R^{d_{\rm ch}\times d_{\rm ch}}$ have ones just above the diagonal,
minus ones just below it, and zeros elsewhere, in the interleaved order.
The affine reference operator
\begin{equation}\label{vi:strong-affine}
 F_0(z)=\bigl(\mu I+\lambda K_{d_{\rm ch}}
 +\lambda q_{\rm geo}\,e_{d_{\rm ch}}e_{d_{\rm ch}}^\top\bigr)z
 -\frac{\lambda c}{q_{\rm geo}}\,e_1
\end{equation}
is the saddle operator of
$\Phi_{\mu,0}(u,v)=\frac\mu2(\|u\|^2-\|v\|^2)
+\lambda\bigl[-\frac c{q_{\rm geo}}u_1+\sum_{i=1}^nu_iv_i
-\sum_{i=1}^{n-1}u_{i+1}v_i-\frac{q_{\rm geo}}2v_n^2\bigr]$.

\begin{lemma}[Strongly monotone chain]\label{vi:strong-core}
The vector $w$ with $w_i=cq_{\rm geo}^{i-1}$,
$i=1,\ldots,d_{\rm ch}$, is the root
of~\eqref{vi:strong-affine}, and $\|w\|=R_0/2$. Let
$\phi_\epsilon(t)=q_\epsilon(|t|-\epsilon)$ with $q_\delta$
from~\eqref{vi:smooth-hinge}, and
\begin{equation*}
 \Psi_\epsilon(u,v)=\lambda\Bigl[-\frac c{q_{\rm geo}}u_1
 +\sum_{i=1}^nE_\epsilon^{+1}(u_i,v_i)
 +\sum_{i=1}^{n-1}E_\epsilon^{-1}(u_{i+1},v_i)
 -q_{\rm geo}\phi_\epsilon(v_n)\Bigr].
\end{equation*}
Its saddle operator $G_{\mu,\epsilon}$ is monotone and $\ell$-Lipschitz, has the
flatness property of Lemma~\ref{vi:core-root}, and satisfies
\[
 \|G_{\mu,\epsilon}-(F_0-\mu I)\|
 \le550\lambda\epsilon\sqrt{d_{\rm ch}},\qquad
 \Psi_\epsilon(0)=0,\qquad
 G_{\mu,\epsilon}(0)=-\frac{\lambda c}{q_{\rm geo}}e_1 .
\]
The operator $F_\epsilon=\mu I+G_{\mu,\epsilon}$ has a unique root
$z_{\mu,\epsilon}$,
and
\begin{equation}\label{vi:strong-root-shift}
 \|z_{\mu,\epsilon}-w\|\le\frac{550\lambda\epsilon\sqrt{d_{\rm ch}}}{\mu}
 \le\frac{550}{10000}R_0q_{\rm geo}^k<\frac{R_0}{16},\qquad
 \|z_{\mu,\epsilon}\|\le\frac{9R_0}{16}<R_{\rm loc}.
\end{equation}
\end{lemma}
\begin{proof}
\emph{The root of the reference.} Using
$\lambda q_{\rm geo}^2+\mu q_{\rm geo}-\lambda=0$,
the first row of $F_0(w)$ is
$\mu c+\lambda cq_{\rm geo}-\lambda c/q_{\rm geo}
=(c/q_{\rm geo})(\lambda q_{\rm geo}^2+\mu q_{\rm geo}-\lambda)=0$,
every interior row is
$-\lambda w_{i-1}+\mu w_i+\lambda w_{i+1}
=cq_{\rm geo}^{i-2}(-\lambda+\mu q_{\rm geo}
+\lambda q_{\rm geo}^2)=0$, and the last row is
$-\lambda w_{d_{\rm ch}-1}+(\mu+\lambda q_{\rm geo})w_{d_{\rm ch}}=0$
for the same reason. Moreover
$\|w\|^2=c^2\sum_{i<d_{\rm ch}}q_{\rm geo}^{2i}
=c^2(1-q_{\rm geo}^{2d_{\rm ch}})/(1-q_{\rm geo}^2)=R_0^2/4$.

\emph{The terminal term.} The function $\phi_\epsilon$ is even, convex and
$C^2$, it vanishes on $[-\epsilon,\epsilon]$, $0\le\phi_\epsilon''\le1$, and
$\phi_\epsilon'(t)=\operatorname{sign}(t)(|t|-\epsilon)^2/(2\epsilon)$ for
$\epsilon<|t|<2\epsilon$ and $\operatorname{sign}(t)(|t|-3\epsilon/2)$ for
$|t|\ge2\epsilon$. Hence $|\phi_\epsilon'(t)-t|\le3\epsilon/2$ for all $t$.
The term $-q_{\rm geo}\phi_\epsilon(v_n)$ is concave in $v$ and flat near $0$.

\emph{The operator $G_{\mu,\epsilon}$.} As in Lemma~\ref{vi:core-root}, the
Jacobian has row and column sums at most
$(2\cdot512+2\cdot82+1)\lambda=1189\lambda<\ell$, its symmetric part is
positive semidefinite, and the flatness property holds. The edges
contribute at most $548\lambda\epsilon\sqrt{d_{\rm ch}}$ to
$\|G_{\mu,\epsilon}-(F_0-\mu I)\|$ and the terminal at most
$3\lambda q_{\rm geo}\epsilon/2
\le2\lambda\epsilon\sqrt{d_{\rm ch}}$. The values at $0$ follow
from the definitions.

\emph{The root.} The operator $F_\epsilon$ is $\mu$-strongly monotone and
$(\mu+\ell)$-Lipschitz, so for $\eta=\mu/(\mu+\ell)^2$ the map
$T(z)=z-\eta F_\epsilon(z)$ satisfies
\[
 \|T(z)-T(y)\|^2\le\bigl(1-2\eta\mu+\eta^2(\mu+\ell)^2\bigr)\|z-y\|^2
 =\Bigl(1-\frac{\mu^2}{(\mu+\ell)^2}\Bigr)\|z-y\|^2 ,
\]
and Banach's fixed point theorem gives a unique root $z_{\mu,\epsilon}$. Since
$F_0(w)=0$, $\|F_\epsilon(w)\|
=\|G_{\mu,\epsilon}(w)-(F_0-\mu I)(w)\|\le
550\lambda\epsilon\sqrt{d_{\rm ch}}$, and strong monotonicity gives
$\mu\|z_{\mu,\epsilon}-w\|^2
\le\langle F_\epsilon(w),w-z_{\mu,\epsilon}\rangle\le
\|F_\epsilon(w)\|\,\|z_{\mu,\epsilon}-w\|$. Substituting $\epsilon$,
\[
 \frac{550\lambda\epsilon\sqrt{d_{\rm ch}}}\mu
 =\frac{550}{10000}\cdot\frac{\lambda}{\mu+\lambda}R_0q_{\rm geo}^k
 \le\frac{550}{10000}R_0q_{\rm geo}^k,
\]
and $\|z_{\mu,\epsilon}\|\le R_0/2+R_0/16$.
\end{proof}

\begin{theorem}[Distance lower bound]\label{vi:strong-lower}
For the parameters~\eqref{vi:strong-params} and
$d_{\rm amb}\ge n+2(410R_{\rm loc})^2\epsilon^{-2}
\log(32mKNd_{\rm ch})$, every algorithm of
Definition~\ref{def:ic-main} on
$\mathbb R^{d_{\rm amb}}\times\mathbb R^{d_{\rm amb}}$ has a $C^2$,
$\mu$-strongly convex--concave instance with $L$-Lipschitz saddle
operator, a root $z_\star$ with $\|z_\star\|\le R_0$, and an oracle
satisfying SGC with equality, such that
\begin{equation*}
 \E\|\widehat z-z_\star\|^2\ge\frac{R_0^2q_{\rm geo}^{2k}}{50}.
\end{equation*}
\end{theorem}
\begin{proof}
\emph{The instance.} For independent $n$-frames $U,V$ let
$\Psi(u,v)=\Psi_\epsilon(U^\top u,V^\top v)$ and
\begin{equation*}
 \overline\Phi(u,v)=\frac\mu2(\|u\|^2-\|v\|^2)
 +\chi_{R_{\rm loc}}(u)\chi_{R_{\rm loc}}(v)\Psi(u,v)
 +P_{R_{\rm loc}}(u)-P_{R_{\rm loc}}(v).
\end{equation*}
Since $\|\nabla\Psi(0)\|=\lambda c/q_{\rm geo}
\le\ell R_{\rm loc}$, the part of
Lemma~\ref{vi:product-cutoff} that does not use a saddle point applies:
the localized residual is $C^2$, convex--concave, and has
$356\ell$-Lipschitz gradient. Adding the quadratic makes $\overline\Phi$
$\mu$-strongly convex--concave, and its saddle operator $\overline F$ is
$(\mu+356\ell)$-Lipschitz, where
$\mu+356\ell=\mu+\frac{356}{400}(L-\mu)\le L$. The point
$z_\star=(Uu_{\mu,\epsilon},Vv_{\mu,\epsilon})$, where
$z_{\mu,\epsilon}=(u_{\mu,\epsilon},v_{\mu,\epsilon})$,
has norm at most $9R_0/16<R_{\rm loc}$, where the masks equal one and
$P_{R_{\rm loc}}$
vanishes; there $\overline F(z_\star)=0$, so $z_\star$ is the unique
saddle point.

\emph{Interior outputs.} Let both output blocks have norm below
$410R_{\rm loc}$
and the hidden coordinates after position $k$ be at most $\epsilon$.
Projecting onto the frames,
\[
 \|\widehat z-z_\star\|
 \ge\|w_{>k}\|-\|z_{\mu,\epsilon}-w\|-\epsilon\sqrt{d_{\rm ch}},
 \qquad w_{>k}=(w_{k+1},\ldots,w_{d_{\rm ch}}).
\]
Since $d_{\rm ch}-k=k+2\ge d_{\rm ch}/2$, we have
$q_{\rm geo}^{4(d_{\rm ch}-k)}\le q_{\rm geo}^{2d_{\rm ch}}$, hence
$1-q_{\rm geo}^{2d_{\rm ch}}
\le(1-q_{\rm geo}^{2(d_{\rm ch}-k)})(1+q_{\rm geo}^{2(d_{\rm ch}-k)})
\le2(1-q_{\rm geo}^{2(d_{\rm ch}-k)})$ and
\[
 \|w_{>k}\|^2=\frac{R_0^2}4\,q_{\rm geo}^{2k}\,
 \frac{1-q_{\rm geo}^{2(d_{\rm ch}-k)}}{1-q_{\rm geo}^{2d_{\rm ch}}}
 \ge\frac{R_0^2}8\,q_{\rm geo}^{2k}.
\]
By~\eqref{vi:strong-root-shift} and the choice of $\epsilon$,
$\|z_{\mu,\epsilon}-w\|+\epsilon\sqrt{d_{\rm ch}}
\le(551/10000)R_0q_{\rm geo}^k$. Since
$1/(2\sqrt2)>1/3$ and $1/3-551/10000>1/5$,
\begin{equation}\label{vi:strong-good-event}
 \|\widehat z-z_\star\|\ge\frac{R_0q_{\rm geo}^k}5 .
\end{equation}

\emph{Exterior outputs.} If a block has norm at least
$410R_{\rm loc}\ge410R_0$, then
$\|\widehat z-z_\star\|\ge410R_0-9R_0/16>R_0q_{\rm geo}^k/5$.

\emph{Oracle and simulation.} The oracle returns the value and saddle
operator of $B_{\rm mask}\overline\Phi/p$ and the bit $B_{\rm mask}$, with
an independent $B_{\rm mask}\sim\operatorname{Bernoulli}(p)$; it is unbiased and
$\E\|B_{\rm mask}\overline F/p\|^2=\rho\|\overline F\|^2$. As in the proof
of Theorem~\ref{vi:sp-lower}, a simulator with access to
$B_{\rm mask}\Psi/p$
answers every interior query with one call at the same point and every
exterior query with one call at $0$; the term $\mu z$ is known. Its
queries and output have blocks of norm at most $410R_{\rm loc}$, and $\Psi$ has
the flatness property, so the two-frame version of Lemma~\ref{vi:haar}
applies. Some fixed pair of frames gives~\eqref{vi:strong-good-event}
probability at least $1/2$, hence
$\E\|\widehat z-z_\star\|^2
\ge\frac12(R_0q_{\rm geo}^k/5)^2=R_0^2q_{\rm geo}^{2k}/50$.
Every sample function $B_{\rm mask}\overline\Phi/p$ is convex--concave
and has the same saddle point.
\end{proof}

\begin{corollary}[Exponential scale]\label{vi:strong-exponent}
Let $\kappa=L/\mu\ge2$. In the setting of Theorem~\ref{vi:strong-lower},
\[
 \E\|\widehat z-z_\star\|^2
 \ge\frac{R_0^2}{50}\exp\Bigl[-\frac{15372800\,N}{\kappa}\InfoDepth\Bigr].
\]
\end{corollary}
\begin{proof}
With $x=\mu/(2\lambda)$ we have
$1/q_{\rm geo}=x+\sqrt{x^2+1}$, so
$\log(1/q_{\rm geo})=\operatorname{arsinh}(x)\le x$ and
$q_{\rm geo}^{2k}\ge e^{-k\mu/\lambda}$.
Since $k\le4N\beta_p$ and $\lambda=(L-\mu)/480400$,
\[
 \frac{k\mu}\lambda\le\frac{1921600\,N\beta_p}{\kappa-1}
 \le\frac{3843200\,N\beta_p}{\kappa},
\]
where we used $1/(\kappa-1)\le2/\kappa$ for $\kappa\ge2$. Finally
$\beta_p\le4\InfoDepth$ by Theorem~\ref{thm:clock:asymptotic}, and
Theorem~\ref{vi:strong-lower} gives the claim.
\end{proof}

\section{Strongly monotone VI under mixed noise}
\label{app:strong-mixed-refinement}
This appendix combines serial, common-query, and two-phase schedules,
proves the upper bound of Theorem~\ref{vi:main-mixed-accuracy} and the
statistical part of its lower bound, and converts both into numbers of
rounds.

Throughout this appendix $F$ is $L$-Lipschitz and $\mu$-strongly
monotone with $0<\mu\le L$, $\kappa=L/\mu$, $\|z_0-z_\star\|\le R_0$, and
the oracle satisfies Assumption~\ref{ass:mixed}. We abbreviate
\[
 K_{\rm ser}=\frac K\rho,\qquad
 K_{\rm cq}=\min\Bigl\{1,\frac{mK}\rho\Bigr\}.
\]
The serial schedule makes progress $K_{\rm ser}$ per round and averages over $K$
samples per round; the common-query schedule makes progress $K_{\rm cq}$ per round
and averages over $mK$ samples. Running the serial schedule first and the
common-query schedule second combines the progress of both with the
statistical error of the second.

\begin{lemma}[The two schedules for every budget]\label{lem:smr-schedules}
Run the serial or the common-query schedule
of~\eqref{vi:all-budget-schedules} with the step sizes of
Proposition~\ref{nn:strong-distance} for $N\ge1$ rounds from a point $x$,
and return $x$ if the schedule makes no complete step. The output
$\widehat z$ satisfies, respectively,
\begin{align*}
 \E\|\widehat z-z_\star\|^2
 &\le96\Bigl[\|x-z_\star\|^2e^{-3NK_{\rm ser}/(640\kappa)}
      +\frac{\sigma^2}{\mu^2KN}\Bigr],\\
 \E\|\widehat z-z_\star\|^2
 &\le96\Bigl[\|x-z_\star\|^2e^{-3NK_{\rm cq}/(640\kappa)}
      +\frac{\sigma^2}{\mu^2mKN}\Bigr].
\end{align*}
The bounds also hold conditionally on a random starting point $x$.
\end{lemma}
\begin{proof}
Consider the serial schedule with batch size $b\le4\rho$ and
$T_{\rm sd}=\lfloor KN/(2b)\rfloor$ steps. If $KN\ge4b$, this is the
first case in the proof of Proposition~\ref{nn:strong-distance}. If
$KN<4b$, then $NK_{\rm ser}/\kappa<16\mu/L\le16$, so
$\|x-z_\star\|^2\le e^{3/40}\|x-z_\star\|^2
 e^{-3NK_{\rm ser}/(640\kappa)}$. If moreover
$T_{\rm sd}=0$, the output is $x$. If $T_{\rm sd}\ge1$,
\eqref{nn:strong-trajectory} gives the bias $\|x-z_\star\|^2$ at most and
the variance $\frac{1932}{81}\cdot\frac{\sigma^2}{\mu^2bT_{\rm sd}}$, and
$bT_{\rm sd}\ge b>KN/4$. In all cases the first bound holds. The
common-query schedule is treated in the same way with
$s_{\rm cq}\le5/K_{\rm cq}$ rounds per stage: if $N<4s_{\rm cq}$, then
$NK_{\rm cq}<20$, and if $T_{\rm cq}\ge1$, then
$mKs_{\rm cq}T_{\rm cq}\ge mKs_{\rm cq}>mKN/4$. The step sizes depend only on the
horizon and known parameters, which gives the conditional version.
\end{proof}

\begin{theorem}[Serial phase followed by common queries]
\label{thm:strong-mixed-hybrid}
For all integers $m,K,N\ge1$ there is an algorithm of
Definition~\ref{def:ic-main}. Set $c_{\rm hyb}:=1/640$ and
$C_{\rm hyb}:=3\cdot96^2$. Then
\begin{equation}\label{eq:smr-hybrid}
 \E\|\widehat z-z_\star\|^2\le C_{\rm hyb}\min\Bigl\{R_0^2,\quad
 R_0^2e^{-c_{\rm hyb}N(K_{\rm ser}+K_{\rm cq})/\kappa}
 +\frac{\sigma^2}{\mu^2mKN}
   \bigl[1+m\,e^{-c_{\rm hyb}NK_{\rm cq}/\kappa}\bigr]\Bigr\}.
\end{equation}
\end{theorem}
\begin{proof}
The bound $R_0^2$ is attained by $z_0$; the algorithm returns $z_0$ if
this is the smaller bound.

\emph{One round.} For $N=1$ we run the serial schedule. Since
$K_{\rm cq}\le1\le\kappa$, we have
$e^{-c_{\rm hyb}K_{\rm cq}/\kappa}\ge e^{-c_{\rm hyb}}$, so
Lemma~\ref{lem:smr-schedules} and $3/640\ge c_{\rm hyb}$ give
\[
 \E\|\widehat z-z_\star\|^2
 \le96\Bigl[R_0^2e^{-c_{\rm hyb}K_{\rm ser}/\kappa}
      +\frac{\sigma^2}{\mu^2K}\Bigr]
 \le96e^{c_{\rm hyb}}\Bigl[
 R_0^2e^{-c_{\rm hyb}(K_{\rm ser}+K_{\rm cq})/\kappa}
 +\frac{\sigma^2}{\mu^2mK}\,m\,
   e^{-c_{\rm hyb}K_{\rm cq}/\kappa}\Bigr],
\]
which is~\eqref{eq:smr-hybrid}.

\emph{Two phases.} Let $N\ge2$, $N_s=\lfloor N/2\rfloor$ and
$N_c=N-N_s$; both are at least $N/3$. Run the serial schedule from $z_0$
for $N_s$ rounds and let $z_{\rm mid}$ be its output; it is broadcast at the end
of round $N_s$. Then run the common-query schedule from $z_{\rm mid}$ for $N_c$
rounds. Each phase uses at most $K$ calls per worker and round, and the
two phases use $N$ rounds. By Lemma~\ref{lem:smr-schedules},
\begin{align}
 \E\|z_{\rm mid}-z_\star\|^2
 &\le96\Bigl[R_0^2e^{-3N_sK_{\rm ser}/(640\kappa)}
      +\frac{\sigma^2}{\mu^2KN_s}\Bigr],
 \label{eq:smr-first-phase}\\
 \E\bigl[\|\widehat z-z_\star\|^2\bigm|z_{\rm mid}\bigr]
 &\le96\Bigl[\|z_{\rm mid}-z_\star\|^2e^{-3N_cK_{\rm cq}/(640\kappa)}
 +\frac{\sigma^2}{\mu^2mKN_c}\Bigr].
 \label{eq:smr-second-phase}
\end{align}
Taking the expectation of~\eqref{eq:smr-second-phase} and
inserting~\eqref{eq:smr-first-phase},
\[
 \E\|\widehat z-z_\star\|^2
 \le96^2R_0^2e^{-3(N_sK_{\rm ser}+N_cK_{\rm cq})/(640\kappa)}
 +\frac{96^2\sigma^2}{\mu^2KN_s}
   e^{-3N_cK_{\rm cq}/(640\kappa)}
 +\frac{96\sigma^2}{\mu^2mKN_c}.
\]
With $N_s,N_c\ge N/3$, the exponents are at least
$N(K_{\rm ser}+K_{\rm cq})/(640\kappa)$ and
$NK_{\rm cq}/(640\kappa)$, and $1/(KN_s)\le3m/(mKN)$,
$1/(mKN_c)\le3/(mKN)$. This gives~\eqref{eq:smr-hybrid}.
\end{proof}

\begin{lemma}[Statistical lower bound for the distance]
\label{lem:smr-statistical}
In every dimension $d\ge2$, every algorithm of Definition~\ref{def:ic-main}
has an instance with a $\mu$-strongly monotone $L$-Lipschitz operator,
$\|z_0-z_\star\|\le R_0$, and an oracle with variance $\sigma^2$ at every point,
so that Assumption~\ref{ass:mixed} holds for every $\rho\ge1$, such that
\[
 \E\|\widehat z-z_\star\|^2
 \ge\frac{3}{128}\min\Bigl\{R_0^2,\frac{\sigma^2}{\mu^2mKN}\Bigr\}.
\]
\end{lemma}
\begin{proof}
For $\sigma=0$ there is nothing to prove. Otherwise let $Q_{\rm tot}=mKN$,
$r_{\rm test}=\min\{R_0,\sigma/(4\mu\sqrt {Q_{\rm tot}})\}$ and
$\theta_{\rm test}\in\{-r_{\rm test},r_{\rm test}\}$, and
consider on $\mathbb R^2$, with $z_0=0$,
\[
 F_{\theta_{\rm test}}(x)
 =\bigl(\mu(x_1-\theta_{\rm test}),Lx_2\bigr),\qquad
 \widehat F_{\theta_{\rm test}}(x)
 =F_{\theta_{\rm test}}(x)+(\sigma\zeta,0),\qquad
 \zeta\sim\mathcal N(0,1).
\]
This operator is $\mu$-strongly monotone and $L$-Lipschitz, its root
$(\theta_{\rm test},0)$ satisfies $\|z_0-z_\star\|\le R_0$, and the oracle satisfies
Assumption~\ref{ass:mixed} for every $\rho\ge1$. Subtracting the known
value $\mu x_1$ from the first coordinate of every response leaves
$Q_{\rm tot}$ observations from
$\mathcal N(-\mu\theta_{\rm test},\sigma^2)$. Since
$Q_{\rm tot}\mu^2r_{\rm test}^2/\sigma^2\le1/16$,
Lemma~\ref{lem:gaussian-testing} shows that the sign of $\widehat z_1$
differs from the sign of $\theta_{\rm test}$ with average probability
at least $3/8$, and then
$\|\widehat z-z_\star\|^2\ge r_{\rm test}^2$.
Hence one of the two signs gives
\[
 \E\|\widehat z-z_\star\|^2\ge\frac38r_{\rm test}^2
 \ge\frac{3R_0^2}{128}\min\Bigl\{1,\frac{\sigma^2}{\mu^2R_0^2mKN}\Bigr\}.
\]
Padding with coordinates on which
the operator is $Lx_j$ embeds this example in every dimension.
\end{proof}

Let $N^*_{\rm dist}(\varepsilon)$ be the smallest positive integer $N$ for
which some algorithm, which may use the parameters and the dimension,
guarantees $\E\|\widehat z-z_\star\|^2\le\varepsilon$ uniformly over the class
of this appendix, in every finite dimension. Write
$\varepsilon_{\rm rel}=\varepsilon/R_0^2$.

\begin{corollary}[Number of rounds]
\label{cor:strong-mixed-accuracy}
Suppose $\kappa\ge2$ and $0<\varepsilon_{\rm rel}<1/200$. Then
\begin{equation*}
 N^*_{\rm dist}(\varepsilon)=\Omega\!\left(\max\Bigl\{1,\
 \frac{\kappa\log(1/\varepsilon_{\rm rel})}{\InfoDepth},\
 \frac{\sigma^2}{\mu^2mK\varepsilon}\Bigr\}\right),
\end{equation*}
and
\begin{equation}\label{eq:smr-round-upper}
 \begin{aligned}
 \mathcal T_{\rm ser}
   &:=\frac{\kappa}{K_{\rm ser}}\log\frac1{\varepsilon_{\rm rel}}
      +\frac{\sigma^2}{\mu^2K\varepsilon},\\
 \mathcal T_{\rm pool}
   &:=\frac{\kappa}{K_{\rm cq}}\log\frac1{\varepsilon_{\rm rel}}
      +\frac{\sigma^2}{\mu^2mK\varepsilon},\\
 \mathcal T_{\rm hyb}
   &:=\frac{\kappa}{K_{\rm ser}+K_{\rm cq}}
      \log\frac1{\varepsilon_{\rm rel}}
      +\frac{\kappa\log(2m)}{K_{\rm cq}}
      +\frac{\sigma^2}{\mu^2mK\varepsilon},\\
 N^*_{\rm dist}(\varepsilon)
   &=O\!\left(1+\min\{\mathcal T_{\rm ser},\mathcal T_{\rm pool},
                        \mathcal T_{\rm hyb}\}\right).
 \end{aligned}
\end{equation}
\end{corollary}
\begin{proof}
\emph{Optimization lower bound.} By Corollary~\ref{vi:strong-exponent},
every algorithm has an instance with
$\E\|\widehat z-z_\star\|^2\ge\frac{R_0^2}{50}\exp[-15372800N\InfoDepth/\kappa]$.
Accuracy $\varepsilon$ therefore requires
\[
 N\ge\frac{\kappa\log(1/(50\varepsilon_{\rm rel}))}
          {15372800\InfoDepth},
 \qquad
 \log\frac1{50\varepsilon_{\rm rel}}
 \ge\Bigl(1-\frac{\log50}{\log200}\Bigr)
      \log\frac1{\varepsilon_{\rm rel}},
\]
where the second inequality uses $\varepsilon_{\rm rel}<1/200$.

\emph{Statistical lower bound.} By Lemma~\ref{lem:smr-statistical} and
$\varepsilon_{\rm rel}<3/128$, accuracy $\varepsilon$ requires
$N\ge3\sigma^2/(128\mu^2mK\varepsilon)$.

\emph{Upper bound.} We give three choices of $N$ for which an algorithm
reaches accuracy $\varepsilon$. By Lemma~\ref{lem:smr-schedules} with
$x=z_0$, the serial schedule does so if
\[
 N\ge\frac{640\kappa}{3K_{\rm ser}}
      \log\frac{192}{\varepsilon_{\rm rel}}
 \quad\text{and}\quad
 N\ge\frac{192\sigma^2}{\mu^2K\varepsilon},
\]
since then each of the two terms is at most $\varepsilon/2$; the
common-query schedule does so under the same conditions with
$K_{\rm ser}$ and $K$ replaced by $K_{\rm cq}$ and $mK$.
By Theorem~\ref{thm:strong-mixed-hybrid}, the two-phase algorithm does so if
\[
 N\ge\frac{640\kappa}{K_{\rm ser}+K_{\rm cq}}
      \log\frac{3C_{\rm hyb}}{\varepsilon_{\rm rel}},\qquad
 N\ge\frac{640\kappa\log m}{K_{\rm cq}},\qquad
 N\ge\frac{6C_{\rm hyb}\sigma^2}{\mu^2mK\varepsilon},
\]
because then
$C_{\rm hyb}R_0^2e^{-N(K_{\rm ser}+K_{\rm cq})/(640\kappa)}
\le\varepsilon/3$,
$m\,e^{-NK_{\rm cq}/(640\kappa)}\le1$, and the statistical term is at most
$2C_{\rm hyb}\sigma^2/(\mu^2mKN)\le\varepsilon/3$.
For $\varepsilon_{\rm rel}<1/200$, there is a universal constant
$C_{\log}>0$ such that
$\log(3C_{\rm hyb}/\varepsilon_{\rm rel})
\le C_{\log}\log(1/\varepsilon_{\rm rel})$, and the
smallest integer satisfying each set of conditions exceeds its largest
right-hand side by at most one. This proves~\eqref{eq:smr-round-upper}.
\end{proof}

\section{Lower bounds under mixed noise}
\label{gm:section}

This appendix proves the lower bounds of Theorem~\ref{vi:main-mixed}
(Corollary~\ref{gm:minimax}) and Theorem~\ref{vi:main-mixed-accuracy}
(Theorem~\ref{sf:rate-lower} and Corollary~\ref{sf:accuracy-match}). Besides the SGC lower bounds of
Appendix~\ref{vi:lower-section}, two new ingredients are needed. The
first is the classical statistical lower bound for $mKN$ Gaussian
observations. The second couples noise with local computation: a
\emph{clipped} chain whose oracle randomizes only the edge at the
boundary of the known coordinates. Its noise has bounded variance, so
it satisfies the mixed condition, and a single worker still discovers
one new coordinate only through a successful call. The clipped
instances are exact at their roots, so they also show that a small
noise at the solution does not help.

Throughout, $m,K,N\ge1$ are integers, $L,D>0$, $\rho\ge1$, $\sigma\ge0$,
and $h=\log(2m)$. The operators are monotone and $L$-Lipschitz; the
sampled maps need not be.

\subsection{The statistical lower bound}

\begin{proposition}[Statistical lower bound]\label{nn:statistical-lower}
Every algorithm of Definition~\ref{def:ic-main} has an affine monotone
instance with $L$-Lipschitz operator, root at distance $D$ from
$z_0=0$, and an oracle with variance $\sigma^2$ at every point (so
Assumption~\ref{ass:mixed} holds for every $\rho\ge1$), such that
\begin{equation*}
 \E\Gap_{B(0,D)}^{\rm M}(\widehat z)
 \ge\frac3{128}\min\left\{LD^2,\
                   \frac{\sigma D}{\sqrt{mKN}}\right\}.
\end{equation*}
\end{proposition}
\begin{proof}
For $\sigma=0$ there is nothing to prove. Let $Q_{\rm tot}=mKN$,
$\lambda=\min\{L,\sigma/(4D\sqrt{Q_{\rm tot}})\}$,
$\theta_{\rm test}\in\{-D,D\}$, and on
$\mathbb R^2$ (padded by coordinates with operator $Lx_j$)
\[
 F_{\theta_{\rm test}}(x)
 =\bigl(\lambda(x_1-\theta_{\rm test}),Lx_2\bigr),\qquad
 g_{\theta_{\rm test}}(x)=F_{\theta_{\rm test}}(x)+(\sigma\zeta,0),
 \quad \zeta\sim\mathcal N(0,1).
\]
The root is $(\theta_{\rm test},0)$. Subtracting the known value
$\lambda x_1$ from
the first coordinate of every response leaves $Q_{\rm tot}$ observations
from $\mathcal N(-\lambda\theta_{\rm test},\sigma^2)$. Since
$Q_{\rm tot}\lambda^2D^2/\sigma^2\le1/16$,
Lemma~\ref{lem:gaussian-testing} shows that the decision
``$\theta_{\rm test}=D$ if $\widehat z_1\ge0$,
$\theta_{\rm test}=-D$ otherwise'' errs with average probability at least
$3/8$. On an error, test at
$u=(\theta_{\rm test}/2,0)\in B(0,D)$:
\[
 \langle F_{\theta_{\rm test}}(u),\widehat z-u\rangle
 =-\frac{\lambda\theta_{\rm test}}2
   \Bigl(\widehat z_1-\frac{\theta_{\rm test}}2\Bigr)
 =\frac{\lambda D}2\Bigl(\frac D2+|\widehat z_1|\Bigr)
 \ge\frac{\lambda D^2}4,
\]
because $\widehat z_1$ and $\theta_{\rm test}$ have opposite signs. The gap is
nonnegative otherwise, so the average over the two signs is at least
$3\lambda D^2/32$, and
$\lambda D^2=\min\{LD^2,\sigma D/(4\sqrt{Q_{\rm tot}})\}
\ge\frac14\min\{LD^2,\sigma D/\sqrt{Q_{\rm tot}}\}$.
\end{proof}

\subsection{A clipped edge and a clipped chain}

The chain of Appendix~\ref{vi:robust-core} has an unbounded gradient,
so randomizing one of its edges would give unbounded variance. We clip
each edge. Let $b_\epsilon,\zeta_\epsilon$ be as in Appendix~\ref{vi:robust-core},
let $E_\epsilon^\varsigma$ denote the edge in~\eqref{vi:flat-edge},
and let $\psi$ be as in~\eqref{vi:psi-bounds}. For $a_{\rm ch}>0$ let
\[
 \omega_{a_{\rm ch}}(t)=1-\psi\Bigl(\frac{|t|-2a_{\rm ch}}{a_{\rm ch}}\Bigr),\qquad
 s_{a_{\rm ch}}(t)=\int_0^t\omega_{a_{\rm ch}}(v)\,dv,
\]
so that $\omega_{a_{\rm ch}}=1$ on $[-2a_{\rm ch},2a_{\rm ch}]$,
$\omega_{a_{\rm ch}}=0$ outside $[-3a_{\rm ch},3a_{\rm ch}]$, and
$s_{a_{\rm ch}}(t)=t$ on $[-2a_{\rm ch},2a_{\rm ch}]$.
Let $q_{a_{\rm ch}}$ be the even function with
$q_{a_{\rm ch}}(0)=q_{a_{\rm ch}}'(0)=0$ and
$q_{a_{\rm ch}}''=\frac{5a_{\rm ch}}2|s_{a_{\rm ch}}''|$; it is convex and $C^2$,
and $q_{a_{\rm ch}}=q_{a_{\rm ch}}'=0$ on $[-2a_{\rm ch},2a_{\rm ch}]$.

\begin{lemma}[Clipped edge]\label{gm:edge}
For $0<\epsilon\le a_{\rm ch}/274$ and $\varsigma\in\{-1,1\}$, the function
\begin{equation*}
 \widehat E_{\epsilon,a_{\rm ch}}^{\varsigma}(u,v)
 =\varsigma\bigl[s_{a_{\rm ch}}(u)s_{a_{\rm ch}}(v)-b_\epsilon(u)b_\epsilon(v)\bigr]
 +\zeta_\epsilon(u)-\zeta_\epsilon(v)+q_{a_{\rm ch}}(u)-q_{a_{\rm ch}}(v)
\end{equation*}
is $C^2$ and convex--concave, and its value and gradient vanish on
$\{|u|,|v|\le\epsilon\}$. Both partial derivatives are at most $6a_{\rm ch}$ in
absolute value, so the gradient has norm at most $9a_{\rm ch}$. The diagonal
second derivatives are at most $522$ and the mixed one at most $82$ in
absolute value. On $[-2a_{\rm ch},2a_{\rm ch}]^2$,
$\widehat E_{\epsilon,a_{\rm ch}}^\varsigma$ equals $E_\epsilon^\varsigma$.
\end{lemma}
\begin{proof}
Since $\psi(1-s)=1-\psi(s)$, $\int_0^1(1-\psi)=1/2$, hence
$|s_{a_{\rm ch}}|\le2a_{\rm ch}+a_{\rm ch}/2=5a_{\rm ch}/2$.
By~\eqref{vi:psi-bounds}, $|s_{a_{\rm ch}}'|=\omega_{a_{\rm ch}}\le1$
and $|s_{a_{\rm ch}}''|\le15/(8a_{\rm ch})$. On each half-line
$s_{a_{\rm ch}}'$ decreases monotonically from $1$ to $0$, so
$\int_0^\infty|s_{a_{\rm ch}}''|=1$ and
$|q_{a_{\rm ch}}'|\le5a_{\rm ch}/2$; also $q_{a_{\rm ch}}''\le75/16$.
Lemma~\ref{vi:edge-lemma} and its
proof give $|b_\epsilon|\le2\epsilon$, $|b_\epsilon'|\le9$,
$|\zeta_\epsilon'|\le256\epsilon$ and
$0\le \zeta_\epsilon''=2\epsilon|b_\epsilon''|\le256$.

Write $\widehat E$ for $\widehat E_{\epsilon,a_{\rm ch}}^\varsigma$. Then
\[
 \partial_{uu}\widehat E=\varsigma\bigl[s_{a_{\rm ch}}''(u)s_{a_{\rm ch}}(v)-b_\epsilon''(u)b_\epsilon(v)\bigr]
 +\zeta_\epsilon''(u)+q_{a_{\rm ch}}''(u)\ge0,
\]
because $|s_{a_{\rm ch}}''(u)s_{a_{\rm ch}}(v)|
\le\frac{5a_{\rm ch}}2|s_{a_{\rm ch}}''(u)|=q_{a_{\rm ch}}''(u)$ and
$|b_\epsilon''(u)b_\epsilon(v)|\le2\epsilon|b_\epsilon''(u)|=\zeta_\epsilon''(u)$;
likewise $\partial_{vv}\widehat E\le0$. These bounds also give
$|\partial_{uu}\widehat E|\le2(256+75/16)<522$. The mixed derivative is
$\varsigma[s_{a_{\rm ch}}'(u)s_{a_{\rm ch}}'(v)-b_\epsilon'(u)b_\epsilon'(v)]$, at most $1+81$ in
absolute value. The first derivative
$\partial_u\widehat E
=\varsigma[s_{a_{\rm ch}}'(u)s_{a_{\rm ch}}(v)-b_\epsilon'(u)b_\epsilon(v)]
+\zeta_\epsilon'(u)+q_{a_{\rm ch}}'(u)$ is at most
$5a_{\rm ch}/2+18\epsilon+256\epsilon+5a_{\rm ch}/2\le6a_{\rm ch}$,
and similarly for $\partial_v\widehat E$; hence the gradient norm is at
most $6\sqrt2a_{\rm ch}<9a_{\rm ch}$. For
$|u|,|v|\le\epsilon\le2a_{\rm ch}$ we have
$s_{a_{\rm ch}}=b_\epsilon=\mathrm{id}$ and $\zeta_\epsilon,q_{a_{\rm ch}}$
vanish with their derivatives, so $\widehat E$ and $\nabla\widehat E$ vanish.
On $[-2a_{\rm ch},2a_{\rm ch}]^2$, $s_{a_{\rm ch}}=\mathrm{id}$ and
$q_{a_{\rm ch}}=0$.
\end{proof}

For an integer $n\ge1$, $d_{\rm ch}=2n$, $r>0$,
$a_{\rm ch}=r/\sqrt n$ and $\lambda>0$, define
in the order $(u_1,v_1,\ldots,u_n,v_n)$
\begin{equation*}
 \widehat\Phi(u,v)=\lambda\Bigl[a_{\rm ch}u_1
 +\sum_{i=1}^n\widehat E_{\epsilon,a_{\rm ch}}^{+1}(u_i,v_i)
 +\sum_{i=1}^{n-1}\widehat E_{\epsilon,a_{\rm ch}}^{-1}(u_{i+1},v_i)\Bigr],\qquad
 \widehat G=(\nabla_u\widehat\Phi,-\nabla_v\widehat\Phi).
\end{equation*}
The $d_{\rm ch}-1$ edges couple consecutive coordinates. Let $E_j(w)$,
$1\le j<d_{\rm ch}$, be $\lambda$ times the saddle operator of the edge on
coordinates $j,j+1$, placed in these two coordinates.

\begin{lemma}[Clipped chain]\label{gm:chain}
Let $\epsilon=a_{\rm ch}/(2000000\,n\sqrt{d_{\rm ch}})$. The operator $\widehat G$ is monotone
and $\ell$-Lipschitz with $\ell=1300\lambda$, and
\begin{equation*}
 \widehat G(w)=\lambda a_{\rm ch}e_1+\sum_{j=1}^{d_{\rm ch}-1}E_j(w),
\end{equation*}
where $E_j$ depends only on coordinates $j,j+1$,
$\|E_j\|\le9\lambda a_{\rm ch}$,
and $E_j$ vanishes when both coordinates are at most $\epsilon$ in absolute
value. Moreover $\widehat G$ has a root within distance
$548n\epsilon\sqrt{d_{\rm ch}}<a_{\rm ch}/4$ of
$z_{\rm ref}=(0,-a_{\rm ch}\mathbf1)$.
\end{lemma}
\begin{proof}
Monotonicity, the decomposition and the bounds on $E_j$ follow from
Lemma~\ref{gm:edge}. Each coordinate belongs to at most two edges, so
the rows and columns of the Jacobian have absolute sums at most
$(2\cdot522+2\cdot82)\lambda=1208\lambda<\ell$. On the ball of radius
$548n\epsilon\sqrt{d_{\rm ch}}<a_{\rm ch}/4$ around
$z_{\rm ref}$, all coordinates have absolute value below $2a_{\rm ch}$,
so $\widehat G$ coincides with $G_\epsilon$ of~\eqref{vi:robust-chain}
there, and $\|\widehat G-G_0\|\le548\lambda\epsilon\sqrt{d_{\rm ch}}$
by Lemma~\ref{vi:core-root}. The Brouwer argument in the proof of
Lemma~\ref{vi:core-root} gives a root in this ball.
\end{proof}

Let $R_{\rm loc}=2r$, $R_{\rm cut}=410R_{\rm loc}$, and let
$d_{\rm amb}\ge d_{\rm ch}$. For an orthonormal $d_{\rm ch}$-frame $U$ in
$\mathbb R^{d_{\rm amb}}$, let
$G_U(x)=U\widehat G(U^\top x)$. Its root has norm at most
$r+a_{\rm ch}/4<R_{\rm loc}$, so Lemma~\ref{vi:operator-cutoff} applies:
\begin{equation}\label{gm:localized-field}
 F_U(x)=\chi(x)G_U(x)+V(x),\qquad
 \chi(x)=1-\psi\Bigl(\frac{\|x\|-2R_{\rm loc}}{408R_{\rm loc}}\Bigr),\qquad
 V(x)=84\ell(\|x\|-R_{\rm loc})_+\frac x{\|x\|}
\end{equation}
is monotone and $90\ell$-Lipschitz, has the same root, and equals the
known map $V$ for $\|x\|\ge R_{\rm cut}$.

\begin{lemma}[Gap on a hidden cube]\label{gm:cube}
Let $r=D/4$ and $k=n-1$. If $\|x\|<R_{\rm cut}$ and all hidden coordinates of
$U^\top x$ after position $k$ are at most $\epsilon$ in absolute value, or
if $\|x\|\ge R_{\rm cut}$, then
\begin{equation*}
 \Gap^{\rm M}_{B(0,D)}(x)\ge\frac{\lambda r^2}{2n}.
\end{equation*}
\end{lemma}
\begin{proof}
\emph{Interior points.} Test on the hidden cube
$\{U(u,v):|u_i|,|v_i|\le a_{\rm ch}\}$, whose points have norm at most
$\sqrt2r<R_{\rm loc}$, so $F_U=G_U$ there. Write
$U^\top x=(\widehat u,\widehat v)$. The support function of the cube of
half-width $a_{\rm ch}$ is $a_{\rm ch}\|\cdot\|_1$, so,
as in the proof of Lemma~\ref{vi:affine-comparison}, the reference
operator $G_0$ gives
\[
 \sup_{(u,v)\in[-a_{\rm ch},a_{\rm ch}]^{d_{\rm ch}}}
 \langle G_0(u,v),U^\top x-(u,v)\rangle
 =\lambda\bigl[a_{\rm ch}\widehat u_1+a_{\rm ch}\|D_n\widehat u\|_1
 +a_{\rm ch}\|D_n^\top\widehat v+a_{\rm ch}e_1\|_1\bigr]
 \ge\lambda a_{\rm ch}(a_{\rm ch}-\epsilon).
\]
Indeed $\|D_n\widehat u\|_1\ge|\mathbf1^\top D_n\widehat u|=|\widehat u_1|$,
and the coordinates of $D_n^\top\widehat v+a_{\rm ch}e_1$ sum to
$a_{\rm ch}+\widehat v_n$,
where $|\widehat v_n|\le\epsilon$ because $\widehat v_n$ is the last
coordinate. On the cube $\widehat G=G_\epsilon$, so
by~\eqref{vi:uniform-core-error} replacing $G_0$ by $\widehat G$ changes
each tested inner product by at most
\[
 548\lambda\epsilon\sqrt{d_{\rm ch}}\,(R_{\rm cut}+\sqrt2r)
 \le\frac{548\cdot822}{2000000}\cdot
 \frac{\lambda r^2}{n\sqrt n}<\frac{\lambda r^2}{4n}.
\]
Since $\lambda a_{\rm ch}(a_{\rm ch}-\epsilon)
=\frac{\lambda r^2}n(1-\epsilon/a_{\rm ch})$ and
$\epsilon/a_{\rm ch}<10^{-6}$, the gap is at least $\lambda r^2/(2n)$.

\emph{Exterior points.} The argument of Theorem~\ref{vi:uvi-lower} with
the test point $2R_{\rm loc}x/\|x\|$ gives a gap of at least
$81\cdot408\,\ell R_{\rm loc}^2\ge\lambda r^2/(2n)$.
\end{proof}

\subsection{An oracle that randomizes the boundary edge}

For $w\in\mathbb R^{d_{\rm ch}}$ let $s(w)$ be the largest index $j$ with
$|w_j|>\epsilon$, or $0$ if there is none, and let
\[
 E_s^{\rm bd}(w)=
 \begin{cases}
 \lambda a_{\rm ch}e_1,&s=0,\\
 E_s(w),&1\le s<d_{\rm ch},\\
 0,&s=d_{\rm ch},
 \end{cases}
 \qquad E_{\rm bd}(w):=E_{s(w)}^{\rm bd}(w).
\]
At every call the oracle draws an independent
$B_{\rm mask}\sim\operatorname{Bernoulli}(p)$,
with a known $0<p\le1$, and returns
\begin{equation}\label{gm:oracle}
 g_U(x,B_{\rm mask})=F_U(x)
 +\Bigl(\frac{B_{\rm mask}}p-1\Bigr)\chi(x)\,UE_{\rm bd}(U^\top x).
\end{equation}
It is unbiased, and since $\E(B_{\rm mask}/p-1)^2=1/p-1$ and
$\|E_{\rm bd}\|\le9\lambda a_{\rm ch}$,
\begin{equation}\label{gm:variance}
 \E\bigl[\|g_U(x,B_{\rm mask})-F_U(x)\|^2\bigr]
 =\Bigl(\frac1p-1\Bigr)\chi(x)^2\|E_{\rm bd}(U^\top x)\|^2
 \le\frac{81\lambda^2a_{\rm ch}^2}p .
\end{equation}
After a failure ($B_{\rm mask}=0$) the response is
$V(x)+\chi(x)U(\widehat G-E_{\rm bd})(U^\top x)$: the edge at the boundary
of the large coordinates is removed. At a root of $\widehat G$ the last
coordinate is close to $-a_{\rm ch}$, so $s=d_{\rm ch}$ and the oracle is
exact there.

\begin{lemma}[Hiding for the boundary-edge oracle]\label{gm:hiding}
Let $S_N$ be as in Section~\ref{vi:haar-subsection}, with the bits
$B_{\rm mask}$ of
the oracle, and let $k\ge0$ be an integer with $\Prob\{S_N\le k\}\ge3/4$;
for instance $k=\min\{KN,\lfloor4N\beta_p\rfloor\}$. If
\begin{equation}\label{gm:dimension}
 d_{\rm amb}\ge d_{\rm ch}
 +\frac{2R_{\rm cut}^2}{\epsilon^2}\log(32mKNd_{\rm ch}),
\end{equation}
then for every algorithm of Definition~\ref{def:ic-main}, with probability
at least $5/8$ over the frame, the bits and the algorithm's randomness,
the output either has norm at least $R_{\rm cut}$ or has all hidden coordinates
after position $k$ at most $\epsilon$ in absolute value.
\end{lemma}
\begin{proof}
We follow the proof of Lemma~\ref{vi:haar}; only the fictitious oracle
changes. Fix the bits and the algorithm's randomness, and let $c$ be
the number of granted columns before a call, as there. For a query $x$
with $\|x\|<R_{\rm cut}$ write $w=U^\top x$, and let $w^{(j)}$ keep the first $j$
coordinates of $w$ and set the others to zero.

Let $\pi_j$ keep the first $j$ components of a vector of
$\mathbb R^{d_{\rm ch}}$ and
$U_{\le j}=(U_1,\ldots,U_j)$. The known radial term $V(x)$ and the known
factor $\chi(x)$ are kept outside the projections below.

\emph{Failure.} The fictitious response is
$V(x)+\chi(x)U_{\le c}\,\pi_c(\widehat G-E_{\rm bd})(w^{(c)})$, which uses $c$
columns. Suppose all coordinates of $w$ after $c$ are at most $\epsilon$.
Then $s(w)=s(w^{(c)})=:s\le c$, and in
$\widehat G-E_s^{\rm bd}$ the edges $E_j$ with
$j>s$ vanish, the edges with $j<s$ involve and are supported on
coordinates up to $s$, and the edge $E_s$ (or the term
$\lambda a_{\rm ch}e_1$ if $s=0$) is removed. Hence
$(\widehat G-E_{\rm bd})(w)$ depends only on the first
$s\le c$ coordinates and is supported on them, so it equals
$\pi_c(\widehat G-E_{\rm bd})(w^{(c)})$, and the fictitious response equals the
true one.

\emph{Success.} Grant the column $c+1$ (if $c<d_{\rm ch}$), let
$c'=\min\{d_{\rm ch},c+1\}$,
and let the fictitious response be
\[
 V(x)+\chi(x)\,U_{\le c'}\,\pi_{c'}\bigl[\widehat G(w^{(c')})
 +(p^{-1}-1)E_{s(w^{(c)})}^{\rm bd}(w^{(c')})\bigr],
\]
where $E_s^{\rm bd}$ is the edge selected by $s$. It uses $c'$ columns; without
$\pi_{c'}$ it would not, since the edge $E_{c+1}$ evaluated at $w^{(c')}$
has a nonzero component $c+2$ when $|w_{c+1}|>\epsilon$. Under the same
assumption as above, the edges $E_j$ with $j\ge c+1$ have both coordinates
at most $\epsilon$ and vanish for $w$ and for $w^{(c')}$, the other edges
and $E_s^{\rm bd}$ with $s\le c$ involve and are supported on the first $c+1$
coordinates, and $s(w)=s(w^{(c)})$. Hence the true response is supported
on the first $c'$ coordinates and equals the fictitious one.

\emph{Exterior queries.} For $\|x\|\ge R_{\rm cut}$ the response is the known
vector $V(x)$; the simulator queries $0$ to use the bit, and its output
is $0$ whenever the true output has norm at least $R_{\rm cut}$.

All queries and outputs of the fictitious transcript have norm at most
$R_{\rm cut}$, and the number of granted columns at the end is at most
$\min\{d_{\rm ch},S_N\}$. The concentration and union-bound steps of
Lemma~\ref{vi:haar} apply without change: with probability at least $7/8$,
every fictitious query has all hidden coordinates after its own granted
prefix $c$ bounded by $\epsilon$, and the fictitious output has all hidden
coordinates after the final prefix bounded by $\epsilon$. On this event
the true and fictitious transcripts coincide by induction over calls.
Intersecting with $\{S_N\le k\}$, on which the final prefix has length at
most $k$, gives probability at least $5/8$.
\end{proof}

\subsection{Monotone VI}

\begin{theorem}[Lower bound for any success probability]
\label{gm:lower}\label{gm:frontier}
Let $0<p\le1$, $\sigma>0$, and
\begin{equation*}
 k=\min\{KN,\lfloor4N\beta_p\rfloor\},\quad n=k+1,\quad d_{\rm ch}=2n,\quad
 r=\frac D4,\quad a_{\rm ch}=\frac r{\sqrt n},\quad
 \lambda=\min\Bigl\{\frac L{120000},\ \frac{\sigma\sqrt{pn}}{9r}\Bigr\}.
\end{equation*}
In every dimension satisfying~\eqref{gm:dimension}, every algorithm of
Definition~\ref{def:ic-main} has an instance with monotone
$L$-Lipschitz operator, root in $B(0,D)$, and an oracle with
variance at most $\sigma^2$ at every point, such that
\begin{equation}\label{gm:lower-bound}
 \E\Gap^{\rm M}_{B(0,D)}(\widehat z)
 \ge c\min\left\{\frac{LD^2}{1+N\beta_p},\
 \frac{\sigma D\sqrt p}{\sqrt{1+N\beta_p}}\right\}
\end{equation}
with a universal $c>0$. In particular Assumption~\ref{ass:mixed} holds
for every $\rho\ge1$.
\end{theorem}
\begin{proof}
The operator $F_U$ is $90\cdot1300\lambda\le L$-Lipschitz, its root has
norm below $R_{\rm loc}<D$, and by~\eqref{gm:variance} and the choice of $\lambda$
the variance is at most $81\lambda^2r^2/(np)\le\sigma^2$.
Lemmas~\ref{gm:hiding} and~\ref{gm:cube} give an event of probability at
least $5/8$ on which the gap is at least $\lambda r^2/(2n)$. The gap is
nonnegative, so, as after Lemma~\ref{vi:haar}, a fixed frame gives
\[
 \E\Gap^{\rm M}_{B(0,D)}(\widehat z)\ge\frac{\lambda r^2}{4n}
 =\min\Bigl\{\frac{LD^2}{7680000\,n},\ \frac{\sigma D\sqrt p}{144\sqrt n}\Bigr\}.
\]
Since $n\le1+4N\beta_p\le4(1+N\beta_p)$, this
gives~\eqref{gm:lower-bound}.
\end{proof}

\begin{corollary}[Local calls and noise]\label{gm:local-obstruction}
In the dimension of Theorem~\ref{gm:lower} with $p=1/K$, every algorithm
has an instance with
\begin{equation*}
 \E\Gap^{\rm M}_{B(0,D)}(\widehat z)
 =\Omega\!\left(\min\left\{\frac{LD^2}{Nh},\ \frac{\sigma D}{\sqrt{KNh}}\right\}\right).
\end{equation*}
\end{corollary}
\begin{proof}
Lemma~\ref{lem:clock:elementary} gives
$\beta_{1/K}\le(e-1)+\log m$. Hence $\beta_{1/K}=O(h)$ and
$1+N\beta_{1/K}=O(Nh)$, and Theorem~\ref{gm:lower} with $p=1/K$ gives
the claim.
\end{proof}

For an ambient dimension $d_{\rm amb}$, let
$\mathfrak R_{d_{\rm amb}}(\rho,\sigma)$ be the minimax
expected Minty gap on $B(0,D)$ over monotone $L$-Lipschitz
operators with a root in this ball and oracles satisfying
Assumption~\ref{ass:mixed}, with budgets $m,K,N$.

\begin{corollary}[Monotone VIs under mixed noise]
\label{gm:equivalence}\label{gm:minimax}
Let
\[
 \mathcal U=\min\Bigl\{L,\ \frac{\rho L}{KN}+\frac{\sigma}{D\sqrt{KN}},\
 \frac LN\Bigl(1+\frac\rho{mK}\Bigr)+\frac{\sigma}{D\sqrt{mKN}}\Bigr\}.
\]
In every dimension
$\mathfrak R_{d_{\rm amb}}(\rho,\sigma)=O(D^2\,\mathcal U)$, and in
a sufficiently large dimension
\begin{equation}\label{gm:equivalent-lower}
 \mathfrak R_{d_{\rm amb}}(\rho,\sigma)
 =\Omega\!\left(\frac{D^2\,\mathcal U}{\log(2m)}\right).
\end{equation}
\end{corollary}
\begin{proof}
\emph{Upper bound.} This is~\eqref{nn:ic-upper}, since
$1+\rho/(mK)$ and $1/\min\{1,mK/\rho\}$ differ by at most a factor two.

\emph{Lower bounds available.} Write $\alpha=K/\rho$ and, after division
by $D^2$,
\[
 x_1=\frac{\rho L}{KN},\quad y_1=\frac{\sigma}{D\sqrt{KN}},\quad
 x_0=\frac LN,\quad x_m=\Bigl(1+\frac1{m\alpha}\Bigr)x_0,\quad
 y_m=\frac{y_1}{\sqrt m},
\]
so that $\mathcal U=\min\{L,x_1+y_1,x_m+y_m\}$. Operators satisfying SGC
satisfy Assumption~\ref{ass:mixed}, so Theorem~\ref{vi:uvi-lower} gives
the lower bound $c\,r_{\rm info}$ with
$r_{\rm info}=L/(1+N\InfoDepth)$.
Proposition~\ref{nn:statistical-lower} gives $c\min\{L,y_m\}$ and
Corollary~\ref{gm:local-obstruction} gives
$c\min\{x_0/h,y_1/\sqrt h\}\ge c'\min\{x_0,y_1\}/h$, using $h\ge\log2$.
For each algorithm we use the instance with the largest of these bounds,
after padding all instances to a common dimension with coordinates on
which the operator is $Lx_j$.

\emph{Case $K\ge2\rho$.} Then $\alpha\ge2$, $x_1\le x_0/2\le L$ and
$x_0\le x_m\le\frac32x_0$. Since
$\InfoDepth\le\alpha+h/\log2\le(1+1/\log2)h\alpha$, we get
$r_{\rm info}=\Omega(\min\{L,x_1/h\})$. Since $x_1/h\le L$, this gives
$r_{\rm info}=\Omega(x_1/h)$. If $y_1\le x_0$, then
$\mathcal U\le x_1+\min\{y_1,x_0\}$. If $y_1>x_0$, then
$\mathcal U\le\min\{L,\frac32x_0+y_m\}\le\frac32\min\{y_1,x_0\}+\min\{L,y_m\}$.
In both cases
$\mathcal U\le3\max\{x_1,\min\{L,y_m\},\min\{y_1,x_0\}\}$, and the maximum
of the three lower bounds is at least $c\,\mathcal U/h$.

\emph{Case $K<2\rho$.} Let $w=\min\{1,m\alpha\}$; then
$\max\{\alpha,w\}\le2w$. Lemma~\ref{lem:clock:constructive} and
Theorem~\ref{thm:clock:asymptotic} give $\InfoDepth=O(hw)$, so
$r_{\rm info}=\Omega(\min\{L,L/(Nw)\}/h)$, and
\[
 \max\bigl\{r_{\rm info},\min\{L,y_m\}\bigr\}
 =\Omega\!\left(\frac1h\min\Bigl\{L,\frac L{Nw}+y_m\Bigr\}\right).
\]
Since $x_m\le2L/(Nw)$, $\mathcal U\le\min\{L,x_m+y_m\}\le
2\min\{L,L/(Nw)+y_m\}$, which proves~\eqref{gm:equivalent-lower}.
\end{proof}

\subsection{Strongly monotone VI}
\label{sf:section}

Let $\kappa=L/\mu\ge2$, measure accuracy by the squared distance, and
let $\|z_0-z_\star\|\le R_0$ with $z_0=0$. We add the known operator $\mu I$
to the clipped chain. \revision{Here $n\ge1$ is the chain length,
$a_{\rm ch}>0$ is its amplitude, and $d_{\rm ch}:=2n$.} Set
\begin{equation*}
 \lambda=16n\mu,\qquad
 \epsilon=\frac{a_{\rm ch}}{2000000\,n\sqrt{d_{\rm ch}}},\qquad
 \ell=1300\lambda,\qquad r=a_{\rm ch}\sqrt n .
\end{equation*}
On the box $\{\|w\|_\infty\le2a_{\rm ch}\}$ the clipping is inactive, and
Lemma~\ref{vi:core-root} gives
\begin{equation}\label{sf:local-error}
 \|\widehat G(w)-G_0(w)\|\le548\lambda\epsilon\sqrt{d_{\rm ch}} .
\end{equation}

\begin{lemma}[Root of the strongly monotone chain]\label{sf:root}
The operator $\mu I+\widehat G$ has a unique root $w_\star=(u_\star,v_\star)$,
and
\begin{equation*}
 \|w_\star\|<1.072\,r,\qquad
 (v_\star)_i<-0.928\,a_{\rm ch}\quad(1\le i\le n).
\end{equation*}
\end{lemma}
\begin{proof}
Let $A=\lambda\begin{psmallmatrix}0&D_n^\top\\-D_n&0\end{psmallmatrix}$, so
that $G_0(w)=A(w-z_{\rm ref})$. Since $D_n^{-1}$ is the upper triangular
matrix of ones, $\|A^{-1}\|\le n/\lambda$ and also
$\|A^{-1}\|_{\infty}\le n/\lambda$ (maximal absolute row sum), so
$\|\mu A^{-1}\|\le1/16$ in both norms. The root of $\mu I+G_0$ is
$z_\mu=(I+\mu A^{-1})^{-1}z_{\rm ref}$, and the Neumann series gives
\begin{equation*}
 \|z_\mu-z_{\rm ref}\|\le\frac{1/16}{1-1/16}\|z_{\rm ref}\|=\frac r{15},\qquad
 \|z_\mu-z_{\rm ref}\|_\infty\le\frac{a_{\rm ch}}{15}.
\end{equation*}
All coordinates of $z_\mu$ are below $2a_{\rm ch}$ in absolute value,
so~\eqref{sf:local-error} holds at $z_\mu$. The operator $\mu I+\widehat G$
is $\mu$-strongly monotone and $(\mu+\ell)$-Lipschitz, so, as in the proof
of Lemma~\ref{vi:strong-core}, it has a unique root $w_\star$ and
\begin{equation*}
 \|w_\star-z_\mu\|\le\frac{\|(\mu I+\widehat G)(z_\mu)\|}\mu
 =\frac{\|(\widehat G-G_0)(z_\mu)\|}\mu
 \le\frac{548\lambda\epsilon\sqrt{d_{\rm ch}}}\mu
 =0.004384\,a_{\rm ch}.
\end{equation*}
Hence $\|w_\star\|\le(1+\frac1{15})r+0.004384a_{\rm ch}<1.072r$ and
$(v_\star)_i\le-a_{\rm ch}+\frac{a_{\rm ch}}{15}
+0.004384a_{\rm ch}<-0.928a_{\rm ch}$.
\end{proof}

With $R_{\rm loc}=2r$, $R_{\rm cut}=410R_{\rm loc}$, $\chi$, $V$ and
$G_U$ as in~\eqref{gm:localized-field},
let
\begin{equation*}
 F_{U,\mu}(x)=\mu x+\chi(x)G_U(x)+V(x).
\end{equation*}
It is $\mu$-strongly monotone and Lipschitz with constant
\begin{equation}\label{sf:lipschitz}
 \mu+90\ell=\mu(1+1872000\,n).
\end{equation}
Its root is $z_\star=Uw_\star$: this point has norm below $R_{\rm loc}$, where
$\chi=1$ and $V=0$, so
$F_{U,\mu}(z_\star)=U(\mu w_\star+\widehat G(w_\star))=0$.
For $\|x\|\ge R_{\rm cut}$, $F_{U,\mu}(x)=\mu x+V(x)$ is known. The oracle
is~\eqref{gm:oracle} with this $F_{U,\mu}$,
\begin{equation*}
 g_{U,\mu}(x,B_{\rm mask})=F_{U,\mu}(x)
 +\Bigl(\frac{B_{\rm mask}}p-1\Bigr)\chi(x)UE_{\rm bd}(U^\top x),
\end{equation*}
and its variance satisfies~\eqref{gm:variance}:
\begin{equation}\label{sf:variance}
 \E\|g_{U,\mu}(x,B_{\rm mask})-F_{U,\mu}(x)\|^2
 \le\frac{81\lambda^2a_{\rm ch}^2}p .
\end{equation}

\begin{lemma}[Hiding with a prescribed chain length]\label{sf:hiding}
Suppose $Nh\le n/64$. Let $p=n/(128KN)$ if $KN\ge n/128$ and $p=1$
otherwise, and let $k=\lfloor n/4\rfloor$. If
\begin{equation}\label{sf:dimension}
 d_{\rm amb}\ge d_{\rm ch}
 +\frac{2R_{\rm cut}^2}{\epsilon^2}\log(32mKNd_{\rm ch})
 =2n+10758400000000000000\,n^4\log(64mKNn),
\end{equation}
then for every algorithm there is a fixed frame $U$ such that, with
probability at least $1/2$ over the bits and the algorithm's randomness,
the output either has norm at least $R_{\rm cut}$ or has all hidden coordinates
after position $k$ at most $\epsilon$.
\end{lemma}
\begin{proof}
If $p=n/(128KN)$, Lemma~\ref{lem:clock:elementary} gives
\begin{equation*}
 \E S_N=N\beta_p\le N\bigl((e-1)Kp+\log m\bigr)
 \le\frac{(e-1)n}{128}+\frac n{64}<\frac n{32},
\end{equation*}
and since $S_N>k$ implies $S_N>n/4$, Markov's inequality gives
$\Prob\{S_N\le k\}>7/8$. If $p=1$, then $S_N\le KN<n/128\le k$ (this case
requires $n>128$). Subtracting the known term $\mu x$ from every response
reduces the oracle to~\eqref{gm:oracle}, so Lemma~\ref{gm:hiding} applies
with this $k$; together with $\Prob\{S_N\le k\}\ge7/8$, its
event has probability at least $3/4$, and some fixed frame gives
probability at least $1/2$. Finally
$R_{\rm cut}=820a_{\rm ch}\sqrt n$ and
$\epsilon=a_{\rm ch}/(2000000\,n\sqrt{2n})$ give
$2R_{\rm cut}^2/\epsilon^2
=2\cdot820^2\cdot(2\cdot10^6)^2\cdot2\,n^4$, which
is~\eqref{sf:dimension}.
\end{proof}

\begin{theorem}[Few rounds]\label{sf:short-lower}
Let $\sigma>0$ and
\begin{equation}\label{sf:length}
 n=\Bigl\lfloor\frac{\kappa-1}{1872000}\Bigr\rfloor .
\end{equation}
If $n\ge1$ and $Nh\le n/64$, then in the dimension~\eqref{sf:dimension}
every algorithm has an instance with $\mu$-strongly monotone
$L$-Lipschitz operator, $\|z_0-z_\star\|\le R_0$, and an oracle with
variance at most $\sigma^2$, such that
\begin{equation*}
 \E\|\widehat z-z_\star\|^2
 \ge\min\left\{\frac{R_0^2}{80},\
       \frac{\sigma^2}{13271040\,\mu^2KN}\right\}.
\end{equation*}
If $KN<n/128$, the oracle is exact and the bound is $R_0^2/80$, also for
$\sigma=0$.
\end{theorem}
\begin{proof}
By~\eqref{sf:lipschitz} and~\eqref{sf:length}, the Lipschitz constant is at
most $\mu\kappa=L$. If $KN\ge n/128$, let
\begin{equation*}
 p=\frac n{128KN},\qquad
 a_{\rm ch}=\min\Bigl\{\frac{R_0}{4\sqrt n},\
 \frac{\sigma\sqrt p}{9\lambda}\Bigr\},
\end{equation*}
so that the variance~\eqref{sf:variance} is at most $\sigma^2$; if
$KN<n/128$, let $p=1$ and $a_{\rm ch}=R_0/(4\sqrt n)$. In both cases
$\|z_\star\|<1.072\,a_{\rm ch}\sqrt n<R_0$.

Consider the event of Lemma~\ref{sf:hiding}. If the output has norm
below $R_{\rm cut}$, its hidden coordinates after $k=\lfloor n/4\rfloor$
are at most $\epsilon<0.001a_{\rm ch}$. The coordinate $v_i$ has position $2i$, so at least
$n-k/2\ge n/2$ of the $v$-coordinates lie after $k$, and each differs from
$(v_\star)_i<-0.928a_{\rm ch}$ by more than $0.9a_{\rm ch}$. Hence
$\|\widehat z-z_\star\|^2\ge\frac n2(0.9a_{\rm ch})^2
=0.405\,na_{\rm ch}^2$. If the output has norm at least $R_{\rm cut}$, then
$\|\widehat z-z_\star\|\ge R_{\rm cut}-1.072r>818r$, which is larger. Therefore
\[
 \E\|\widehat z-z_\star\|^2
 \ge\frac12\cdot0.405\,na_{\rm ch}^2\ge\frac{na_{\rm ch}^2}5 .
\]
For $a_{\rm ch}=R_0/(4\sqrt n)$ this is $R_0^2/80$. For
$a_{\rm ch}=\sigma\sqrt p/(9\lambda)=\sigma\sqrt p/(144n\mu)$ it is
$\sigma^2p/(103680\,\mu^2n)=\sigma^2/(13271040\,\mu^2KN)$.
\end{proof}

\begin{theorem}[Lower bound after $N$ rounds]\label{sf:rate-lower}
Let $\kappa\ge2$ and $T_{\rm ch}=1872000$. In a sufficiently large
dimension, every algorithm of Definition~\ref{def:ic-main} has, for each
of the three quantities
\begin{equation*}
 \begin{gathered}
 \frac{R_0^2}{50}e^{-15372800N\InfoDepth/\kappa},\qquad
 \frac3{128}\min\Bigl\{R_0^2,\frac{\sigma^2}{\mu^2mKN}\Bigr\},\\
 \mathbf 1\Bigl\{N\log(2m)\le\frac{\kappa}{256T_{\rm ch}}\Bigr\}
 \min\Bigl\{\frac{R_0^2}{80},\frac{\sigma^2}{13271040\,\mu^2KN}\Bigr\},
 \end{gathered}
\end{equation*}
an instance with a $\mu$-strongly monotone $L$-Lipschitz operator,
$\|z_0-z_\star\|\le R_0$, and an oracle satisfying
Assumption~\ref{ass:mixed}, on which $\E\|\widehat z-z_\star\|^2$ is at
least this quantity.
\end{theorem}
\begin{proof}
The first quantity is Corollary~\ref{vi:strong-exponent}; its oracle
satisfies SGC and therefore Assumption~\ref{ass:mixed} for every
$\sigma\ge0$. The second is Lemma~\ref{lem:smr-statistical}. The third is
zero for $\sigma=0$. For $\sigma>0$ let $h=\log(2m)$ and suppose
$Nh\le\kappa/(256T_{\rm ch})$. Since $N\ge1$ and $h\ge\log2$, this gives
$\kappa\ge256T_{\rm ch}\log2>2T_{\rm ch}+1$, so
$n=\lfloor(\kappa-1)/T_{\rm ch}\rfloor\ge(\kappa-1)/T_{\rm ch}-1
\ge\kappa/(4T_{\rm ch})\ge1$ and $Nh\le n/64$; Theorem~\ref{sf:short-lower}
gives the bound. Coordinates on which the operator is $\mu x_j$ embed all
instances in a common dimension.
\end{proof}

We use $N^*_{\rm dist}(\varepsilon)$ from
Appendix~\ref{app:strong-mixed-refinement}, and write
$\varepsilon_{\rm rel}=\varepsilon/R_0^2$,
$s_1=\sigma^2/(\mu^2K\varepsilon)$ and
$s_m=s_1/m$.

\begin{corollary}[Lower bound on the number of rounds]
\label{sf:accuracy-lower}
For $\kappa\ge2$ and $0<\varepsilon_{\rm rel}<1/200$,
\begin{equation}\label{sf:accuracy-bridge}
 N^*_{\rm dist}(\varepsilon)=\Omega\!\left(\min\Bigl\{\frac\kappa h,\ s_1\Bigr\}\right),
\end{equation}
and, together with Corollary~\ref{cor:strong-mixed-accuracy},
\begin{equation}\label{sf:accuracy-full-lower}
 N^*_{\rm dist}(\varepsilon)=\Omega\!\Biggl(1
 +\frac{\kappa\log(1/\varepsilon_{\rm rel})}{\InfoDepth}
 +\min\Bigl\{s_1,\ \frac\kappa h+s_m\Bigr\}\Biggr).
\end{equation}
\end{corollary}
\begin{proof}
Let $T=1872000$ and $n=\lfloor(\kappa-1)/T\rfloor$. Suppose
$N<\min\{n/(64h),\,s_1/13271040\}$. Then $n\ge1$ and $Nh\le n/64$, and
Theorem~\ref{sf:short-lower} gives an instance with
$\E\|\widehat z-z_\star\|^2>\varepsilon$, because $R_0^2/80>\varepsilon$ and
$\sigma^2/(13271040\mu^2KN)>\varepsilon$. Hence
$N^*_{\rm dist}(\varepsilon)\ge\min\{n/(64h),s_1/13271040\}$. If
$\kappa\ge2T+1$, then $n\ge(\kappa-1)/T-1\ge\kappa/(4T)$, which
gives~\eqref{sf:accuracy-bridge}. If $2\le\kappa<2T+1$, then
$\kappa/h\le(2T+1)/\log2$, and~\eqref{sf:accuracy-bridge} follows from
$N^*_{\rm dist}(\varepsilon)\ge1$.

Corollary~\ref{cor:strong-mixed-accuracy} gives the lower bounds $1$,
$\kappa\log(1/\varepsilon_{\rm rel})/\InfoDepth$ and $s_m$.
Since $s_m\le s_1$,
\[
 \min\Bigl\{s_1,\frac\kappa h+s_m\Bigr\}
 \le s_m+\min\Bigl\{s_1,\frac\kappa h\Bigr\}
 \le2\max\Bigl\{s_m,\min\Bigl\{s_1,\frac\kappa h\Bigr\}\Bigr\},
\]
so the sum in~\eqref{sf:accuracy-full-lower} is at most four times the
largest of the four lower bounds.
\end{proof}

\begin{corollary}[Number of rounds up to worker logarithms]
\label{sf:accuracy-match}
For $\kappa\ge2$ and $0<\varepsilon_{\rm rel}<1/200$, the lower
bound~\eqref{sf:accuracy-full-lower} and the upper bound
\begin{equation}\label{sf:accuracy-full-upper}
 N^*_{\rm dist}(\varepsilon)=O\!\Biggl(1
 +\frac{\kappa\log(1/\varepsilon_{\rm rel})}
 {\max\{K/\rho,\min\{1,mK/\rho\}\}}
 +\min\bigl\{s_1,\ \kappa h+s_m\bigr\}\Biggr)
\end{equation}
differ by at most a factor $C\log^2(2m)$ with a universal $C$.
\end{corollary}
\begin{proof}
Let $\alpha=K/\rho$, $w=\min\{1,m\alpha\}$ and $\Delta=\max\{\alpha,w\}$.

\emph{Upper bound.} If $\alpha\le w$, then $\Delta=w$, and the
common-query entry of~\eqref{eq:smr-round-upper} gives
$1+\kappa\log(1/\varepsilon_{\rm rel})/w+s_m$, which is at most the right side
of~\eqref{sf:accuracy-full-upper} because $s_m\le s_1$ and
$s_m\le\kappa h+s_m$. If $\alpha>w$, then $w=1<\alpha=\Delta$, and the
serial and two-phase entries give
$1+\kappa\log(1/\varepsilon_{\rm rel})/\alpha+s_1$ and
$1+\kappa\log(1/\varepsilon_{\rm rel})/(\alpha+1)+\kappa h+s_m$;
their minimum is at most
the right side of~\eqref{sf:accuracy-full-upper}.

\emph{Comparison.} If $\alpha\ge1$, then $\Delta\ge\alpha$ and
$\InfoDepth\le\alpha+h/\log2\le C_1h\Delta$. If $\alpha<1$, then
$\Delta=w$, and Lemma~\ref{lem:clock:constructive} with
Theorem~\ref{thm:clock:asymptotic} gives $\InfoDepth\le C_1h\Delta$. Hence
the second terms of the two bounds differ by at most a factor $C_1h$.
For the third terms, $\kappa h+s_m\le\max\{1,h^2\}(\kappa/h+s_m)$ and
$\max\{1,h^2\}\le h^2/\log^22$, so
$\min\{s_1,\kappa h+s_m\}\le(h^2/\log^22)\min\{s_1,\kappa/h+s_m\}$.
\end{proof}

\section{Generalized smoothness: adaptive steps and restarts}
\label{gs:appendix}\label{grv:section}

This appendix proves Theorem~\ref{ghs:upper} and, from it,
Corollary~\ref{gs:main-distance-complexity} (Corollary~\ref{gs:rate-sgc}). Throughout, Assumptions~\ref{ass:generalized-smoothness},
\ref{ass:strong-monotonicity} and~\ref{ass:mixed} hold with
$L_0\ge\mu>0$; we write $\kappa_0=L_0/\mu$, $a_0=L_1R_0$, and fix the
target accuracy $0<\varepsilon<R_0^2$. The method has three layers.
\begin{enumerate}[label=(\arabic*),leftmargin=*]
\item \emph{One stage} (Section~\ref{grv:adaptive-subsection}). Given a
center $c$ with $\|c-z_\star\|\le r$, projected extragradient on the ball
$B(c,2r)$ with a step adapted to a robust estimate of
$\|F\|$ reaches distance $r/8$ with probability at least $15/16$.
\item \emph{Restarts} (Section~\ref{grv:restart-subsection}). Repeating
stages with halved radii and taking geometric medians of independent
copies reaches accuracy $\varepsilon$.
\item \emph{Communication} (Section~\ref{ghs:section}). The first stages
run on one worker and the later stages pool responses over all workers.
\end{enumerate}
Two tools are proved first.

\subsection{Two tools}

\begin{lemma}[Lipschitz bound along a segment]\label{gs:gronwall}
For all $x,y\in\mathbb R^d$,
\[
 \|F(y)-F(x)\|\le(L_0+L_1\|F(x)\|)\,e^{L_1\|y-x\|}\,\|y-x\| .
\]
\end{lemma}
\begin{proof}
Let $u=(y-x)/\|y-x\|$ and $\varphi_{\rm seg}(t)=\|F(x+tu)\|$. Then
$\varphi_{\rm seg}$ is absolutely continuous and
$\varphi_{\rm seg}'(t)\le\|\mathrm DF(x+tu)\|_{\rm op}\le L_0+L_1\varphi_{\rm seg}(t)$
almost everywhere, so $L_0+L_1\varphi_{\rm seg}(t)\le(L_0+L_1\varphi_{\rm seg}(0))e^{L_1t}$ (for $L_1=0$
directly, for $L_1>0$ by Gr\"onwall's inequality). Integrating
$\|\mathrm DF(x+tu)u\|\le(L_0+L_1\varphi_{\rm seg}(0))e^{L_1t}$ over $t\in[0,\|y-x\|]$
gives the claim.
\end{proof}

\begin{lemma}[Projected extragradient step]\label{gs:seg-residual}
Let $C$ be closed and convex, $x,u\in C$, $\eta>0$, and
\[
 y=\Pi_C\bigl(x-\eta(F(x)+e)\bigr),\qquad
 x^+=\Pi_C\bigl(x-\eta(F(y)+h)\bigr).
\]
If $\|F(y)-F(x)\|\le L_{\rm seg}\|y-x\|$, then
\begin{equation}\label{gs:projected-pathwise}
 \|x^+-u\|^2\le\|x-u\|^2-2\eta\langle F(y),y-u\rangle
 -(1-3\eta^2L_{\rm seg}^2)\|x-y\|^2
 +3\eta^2\bigl(\|e\|^2+\|h\|^2\bigr)-2\eta\langle h,y-u\rangle .
\end{equation}
If moreover $C=B(c,2r)$, $\|c-z_\star\|\le r$,
$\eta\|F(x)\|\le3r$, and $y^0=\Pi_C(x-\eta F(x))$, then
\begin{equation}\label{gs:projection-margin}
 \|x-y^0\|\ge\frac{\eta\|F(x)\|}8 .
\end{equation}
\end{lemma}
\begin{proof}
The optimality conditions of the two projections,
$\langle x-\eta(F(y)+h)-x^+,u-x^+\rangle\le0$ and
$\langle x-\eta(F(x)+e)-y,x^+-y\rangle\le0$, give the standard bound
\[
 \|x^+-u\|^2\le\|x-u\|^2-2\eta\langle F(y)+h,y-u\rangle
 -\|x-y\|^2-\|y-x^+\|^2
 +2\eta\langle F(y)-F(x)+h-e,\,y-x^+\rangle .
\]
By $2\langle a,b\rangle\le\|a\|^2+\|b\|^2$ the last term is at most
$\eta^2\|F(y)-F(x)+h-e\|^2+\|y-x^+\|^2$, and
$\|F(y)-F(x)+h-e\|^2\le3(L_{\rm seg}^2\|y-x\|^2+\|h\|^2+\|e\|^2)$. This
proves~\eqref{gs:projected-pathwise}.

For~\eqref{gs:projection-margin} let $M_x=\|F(x)\|$ and
$\Delta_0=\|x-y^0\|$.
If $M_x=0$ there is nothing to prove, and if $\Delta_0\ge r/2$ the claim
follows from $\eta M_x\le3r$. Otherwise let
$u=z_\star-rF(x)/M_x$, which lies in
$C$. Monotonicity and $F(z_\star)=0$ give
$\langle F(x),x-z_\star\rangle\ge0$, hence
\[
 \langle F(x),y^0-u\rangle
 =\langle F(x),y^0-x\rangle+\langle F(x),x-z_\star\rangle+rM_x
 \ge M_x(r-\Delta_0)\ge\frac{M_xr}2 .
\]
The optimality condition of $y^0$ gives
$\eta\langle F(x),y^0-u\rangle\le\langle x-y^0,y^0-u\rangle\le4r\Delta_0$,
since $C$ has diameter $4r$. Hence $\eta M_xr/2\le4r\Delta_0$.
\end{proof}

\subsection{A robust estimate of the operator}

\begin{lemma}[Median of block means]\label{grv:robust}
Fix a query $x$ determined by the history $\mathcal F$, a reference
magnitude $F_{\rm ref}>0$, a tolerance $\delta_{\rm rob}>0$, and a
failure probability $0<\tau_{\rm fail}<1$. Let
\begin{equation}\label{grv:robust-budget}
 n=\max\Bigl\{1,\Bigl\lceil\frac{256}{\delta_{\rm rob}^2}
       \Bigl(\rho-1+\frac{\sigma^2}{F_{\rm ref}^2}\Bigr)\Bigr\rceil\Bigr\},
 \qquad k=\max\{1,\lceil8\log(1/\tau_{\rm fail})\rceil\},\qquad b=nk .
\end{equation}
The geometric median $\widehat F_x$ of $k$ block means, each of $n$ independent
responses at $x$, satisfies
\begin{equation*}
 \Prob\bigl\{\|\widehat F_x-F(x)\|>
 \delta_{\rm rob}\max\{\|F(x)\|,F_{\rm ref}\}\bigm|\mathcal F\bigr\}
 \le\tau_{\rm fail} .
\end{equation*}
If $\rho=1$ and $\sigma=0$, one exact response is used instead ($b=1$). The
median is the minimum-norm minimizer of the sum of distances, which is
measurable and lies in the convex hull of the block means.
\end{lemma}
\begin{proof}
Let $M_x=\|F(x)\|$ and $M_{\rm rob}=\max\{M_x,F_{\rm ref}\}$. A block
mean $h_j$ has conditional variance at most
$((\rho-1)M_x^2+\sigma^2)/n$, so Chebyshev's inequality and
the choice of $n$ give
\[
 \Prob\Bigl\{\|h_j-F(x)\|>
       \frac{\delta_{\rm rob}M_{\rm rob}}4\Bigm|\text{earlier blocks}\Bigr\}
 \le\frac{16\bigl((\rho-1)M_x^2+\sigma^2\bigr)}
          {n\delta_{\rm rob}^2M_{\rm rob}^2}
 \le\frac{16}{n\delta_{\rm rob}^2}
       \Bigl(\rho-1+\frac{\sigma^2}{F_{\rm ref}^2}\Bigr)\le\frac1{16}.
\]

\emph{Medians.} Suppose that at least a fraction $1-\vartheta$ of points
$x_1,\ldots,x_k$ lie within $r_{\rm med}$ of a point $v$, and let
$z_{\rm med}$ be a geometric
median. For such a good point,
$\|x_i-z_{\rm med}\|-\|x_i-v\|\ge\|z_{\rm med}-v\|-2r_{\rm med}$,
and for every other point
$\|x_i-z_{\rm med}\|-\|x_i-v\|\ge-\|z_{\rm med}-v\|$.
Since $z_{\rm med}$ minimizes the sum of
distances, summing gives
\begin{equation}\label{grv:median-distance}
 (1-2\vartheta)\|z_{\rm med}-v\|\le2(1-\vartheta)r_{\rm med} .
\end{equation}

\emph{Counting bad blocks.} Let $I_j$ indicate
$\|h_j-F(x)\|>\delta_{\rm rob}M_{\rm rob}/4$. Then
$\E[4^{I_j}\mid\text{earlier blocks}]\le1+3/16$,
so $\E4^{\sum_jI_j}\le(19/16)^k$ and, by Markov's inequality,
\[
 \Prob\Bigl\{\sum_jI_j\ge\frac k4\Bigr\}\le\Bigl(\frac{19}{16}\cdot4^{-1/4}\Bigr)^k
 \le e^{-k/8}\le\tau_{\rm fail} .
\]
Otherwise $\vartheta<1/4$, and~\eqref{grv:median-distance} with
$r_{\rm med}=\delta_{\rm rob}M_{\rm rob}/4$ gives
\[
 \|\widehat F_x-F(x)\|
 \le\frac{2\cdot3/4}{1/2}\cdot\frac{\delta_{\rm rob}M_{\rm rob}}4
 =\frac{3\delta_{\rm rob}M_{\rm rob}}4.
\]
\end{proof}

\subsection{One stage}
\label{grv:adaptive-subsection}

Fix a center $c$ and a radius $r>0$, let
$\mathcal B(c,r)=B(c,2r)$, and
\begin{equation*}
 S_r=\max\{(1+L_1r)^2,\kappa_0\},\quad M_r=\mu S_r,\quad
 F_r=\mu r\sqrt{S_r},\quad \ell_r=L_1+\frac1r,\quad
 \theta_r=\frac{2^{-32}}{S_r},\quad \delta_{\rm rob}=2^{-20}.
\end{equation*}
Then $M_r\ge\mu\kappa_0=L_0$ and
$\ell_rF_r=\mu\sqrt{S_r}(1+L_1r)\le M_r$. Let
$T_r=\lceil\theta_r^{-1}\log4096\rceil$,
$\tau_{{\rm fail},r}=\theta_r/(9\cdot4096)$, and
let $n_r$, $k_r$, and $b_r=n_rk_r$ be the quantities $n$, $k$, and $b$
from Lemma~\ref{grv:robust} with
$F_{\rm ref}=F_r$, $\delta_{\rm rob}$, and
$\tau_{\rm fail}=\tau_{{\rm fail},r}$. Starting from $x_0=c$, a stage performs
$T_r$ iterations
\begin{equation*}
 \eta=\frac1{64(M_r+\ell_r\|\widehat F_x\|)},\qquad
 y=\Pi_{\mathcal B(c,r)}(x-\eta\widehat F_x),\qquad
 \widehat h=\frac1{b_r}\sum_{l=1}^{b_r}g(y,\xi_l),\qquad
 x^+=\Pi_{\mathcal B(c,r)}(x-\eta\widehat h),
\end{equation*}
where $\widehat F_x$ is the robust estimate of $F(x)$ from
Lemma~\ref{grv:robust}, and the responses $g(y,\xi_l)$ are independent. Every
iteration uses exactly $2b_r$ calls. The step is small when the estimated
operator is large; this replaces a global Lipschitz constant.

\begin{lemma}[One stage]\label{grv:adaptive-stage}
Conditionally on any history at the start of the stage with
$\|c-z_\star\|\le r$, the output $z_{\rm out}(c,r)=x_{T_r}$ satisfies
\begin{equation*}
 \E\|z_{\rm out}(c,r)-z_\star\|^2\le\frac{r^2}{1024},
 \qquad
 \Prob\Bigl\{\|z_{\rm out}(c,r)-z_\star\|>\frac r8\Bigr\}\le\frac1{16}.
\end{equation*}
All queries and the output lie in $\mathcal B(c,r)$ on every history.
For a universal constant $C_{\rm stage}>0$,
\begin{equation}\label{grv:adaptive-stage-work}
 T_r\le C_{\rm stage}S_r,\qquad
 b_r\le C_{\rm stage}\Bigl(\rho+\frac{\sigma^2}{\mu^2r^2S_r}\Bigr)\log(64S_r),\qquad
 b_rT_r\le C_{\rm stage}\Bigl(\rho S_r+\frac{\sigma^2}{\mu^2r^2}\Bigr)\log(64S_r).
\end{equation}
\end{lemma}
\begin{proof}
Fix an iteration and condition on the history before it. Drop the index
$r$ by writing
\[
 S=S_r,\quad M=M_r,\quad \ell=\ell_r,\quad F_{\rm ref}=F_r,
 \quad \tau_{\rm fail}=\tau_{{\rm fail},r},\quad
 n=n_r,\quad b=b_r.
\]
Also write $M_x=\|F(x)\|$,
$\widehat M_x=\|\widehat F_x\|$,
$e_x=\widehat F_x-F(x)$,
$E_x=\|x-z_\star\|^2\le9r^2$,
$R_y=\|y-z_\star\|$, and $\Delta_{xy}=\|x-y\|$.

\emph{A good estimate controls the step.} Consider the event
$\{\|e_x\|\le\delta_{\rm rob}\max\{M_x,F_{\rm ref}\}\}$, which fails
with probability at most $\tau_{\rm fail}$ (Lemma~\ref{grv:robust}). On it,
$(1-\delta_{\rm rob})M_x\le\widehat M_x+\delta_{\rm rob}F_{\rm ref}$,
and since $\ell F_{\rm ref}\le M$, we have
$(1-\delta_{\rm rob})(M+\ell M_x)\le M+\ell\widehat M_x$. Hence
\begin{equation*}
 \eta(L_0+L_1M_x)\le\eta(M+\ell M_x)
      \le\frac1{64(1-\delta_{\rm rob})},\qquad
 \Delta_{xy}\le\eta\widehat M_x\le\frac1{64\ell},\qquad
 \eta M_x<\frac1{32\ell}\le\frac r{32}.
\end{equation*}
By Lemma~\ref{gs:gronwall}, $F$ is $L_{xy}$-Lipschitz on the segment
$[x,y]$ with $L_{xy}=(L_0+L_1M_x)e^{L_1\Delta_{xy}}$, and
$L_1\Delta_{xy}\le\ell\Delta_{xy}\le1/64$ gives
$\eta L_{xy}\le e^{1/64}/(64(1-\delta_{\rm rob}))<1/16$. Therefore
$\|F(y)\|\le M_x+\widehat M_x/16
\le\frac{17}{16}M_x+\frac{\|e_x\|}{16}$ and
$\|F(y)\|^2\le3M_x^2+\delta_{\rm rob}^2F_{\rm ref}^2/128$.

\emph{The step is not too short.} Since $\eta M_x<r/32$,
Lemma~\ref{gs:seg-residual} gives $\|x-y^0\|\ge\eta M_x/8$ for
$y^0=\Pi_{\mathcal B(c,r)}(x-\eta F(x))$. Since
$\|y-y^0\|\le\eta\|e_x\|$ and
$(s-t)^2\ge s^2/2-t^2$,
\begin{equation}\label{grv:adaptive-displacement}
 \Delta_{xy}^2\ge\frac{\eta^2M_x^2}{128}-\eta^2\|e_x\|^2 .
\end{equation}

\emph{The corrector.} Let $e_y=\widehat h-F(y)$. The batch size $b\ge n$ gives
$(\rho-1)/b\le\delta_{\rm rob}^2/256$ and
$\sigma^2/b\le\delta_{\rm rob}^2F_{\rm ref}^2/256$, so
\begin{equation}\label{grv:adaptive-corrector}
 \E[e_y\mid\widehat F_x,\mathcal F]=0,\qquad
 \E[\|e_y\|^2\mid\widehat F_x,\mathcal F]
 \le\frac{\rho-1}b\|F(y)\|^2+\frac{\sigma^2}b
 \le\frac{\delta_{\rm rob}^2}{64}(M_x^2+F_{\rm ref}^2).
\end{equation}

\emph{One iteration on the good event.} Apply~\eqref{gs:projected-pathwise}
with $u=z_\star$ and $L_{\rm seg}=L_{xy}$; the last term has zero conditional mean,
$1-3\eta^2L_{xy}^2\ge3/4$, and
$\|e_x\|^2\le\delta_{\rm rob}^2(M_x^2+F_{\rm ref}^2)$.
By~\eqref{grv:adaptive-displacement} and~\eqref{grv:adaptive-corrector}, the
terms other than $-2\eta\langle F(y),y-z_\star\rangle$ contribute at most
\[
 \eta^2\Bigl[-\frac{3M_x^2}{512}
 +\Bigl(\frac34+3\Bigr)\delta_{\rm rob}^2(M_x^2+F_{\rm ref}^2)
 +\frac3{64}\delta_{\rm rob}^2(M_x^2+F_{\rm ref}^2)\Bigr]
 \le-\frac{\eta^2M_x^2}{256}
      +4\delta_{\rm rob}^2\eta^2F_{\rm ref}^2,
\]
because $\delta_{\rm rob}=2^{-20}$. With
$\langle F(y),y-z_\star\rangle\ge\mu R_y^2$,
\begin{equation}\label{grv:adaptive-good-drift}
 \E[\|x^+-z_\star\|^2\mid\widehat F_x,\mathcal F]
 \le E_x-2\eta\mu R_y^2-\frac{\eta^2M_x^2}{256}
      +4\delta_{\rm rob}^2\eta^2F_{\rm ref}^2 .
\end{equation}

\emph{Contraction.} Let $\theta_{\rm ctr}=\theta_r$. If
$\ell M_x>M$, then $M_x>F_{\rm ref}$,
$\widehat M_x\le(1+\delta_{\rm rob})M_x$ and
$\eta M_x\ge1/(192\ell)$; since
$S\ge(1+L_1r)^2=\ell^2r^2$ and $E_x\le9r^2$,
\[
 \frac{\eta^2M_x^2}{256}\ge\frac1{256\cdot192^2\ell^2}
 \ge\frac {E_x}{9\cdot256\cdot192^2S}\ge\theta_{\rm ctr}E_x,
\]
because $9\cdot256\cdot192^2<2^{32}$. If $\ell M_x\le M$, then
$\ell\widehat M_x\le(1+\delta_{\rm rob})M$, so
$\eta\mu\ge\mu/(192M)=1/(192S)\ge\theta_{\rm ctr}$;
moreover
$E_x\le2R_y^2+2\Delta_{xy}^2
\le2R_y^2+8\eta^2M_x^2+4\delta_{\rm rob}^2\eta^2F_{\rm ref}^2$, and
since $8\theta_{\rm ctr}\le1/256$, the two negative terms
of~\eqref{grv:adaptive-good-drift} are at least
$\theta_{\rm ctr}E_x
-4\theta_{\rm ctr}\delta_{\rm rob}^2\eta^2F_{\rm ref}^2$. In both cases
\[
 \E[\|x^+-z_\star\|^2\mid\widehat F_x,\mathcal F]
 \le(1-\theta_{\rm ctr})E_x
      +5\delta_{\rm rob}^2\eta^2F_{\rm ref}^2
 \le(1-\theta_{\rm ctr})E_x+\frac{\theta_{\rm ctr} r^2}{4096},
\]
using
$\eta^2F_{\rm ref}^2\le F_{\rm ref}^2/(4096M^2)=r^2/(4096S)$ and
$5\delta_{\rm rob}^2\le2^{-32}$.

\emph{Failures and iteration.} On the event that the estimate fails,
$\|x^+-z_\star\|^2\le9r^2$. Hence
\begin{equation*}
 \E[\|x^+-z_\star\|^2\mid\mathcal F]
 \le(1-\theta_{\rm ctr})E_x
      +\frac{\theta_{\rm ctr}r^2}{4096}+9r^2\tau_{\rm fail} .
\end{equation*}
Iterating from $\|x_0-z_\star\|^2\le r^2$ and using
$(1-\theta_{\rm ctr})^{T_r}\le e^{-\theta_{\rm ctr}T_r}\le1/4096$
and $9\tau_{\rm fail}/\theta_{\rm ctr}=1/4096$,
\[
 \E\|x_{T_r}-z_\star\|^2\le\frac{r^2}{4096}+\frac{r^2}{4096}+\frac{r^2}{4096}
 \le\frac{r^2}{1024},
\]
and Markov's inequality gives the probability bound.

\emph{Work.} $T_r\le2^{32}S\log4096+1$ and
$\log(1/\tau_{{\rm fail},r})
=\log(9\cdot4096\cdot2^{32}S)\le C_{\rm stage}\log(64S)$; substituting
$F_r^2=\mu^2r^2S$ into~\eqref{grv:robust-budget}
gives~\eqref{grv:adaptive-stage-work}.
\end{proof}

\subsection{Restarts}
\label{grv:restart-subsection}

Let
\begin{equation*}
 J=\Bigl\lceil\frac{\log(2R_0^2/\varepsilon)}{2\log2}\Bigr\rceil,\qquad
 r_j=2^{-j}R_0,\qquad
 k_{\rm rep}=\Bigl\lceil8\log\frac{50JR_0^2}{\varepsilon}\Bigr\rceil .
\end{equation*}
Starting from $c_0=z_0$, stage $j=0,\ldots,J-1$ runs $k_{\rm rep}$ independent
copies of the stage of Lemma~\ref{grv:adaptive-stage} from $c_j$ with radius
$r_j$, and lets $c_{j+1}$ be the geometric median of their outputs. The
method returns $c_J$. Write $T_j=T_{r_j}$, $b_j=b_{r_j}$, $S_j=S_{r_j}$, and
\begin{equation*}
 A_{\rm rst}=a_0^2+\kappa_0J,\qquad
 S_0=\max\{(1+a_0)^2,\kappa_0\},\qquad \ell_0=\log(64S_0).
\end{equation*}

\begin{theorem}[Restarts]\label{grv:restarted}
In every dimension, $\E\|c_J-z_\star\|^2\le\varepsilon$. The method uses
$2k_{\rm rep}\sum_{j<J}b_jT_j$ calls. For a universal constant
$C_{\rm rst}>0$,
\begin{equation}\label{grv:restart-sums}
 \sum_{j<J}T_j\le C_{\rm rst}A_{\rm rst},\qquad
 \sum_{j<J}b_jT_j\le C_{\rm rst}\ell_0
 \Bigl(\rho A_{\rm rst}+\frac{\sigma^2}{\mu^2\varepsilon}\Bigr).
\end{equation}
\end{theorem}
\begin{proof}
\emph{Accuracy.} Suppose $\|c_j-z_\star\|\le r_j$. Each copy ends within
$r_j/8$ of $z_\star$ with probability at least $15/16$, independently given
$c_j$. As in the proof of Lemma~\ref{grv:robust}, fewer than a quarter of the
copies fail with probability at least $1-e^{-k_{\rm rep}/8}$, and then
\eqref{grv:median-distance} gives $\|c_{j+1}-z_\star\|\le3r_j/8\le r_{j+1}$.
By the choice of $k_{\rm rep}$, each stage fails with probability at most
$\varepsilon/(50JR_0^2)$, so all stages succeed except with probability at most
$\varepsilon/(50R_0^2)$. On every history the copies stay in
$B(c_j,2r_j)$, hence so does their median, and
$\|c_J-z_\star\|\le R_0+2\sum_jr_j<5R_0$. The choice of $J$ gives
$\varepsilon/8<r_J^2\le\varepsilon/2$. Therefore
\[
 \E\|c_J-z_\star\|^2\le r_J^2+25R_0^2\cdot\frac\varepsilon{50R_0^2}\le\varepsilon .
\]

\emph{Work.} Since $\kappa_0\ge1$, $S_j\le(1+L_1r_j)^2+\kappa_0\le3\kappa_0+2L_1^2r_j^2$,
so $\sum_{j<J}S_j\le3\kappa_0J+\frac83a_0^2\le3A_{\rm rst}$, and
$\sum_{j<J}r_j^{-2}=(r_J^{-2}-R_0^{-2})/3<8/(3\varepsilon)$. Since
$S_j\le S_0$, \eqref{grv:adaptive-stage-work} gives~\eqref{grv:restart-sums}.
\end{proof}

\subsection{Serial stages followed by pooled stages}
\label{ghs:section}

A stage can run on one worker, which keeps its state across round
boundaries, or with common queries, where each batch of $b_j$ responses is
collected over $\lceil b_j/(mK)\rceil$ rounds by all workers.

\begin{theorem}[One switch]\label{ghs:upper}
For $j_{\rm sw}\in\{0,\ldots,J\}$, run stages $j<j_{\rm sw}$ on one
worker and stages $j\ge j_{\rm sw}$
with common queries. Then $\E\|\widehat z-z_\star\|^2\le\varepsilon$, with
\begin{equation}\label{ghs:exact-budget}
 \Bigl\lceil\frac{2k_{\rm rep}}K\sum_{j<j_{\rm sw}}b_jT_j\Bigr\rceil
 +2k_{\rm rep}\sum_{j_{\rm sw}\le j<J}T_j
      \Bigl\lceil\frac{b_j}{mK}\Bigr\rceil
\end{equation}
rounds. Choosing $j_{\rm sw}$ to minimize~\eqref{ghs:exact-budget}, the number of
rounds is at most
\begin{equation}\label{ghs:rate}
 \begin{aligned}
 \widetilde O\Bigl(&1
 +\min\Bigl\{\frac\rho K,1+\frac\rho{mK}\Bigr\}
       \Bigl[a_0^2+\kappa_0\log\frac{2R_0^2}{\varepsilon}\Bigr]\\
 &+\min\Bigl\{\frac{\sigma^2}{\mu^2K\varepsilon},\
 \kappa_0+\frac{\sigma L_1}{\mu\sqrt K}
 +\frac{\sigma^2}{\mu^2mK\varepsilon}\Bigr\}\Bigr)
 \end{aligned}
\end{equation}
where the tilde hides only logarithmic factors in $2+a_0$, $2+\kappa_0$ and
$2R_0^2/\varepsilon$.
\end{theorem}
\begin{proof}
\emph{Accuracy and budget.} The way responses are distributed over workers
changes neither the stage algorithm nor the conditional moments, so
Theorem~\ref{grv:restarted} applies. The serial part needs
$\lceil(2k_{\rm rep}/K)\sum_{j<j_{\rm sw}}b_jT_j\rceil$ rounds, its last round
communicates $c_{j_{\rm sw}}$, and each pooled iteration needs
$2\lceil b_j/(mK)\rceil$ rounds; this is~\eqref{ghs:exact-budget}.

\emph{Cost per stage.} Let $\vartheta_1=\rho/K$,
$\vartheta_m=1+\rho/(mK)$, $U_j=L_1^2r_j^2+\kappa_0$ and
$V_j=\sigma^2/(\mu^2r_j^2)$, so that $S_j\le3U_j$. By
\eqref{grv:adaptive-stage-work} and $\lceil t\rceil\le1+t$,
for a universal constant $C_{\rm comm}>0$,
\eqref{ghs:exact-budget} is at most $C_{\rm comm}k_{\rm rep}\ell_0$ times
\begin{equation}\label{ghs:stage-cost}
 1+\vartheta_1\sum_{j<j_{\rm sw}}U_j
 +\vartheta_m\sum_{j_{\rm sw}\le j<J}U_j
 +\frac1K\sum_{j<j_{\rm sw}}V_j
 +\frac1{mK}\sum_{j_{\rm sw}\le j<J}V_j .
\end{equation}
We use $\sum_{j<J}U_j\le\frac43a_0^2+\kappa_0J$,
$\sum_{j_{\rm sw}\le j<J}U_j
\le\frac43L_1^2r_{j_{\rm sw}}^2+\kappa_0J$,
$\sum_{j<j_{\rm sw}}V_j
\le\sigma^2/(3\mu^2r_{j_{\rm sw}}^2)$ and
$\sum_{j<J}V_j\le8\sigma^2/(3\mu^2\varepsilon)$.

\emph{The case $\vartheta_1\ge\vartheta_m$.} Take $j_{\rm sw}=0$:
\eqref{ghs:stage-cost} is at most
$1+\vartheta_m\sum_jU_j+8\sigma^2/(3\mu^2mK\varepsilon)$, which is bounded
by~\eqref{ghs:rate}.

\emph{The case $\vartheta_1<\vartheta_m$.} Then $0<\vartheta_m-\vartheta_1\le1$
and~\eqref{ghs:stage-cost} is at most
\[
 1+\vartheta_1\sum_{j<J}U_j+\sum_{j_{\rm sw}\le j<J}U_j
 +\frac1K\sum_{j<j_{\rm sw}}V_j+\frac1{mK}\sum_{j<J}V_j .
\]
With $j_{\rm sw}=J$ the third term vanishes, and this gives the first entry of the
minimum in~\eqref{ghs:rate}. For the second entry, note that
$\sum_{j_{\rm sw}\le j<J}U_j\le\kappa_0J+\Lambda_{j_{\rm sw}}$ with
$\Lambda_{j_{\rm sw}}=\frac43L_1^2r_{j_{\rm sw}}^2$ for
$j_{\rm sw}<J$ and $\Lambda_J=0$, so it suffices to choose $j_{\rm sw}$ with
\[
 \Lambda_{j_{\rm sw}}+\frac1K\sum_{j<j_{\rm sw}}V_j
 \le\frac{8\sigma L_1}{3\mu\sqrt K}.
\]
If $L_1=0$, take $j_{\rm sw}=0$; if $\sigma=0$, take $j_{\rm sw}=J$.
Otherwise let $r_{\rm bal}^2=\sigma/(\mu L_1\sqrt K)$.
If $r_J^2<r_{\rm bal}^2<R_0^2$, take the unique
$j_{\rm sw}\in\{1,\ldots,J\}$ with
$r_{j_{\rm sw}}^2\le r_{\rm bal}^2<4r_{j_{\rm sw}}^2$; then the left side is at most
$\frac43L_1^2r_{\rm bal}^2
 +\frac{\sigma^2}{3\mu^2Kr_{j_{\rm sw}}^2}
\le\frac43L_1^2r_{\rm bal}^2
 +\frac{4\sigma^2}{3\mu^2Kr_{\rm bal}^2}
=\frac{8\sigma L_1}{3\mu\sqrt K}$.
If $r_{\rm bal}^2\ge R_0^2$, take $j_{\rm sw}=0$: the left side is
$\frac43L_1^2R_0^2\le\frac43L_1^2r_{\rm bal}^2
=\frac{4\sigma L_1}{3\mu\sqrt K}$. If
$r_{\rm bal}^2\le r_J^2$, take $j_{\rm sw}=J$: the left side is at most
$\frac{\sigma^2}{3\mu^2Kr_J^2}
\le\frac{\sigma^2}{3\mu^2Kr_{\rm bal}^2}
=\frac{\sigma L_1}{3\mu\sqrt K}$.
In all cases we obtain the second entry, with $\kappa_0J$ in place of
$\kappa_0$.

Finally $J$, $k_{\rm rep}$ and $\ell_0$ are logarithmic in $2R_0^2/\varepsilon$,
$2+a_0$ and $2+\kappa_0$, which gives~\eqref{ghs:rate}.
\end{proof}

\begin{corollary}[Rate under SGC]\label{gs:rate-sgc}
Let $\sigma=0$, $\vartheta=\min\{\rho/K,1+\rho/(mK)\}$, let $C_{\rm comm}$ be
the constant in the proof of Theorem~\ref{ghs:upper}, and
$C_{\rm rate}=51C_{\rm comm}$. For every $N\ge1$ there is a number of stages
$J$ for which the method of Theorem~\ref{ghs:upper} uses at most $N$
rounds and
\begin{equation}\label{gs:rate-sgc-bound}
 \E\|\widehat z-z_\star\|^2\le8R_0^2\cdot4^{-x},\qquad
 x=\min\Bigl\{\frac{N}{2C_{\rm rate}\ell_0(1+\vartheta a_0^2)},\
 \sqrt{\frac{N}{2C_{\rm rate}\ell_0\vartheta\kappa_0}}\Bigr\}.
\end{equation}
\end{corollary}
\begin{proof}
For an integer $J\ge1$ take the target $\varepsilon_J=2R_0^2\cdot4^{-J}$; then
the method has $J$ stages and
$k_{\rm rep}=\lceil8\log(25J\cdot4^J)\rceil
\le8\log25+8(J-1)+8J\log4+1\le38J$.
With $\sigma=0$ the terms with $V_j$ in~\eqref{ghs:stage-cost} vanish, and the
choice $j_{\rm sw}=0$ if $\vartheta_1\ge\vartheta_m$ and $j_{\rm sw}=J$
otherwise, as in the proof of Theorem~\ref{ghs:upper}, bounds the number of
rounds by
\[
 C_{\rm comm}k_{\rm rep}\ell_0\Bigl[1+\vartheta\Bigl(\tfrac43a_0^2+\kappa_0J\Bigr)\Bigr]
 \le C_{\rm rate}\ell_0\bigl[J(1+\vartheta a_0^2)+\vartheta\kappa_0J^2\bigr].
\]
Let $J=\lfloor x\rfloor$. If $J\ge1$, each of the two parts of the right
side is at most $N/2$, so the method uses at most $N$ rounds, and
Theorem~\ref{grv:restarted} gives
$\E\|\widehat z-z_\star\|^2\le\varepsilon_J=2R_0^2\cdot4^{-J}\le8R_0^2\cdot4^{-x}$.
If $J=0$, the method returns $z_0$, and
$\|z_0-z_\star\|^2\le R_0^2\le8R_0^2\cdot4^{-x}$ because $x<1$.
\end{proof}

\section{Composite VI: accelerated mirror-prox}
\label{interp:app-proofs}
\label{interp:app-communication}

This appendix proves the single-worker entry of
Theorem~\ref{thm:composite-local-main} and
Proposition~\ref{interp:separation}, which shows that interpolation does
not imply SGC. Under Assumption~\ref{ass:interp-main} we write
\[
 D_f(x,y)=f(x)-f(y)-\langle\nabla f(y),x-y\rangle,\qquad
 \phi(x)=D_f(x,x_\star),\qquad
 \mathcal M(x,u)=f(x)-f(u)+\langle H(u),x-u\rangle,
\]
so that $\CompGap(x)=\sup_{\|u-x_\star\|\le R_0}\mathcal M(x,u)$ and
$\mathcal M(x,x_\star)=\phi(x)\ge0$, since $\nabla f(x_\star)=-H(x_\star)$.
The function $\phi$ is convex, nonnegative, $L_f$-smooth and vanishes at
$x_\star$.

\begin{lemma}[Variance of smooth convex samples]
\label{lem:composite:variance}\label{lem:samples:variance}
\label{prop:composite-potential-main}
\label{interp:samples-subsection}
For every query $x$, conditionally on the history before the call,
\begin{equation*}
 \operatorname{Var}(\nabla f_\xi(x))
 \le4L_{\max}\min\{D_f(x,x_\star),D_f(x_\star,x)\}+2\sigma_\star^2.
\end{equation*}
Under interpolation, $4L_{\max}$ can be replaced by $2L_{\max}$ and the
last term vanishes.
\end{lemma}
\begin{proof}
For a convex function with $L_{\max}$-Lipschitz gradient,
$\|\nabla f_\xi(x)-\nabla f_\xi(y)\|^2\le2L_{\max}D_{f_\xi}(x,y)$ for
both orders of $x$ and $y$. Moreover $\E D_{f_\xi}(x,y)=D_f(x,y)$: with
$e=x-y$,
$D_{f_\xi}(x,y)=\int_0^1\langle\nabla f_\xi(y+te)-\nabla f_\xi(y),e\rangle\,dt$,
the integrand is at most $L_{\max}t\|e\|^2$ in absolute value, and
Fubini's theorem and unbiasedness of $\nabla f_\xi$ apply. The variance
is at most the second moment about $\nabla f(x_\star)$, so
\[
 \operatorname{Var}(\nabla f_\xi(x))
 \le\E\|\nabla f_\xi(x)-\nabla f(x_\star)\|^2
 \le2\E\|\nabla f_\xi(x)-\nabla f_\xi(x_\star)\|^2
 +2\E\|\nabla f_\xi(x_\star)-\nabla f(x_\star)\|^2,
\]
and the last expectation is $\sigma_\star^2$
by~\eqref{eq:composite-root-variance-main}. Under interpolation
$\nabla f_\xi(x_\star)=\nabla f(x_\star)$ almost surely, and the factor two
is not needed. Each call uses an independent sample, which gives the
conditional statement.
\end{proof}

\subsection{The method and its guarantee}
\label{interp:smooth-algorithm}
We first state the accelerated mirror-prox scheme and then isolate the
one-step estimate used by both schedules.

Fix integers $b,T\ge1$ and let
\begin{equation}
 \begin{aligned}
 a_\phi&=\frac{4L_{\max}}b,
 &\sigma_{\star,b}&=\frac{\sqrt2\,\sigma_\star}{\sqrt b},\\
 \Gamma_{\rm amp}&=64\Bigl(L_f+L_HT+a_\phi T+\frac{\sigma_{\star,b}T^{3/2}}{R_0}\Bigr),
 &\alpha_t&=\frac2{t+1},\qquad \eta_t=\frac t{\Gamma_{\rm amp}}.
 \end{aligned}
 \label{interp:parameters}
\end{equation}
Under interpolation one may take $a_\phi=2L_{\max}/b$. If $\Gamma_{\rm amp}=0$, return
$x_0$. Otherwise set $r_1=z_0=x_0$ and, for $t=1,\ldots,T$,
\begin{align}
 y_t&=(1-\alpha_t)z_{t-1}+\alpha_t r_t,
 &g_t&=\frac1b\sum_{j=1}^b\nabla f_{\xi_{t,j}}(y_t),\notag\\
 w_t&=r_t-\eta_t\bigl(H(r_t)+g_t\bigr),
 &r_{t+1}&=r_t-\eta_t\bigl(H(w_t)+g_t\bigr),\label{interp:algorithm}\\
 z_t&=(1-\alpha_t)z_{t-1}+\alpha_t w_t,\notag
\end{align}
with independent samples $\xi_{t,j}$. These are the Euclidean updates of
accelerated mirror-prox~\citep{chen2017amp} without projection; the
output is $z_T$, and $r_{T+1}$ is not computed.

\begin{theorem}[Accelerated mirror-prox]
\label{thm:interp-main}\label{interp:accelerated}
Under Assumption~\ref{ass:interp-main}, in every dimension,
\begin{align}
 \E\CompGap(z_T)
 &\le\frac{10\Gamma_{\rm amp}R_0^2}{T(T+1)}\label{interp:main}\\
 &\le640\Bigl[R_0^2\Bigl(\frac{L_f}{T^2}+\frac{L_H}T
       +\frac{4L_{\max}}{bT}\Bigr)
       +\frac{\sqrt2\,\sigma_\star R_0}{\sqrt{bT}}\Bigr].\notag
\end{align}
The method uses $bT$ stochastic gradients and $2T-1$ evaluations of $H$.
Under interpolation, $4L_{\max}$ may be replaced by $2L_{\max}$.
\end{theorem}
\begin{proof}
If $\Gamma_{\rm amp}=0$, then $L_f=L_H=0$, so $\nabla f$ and $H$ are constant with zero
sum, and $\CompGap\equiv0$. Let $\Gamma_{\rm amp}>0$ and put
\[
 \delta_t=g_t-\nabla f(y_t),\qquad d_t=r_t-w_t,\qquad
 \omega_t=\frac2{t(t+1)} .
\]
By Lemma~\ref{lem:composite:variance} and independence within a batch,
\begin{equation}
 \E[\delta_t\mid\mathcal F_{t-1}]=0,\qquad
 \E[\|\delta_t\|^2\mid\mathcal F_{t-1}]\le a_\phi\phi(y_t)+\sigma_{\star,b}^2,
 \label{interp:variance}
\end{equation}
where $\mathcal F_{t-1}$ is the history before the batch of step $t$. All
iterates have finite second moments by induction.

\emph{Step 1: the function part.} Since $z_t-y_t=-\alpha_td_t$ and
$z_t=(1-\alpha_t)z_{t-1}+\alpha_tw_t$, smoothness and convexity of $f$
give, for every $u$,
\begin{align*}
 f(z_t)&\le f(y_t)+\langle\nabla f(y_t),z_t-y_t\rangle
 +\frac{L_f\alpha_t^2}2\|d_t\|^2\notag\\
 &=(1-\alpha_t)\bigl[f(y_t)+\langle\nabla f(y_t),z_{t-1}-y_t\rangle\bigr]
 +\alpha_t\bigl[f(y_t)+\langle\nabla f(y_t),w_t-y_t\rangle\bigr]
 +\frac{L_f\alpha_t^2}2\|d_t\|^2\notag\\
 &\le(1-\alpha_t)f(z_{t-1})+\alpha_t\bigl[f(u)+\langle\nabla f(y_t),w_t-u\rangle\bigr]
 +\frac{L_f\alpha_t^2}2\|d_t\|^2 .
\end{align*}
Moreover
$\langle H(u),z_t-u\rangle-(1-\alpha_t)\langle H(u),z_{t-1}-u\rangle
=\alpha_t\langle H(u),w_t-u\rangle\le\alpha_t\langle H(w_t),w_t-u\rangle$
by monotonicity. Adding,
\[
 \mathcal M(z_t,u)-(1-\alpha_t)\mathcal M(z_{t-1},u)
 \le\alpha_t\langle H(w_t)+g_t,w_t-u\rangle
 -\alpha_t\langle\delta_t,w_t-u\rangle+\frac{L_f\alpha_t^2}2\|d_t\|^2 .
\]

\emph{Step 2: the extragradient part.} Expanding the squares in the
updates of $w_t$ and $r_{t+1}$,
\begin{align*}
 \eta_t\langle H(w_t)+g_t,w_t-u\rangle
 &=\frac{\|r_t-u\|^2-\|r_{t+1}-u\|^2}2
  -\frac{\|d_t\|^2}2-\frac{\|w_t-r_{t+1}\|^2}2\notag\\
 &\quad+\eta_t\langle H(w_t)-H(r_t),w_t-r_{t+1}\rangle,
\end{align*}
and Young's inequality bounds the last term by
$\frac12\eta_t^2L_H^2\|d_t\|^2+\frac12\|w_t-r_{t+1}\|^2$. For the noise,
\[
 -\langle\delta_t,w_t-u\rangle
 =\langle\delta_t,u-r_t\rangle+\langle\delta_t,d_t\rangle,\qquad
 \langle\delta_t,d_t\rangle\le2\eta_t\|\delta_t\|^2+\frac{\|d_t\|^2}{8\eta_t}.
\]

\emph{Step 3: one step.} Combining the last three displays, the
coefficient of $\|d_t\|^2$ is
$\frac{\alpha_t}{\eta_t}\bigl(-\frac38+\frac12(\eta_t^2L_H^2+L_f\alpha_t\eta_t)\bigr)$.
By~\eqref{interp:parameters}, $\eta_t^2L_H^2\le(TL_H/\Gamma_{\rm amp})^2\le1/4096$ and
$L_f\alpha_t\eta_t\le2L_f/\Gamma_{\rm amp}\le1/32$, so this coefficient is at most
$-\alpha_t/(4\eta_t)$, and for every $u$
\begin{align}\label{interp:onestep}
 \mathcal M(z_t,u)&\le(1-\alpha_t)\mathcal M(z_{t-1},u)
 +\frac{\alpha_t}{2\eta_t}\bigl(\|r_t-u\|^2-\|r_{t+1}-u\|^2\bigr)\notag\\
 &\quad-\frac{\alpha_t}{4\eta_t}\|d_t\|^2
       +2\alpha_t\eta_t\|\delta_t\|^2
       +\alpha_t\langle\delta_t,u-r_t\rangle.
\end{align}

\emph{Step 4: summation.} We have $\alpha_t/\omega_t=t$, $t/\eta_t=\Gamma_{\rm amp}$,
$(1-\alpha_t)/\omega_t=1/\omega_{t-1}$ for $t\ge2$, and $\alpha_1=1$.
Dividing~\eqref{interp:onestep} by $\omega_t$ and summing,
\begin{align}\label{interp:weighted}
 \frac{\mathcal M(z_t,u)}{\omega_t}
 +\frac{\Gamma_{\rm amp}}4\sum_{j\le t}\|d_j\|^2
 &\le\frac{\Gamma_{\rm amp}}2\bigl(\|x_0-u\|^2-\|r_{t+1}-u\|^2\bigr)
 +\frac2{\Gamma_{\rm amp}}\sum_{j\le t}j^2\|\delta_j\|^2
 +\sum_{j\le t}j\langle\delta_j,u-r_j\rangle .
\end{align}

\emph{Step 5: the noise is controlled by the gap at the root.} Since
$\phi$ is nonnegative and $L_f$-smooth, $\|\nabla\phi(z)\|^2\le2L_f\phi(z)$,
and therefore
$\phi(y)\le\phi(z)+\|\nabla\phi(z)\|\|y-z\|+\frac{L_f}2\|y-z\|^2
\le2\phi(z)+L_f\|y-z\|^2$. With $y_t-z_t=\alpha_td_t$,
\begin{equation}\label{interp:newaggregate}
 \phi(y_t)\le2\phi(z_t)+L_f\alpha_t^2\|d_t\|^2 .
\end{equation}
Let $V_t=\E\|\delta_t\|^2$, $p_t=\E\phi(z_t)/\omega_t$ and let the
accumulated step energy be
$\mathsf E_t^{\rm step}=\sum_{j\le t}\E\|d_j\|^2$. By~\eqref{interp:variance}
and~\eqref{interp:newaggregate},
\begin{equation}\label{interp:varianceused}
 V_t\le2a_\phi\omega_tp_t+a_\phi L_f\alpha_t^2\E\|d_t\|^2+\sigma_{\star,b}^2 .
\end{equation}
Take $u=x_\star$ in~\eqref{interp:weighted}: then
$\mathcal M(z_t,x_\star)=\phi(z_t)$ and the last sum has zero mean. Using
$j^2\omega_j\le2$, $j^2\alpha_j^2\le4$ and $\sum_{j\le t}j^2\le T^3$,
\begin{equation*}
 p_t+\frac{\Gamma_{\rm amp}}4\mathsf E_t^{\rm step}
 \le\frac{\Gamma_{\rm amp}R_0^2}2+\frac{8a_\phi}{\Gamma_{\rm amp}}\sum_{j\le t}p_j
 +\frac{8a_\phi L_f}{\Gamma_{\rm amp}}\mathsf E_t^{\rm step}
 +\frac{2\sigma_{\star,b}^2T^3}{\Gamma_{\rm amp}} .
\end{equation*}
By~\eqref{interp:parameters}, $a_\phi T\le \Gamma_{\rm amp}/64$,
$L_f\le \Gamma_{\rm amp}/64$ and
$\sigma_{\star,b}^2T^3\le \Gamma_{\rm amp}^2R_0^2/4096$; in particular
$8a_\phi L_f\le8\Gamma_{\rm amp}^2/64^2\le \Gamma_{\rm amp}^2/8$. Hence
\begin{equation}\label{interp:closure}
 p_t+\frac{\Gamma_{\rm amp}}8\mathsf E_t^{\rm step}
 \le\Bigl(\frac12+\frac1{2048}\Bigr)\Gamma_{\rm amp}R_0^2
 +\frac{8a_\phi}{\Gamma_{\rm amp}}\sum_{j\le t}p_j .
\end{equation}
Let $p_{\max}=\max_{t\le T}p_t$. Since $8a_\phi T/\Gamma_{\rm amp}\le1/8$,
\eqref{interp:closure} at a maximizing index gives
$p_{\max}\le(\frac12+\frac1{2048})\Gamma_{\rm amp}R_0^2+\frac18p_{\max}$, so
$p_{\max}\le \Gamma_{\rm amp}R_0^2$, and then~\eqref{interp:closure} with $t=T$ gives
$\mathsf E_T^{\rm step}\le8R_0^2$. Consequently, by~\eqref{interp:varianceused},
\begin{align*}
 \mathsf V_{\rm wt}:=\sum_{t\le T}t^2V_t
 &\le4a_\phi T\Gamma_{\rm amp}R_0^2
 +4a_\phi L_f\mathsf E_T^{\rm step}+\sigma_{\star,b}^2T^3\notag\\
 &\le \Gamma_{\rm amp}^2R_0^2\Bigl(\frac4{64}+\frac{32}{4096}+\frac1{4096}\Bigr)
 \le \Gamma_{\rm amp}^2R_0^2 .
\end{align*}

\emph{Step 6: the gap.} Let $\|u-x_\star\|\le R_0$, so that
$\|x_0-u\|\le2R_0$. Write
$\sum_tt\langle\delta_t,u-r_t\rangle
=\langle\sum_tt\delta_t,u-x_\star\rangle+\sum_tt\langle\delta_t,x_\star-r_t\rangle$.
The second sum has zero mean, and the supremum of the first over $u$ is
$R_0\|\sum_tt\delta_t\|$, whose expectation is at most
$R_0\sqrt{\mathsf V_{\rm wt}}$ by
orthogonality of martingale differences. Taking the supremum
in~\eqref{interp:weighted} at $t=T$ and then the expectation,
\[
 \frac{\E\CompGap(z_T)}{\omega_T}
 \le2\Gamma_{\rm amp}R_0^2
 +\frac{2\mathsf V_{\rm wt}}{\Gamma_{\rm amp}}
 +R_0\sqrt{\mathsf V_{\rm wt}}\le5\Gamma_{\rm amp}R_0^2,
\]
which is~\eqref{interp:main}. The second line follows from
$T(T+1)\ge T^2$ and the definition of $\Gamma_{\rm amp}$.
\end{proof}

\subsection{Serial and synchronized schedules}
\label{interp:schedules-subsection}

\begin{proposition}[Two schedules]
\label{prop:composite-schedules-main}
Under Assumption~\ref{ass:interp-main}, an algorithm of
Definition~\ref{def:ic-main} satisfies
\begin{equation}
 \mathbb E\CompGap(\widehat x)=O\!\left(R_0^2
 \min\left\{\begin{aligned}
 &L_f+L_H,\\[-2pt]
 &\frac{L_f}{(KN)^2}+\frac{L_H+L_{\max}}{KN}
       +\frac{\sigma_\star}{R_0\sqrt{KN}}
 \end{aligned}\right\}\right)
 \label{eq:comm-upper-main}
\end{equation}
and, if $K\ge3$, also
$\mathbb E\CompGap(\widehat x)=O\!\left(R_0^2\bigl(\frac{L_f}{N^2}
+\frac{L_H}N+\frac{L_{\max}}{mKN}+\frac{\sigma_\star}{R_0\sqrt{mKN}}\bigr)\right)$.
\end{proposition}
\begin{proof}
\emph{Initial point.} Let $\widetilde H(u)=H(u)-H(x_\star)$. Since
$\nabla f(x_\star)+H(x_\star)=0$,
$\mathcal M(x_0,u)=\phi(x_0)-\phi(u)+\langle\widetilde H(u),x_0-u\rangle$.
For $\|u-x_\star\|\le R_0$, monotonicity gives
$\langle\widetilde H(u),x_\star-u\rangle\le0$, hence
$\langle\widetilde H(u),x_0-u\rangle\le\langle\widetilde H(u),x_0-x_\star\rangle
\le L_HR_0^2$, and $\phi(x_0)\le L_fR_0^2/2$, $\phi(u)\ge0$. Thus
$\CompGap(x_0)\le(L_f/2+L_H)R_0^2$.

\emph{Serial schedule.} One worker runs~\eqref{interp:algorithm} with $b=1$
and $T=\lfloor KN/3\rfloor$. Each step uses one gradient and at most two
evaluations of $H$, so the $3T\le KN$ calls can be split into rounds of
$K$ calls, the worker keeping its state across round boundaries. If
$KN\ge3$, then $T\ge KN/6$ and Theorem~\ref{thm:interp-main} gives the
second entry of~\eqref{eq:comm-upper-main}. If $KN<3$, the initial point
gives $(L_f/2+L_H)R_0^2$, which is at most twice the second entry.

\emph{Synchronized schedule.} Let $K\ge3$, $T=N$ and $b=m(K-2)\ge mK/3$.
At the start of round $t\ge2$ every worker knows
$r_{t-1},w_{t-1},z_{t-1},g_{t-1}$. It evaluates $H(w_{t-1})$ and computes
$r_t=r_{t-1}-\eta_{t-1}(H(w_{t-1})+g_{t-1})$, then evaluates $H(r_t)$,
computes $y_t$, and spends its remaining $K-2$ calls on independent
gradients at $y_t$. Averaging these gradients over the workers at the
end of the round gives $g_t$ with an independent batch of size $m(K-2)$, after
which all workers compute $w_t$ and $z_t$. In round $1$ only $H(r_1)$ is
needed. This reproduces~\eqref{interp:algorithm} with at most $K$ calls
per worker and round, and Theorem~\ref{thm:interp-main} gives the third
bound, since $4L_{\max}/(bN)\le12L_{\max}/(mKN)$ and
$\sigma_\star/\sqrt{bN}\le\sqrt3\,\sigma_\star/\sqrt{mKN}$.

The algorithm returns the candidate with the smallest bound.
\end{proof}

\subsection{Interpolation does not imply SGC}
\label{interp:separation-subsection}

\begin{proposition}[Interpolation without strong growth]
\label{interp:separation}
There is a distribution of nonnegative convex $1$-smooth functions
$f_\xi$ on $\mathbb R$, all minimized at $0$, whose mean $f=\E f_\xi$ has
the unique minimizer $0$, such that no finite $\rho$ satisfies
$\E|f_\xi'(x)|^2\le\rho|f'(x)|^2$ for all $x$.
\end{proposition}
\begin{proof}
Let $U$ be uniform on $[0,1]$ and
\[
 f_U(x)=\tfrac12\min\{x,0\}^2+\tfrac12\max\{x-U,0\}^2,\qquad
 f_U'(x)=\min\{x,0\}+\max\{x-U,0\}.
\]
The derivative is continuous, nondecreasing and piecewise affine with
slopes $0$ and $1$, so every $f_U$ is convex, $1$-smooth, nonnegative,
and zero at $0$. For $0<x<1$,
\begin{equation*}
 f'(x)=\E f_U'(x)=\int_0^x(x-u)\,du=\frac{x^2}2,\qquad
 \E f_U'(x)^2=\int_0^x(x-u)^2\,du=\frac{x^3}3,
\end{equation*}
while $f'(x)=x$ for $x<0$ and $f'(x)=x-1/2$ for $x\ge1$. Hence $0$ is the
unique zero of $f'$. On $(0,1)$,
$\E f_U'(x)^2/f'(x)^2=4/(3x)\to\infty$ as $x\downarrow0$.
\end{proof}
In the composite model with $H=0$, these samples satisfy
Assumption~\ref{ass:interp-main} with $\sigma_\star=0$.

\section{Composite VI: pooled gradients and local solves}
\label{comp-local:section}
\label{app:composite-local-refinement}

This appendix proves the pooled entry of
Theorem~\ref{thm:composite-local-main}. The idea is to evaluate stochastic
gradients in one round with all workers and then, during the next rounds,
let one worker solve the resulting proximal subproblem for the exact
component $H$. Assumption~\ref{ass:interp-main} holds throughout, and we use
$D_f$, $\phi$ and $\mathcal M$ from Appendix~\ref{interp:app-proofs}. Let
\[
 \mathcal X_{\rm loc}=B(x_0,2R_0).
\]
Then $B(x_\star,R_0)\subseteq\mathcal X_{\rm loc}$,
$\mathcal X_{\rm loc}$ has diameter $4R_0$, and
$\|v-x_\star\|\le3R_0$ for $v\in\mathcal X_{\rm loc}$. Let
$g_\star=\nabla f(x_\star)=-H(x_\star)$.

We use two facts about a VI on $\mathcal X_{\rm loc}$ with a $\mu$-strongly monotone,
$L_A$-Lipschitz operator $A$ and solution $w^\star$. First, projected
extragradient with step $\gamma=1/(4L_A)$,
$v=\Pi_{\mathcal X_{\rm loc}}(x-\gamma A(x))$,
$x^+=\Pi_{\mathcal X_{\rm loc}}(x-\gamma A(v))$,
satisfies
\begin{equation}\label{comp-local:peg}
 \|x^+-w^\star\|^2\le(1-\gamma\mu)\|x-w^\star\|^2 .
\end{equation}
Indeed, the standard analysis of projected extragradient gives
$\|x^+-w^\star\|^2\le\|x-w^\star\|^2-(1-\gamma^2L_A^2)\|x-v\|^2
-2\gamma\langle A(v),v-w^\star\rangle$, and
$\langle A(v),v-w^\star\rangle\ge\mu\|v-w^\star\|^2$ by strong monotonicity
and the VI condition at $w^\star$. Since
$1-\gamma^2L_A^2=15/16\ge2\gamma\mu$
and $\|x-w^\star\|^2\le2\|x-v\|^2+2\|v-w^\star\|^2$, the decrease is at least
$\gamma\mu\|x-w^\star\|^2$. Second, for every $w\in\mathcal X_{\rm loc}$,
\begin{equation}\label{comp-local:residual-distance}
 \max_{v\in\mathcal X_{\rm loc}}\langle A(w),w-v\rangle
 \le\Bigl(4L_AR_0+\sup_{\mathcal X_{\rm loc}}\|A\|\Bigr)\|w-w^\star\|,
\end{equation}
because
$\langle A(w),w-v\rangle=\langle A(w)-A(w^\star),w-v\rangle
+\langle A(w^\star),w-w^\star\rangle+\langle A(w^\star),w^\star-v\rangle$
and the last term is nonpositive.

\subsection{Accelerated steps with local solves}

Let
\begin{equation*}
 J_{\rm loc}=\Bigl\lceil\log\bigl(64N^2(32+6mK+14K)\bigr)\Bigr\rceil,
 \qquad D_{\rm loc}=1+\Bigl\lceil\frac{64J_{\rm loc}}K\Bigr\rceil .
\end{equation*}
Each outer step uses a window of $D_{\rm loc}$ rounds: one round for a
gradient batch of size $b=mK$ and $D_{\rm loc}-1$ rounds for the local
solve.

\begin{theorem}[Pooled gradients and local solves]
\label{comp-local:near}\label{thm:composite-local-resolvent}
If $N\ge D_{\rm loc}$, there is an algorithm of
Definition~\ref{def:ic-main} with output $\widehat x\in\mathcal X_{\rm loc}$ and
\begin{equation}\label{comp-local:near-bound}
 \begin{aligned}
 \E\CompGap(\widehat x)\le16R_0^2\Bigl[{}
 &\frac{32D_{\rm loc}^2L_f}{N^2}
 +\frac{16D_{\rm loc}L_{\max}}{mKN}\\
 &+\frac{64J_{\rm loc}L_H}{KN}
 +\frac{2\sqrt{2D_{\rm loc}}\,\sigma_\star}{R_0\sqrt{mKN}}\Bigr].
 \end{aligned}
\end{equation}
\end{theorem}
\begin{proof}
\emph{Parameters.} Let $T=\lfloor N/D_{\rm loc}\rfloor\ge1$,
$B_{\rm loc}=K(D_{\rm loc}-1)$, so that
$64J_{\rm loc}\le B_{\rm loc}\le64J_{\rm loc}+K$, and
\begin{equation*}
 \vartheta_{\rm loc}=\frac{32J_{\rm loc}}{B_{\rm loc}}\le\frac12,\qquad
 \Gamma_{\rm loc}=8L_f+\frac{16L_{\max}T}b
 +\frac{2\sigma_\star T^{3/2}}{R_0\sqrt b}+\vartheta_{\rm loc}L_HT .
\end{equation*}
If $L_f=L_{\max}=L_H=0$, the operator $G$ is constant and has a root, so
$G\equiv0$ and we return $x_0$. Otherwise $\Gamma_{\rm loc}>0$.

\emph{The method.} Set $z_0=r_1=x_0$ and, for $t=1,\ldots,T$,
\begin{equation*}
 \alpha_t=\frac2{t+1},\quad\eta_t=\frac t{\Gamma_{\rm loc}},\quad
 y_t=(1-\alpha_t)z_{t-1}+\alpha_tr_t,\quad
 z_t=(1-\alpha_t)z_{t-1}+\alpha_tw_t,\quad r_{t+1}=w_t .
\end{equation*}
In the first round of window $t$, every worker makes $K$ independent gradient
calls at $y_t$, and one worker receives their average $c_t$. In the
remaining $D_{\rm loc}-1$ rounds this worker approximately solves the VI on
$\mathcal X_{\rm loc}$ with the operator
\begin{equation}\label{comp-local:inner-operator}
 A_t(w)=\frac{w-r_t}{\eta_t}+c_t+H(w),
\end{equation}
which is $(1/\eta_t)$-strongly monotone and $(1/\eta_t+L_H)$-Lipschitz, by
\[
 q_t=\Bigl\lceil8J_{\rm loc}
       \Bigl(1+\frac{L_Ht}{\Gamma_{\rm loc}}\Bigr)\Bigr\rceil
\]
steps of projected extragradient from $r_t$ with step
$1/(4(1/\eta_t+L_H))$. Its last point is $w_t$, broadcast at the end of the
window. All iterates lie in $\mathcal X_{\rm loc}$.

\emph{Budget.} Since
$L_Ht/\Gamma_{\rm loc}\le L_HT/\Gamma_{\rm loc}\le1/\vartheta_{\rm loc}$
and $16J_{\rm loc}/\vartheta_{\rm loc}=B_{\rm loc}/2$, the local solve uses
$2q_t\le16J_{\rm loc}+16J_{\rm loc}/\vartheta_{\rm loc}+2
\le B_{\rm loc}$ evaluations of $H$,
which fit into the $D_{\rm loc}-1$ rounds. Since $TD_{\rm loc}\le N$, the
whole method uses at most $N$ rounds and at most $K$ calls per worker and
round.

\emph{Gradient noise.} Let $\delta_t=c_t-\nabla f(y_t)$ and let
$\mathcal F_{t-1}$ be the history before the batch at $y_t$. By
Lemma~\ref{lem:composite:variance} with $D_f(x_\star,y_t)$,
\begin{equation}\label{comp-local:variance}
 \E[\delta_t\mid\mathcal F_{t-1}]=0,\qquad
 \E[\|\delta_t\|^2\mid\mathcal F_{t-1}]\le\frac{4L_{\max}}bD_f(x_\star,y_t)
 +\frac{2\sigma_\star^2}b .
\end{equation}

\emph{Accuracy of the local solve.} Let $w_t^\star$ be the solution of the
VI with $A_t$ on $\mathcal X_{\rm loc}$; it depends continuously on $c_t$.
By~\eqref{comp-local:peg}, with $\gamma\mu=1/(4(1+\eta_tL_H))$ and
$\eta_tL_H=L_Ht/\Gamma_{\rm loc}$,
\begin{equation*}
 \|w_t-w_t^\star\|\le4R_0\exp\Bigl(-\frac{q_t}{8(1+L_Ht/\Gamma_{\rm loc})}\Bigr)
 \le4R_0e^{-J_{\rm loc}}.
\end{equation*}
Let $R_t=\max_{v\in\mathcal X_{\rm loc}}\langle A_t(w_t),w_t-v\rangle\ge0$. Since
$g_\star+H(x_\star)=0$,
\[
 A_t(w)=\frac{w-r_t}{\eta_t}+(c_t-g_\star)
 +\bigl(H(w)-H(x_\star)\bigr),
\]
and pathwise
$\|c_t-g_\star\|
\le\|\frac1b\sum_j(\nabla f_{\xi_j}(y_t)-\nabla f_{\xi_j}(x_\star))\|
+\|\frac1b\sum_j\nabla f_{\xi_j}(x_\star)-g_\star\|$, where the first
term is
at most $3L_{\max}R_0$ and the second has expectation at most
$\sigma_\star/\sqrt b$. Hence
\[
 \E\sup_{w\in\mathcal X_{\rm loc}}\|A_t(w)\|
 \le\Bigl(\frac4{\eta_t}+3L_{\max}+3L_H\Bigr)R_0+\frac{\sigma_\star}{\sqrt b},
\]
and~\eqref{comp-local:residual-distance} with $L_A=1/\eta_t+L_H$ gives
\[
 \E R_t\le4R_0e^{-J_{\rm loc}}
 \Bigl[\Bigl(\frac8{\eta_t}+7L_H+3L_{\max}\Bigr)R_0
 +\frac{\sigma_\star}{\sqrt b}\Bigr]
 \le4\Gamma_{\rm loc}R_0^2e^{-J_{\rm loc}}
 \Bigl[8+\frac7{\vartheta_{\rm loc}T}
 +\frac{3b}{16T}+\frac1{2T^{3/2}}\Bigr],
\]
using $1/\eta_t=\Gamma_{\rm loc}/t\le \Gamma_{\rm loc}$,
$L_H\le \Gamma_{\rm loc}/(\vartheta_{\rm loc}T)$,
$L_{\max}\le \Gamma_{\rm loc}b/(16T)$ and
$\sigma_\star\le \Gamma_{\rm loc}R_0\sqrt b/(2T^{3/2})$. Since
$1/\vartheta_{\rm loc}=B_{\rm loc}/(32J_{\rm loc})\le2+K/32$,
the bracket is at most
$32+6b+14K$, and the
choice of $J_{\rm loc}$ gives
\begin{equation*}
 \E R_t\le\frac{\Gamma_{\rm loc}R_0^2}{16N^2}
 \le\frac{\Gamma_{\rm loc}R_0^2}{T^2}.
\end{equation*}

\emph{One outer step.} Let $h_t=w_t-r_t$ and $u\in\mathcal X_{\rm loc}$. As in
Step~1 of the proof of Theorem~\ref{thm:interp-main}, but keeping the
Bregman term from convexity,
\[
 \mathcal M(z_t,u)-(1-\alpha_t)\mathcal M(z_{t-1},u)
 \le\alpha_t\langle\nabla f(y_t)+H(w_t),w_t-u\rangle
 +\frac{L_f\alpha_t^2}2\|h_t\|^2-\alpha_tD_f(u,y_t).
\]
By~\eqref{comp-local:inner-operator},
$\nabla f(y_t)+H(w_t)=A_t(w_t)-h_t/\eta_t-\delta_t$. Here
$\langle A_t(w_t),w_t-u\rangle\le R_t$;
$-\langle h_t,w_t-u\rangle=\frac12(\|r_t-u\|^2-\|w_t-u\|^2-\|h_t\|^2)$ with
$w_t=r_{t+1}$; and
$-\langle\delta_t,w_t-u\rangle=\langle\delta_t,u-r_t\rangle-\langle\delta_t,h_t\rangle$
with $-\langle\delta_t,h_t\rangle\le2\eta_t\|\delta_t\|^2+\|h_t\|^2/(8\eta_t)$.
The coefficient of $\|h_t\|^2$ is
$\frac{\alpha_t}{\eta_t}(-\frac38+\frac{L_f\alpha_t\eta_t}2)\le-\frac{\alpha_t}{4\eta_t}$,
since $\alpha_t\eta_t\le2/\Gamma_{\rm loc}$ and
$L_f\le \Gamma_{\rm loc}/8$. Hence
\begin{align*}
 \mathcal M(z_t,u)\le{}&(1-\alpha_t)\mathcal M(z_{t-1},u)
 +\frac{\alpha_t}{2\eta_t}\bigl(\|r_t-u\|^2-\|r_{t+1}-u\|^2\bigr)
 -\frac{\alpha_t}{4\eta_t}\|h_t\|^2\notag\\
 &+2\alpha_t\eta_t\|\delta_t\|^2
 +\alpha_t\langle\delta_t,u-r_t\rangle
 -\alpha_tD_f(u,y_t)+\alpha_tR_t .
\end{align*}
With $\omega_t=2/(t(t+1))$, dividing by $\omega_t$ and summing as in Step~4
of the proof of Theorem~\ref{thm:interp-main},
\begin{align}\label{comp-local:telescope}
 \frac{T(T+1)}2\mathcal M(z_T,u)\le{}&\frac{\Gamma_{\rm loc}}2\|x_0-u\|^2
 +\sum_{t=1}^T\Bigl[-\frac{\Gamma_{\rm loc}}4\|h_t\|^2
 +\frac{2t^2}{\Gamma_{\rm loc}}\|\delta_t\|^2
 +t\langle\delta_t,u-r_t\rangle-tD_f(u,y_t)+tR_t\Bigr].
\end{align}

\emph{Noise energy.} Let
$\mathsf R_{\rm wt}=\sum_tt\,\E R_t\le \Gamma_{\rm loc}R_0^2$,
$\mathsf B_t=\E D_f(x_\star,y_t)$,
$V_t=\E\|\delta_t\|^2$ and
$\mathsf V_{\rm wt}=\sum_tt^2V_t$. By~\eqref{comp-local:variance},
$V_t\le A_\phi\mathsf B_t+\sigma_{\rm bat}^2$ with
$A_\phi=4L_{\max}/b$ and $\sigma_{\rm bat}^2=2\sigma_\star^2/b$, and the
parameters satisfy $\Gamma_{\rm loc}\ge4A_\phi T$ and
$\Gamma_{\rm loc}\ge\sqrt2\,\sigma_{\rm bat}T^{3/2}/R_0$.
Take $u=x_\star$ in~\eqref{comp-local:telescope}:
the left side is nonnegative and the inner products have zero mean, so
\[
 \sum_tt\mathsf B_t\le\frac{\Gamma_{\rm loc}}2R_0^2
 +\frac2{\Gamma_{\rm loc}}\sum_tt^2V_t+\mathsf R_{\rm wt}
 \le\frac{\Gamma_{\rm loc}}2R_0^2
 +\frac{2A_\phi T}{\Gamma_{\rm loc}}\sum_tt\mathsf B_t
 +\frac{2\sigma_{\rm bat}^2T^3}{\Gamma_{\rm loc}}+\mathsf R_{\rm wt},
\]
and $2A_\phi T/\Gamma_{\rm loc}\le1/2$ gives
$\sum_tt\mathsf B_t\le \Gamma_{\rm loc}R_0^2+2\mathsf R_{\rm wt}
+4\sigma_{\rm bat}^2T^3/\Gamma_{\rm loc}$.
Therefore
\[
 \mathsf V_{\rm wt}\le A_\phi T\sum_tt\mathsf B_t+\sigma_{\rm bat}^2T^3
 \le\frac{\Gamma_{\rm loc}^2R_0^2}4
 +\frac{\Gamma_{\rm loc}\mathsf R_{\rm wt}}2+2\sigma_{\rm bat}^2T^3
 \le\frac74\Gamma_{\rm loc}^2R_0^2 .
\]

\emph{The gap.} For $\|u-x_\star\|\le R_0$ we have $\|x_0-u\|\le2R_0$. Split
$\langle\delta_t,u-r_t\rangle=\langle\delta_t,u-x_\star\rangle
+\langle\delta_t,x_\star-r_t\rangle$, drop the nonpositive terms, take the
supremum over $u$ and then the expectation, and use
$\E\|\sum_tt\delta_t\|\le\sqrt{\mathsf V_{\rm wt}}$:
\[
 \frac{T(T+1)}2\E\CompGap(z_T)
 \le2\Gamma_{\rm loc}R_0^2
 +\frac{2\mathsf V_{\rm wt}}{\Gamma_{\rm loc}}
 +R_0\sqrt{\mathsf V_{\rm wt}}+\mathsf R_{\rm wt}
 \le8\Gamma_{\rm loc}R_0^2 .
\]
Thus $\E\CompGap(z_T)\le16\Gamma_{\rm loc}R_0^2/(T(T+1))$.
Finally $T\ge N/(2D_{\rm loc})$, $T+1\ge N/D_{\rm loc}$ and
$\vartheta_{\rm loc}D_{\rm loc}\le64J_{\rm loc}/K$; substituting $\Gamma_{\rm loc}$
gives~\eqref{comp-local:near-bound}.
\end{proof}

\begin{lemma}[A variant with a uniform variance bound]
\label{comp-local:uniform}
If $N\ge D_{\rm loc}$, there is an algorithm of
Definition~\ref{def:ic-main} with
\begin{equation*}
 \E\CompGap(\widehat x)=O\!\Biggl(R_0^2\Bigl[
 \frac{D_{\rm loc}^2L_f}{N^2}+\frac{J_{\rm loc}L_H}{KN}
 +\sqrt{\frac{D_{\rm loc}L_{\max}L_f}{mKN}}\Bigr]
 +\sigma_\star R_0\sqrt{\frac{D_{\rm loc}}{mKN}}\Biggr).
\end{equation*}
\end{lemma}
\begin{proof}
On $\mathcal X_{\rm loc}$,
$D_f(x_\star,y)\le\frac{L_f}2\|y-x_\star\|^2\le\frac92L_fR_0^2$,
so by~\eqref{comp-local:variance} the batch variance is at most
$s_0^2=(18L_{\max}L_fR_0^2+2\sigma_\star^2)/b$. Use the method of
Theorem~\ref{comp-local:near} with
$\Gamma_{\rm loc}=8L_f+2s_0T^{3/2}/R_0+\vartheta_{\rm loc}L_HT$.
If $L_f=L_H=0$, the operator $G$ is
constant with a root, hence zero, and we return $x_0$. Otherwise the
budget calculation is unchanged. Now
$\|c_t-g_\star\|\le\|c_t-\nabla f(y_t)\|+3L_fR_0$, so
\[
 \E R_t\le4\Gamma_{\rm loc}R_0^2e^{-J_{\rm loc}}
 \Bigl[8+\frac{3L_f}{\Gamma_{\rm loc}}
 +\frac{7L_H}{\Gamma_{\rm loc}}+\frac{s_0}{\Gamma_{\rm loc}R_0}\Bigr]
 \le4\Gamma_{\rm loc}R_0^2e^{-J_{\rm loc}}(32+6mK+14K),
\]
because $L_f/\Gamma_{\rm loc}\le1/8$,
$L_H/\Gamma_{\rm loc}\le1/(\vartheta_{\rm loc}T)$ and
$s_0/(\Gamma_{\rm loc}R_0)\le1/(2T^{3/2})$; hence again
$\mathsf R_{\rm wt}\le \Gamma_{\rm loc}R_0^2$. The uniform variance bound
gives directly
$\mathsf V_{\rm wt}\le s_0^2T^3\le \Gamma_{\rm loc}^2R_0^2/4$, and the last
step of the proof of Theorem~\ref{comp-local:near} gives
$\frac{T(T+1)}2\E\CompGap(z_T)
\le4\Gamma_{\rm loc}R_0^2$. Substituting $\Gamma_{\rm loc}$ and the
bounds on $T$, $D_{\rm loc}$ and $\vartheta_{\rm loc}$ proves the claim.
\end{proof}

\subsection{One pooled gradient and one local solve}
\label{comp-pool:section}

For budgets with $N<D_{\rm loc}$ we use a simpler method.

\begin{lemma}[One pooled gradient]\label{comp-pool:theorem}
For every $m,K\ge1$ and $N\ge2$ there is an algorithm of
Definition~\ref{def:ic-main} with
\begin{equation*}
 \E\CompGap(\widehat x)\le CR_0^2
 \min\Bigl\{L_f+L_H,\ L_f+\frac{L_H}{KN}
 +\sqrt{\frac{L_{\max}L_f}{mKN}}+\frac{\sigma_\star}{R_0\sqrt{mKN}}\Bigr\},
\end{equation*}
where $C$ is a universal constant.
\end{lemma}
\begin{proof}
\emph{Linearization at $x_0$.} Let $g_{\rm init}=\nabla f(x_0)$. For
$z,u\in\mathcal X_{\rm loc}$, smoothness and convexity give
$f(z)\le f(x_0)+\langle g_{\rm init},z-x_0\rangle+2L_fR_0^2$ and
$f(u)\ge f(x_0)+\langle g_{\rm init},u-x_0\rangle$, hence
\begin{equation}\label{comp-pool:transfer}
 f(z)-f(u)\le\langle g_{\rm init},z-u\rangle+2L_fR_0^2 .
\end{equation}

\emph{The initial point and small budgets.} The initial point gives
$(L_f/2+L_H)R_0^2$ (Proposition~\ref{prop:composite-schedules-main}). If
$KN<8$, this is at most $8R_0^2(L_f+L_H/(KN))$.

\emph{Pooled gradient.} Let $KN\ge8$. During the first $n_g=\lfloor N/2\rfloor\ge N/3$
rounds all workers sample gradients at $x_0$; denote the average formed
at the end by $g_{\rm pool}$ and set
$Q_{\rm pool}=mKn_g\ge mKN/3$. By
Lemma~\ref{lem:composite:variance} and $D_f(x_\star,x_0)\le L_fR_0^2/2$,
\begin{equation}\label{comp-pool:mean-error}
 \E\|g_{\rm pool}-g_{\rm init}\|
 \le\sqrt{\frac{2L_{\max}L_fR_0^2+2\sigma_\star^2}{Q_{\rm pool}}}.
\end{equation}

\emph{Local solve.} If $L_H>0$, during the remaining $n_h=\lceil N/2\rceil$
rounds one worker runs projected mirror-prox~\citep{nemirovski2004} for
the operator $H+g_{\rm pool}$ on $\mathcal X_{\rm loc}$ from $x_0$,
with step $1/(2L_H)$ and
$T=\lfloor Kn_h/2\rfloor\ge KN/8$ steps of two evaluations of $H$ each.
Its averaged point $z$ satisfies, for all $u\in\mathcal X_{\rm loc}$,
\[
 \langle H(u)+g_{\rm pool},z-u\rangle
 \le\frac{L_H\|x_0-u\|^2}T\le\frac{32L_HR_0^2}{KN}.
\]
If $L_H=0$, one evaluation gives the constant $H\equiv h$, and $z$ is chosen
to minimize $\langle h+g_{\rm pool},z\rangle$ over $\mathcal X_{\rm loc}$,
so that the same
inequality holds with right side $0$.

\emph{The gap.} For $\|u-x_\star\|\le R_0$, both $u$ and $z$ lie in
$\mathcal X_{\rm loc}$ and $\|z-u\|\le4R_0$. By~\eqref{comp-pool:transfer},
\[
 \mathcal M(z,u)\le\langle H(u)+g_{\rm pool},z-u\rangle
 +\langle g_{\rm init}-g_{\rm pool},z-u\rangle+2L_fR_0^2
 \le\frac{32L_HR_0^2}{KN}
 +4R_0\|g_{\rm init}-g_{\rm pool}\|+2L_fR_0^2 .
\]
The right side does not depend on $u$. Taking the expectation and
using~\eqref{comp-pool:mean-error} with $Q_{\rm pool}\ge mKN/3$,
\[
 \E\CompGap(z)\le2L_fR_0^2+\frac{32L_HR_0^2}{KN}
 +4\sqrt6\,R_0^2\sqrt{\frac{L_{\max}L_f}{mKN}}
 +\frac{4\sqrt6\,\sigma_\star R_0}{\sqrt{mKN}} .
\]
The algorithm returns the better of $z$ and $x_0$.
\end{proof}

\subsection{Proof of the pooled entry}
\label{comp-pool:all-budgets}

\begin{proof}[Proof of Theorem~\ref{thm:composite-local-main}]
The single-worker entry and the cap $L_f+L_H$ are given by
Proposition~\ref{prop:composite-schedules-main}. For the pooled entry let
$Q_{\rm tot}=mKN$. Then $J_{\rm loc}=O(\log(2Q_{\rm tot}))$ and
$D_{\rm loc}=O(1+\log(2Q_{\rm tot})/K)$; in particular,
$D_{\rm loc}=O(\log(2Q_{\rm tot}))$.

If $N\ge D_{\rm loc}$, the better of Theorem~\ref{comp-local:near} and
Lemma~\ref{comp-local:uniform} gives
\[
 R_0^2\Bigl[\frac{L_f}{N^2}+\frac{L_H}{KN}
 +\min\Bigl\{\frac{L_{\max}}{Q_{\rm tot}},
       \sqrt{\frac{L_{\max}L_f}{Q_{\rm tot}}}\Bigr\}\Bigr]
 +\frac{\sigma_\star R_0}{\sqrt{Q_{\rm tot}}}
\]
up to powers of $\log(2Q_{\rm tot})$.

If $N<D_{\rm loc}$, then $N^2=O(\log^2(2Q_{\rm tot}))$; consequently,
the $L_f$ fallback is
$O(\log^2(2Q_{\rm tot})\,L_f/N^2)$. Moreover
$\sqrt{L_{\max}L_f/Q_{\rm tot}}
\le L_f+\min\{L_{\max}/Q_{\rm tot},
\sqrt{L_{\max}L_f/Q_{\rm tot}}\}$: if
$L_{\max}/Q_{\rm tot}\le L_f$, the square root is at most $L_f$,
and otherwise it is
the minimum. Hence Lemma~\ref{comp-pool:theorem} gives the same expression
up to powers of $\log(2Q_{\rm tot})$. This holds for every $N\ge2$ and
without any
condition on $L_{\max}/L_f$.
\end{proof}

\section{Composite VI: lower bounds}
\label{app:composite-lower}

This appendix proves the general lower bound for the composite risk
$\mathfrak R_{\rm comp}(\nu)$ of Section~\ref{sec:acceleration-main}
(Theorem~\ref{thm:composite-lower-main}) and the matching statement of
Corollary~\ref{cor:composite-matching-main}. The bound has five terms, each
given by a separate family of instances:
\begin{itemize}[leftmargin=*,nosep]
\item exact quadratics and exact skew operators, which need many
sequential calls (Section~\ref{interp:affine-depth-input});
\item interpolating quadratics with random masks, the composite analogue
of the SGC lower bound (Section~\ref{interp:communication-lower-subsection});
\item one-dimensional Gaussian instances for the root noise
(Section~\ref{app:composite-root-noise});
\item an instance in which the samples have large smoothness and the
exact component $H$ prevents averaging out the sampling error
(Section~\ref{comp-proj:section}).
\end{itemize}
In all of them $x_0=0$, the instance distribution does not depend on the
algorithm, and queries and outputs are arbitrary.

\subsection{Exact quadratics and skew operators}
\label{interp:affine-depth-input}

Let $m,B_0\ge1$ be integers and $\lambda_0,R_0>0$.
Corollary~\ref{cor:dimref:depth}, applied with $L=\lambda_0$, gives a fixed
distribution of symmetric invertible matrices $M_{\rm quad}$ with
$\|M_{\rm quad}\|_{\rm op}=\lambda_0$ and vectors $b_{\rm quad}$ with
$\|M_{\rm quad}^{-1}b_{\rm quad}\|=R_0$ such that
every algorithm with $B_0$ batches of at most $m$ simultaneous queries to
the skew operator
$G_{\rm quad}(u,v)=(M_{\rm quad}v+b_{\rm quad},-M_{\rm quad}u)$ satisfies
\begin{equation}
 R_0\,\E\sqrt{\|M_{\rm quad}\widehat v+b_{\rm quad}\|^2
       +\|M_{\rm quad}\widehat u\|^2}
 \ge\frac{\lambda_0R_0^2}{6B_0}.
 \label{interp:comm-affine-input}
\end{equation}
Indeed, for a skew-symmetric affine operator the root-centered Minty gap
equals $R_0\|G_{\rm quad}(\widehat u,\widehat v)\|$. The dimension is
$d_{\rm ref}(m,B_0)\le1024mB_0^2$.

\begin{lemma}[Exact quadratics]\label{interp:quadratic-depth}
For integers $m,n_{\rm dep}\ge1$ and $\lambda,R_0>0$ there is a fixed distribution
of convex $\lambda$-smooth quadratics $f\ge0$, each with a unique
minimizer of norm $R_0$, such that every algorithm with $n_{\rm dep}$ batches of
at most $m$ simultaneous exact gradient (and value) queries satisfies
\begin{equation*}
 \E\bigl[f(\widehat v)-\min f\bigr]\ge
 \frac{\lambda R_0^2}{288\,n_{\rm dep}^2}.
\end{equation*}
The dimension is
$d_{\rm ref}(m,2n_{\rm dep})/2\le2048m n_{\rm dep}^2$.
\end{lemma}
\begin{proof}
Take the distribution of~\eqref{interp:comm-affine-input} with
$B_0=2n_{\rm dep}$
and $\lambda_0=\sqrt\lambda$, and let
$f(v)=\frac12\|M_{\rm quad}v+b_{\rm quad}\|^2$. Its Hessian
$M_{\rm quad}^2$ has norm $\lambda$, and its minimizer
$-M_{\rm quad}^{-1}b_{\rm quad}$ has norm $R_0$. A
batch of gradient queries at points $v_i$ is answered by two batches of
queries to $G_{\rm quad}$: the query $G_{\rm quad}(0,v_i)$ returns
$r_i=M_{\rm quad}v_i+b_{\rm quad}$, hence
$f(v_i)=\|r_i\|^2/2$, and the query $G_{\rm quad}(r_i,0)$ returns
$-M_{\rm quad}r_i=-\nabla f(v_i)$.
After $2n_{\rm dep}$ batches, output $(0,\widehat v)$;
\eqref{interp:comm-affine-input} gives
$\E\|M_{\rm quad}\widehat v+b_{\rm quad}\|
\ge\sqrt\lambda R_0/(12n_{\rm dep})$, and Jensen's
inequality gives
$\E[f(\widehat v)-\min f]
=\frac12\E\|M_{\rm quad}\widehat v+b_{\rm quad}\|^2
\ge\lambda R_0^2/(288n_{\rm dep}^2)$.
\end{proof}

\subsection{Interpolating quadratics with random masks}
\label{interp:communication-lower-subsection}

\begin{theorem}[Masked quadratics]\label{thm:comm-main}\label{interp:communication-lower}
Let $0<\lambda\le L_f\le L_{\max}$, $p=\lambda/L_{\max}$, and
\begin{equation*}
 n_{\rm dep}=\min\bigl\{KN,\ \max\{1,\lceil2N\beta_p\rceil\}\bigr\}.
\end{equation*}
There is a fixed distribution of instances of
Assumption~\ref{ass:interp-main} with $H=0$ and $\sigma_\star=0$, whose
mean $f$ is a convex $\lambda$-smooth quadratic and whose samples are
nonnegative, convex and $L_{\max}$-smooth, such that every algorithm
satisfies
\begin{equation}
 \E\CompGap(\widehat x)
 \ge\frac{\lambda R_0^2}{576\,n_{\rm dep}^2}
 \ge\frac{\lambda R_0^2}{2304\,(1+N\beta_p)^2}
 \label{interp:comm-clock-lower}
 \ge\frac{\lambda R_0^2}{36864\,[1+N\BinDepth(m,K,p)]^2},
\end{equation}
where $\BinDepth(m,K,p)$ is defined in~\eqref{eq:clock:explicit-app}. The dimension is at
most $2048m(KN)^2$, and the bound holds even if the masks and the sample
values are revealed.
\end{theorem}
\begin{proof}
Draw $f$ from Lemma~\ref{interp:quadratic-depth} with depth $n_{\rm dep}$,
and let
every call use the sample
\begin{equation*}
 f_{\rm samp}(x)=\frac{B_{\rm mask}}p f(x),\qquad
 B_{\rm mask}\sim\operatorname{Bernoulli}(p)\ \text{independent}.
\end{equation*}
Each sample is nonnegative, convex and $\lambda/p=L_{\max}$-smooth, its
mean is $f$, and all samples vanish at the minimizer $x_\star$ of $f$. This
oracle is masked with probability $p$ in the sense of
Section~\ref{sec:relative-minimax}, and calls to $H=0$ give no information.
By Lemma~\ref{lem:masked:simulation}(iii) with
$\ell_{\rm dep}=\lambda R_0^2/(288n_{\rm dep}^2)$, every algorithm has expected
loss at least $\lambda R_0^2/(576n_{\rm dep}^2)$ for the loss
$f(\widehat x)-\min f$, which equals
$\CompGap(\widehat x)$ because $H=0$ and $x_\star$ is an admissible
comparator. Finally $n_{\rm dep}\le2(1+N\beta_p)$ and
$\beta_p\le4\BinDepth(m,K,p)$
by Theorem~\ref{thm:clock:asymptotic}.
\end{proof}

\subsection{Root noise}
\label{app:composite-root-noise}
\label{comp-local:noise-subsection}

\begin{lemma}[Root noise]
\label{prop:composite-full-noise-lower}
\label{prop:composite-root-noise}\label{comp-local:noise-lower}
For $\nu\ge0$,
\begin{equation*}
 \mathfrak R_{\rm comp}(\nu)\ge\frac7{2048}
 \min\Bigl\{(L_f+L_H)R_0^2,\ \frac{\nu R_0}{\sqrt{mKN}}\Bigr\}.
\end{equation*}
The instances are one-dimensional and have root noise exactly $\nu^2$.
\end{lemma}
\begin{proof}
Let $\nu>0$, $Q_{\rm tot}=mKN$, and
$\theta_{\rm test}\in\{-R_0/2,R_0/2\}$. We use two
families, each only when its budget is positive.
\begin{itemize}[leftmargin=*,nosep]
\item[(a)] If $L_f>0$:
$\lambda=\min\{L_f,\nu/(4R_0\sqrt{Q_{\rm tot}})\}$,
draw $\zeta\sim\mathcal N(\lambda\theta_{\rm test},\nu^2)$, and set
$f_\zeta(x)=\frac\lambda2(x-\zeta/\lambda)^2$ and $H=0$. The samples are
$\lambda$-smooth with $\lambda\le L_f\le L_{\max}$, and
$f(x)=\frac\lambda2(x-\theta_{\rm test})^2+{\rm const}$.
\item[(b)] If $L_H>0$:
$\lambda=\min\{L_H,\nu/(4R_0\sqrt{Q_{\rm tot}})\}$,
draw $\zeta\sim\mathcal N(\lambda\theta_{\rm test},\nu^2)$, and set
$f_\zeta(x)=-\zeta x$ and $H(x)=\lambda x$. The samples are affine, so $L_f$ and
$L_{\max}$ may be zero, $H$ is $\lambda$-Lipschitz with $\lambda\le L_H$,
and $f(x)=-\lambda\theta_{\rm test}x$.
\end{itemize}
In both families $G(x)=\lambda(x-\theta_{\rm test})$, the root is $\theta_{\rm test}$,
and
$\nabla f_\zeta(\theta_{\rm test})+H(\theta_{\rm test})
=\lambda\theta_{\rm test}-\zeta$ has variance $\nu^2$.
Write $L_{\rm active}$ for $L_f$ in family (a) and for $L_H$ in family (b).

Every response is a known function of the query and of $\zeta$, so the
algorithm learns at most $Q_{\rm tot}$ independent observations of
$\zeta\sim\mathcal N(\lambda\theta_{\rm test},\nu^2)$. Define the mean
separation $\mu_{\rm obs}=\lambda R_0/2$. Then
$\sqrt{Q_{\rm tot}}\,\mu_{\rm obs}/\nu\le1/8$, so by
Lemma~\ref{lem:gaussian-testing} the sign of $\widehat x$ differs from
the sign of $\theta_{\rm test}$ with average probability at least $7/16$;
then $|\widehat x-\theta_{\rm test}|\ge R_0/2$. In (a),
$\CompGap(\widehat x)=\frac\lambda2(\widehat x-\theta_{\rm test})^2
\ge\lambda R_0^2/8$.
In (b), reparametrize the comparator as
$w=u_{\rm cmp}-\theta_{\rm test}$:
\[
 \CompGap(\widehat x)
 =\sup_{|w|\le R_0}\lambda
   \bigl[w(\widehat x-\theta_{\rm test})-w^2\bigr]
 \ge\lambda\Bigl[\frac{R_0}4|\widehat x-\theta_{\rm test}|
   -\frac{R_0^2}{16}\Bigr]
 \ge\frac{\lambda R_0^2}{16},
\]
taking
$w=\frac{R_0}4\operatorname{sign}(\widehat x-\theta_{\rm test})$.
Hence one sign of $\theta_{\rm test}$ gives
$\E\CompGap(\widehat x)\ge\frac7{16}\cdot\frac{\lambda R_0^2}{16}
\ge\frac7{1024}\min\{L_{\rm active}R_0^2,
\nu R_0/\sqrt{Q_{\rm tot}}\}$ for each family that is available; a family
with $L_{\rm active}=0$ would only give the bound $0$. Finally, with
$R_{\rm noise}=\nu R_0/\sqrt{Q_{\rm tot}}$,
\[
 \max_{L_{\rm active}\in\{L_f,L_H\}}
 \min\{L_{\rm active}R_0^2,R_{\rm noise}\}
 \ge\frac12\min\{(L_f+L_H)R_0^2,R_{\rm noise}\}.
\]
\end{proof}

\subsection{Large sample smoothness and an exact component}
\label{comp-proj:section}

\begin{theorem}[Sampling lower bound with an exact component]
\label{comp-proj:lower}
Let $Q_{\rm tot}\ge1$ be an integer, $L_f>0$, $L_H\ge0$ and
$L_{\max}>Q_{\rm tot}L_f$. For every algorithm that makes at most
$Q_{\rm tot}$ stochastic gradient calls, with
arbitrary queries, sequential access and free access to $H$, there is
an instance of Assumption~\ref{ass:interp-main} with $\sigma_\star=0$ in
dimension $d=\lceil8192L_{\max}/L_f\rceil$ (and in every larger dimension)
such that
\begin{equation*}
 \E\CompGap(\widehat x)
 \ge\frac{3}{512\sqrt{8193}}R_0^2
 \min\Bigl\{L_H,\sqrt{\frac{L_{\max}L_f}{Q_{\rm tot}}}\Bigr\}.
\end{equation*}
\end{theorem}
\begin{proof}
If $L_H=0$ there is nothing to prove. Write
$L_{\rm samp}=L_{\max}$ and let
\[
 d=\Bigl\lceil\frac{8192L_{\rm samp}}{L_f}\Bigr\rceil,\qquad
 \lambda=\frac{4L_{\rm samp}}d\le L_f,\qquad
 \mu_{\rm ex}=\min\Bigl\{L_H,
 \frac{L_{\rm samp}}{32\sqrt{dQ_{\rm tot}}}\Bigr\}>0 .
\]

\emph{The instance.} Let $x_\star$ be uniform on the sphere of radius $R_0$
in $\mathbb R^d$. Every stochastic call draws, independently, the orthogonal
projector $\Pi_{\rm rnd}$ onto a uniformly distributed $4$-dimensional
subspace, and
returns the gradient of
\begin{equation*}
 f_{\Pi_{\rm rnd}}(x)
 =\frac{L_{\rm samp}}2\|\Pi_{\rm rnd}(x-x_\star)\|^2
 -\mu_{\rm ex}\langle x_\star,x\rangle,
 \qquad\text{while}\qquad H(x)=\mu_{\rm ex}x .
\end{equation*}
The samples are convex and $L_{\rm samp}$-smooth. Since
$\E \Pi_{\rm rnd}=\frac4dI$, the mean is
$f(x)=\frac\lambda2\|x-x_\star\|^2
-\mu_{\rm ex}\langle x_\star,x\rangle$,
which is $\lambda$-smooth, and
$G(x)=(\lambda+\mu_{\rm ex})(x-x_\star)$. Every sample satisfies
$\nabla f_{\Pi_{\rm rnd}}(x_\star)+H(x_\star)=0$, so $\sigma_\star=0$.
The map $H$ is known,
so calls to it carry no information.

\begin{revisionblock}
\begin{lemma}[Spherical posterior under equivariant observations]
\label{comp-proj:posterior}
Let $x_\star$ be uniform on the sphere of radius $R_0$ in $\mathbb R^d$.
Fix $q\in\{1,\ldots,d\}$ and deterministic scalars $a,b\in\mathbb R$.
Starting with $S=\{0\}$, repeat the following adaptive experiment. Given
the past, choose $x$, replace $S$ by $\operatorname{span}(S,x)$, and reveal
$\Pi_Sx_\star$. Then draw, independently of the past, a uniformly random
rank-$q$ orthogonal projector $P$, observe
$y=aP(x-x_\star)-b x_\star$, replace $S$ by
$\operatorname{span}(S,y)$, and again reveal $\Pi_Sx_\star$.
After either reveal, conditionally on the complete augmented transcript,
$(I-\Pi_S)x_\star$ is uniform on the sphere in $S^\perp$ of radius
$(R_0^2-\|\Pi_Sx_\star\|^2)^{1/2}$. The conclusion is unchanged if the
adaptive choices use independent internal randomness.
\end{lemma}
\begin{proof}
Condition on the internal random seed and argue by induction. The claim
is true initially. At a query step, $x$ is fixed by the past. Writing
$v=(I-\Pi_S)x$, revealing the new projection either changes nothing
when $v=0$ or conditions a uniform point of the sphere in $S^\perp$ on
its coordinate along $v$. The remaining component is therefore uniform
on the corresponding sphere in $\operatorname{span}(S,x)^\perp$.

After this enlargement, $x\in S$. Before the response is observed, the
conditional joint law of $(x_\star,P)$ is invariant under
$(x_\star,P)\mapsto(Qx_\star,QPQ^\top)$ for every orthogonal $Q$ fixing
$S$ pointwise, and the response map is equivariant:
$y(Qx_\star,QPQ^\top)=Qy(x_\star,P)$. Put
$S'=\operatorname{span}(S,y)$. After conditioning on $y$ and on the
revealed vector $\Pi_{S'}x_\star$, every orthogonal map fixing $S'$ pointwise
leaves all conditioned data unchanged. Hence the conditional law of
$(I-\Pi_{S'})x_\star$ is invariant under the full orthogonal group of
$S'^\perp$. Its norm is fixed by $\|x_\star\|=R_0$, so uniqueness of the
orthogonally invariant probability measure on a sphere makes this law
uniform. This proves the induction; averaging over the seed gives the
last assertion.
\end{proof}
\end{revisionblock}

\emph{A more informative experiment.} Let $\mathcal S_{\rm obs}$ be the
span of all stochastic queries and responses observed so far, and reveal
$\Pi_{\mathcal S_{\rm obs}}x_\star$ after every change of
$\mathcal S_{\rm obs}$: once when a query $x$ is chosen, and once after
its response
$y=L_{\rm samp}\Pi_{\rm rnd}(x-x_\star)-\mu_{\rm ex}x_\star$.
This can only help the algorithm,
and $\dim\mathcal S_{\rm obs}$ grows by at most two per call, so
$\dim\mathcal S_{\rm obs}\le2Q_{\rm tot}\le d/4096$ and
$d_\perp:=\dim\mathcal S_{\rm obs}^\perp\ge d/2$ throughout.
\revision{By Lemma~\ref{comp-proj:posterior}, conditionally on the
augmented history the unrevealed component}
$\Pi_{\mathcal S_{\rm obs}^\perp}x_\star$ is uniform
on the sphere of radius
$r=(R_0^2-\|\Pi_{\mathcal S_{\rm obs}}x_\star\|^2)^{1/2}$ in
$\mathcal S_{\rm obs}^\perp$.

\emph{Information revealed by a query.} The query direction is fixed
given the past, and a uniform point on a sphere of radius $r$ in
dimension $d_\perp$ has second moment $r^2/d_\perp$ in every direction. Hence adding
$x$ to $\mathcal S_{\rm obs}$ reveals in expectation at most
$r^2/d_\perp\le2r^2/d$.

\emph{Information revealed by a response.} After adding $x$, let
$\mathcal V_{\rm hid}=\mathcal S_{\rm obs}^\perp$,
$\Pi_{\mathcal V_{\rm hid}}x_\star=ru_{\rm hid}$ with
$\|u_{\rm hid}\|=1$, and $w=x-x_\star$, so that
$\Pi_{\mathcal V_{\rm hid}}w=-ru_{\rm hid}$ and $\|w\|\ge r$.
A uniform rank-four projector satisfies
\begin{equation*}
 \Pi_{\rm rnd}w=t\,w+\|w\|\sqrt{t(1-t)}\,e_{\rm tan},
\end{equation*}
with independent $t\sim\operatorname{Beta}(2,(d-4)/2)$ and
$e_{\rm tan}$ uniform on
the unit sphere of $w^\perp$. (Rotating $w/\|w\|$ to a coordinate axis,
$t=\|Pw\|^2/\|w\|^2$ is the squared norm of the projection of a uniform
unit vector onto a fixed $4$-dimensional subspace, a ratio of
$\chi^2_4$ and $\chi^2_4+\chi^2_{d-4}$; the identities
$\langle Pw,w\rangle=\|Pw\|^2=t\|w\|^2$ determine the component along $w$,
and rotations fixing $w$ make $e_{\rm tan}$ uniform and independent of
$t$.) Let
$\mathcal W_{\rm tan}=\mathcal V_{\rm hid}\cap u_{\rm hid}^\perp
\subseteq w^\perp$, of dimension $d_\perp-1$,
$\zeta_{\rm tan}=\|\Pi_{\mathcal W_{\rm tan}}e_{\rm tan}\|^2$ and
$a_{\rm resp}=\Pi_{\mathcal V_{\rm hid}}\Pi_{\rm rnd}w$. Then
\begin{equation}\label{comp-proj:geometry}
 \|a_{\rm resp}\|^2\ge t(1-t)\|w\|^2\zeta_{\rm tan},\qquad
 \langle u_{\rm hid},a_{\rm resp}\rangle
 =-tr+\|w\|\sqrt{t(1-t)}\,\langle u_{\rm hid},e_{\rm tan}\rangle .
\end{equation}
For the beta law with parameters $(\alpha,\beta)=(2,(d-4)/2)$,
$\E\frac t{1-t}=\frac\alpha{\beta-1}$, $\E\frac1t=\frac{\alpha+\beta-1}{\alpha-1}$ and
$\E\frac1{1-t}=\frac{\alpha+\beta-1}{\beta-1}$; writing
$e_{\rm tan}=g/\|g\|$ with a standard Gaussian $g$ in $w^\perp$ and using
that $\|\Pi_{\mathcal W_{\rm tan}}g\|^2\sim\chi^2_{d_\perp-1}$
is independent of the remaining coordinates, one obtains
\begin{equation}\label{comp-proj:moments}
 \begin{aligned}
 \E\frac t{1-t}&=\frac4{d-6},
 &\E\frac1{t(1-t)}&=(d-2)\Bigl(\frac12+\frac1{d-6}\Bigr),\\
 \E\frac1{\zeta_{\rm tan}}&=\frac{d-3}{d_\perp-3},
 &\E\frac{\langle u_{\rm hid},e_{\rm tan}\rangle^2}{\zeta_{\rm tan}}
 &\le\frac1{d_\perp-3}.
 \end{aligned}
\end{equation}
By~\eqref{comp-proj:geometry}, \eqref{comp-proj:moments},
$(x+y)^2\le2x^2+2y^2$, $r\le\|w\|$, $d_\perp\ge d/2$ and $d\ge12$,
\[
 \E\frac{\langle u_{\rm hid},a_{\rm resp}\rangle^2}{\|a_{\rm resp}\|^2}
 \le\frac8{d-6}\cdot\frac{d-3}{d_\perp-3}+\frac2{d_\perp-3}\le\frac{72}d,\qquad
 \E\frac1{\|a_{\rm resp}\|^2}
 \le\frac{d-2}{\|w\|^2}\Bigl(\frac12+\frac1{d-6}\Bigr)
 \frac{d-3}{d_\perp-3}
 \le\frac{4d}{\|w\|^2}.
\]
The new direction of the response in $\mathcal V_{\rm hid}$ is
$L_{\rm samp}a_{\rm resp}-\mu_{\rm ex}ru_{\rm hid}$. For every nonzero
$a_{\rm resp}$, unit $u_{\rm hid}$ and $\beta\ge0$,
\begin{equation*}
 \frac{\langle u_{\rm hid},a_{\rm resp}-\beta u_{\rm hid}\rangle^2}
 {\|a_{\rm resp}-\beta u_{\rm hid}\|^2}
 \le8\Bigl(
 \frac{\langle u_{\rm hid},a_{\rm resp}\rangle^2}{\|a_{\rm resp}\|^2}
 +\frac{\beta^2}{\|a_{\rm resp}\|^2}\Bigr)
\end{equation*}
(if $\beta\le\|a_{\rm resp}\|/2$, then
$\|a_{\rm resp}-\beta u_{\rm hid}\|\ge\|a_{\rm resp}\|/2$; otherwise the right side
exceeds $2$ and the left side is at most $1$). With
$\beta=\mu_{\rm ex}r/L_{\rm samp}$, the
response reveals in expectation at most
$r^2\bigl(\frac{576}d+\frac{32d\mu_{\rm ex}^2}{L_{\rm samp}^2}\bigr)$.

\emph{Remaining uncertainty.} Each call therefore reveals in expectation
at most
$R_0^2(578/d+32d\mu_{\rm ex}^2/L_{\rm samp}^2)$, and since
$d\ge8192Q_{\rm tot}$ and
$32d\mu_{\rm ex}^2Q_{\rm tot}\le L_{\rm samp}^2/32$,
\[
 Q_{\rm tot}\Bigl(\frac{578}d
 +\frac{32d\mu_{\rm ex}^2}{L_{\rm samp}^2}\Bigr)
 \le\frac{578}{8192}+\frac1{32}<\frac14 .
\]
Hence the final unrevealed radius $r_{\rm rem}$ satisfies
$\E r_{\rm rem}^2\ge3R_0^2/4$. Given the revealed information, the
\revision{conditional mean of the unrevealed component is zero by
Lemma~\ref{comp-proj:posterior}, so the posterior mean of $x_\star$ is
$\Pi_{\mathcal S_{\rm obs}}x_\star$. Hence}
every output satisfies
$\E\|\widehat x-x_\star\|^2\ge\E r_{\rm rem}^2\ge3R_0^2/4$.

\emph{The gap.} With $e_{\rm out}=x-x_\star$ and
$u-x_\star$ in place of $u$,
\[
 \CompGap(x)=\frac\lambda2\|e_{\rm out}\|^2
 +\sup_{\|u'\|\le R_0}\bigl\{\mu_{\rm ex}\langle u',e_{\rm out}\rangle
 -(\lambda/2+\mu_{\rm ex})\|u'\|^2\bigr\},
\]
which is a nondecreasing function of $\|e_{\rm out}\|$. Projecting the output onto
$B(0,R_0)$ does not increase $\|x-x_\star\|$, so we may assume
$\|x\|\le R_0$ and $\|e_{\rm out}\|\le2R_0$. Then
$u'=e_{\rm out}/2$ gives
$\CompGap(x)\ge(3\lambda/8+\mu_{\rm ex}/4)\|e_{\rm out}\|^2
\ge \mu_{\rm ex}\|e_{\rm out}\|^2/4$, and
$\E\CompGap(\widehat x)\ge3\mu_{\rm ex}R_0^2/16$; some fixed
$x_\star$ attains this. Finally $d\le8193L_{\rm samp}/L_f$ gives
$\mu_{\rm ex}\ge\min\{L_H,
\sqrt{L_{\rm samp}L_f/Q_{\rm tot}}/(32\sqrt{8193})\}$.
\end{proof}

\subsection{The general lower bound}

\begin{proof}[Proof of Theorem~\ref{thm:composite-lower-main}]
Each term is a lower bound for a family of instances in the class
$\mathfrak I_d(\nu)$; padding all families with inactive coordinates to a
common dimension and taking the largest bound gives the sum up to a
universal factor.

A term whose smoothness budget is zero vanishes and needs no instance;
below each construction is used only when its budgets are positive.

\emph{First two terms.} If $L_f>0$, exact samples $f_\xi=f$ from
Lemma~\ref{interp:quadratic-depth} with $n_{\rm dep}=KN$ and $\lambda=L_f$ give
$L_fR_0^2/(288(KN)^2)$; any algorithm of Definition~\ref{def:ic-main} is
covered, since it makes at most $KN$ batches of $m$ calls. If $L_H>0$, the
skew operator of~\eqref{interp:comm-affine-input} with $f_\xi=0$,
$H=G_{\rm quad}$,
$\lambda_0=L_H$ and $B_0=KN$ gives $L_HR_0^2/(6KN)$.

\emph{Third and fourth terms.} If $L_f=0$, both terms are zero. Let
$L_f>0$; then $L_{\max}\ge L_f>0$, so $p=\lambda/L_{\max}$ below is well
defined. For the third term take
$\lambda=\min\{L_f,L_{\max}/(mKN)\}$ in Theorem~\ref{thm:comm-main}.
Lemma~\ref{lem:clock:elementary} gives $N\beta_p\le NmKp\le1$, so
\begin{equation*}
 \mathfrak R_{\rm comp}(0)\ge\frac{R_0^2}{9216}
 \min\Bigl\{L_f,\frac{L_{\max}}{mKN}\Bigr\}.
\end{equation*}

For the fourth term let $\ell_m=1+N\log m$ and
$\lambda=\min\{L_f,L_{\max}\ell_m/((e-1)KN)\}$.
Lemma~\ref{lem:clock:elementary}
gives $\beta_p\le(e-1)K\lambda/L_{\max}+\log m
\le \ell_m/N+\log m$, hence
$1+N\beta_p\le2\ell_m$, and~\eqref{interp:comm-clock-lower} gives
\begin{equation*}
 \mathfrak R_{\rm comp}(0)\ge\frac{R_0^2}{9216}
 \min\Bigl\{\frac{L_f}{\ell_m^2},
       \frac{L_{\max}}{(e-1)KN\ell_m}\Bigr\}.
\end{equation*}

\emph{Last term.} This is Lemma~\ref{prop:composite-full-noise-lower}.
\end{proof}

\subsection{Proof of the matching regime}
\label{comp-match:proof}
\begin{proof}[Proof of Corollary~\ref{cor:composite-matching-main}]
The upper bound follows from Theorem~\ref{thm:composite-local-main} with
$\sigma_\star\le \nu$. Let $Q_{\rm tot}=mKN$,
$\rho_{\rm noise}=\nu/(R_0\sqrt{Q_{\rm tot}})$, and divide all
bounds by $R_0^2$. We use
$\min\{r_{\rm cap},r_1+r_2\}\le r_1+\min\{r_{\rm cap},r_2\}$ for
$r_{\rm cap},r_1,r_2\ge0$.

\emph{Case $L_{\max}\le Q_{\rm tot}L_f$.} The sampling term in
Corollary~\ref{cor:composite-matching-main} is
$L_{\max}/Q_{\rm tot}$. Let
$\ell_m=1+N\log m\le N(1+\log m)$. Since $L_{\max}\ge KL_f$,
\[
 \min\Bigl\{\frac{L_f}{\ell_m^2},
       \frac{L_{\max}}{KN\ell_m}\Bigr\}
 \ge\frac{L_f}{N^2(1+\log m)^2},
\]
so Theorem~\ref{thm:composite-lower-main} bounds the risk below by a
constant times the maximum of $L_H/(KN)$,
$L_{\max}/Q_{\rm tot}$,
$L_f/(N^2(1+\log m)^2)$ and
$\min\{L_f+L_H,\rho_{\rm noise}\}$. The rate in
Corollary~\ref{cor:composite-matching-main} is at most
$L_f/N^2+L_H/(KN)+L_{\max}/Q_{\rm tot}
+\min\{L_f+L_H,\rho_{\rm noise}\}$, which is at most
$4(1+\log m)^2$ times this maximum.

\emph{Case $L_{\max}>Q_{\rm tot}L_f$.} Let
$\rho_{\rm samp}=\sqrt{L_{\max}L_f/Q_{\rm tot}}$; then
$\rho_{\rm samp}>L_f$ and
the sampling term is $\rho_{\rm samp}$.
Theorems~\ref{thm:composite-lower-main}
and~\ref{comp-proj:lower} bound the risk below by a constant times
$\rho_{\rm lb}=\max\{L_f,L_H/(KN),\min\{L_H,\rho_{\rm samp}\},
\min\{L_f+L_H,\rho_{\rm noise}\}\}$. If $L_H\le \rho_{\rm samp}$,
the rate is at most $L_f+L_H\le2\rho_{\rm lb}$. If
$L_H>\rho_{\rm samp}$, then $\rho_{\rm samp}\le \rho_{\rm lb}$ and the
rate is at most
$L_f/N^2+L_H/(KN)+\rho_{\rm samp}
+\min\{L_f+L_H,\rho_{\rm noise}\}\le4\rho_{\rm lb}$.

All families are admissible for the root-noise bound $\nu$, and padding
puts them into the common dimension stated in the corollary.
\end{proof}

\section{Prior bounds in Table~\ref{tab:results}}
\label{app:comparisons}\label{app:resource-comparison}
Every algorithm may return $z_0$, which gives the cap omitted in each row of
Table~\ref{tab:results}.

\emph{Single-machine bounds.} \citet{juditsky2011} and \citet{chen2017amp}
analyze one machine that makes $T$ iterations with an unbiased oracle of
variance at most $\sigma^2$, and Table~\ref{tab:results} states their bounds in
this form. For stochastic mirror-prox on a compact set of diameter of order $D$,
\citet{juditsky2011} bound the expected gap by a constant multiple of
$LD^2/T+\sigma D/\sqrt T$. For stochastic accelerated mirror-prox with exact
evaluations of $H$, \citet{chen2017amp} bound the composite gap by a constant
multiple of $L_fR_0^2/T^2+L_HR_0^2/T+\sigma R_0/\sqrt T$.

\emph{Running a single-machine method in the intermittent-communication
model.} The comparisons after Theorem~\ref{vi:main-monotone} and in
Section~\ref{sec:acceleration-main} run these methods in our model in the two
standard ways. One worker can run the method with batch $b=1$, keeping its
state across round boundaries, so that $T$ is of order $KN$. Alternatively,
all workers can make their $K$ calls of a round at a common query point, and
the responses are averaged at communication, so that each oracle call of the
method is replaced by the average of $b$ of order $mK$ fresh responses. A step
of an extragradient-type method then needs two rounds, because its second
query depends on the averaged first response, so $T=\lfloor N/2\rfloor$.
Averaging $b$ fresh responses at the same point gives an unbiased response
whose conditional variance is that of one response divided by $b$, because
the cross terms vanish by independence and unbiasedness. Hence a guarantee
that holds for every unbiased oracle with variance at most $\sigma^2$ holds for
the averaged oracle with $\sigma^2/b$. For accelerated mirror-prox with exact
$H$ and $K\ge3$, the schedules $T=\lfloor KN/3\rfloor$, $b=1$ and $T=N$,
$b=m(K-2)$ of Proposition~\ref{prop:composite-schedules-main} also count the
two evaluations of $H$ per step.

\emph{Relative and root noise.} Under Assumption~\ref{ass:mixed}, run
stochastic mirror-prox with projection onto $B(z_0,2D)$. On this ball
$\|F(x)\|\le L\|x-z_\star\|\le3LD$, so
\[
 \E\bigl[\|g(x,\xi)-F(x)\|^2\mid\mathcal F\bigr]
 \le9(\rho-1)L^2D^2+\sigma^2 ,
\]
and the gap over $B(z_0,2D)\supseteq C$ bounds the Minty gap over $C$. Under
Assumption~\ref{ass:interp-main}, Lemma~\ref{lem:samples:variance} and
$\|x-x_\star\|\le3R_0$ on $B(x_0,2R_0)$ give
$\operatorname{Var}(\nabla f_\xi(x))\le18L_{\max}L_fR_0^2+2\sigma_\star^2$
there. These are the variance bounds used in the two comparisons.

\emph{Distributed bounds.} \citet{wang2026faster} state their Theorems~1
and~2 in the same intermittent-communication model, with $m$ workers, $K$
local oracle calls per round, and $N$ rounds, and Table~\ref{tab:results}
reproduces them. The distributed minibatch extragradient of
\citet[Theorem~4.1]{beznosikov2024lower} averages the responses of the
machines at every communication, which is the common-query schedule, and the
table reproduces their bound. Their bound is stated for
strongly-convex--strongly-concave saddle problems, a subclass of the strongly
monotone VIs to which Theorem~\ref{vi:main-mixed-accuracy} applies.

\end{document}